\documentclass[10pt,final]{book}
\usepackage[paperheight=24.6cm,paperwidth=17.6cm, hmargin={2.8cm,1.8cm}, vmargin={2.3cm,1.8cm}, headheight=14pt]{geometry}
\usepackage{color}
\usepackage{amsthm,amsmath}
\usepackage{amssymb}
\usepackage{stmaryrd}
\usepackage[all]{xy}
\usepackage{epsf}
\usepackage{enumerate,enumitem}
\usepackage{fancyhdr}
\usepackage{amsmath,amscd}
\usepackage{fontenc}
\usepackage{fontenc,graphicx,multirow}

\newtheorem{theorem}{Theorem}[section]
\newtheorem{proposition}[theorem]{Proposition}
\newtheorem{lemma}[theorem]{Lemma}
\newtheorem{corollary}[theorem]{Corollary}

\newtheorem{definition}[theorem]{Definition}

\newtheorem{example}[theorem]{Example}
\newtheorem{remark}[theorem]{Remark}
\newtheorem*{theoremone}{Theorem 1}
\newtheorem*{theoremtwo}{Theorem 2}
\newtheorem*{theoremthree}{Theorem 3}
\newtheorem*{theoremfour}{Theorem 4}
\newtheorem*{theoremfive}{Theorem 5}
\newtheorem*{corollarytwo}{Corollary 2}

\newcommand{\Rr}{\mathbb R}

\newcommand{\set}[1]{\left\{#1\right\}}

\newcommand{\eps}{\epsilon}
\newcommand{\To}{\longrightarrow}
\newcommand{\rmap}{\longrightarrow}
\newcommand{\lmap}{\longleftarrow}

\newcommand{\X}{\ensuremath{\mathfrak{X}}}

\renewcommand{\d}{d}

\newcommand{\C}{\ensuremath{\mathrm{C}^{\infty}}}

\DeclareMathOperator{\gl}{\mathfrak{gl}}
\DeclareMathOperator{\GL}{GL}

\DeclareMathOperator{\Diff}{Diff}
\newcommand{\pr}{pr}
\DeclareMathOperator{\Bis}{Bis}
\renewcommand{\hom}{\mathrm{Hom}}
\DeclareMathOperator{\modular}{mod}     
\renewcommand{\mod}{\modular}
\newcommand{\Jet}{J}
\newcommand{\jet}{j}
\newcommand{\ve}{\vartheta}          

\newcommand{\G}{\mathcal{G}}            
\newcommand{\tG}{\tilde{\G}}
\renewcommand{\u}{u}           
\renewcommand{\H}{\mathcal{H}}          
\newcommand{\tH}{\tilde{\H}}
\renewcommand{\tt}{\tilde{\theta}}
\newcommand{\Rep}{\text{\rm Rep}\,}     
\newcommand{\A}{A}                      
\newcommand{\al}{\alpha}                
\newcommand{\be}{\beta}                 
\newcommand{\ga}{\gamma}                
\newcommand{\Lie}{L}          
\newcommand{\g}{\mathfrak{g}}        
\newcommand{\ggl}{\mathfrak{gl}}          
\newcommand{\Sol}{\text{\rm Sol}}
\renewcommand{\Im}{\text{\rm Im}\,}     
\newcommand{\Ad}{\text{\rm Ad}\,}       
\newcommand{\ad}{\text{\rm ad}\,}       
\newcommand{\tto}{\rightrightarrows}    

\newcommand{\clas}{\text{\rm clas}}      

\newcommand{\rank}{\text{\rm rk}\,}   
\newcommand{\B}{\mathfrak{g}}                    

\newcommand{\thesistitle}
{Pfaffian groupoids}

\newcommand{\nativetitle}
{Pfaffiaanse groepo\"\i den}

\begin{document}
 \frontmatter

 \thispagestyle{empty}
\begin{center}

\vspace*{0.2\textheight}

\begin{minipage}{\textwidth}
\begin{center}
  \renewcommand{\baselinestretch}{1.2}
  \bfseries\huge\thesistitle
\end{center}
\end{minipage}

\end{center}

\clearpage


\thispagestyle{empty}




\null
\vfill

\noindent Pfaffian groupoids, Mar\'ia Amelia Salazar Pinz\'on

\noindent Ph.D. Thesis Utrecht University, May 2013


\noindent ISBN: 978-90-393-5961-7

\noindent


\clearpage

\thispagestyle{empty}
\begin{center}

\vspace*{3em}

\begin{minipage}{\textwidth}
\begin{center}
  \renewcommand{\baselinestretch}{1.2}
  \bfseries\huge\thesistitle
\end{center}
\end{minipage}

\vspace{5em}

\begin{minipage}{\textwidth}
\begin{center}
  \renewcommand{\baselinestretch}{1.2}
  \bfseries\Large\nativetitle
\end{center}
\end{minipage}

\vspace{\baselineskip} (met een samenvatting in het Nederlands)

\vfill

{\LARGE Proefschrift}

\vspace{3em}

\begin{minipage}{0.9\textwidth}
\begin{center}
ter verkrijging van de graad van doctor aan de Universiteit Utrecht\\
op gezag van de rector magnificus, prof.\,dr.\ G.J.\ van der Zwaan,\\
ingevolge het besluit van het college voor promoties in het openbaar\\
te verdedigen op vrijdag 24 mei 2013 des middags te 12.45 uur
\end{center}
\end{minipage}
\vspace{\baselineskip}

door \vspace{\baselineskip}

{\large Mar\'\i a Amelia Salazar Pinz\'on}

\vspace{\baselineskip}

geboren op 29 september 1983 te Manizales, Colombia \vspace{\baselineskip}

\end{center}

\clearpage

\thispagestyle{empty}

\begin{tabular}{ll}
Promotor:  & Prof.\,dr.\ M.\ Crainic 
\end{tabular}

\vfill

This thesis was accomplished with financial support from the Netherlands Organisation for Scientific Research (NWO) under the Project ``Fellowship of Geometry and Quantum Theory'' no. 613.000.003

\cleardoublepage

\thispagestyle{empty}


\clearpage

\thispagestyle{empty}

\cleardoublepage

\pagestyle{fancy}

\fancyhead{}
\fancyhead[CE]{Contents} 
\fancyhead[CO]{Contents}

\tableofcontents

\clearpage \pagestyle{plain}

\mainmatter

\fancyhead{} \setcounter{tocdepth}{1}
\chapter*{Introduction}
\addcontentsline{toc}{chapter}{Introduction} \pagestyle{fancy} 
\fancyhead[CE]{Introduction} 
\fancyhead[CO]{Introduction}

This thesis is motivated by the study of symmetries of partial differential equations (PDEs for short). Partial differential equations are used to model many different `real life' phenomena, while also describing many interesting geometric structures. As such, they are a worthwhile object of study, which has interested scores of mathematicians since the 19th century. In fact,  Lie \cite{Lie} first studied the symmetries of PDEs in an attempt to attain a better understanding of their geometry. Ever since many mathematicians considered different aspects of the geometric theory of PDEs and developed various techniques to attack problems in this area; amongst others, this thesis has been inspired by ideas in \cite{Gold1,GuilleminSternberg:deformation,KumperaSpencer,Kuranishi,SingerSternberg}, which are often used and referred to in the chapters below.

This begs the question of what we mean by symmetries of PDEs; in general, symmetries play a fundamental role in mathematics and, in particular, in geometry, as they describe qualitative behaviour of the object under consideration. To give an idea of symmetries in geometry, consider a wheel as the object of interest whose center is fixed; a rotation by any angle (either clockwise or counterclockwise) keeps the shape of the wheel: as such it can be thought of as a symmetry. Rotations have the property that they can be {\it composed}, in the sense that when we rotate the wheel twice, first by an angle $x$ and then by an angle $y$, the final symmetry is a rotation by the angle $x+y$. Also, not moving the wheel can be thought of as a rotation by the $0$ angle (call it the {\it identity} symmetry). Finally, if we rotate the wheel by an angle $x$ and then rotate counterclockwise by the same angle (in other words, rotation by the {\it inverse} of $x$), then at the end we get the identity symmetry. This example shows that when we consider symmetries of a geometric object we are interested in extra structure that can be attached to these, which, in this case, is that of a {\it Lie group}, i.e. the abstract space of rotations thought of as a circle, which {\it acts} on the wheel. It was Lie's original work that gave rise to the modern notion of Lie groups; these encode the abstract smooth objects that describe symmetries. 

In fact, Lie worked with a more general notion, which is nowadays known as a {\it Lie pseudogroup}, in which not all symmetries necessarily act on the whole geometric object under consideration. The presence of {\it local} symmetries allows to study the qualitative behaviour of more general spaces, for instance (systems of) PDEs. One of Lie's greatest achievements was the discovery that Lie pseudogroups are better understood by considering their infinitesimal data (or linearization), which consists of going from very complicated differential equations to much simpler ones which nonetheless encode much geometric information about the original system. Actually, his work concentrated on a special kind of Lie pseudogroups, called of finite type, which in today's language correspond to Lie groups. Predictably, their linearizations are nothing but their associated Lie algebras.
 However, it was Cartan \cite{Cartan1904,Cartan1905, Cartan1937} at the beginning of the 20th century who made further progress on the topic of Lie pseudogroups of infinite type. He discovered that the way to encode the infinitesimal data for Lie pseudogroups of infinite type, in analogy with that of Lie algebras for Lie pseudogroups of finite type, was using differential 1-forms. This approach led him to the discovery of differential forms and the theory of exterior differential systems (EDS for short) \cite{BC}. With his new interpretation he was dealing with a {\it jet bundle} space $X$, together with the so-called Cartan (contact) form $\theta$. The infinitesimal data of the original Lie pseudogroup was encoded in Maurer-Cartan type equations that $\theta$ satisfies. This is the right place to explain the title of this thesis: the reason for the nice interaction between $X$ and $\theta$, relies on their compatibility (i.e. $\theta$ is multiplicative with respect to the structure of $X$), rather than on the fact that $X$ is a jet bundle and $\theta$ is the Cartan form... And this is what a Pfaffian groupoid is: a Lie groupoid together with a multiplicative differential form. Just to give an idea, one of the main theorems (theorem 2) says that Lie's infinitesimal picture corresponding to Pfaffian groupoids, consists of certain connection-like operator (Spencer operator), and this correspondence is one to one under the usual assumptions. I would also like to comment on the relation between theorem 1 and Pfaffian groupoids. 
 The main reason for us to allow general forms is the fact that
multiplicative $2$-forms are central to Poisson and related geometries. Moreover, while multiplicative $2$-forms with trivial
coefficients (!) are well understood, the question of passing from trivial to arbitrary coefficients has been around for a while. 
Surprisingly enough, even the statement of an integrability theorem for multiplicative forms with non-trivial coefficients was completely unclear.
This shows, that even the case of trivial coefficients was still hiding phenomena that we did not understand. The fact that our work related to Pfaffian groupoids clarifies this point
came as a nice surprise to us and, looking back, we can now say what was missing from the existing picture in multiplicative $2$-forms: Spencer operators.\\

Next I will explain  the form of this thesis in more detail. As I mentioned above, this thesis is about the study of Lie groupoids endowed with a compatible (multiplicative) differential 1-form. The motivation and scope of the present work is to study the geometry of PDEs using the formalism of Lie groupoids and multiplicative forms; as such, ideas from the two theories have to be introduced and explained from our point of view (which may not be the same as in the literature!) before new results can be presented. Therefore the thesis can be naturally split in two halves: the first, consisting of chapters 1, 2 and 3, recall the ideas and methods which are used in the second half, where the majority of original results are presented. It is important to remark that when considering multiplicative structures on Lie groupoids we shall employ two (equivalent) points of view: the one using differential forms and the dual picture with distributions. 
 
Chapter 1 provides some preliminaries that are used throughout the thesis, such as jet bundles, Spencer cohomology, Lie groupoids, Lie algebroids, etc. The aim of this chapter is twofold: on the one hand, it introduces concepts which may not be familiar to all readers in a way to ease them into the rest of the thesis, while on the other it provides crucial motivation for this work. In particular, some examples coming from Lie pseudogroups are discussed and a new notion of ``generalized pseudogroups" is proposed. 

As many notions come from the classical theory of Pfaffian systems (where we have a general bundle, instead of a Lie groupoid), it is important to understand their geometry conceptually. Chapters 2 and 3 are devoted to the study of this, starting with the easier-to-handle linear picture of relative connections in Chapter 2, passing to the global description of Chapter 3. While this may not be the most natural order (especially from a historical point of view), we choose to do it the other way because the linear picture is easier to consider. 

In Chapter 4, we move to the theory of multiplicative forms on Lie groupoids and their infinitesimal counterparts on Lie algebroids. While this chapter can be read independently from the rest of the thesis, it presents ideas which are crucial to the understanding of Chapters 5 and 6. The main result of this chapter is the integrability theorem for multiplicative $k$-form with coefficients (cf. Theorem 1 below), which states that under the usual conditions we can recover a given multiplicative $k$-form from its infinitesimal data (namely, a $k$-Spencer operator). Of course, $k=1$ is the relevant case for the original motivation of the thesis. 

Chapter 5 and 6 are the core of the thesis in terms of original results. In Chapter 6 everything comes together. The multiplicativity condition for Pfaffian groupoids simplifies the theory developed in Chapter 3, where all the notions here become ``Lie theoretic''. In contrast with Chapter 3, integrability results (cf. Theorem 2 below) can by applied to ensure that a Pfaffian groupoid can be recovered from its infinitesimal data (namely, its Spencer operator). These are discussed in Chapter 5. Of course, as the linear counterpart of Pfaffian groupoids, they are the natural relative connections on the setting of Lie algebroids: they are compatible with the anchor and the Lie bracket. Again, in some sense it is more natural to start with the global picture of Pfaffian groupoids and then pass by linearization to that of Spencer operators, but we chose to start with the easier to handle picture of Spencer operators in conformity with our introduction to the standard theory of Pfaffian bundles. Chapter 6 also discusses some other results which stem from this thesis, such as the infinitesimal condition that ensures the Frobenius involutivity of the Pfaffian distribution (cf. Theorem 5 below), and integration of Jacobi structures to contact groupoids (cf. Corollary 2 below). 

\subsection*{Description of the chapters}

This is a more detailed description of each chapter with a short review of the most important notions and results. The following table illustrates the dynamics of the thesis. In the table ``Lin'' stands for linearization, and ``Int'' for integration. \\

\begin{table*}[!h]
  \centering
  \large
  \begin{tabular}{|c|p{4.25cm} p{3cm}|}
    \hline
    \mbox{} & {\bf Infinitesimal} &  {\bf Global} \\
    \hline
    \multirow{2}*{{\bf General}} & \multicolumn{2}{p{8cm}|}{Relative connections $\xleftarrow{\mathrm{Lin}}$  Pfaffian bundles} \\ 
    & (Ch. 2) & (Ch. 3) \\
    \cline{2-3}
    \multirow{2}*{ {\bf Multiplicative}} & \multicolumn{2}{p{8cm}|} {$\xymatrix{\mathrm{Spencer\, operators} \ar@/^0.5pc/[r]^-{\mathrm{Lin}}& \mathrm{Pfaffian\, groupoids}   \ar@/^0.8pc/@{-->}[l]^-{\mathrm{(Ch.\,4)}}_-{\mathrm{Int}}} $} \\
     \mbox{} & (Ch. 5) & (Ch. 6) \\
    \hline
  \end{tabular}
\end{table*}
We should also keep in mind that 
$$\text{\bf Multiplicative}\subset \text{\bf General},$$
in the global picture, as in the infinitesimal counterpart.\\

\noindent
\underline{\textbf{Chapter 1: preliminaries.}} This consists of the preliminaries and some motivation for this thesis, while also setting the notation used throughout.\\

\noindent
\underline{\textbf{Chapter 2: relative connections.}} This chapter discusses relative connections. These are linear operators 
\begin{eqnarray*}\Gamma(F)\To\Omega^1(M,E)\end{eqnarray*}
satisfying connection-like properties along a surjective vector bundle map $F\to E$ over a manifold $M$. The example to keep in mind and which will be of relevance is the classical Spencer operator $\Gamma(J^1E)\to\Omega^1(M,E)$ on jets, whose role is to detect holonomic sections (cf. \cite{quillen}). Although relative connections are basically the same thing as linear Pfaffian bundles, I build the theory on its own (and leave the explanation of this correspondence for Chapter 3). I study the geometry of relative connections carrying out standard notions of the geometry of ODEs, such as solutions, Spencer cohomology, curvature map, prolongations, formal integrability, etc. (cf. \cite{BC,Gold1,Gold4,Gold3,Kuranishi1,quillen,SingerSternberg,Spencer}). I pay particular attention to the notion of prolongation: I introduce an abstract notion of prolongation (very natural from the cohomological point of view), and I recover the standard prolongation as the ``universal abstract prolongation''.\\

\noindent
\underline{\textbf{Chapter 3: Pfaffian bundles.}} This chapter is very similar in spirit to Chapter 2, but this time we deal with non-linear objects, namely Pfaffian bundles. A Pfaffian bundle
\begin{eqnarray*}
\pi:R\To M,\quad H\subset TR
\end{eqnarray*}
 is a surjective submersion $\pi$ together with a distribution $H$
which is transversal to the $\pi$-fibers, and whose vertical part $H^\pi:=H\cap\ker d\pi$ is  Frobenius involutive. Dually, I give an alternative equivalent definition in terms of a point-wise surjective 1-form $\theta$ with values on a vector bundle $E$ over $R$, with the property that the vertical part of its kernel is Frobenius involutive. Of course the connection between the two is the kernel of the 1-form. The latter definition admits a slightly more general theory where one allows 1-forms of non-constant rank (and there are still notions that make sense in this case). With the two dual definitions, the main notion is that of a solution $\sigma\in\Gamma(R)$. These are sections with the property that 
\begin{eqnarray*}
d\sigma(TM)\subset H.
\end{eqnarray*}
These objects appear, for instance, in the study of PDE's. For example in \cite{russians} the authors consider the (infinite dimensional) infinite jet space together with a distributions. Solutions of the PDE are sections whose infinite jet is tangent to such a distribution.

I study the geometry of Pfaffian systems using techniques from EDS (cf. \cite{BC, Gardner, Cogliati, Kamran1, Ivey}) and PDE's (cf. \cite{Gold2, Gold3,Stormark}), such as prolongations (cf. \cite{Kuranishi1}), curvature maps, formal integrability (cf. \cite{Gold2,Kakie, Makeev}), and so on. I mention the key properties that govern the classical prolongations, allowing to develop further the abstract notion of prolongation in the case of Pfaffian groupoids. Towards the end of the chapter I study Pfaffian bundles on which $\pi$ is a vector bundle and $H$ (or $\theta$ in the dual picture) is compatible with the linear structure of $R$. These are the so-called linear Pfaffian bundles. They correspond naturally to connections relative to the projection $\theta:R\to E$ (using the identification $T^\pi R|_M\simeq R$), via the formula
\begin{eqnarray*}
D(s)=s^*\theta
\end{eqnarray*}
for $X\in\X(M),s\in\Gamma(R)$. With this, I can apply all the theory of Chapter 2, and I actually go through the notions of chapter 2 (such as prolongations, symbol space, curvature maps, integrability...) and the analogous ones for Pfaffian bundles, showing that in this case they coincide. At the end of the chapter I discuss linearization of Pfaffian bundles (along solutions), which can be described directly and naturally as relative connections.\\

\noindent
\underline{\textbf{Chapter 4: the integrability theorem for multiplicative forms.}} This chapter contains the integrability result (cf. Theorem 1 below) for multiplicative $k$-forms on a Lie groupoid $\G$ with coefficients in a representation $E$ of $\G$. This result describes multiplicative forms of degree $k$ in terms of their infinitesimal data, i.e. $k$-Spencer operators. More precisely, for a representation $E$ of a Lie algebroid $A$ over $M$, a $k$-Spencer operator is a linear operator
\begin{eqnarray*}
D:\Gamma(A)\To \Omega^{k}(M,E)
\end{eqnarray*}
together with a vector bundle map $l:A\to \wedge^{k-1}T^*M\otimes E$, satisfying some compatibility conditions with the bracket and the anchor. We remark that  for Pfaffian groupoids, the relevant case is $k =1$, and the resulting notion is that of Spencer operator (cf. Chapter 5).

\begin{theoremone}
Any multiplicative form $\theta \in \Omega^k(\G, t^{\ast}E)$ induces a
$k$-Spencer operator $D_{\theta}$ on $A=Lie(\G)$, given by 
\begin{equation*}\small{
\left\{\begin{aligned}
D_{\theta}(\al)_x(X_1, \ldots, X_k) &= \frac{\d}{\d \eps}\Big|_{\eps = 0} \phi^{\eps}_{\al}(x)^{-1}\cdot\theta((\d \phi^{\eps}_{\al})_x(X_1), \ldots, (\d \phi^{\eps}_{\al})_x(X_k)), \\ \\l_{\theta}(\al) &= u^{\ast}(i_{\al}\theta) .\end{aligned}\right.}
\end{equation*} 
where $\phi_\al^\epsilon:M\to R$ is the flow of $\alpha$.

If $\G$ is $s$-simply connected, then this construction defines a 1-1 correspondence between  $E$-valued $k$-forms on $\G$ and 
$E$-valued $k$-Spencer operators on $A$.  
\end{theoremone}

This is more general than what is needed for Pfaffian groupoids (which only requires 1-forms). However, this greater generality settles the question of integrability of multiplicative forms with non-trivial coefficients, which, to the best of my knowledge, was not solved before. This also leads to interesting developments in geometric settings other than those arising from Lie pseudogroups, e.g. in understanding the relation between contact and Jacobi manifolds (cf. Chapter 6). Moreover, theorem \ref{t1} can be applied to multiplicative forms of arbitrary degree, e.g. to recover results on the integrability of Poisson manifolds.\\

\noindent
\underline{\textbf{Chapter 5: Spencer operators.}} Spencer operators are relative connections on a Lie algebroid which are compatible with the algebroid structure. The compatibility conditions are far from obvious, and the main explanation is that they arise as the infinitesimal counterpart of multiplicative 1-forms. In other words, Spencer operators are the linearization of Pfaffian groupoids along the unit map (actually, Theorem 6 asserts that, under the usual assumptions, the Pfaffian groupoid can be recovered from its Spencer operator). Hence, in this chapter I go through all the notions of chapter 2 and show that they become ``Lie theoretic''.\\

\noindent
\underline{\textbf{Chapter 6: Pfaffian groupoids.}} In this chapter everything is brought together. The theory of Pfaffian bundles becomes simpler when the objects are multiplicative (e.g. the vector bundles involved are pull-backs from the base, etc.) and all the notions (coming from Pfaffian bundles) become ``Lie theoretic''. On the infinitesimal side, I get that the Spencer operators are the linearization of Pfaffian groupoids along the unit map; in contrast with the theory of Pfaffian bundles, Pfaffian groupoids, under the usual conditions can be recovered from the associated Spencer operators. More precisely, I have the following result:

\begin{theoremtwo} Let $\G\rightrightarrows M$ be a $s$-simply connected Lie groupoid with Lie algebroid $A\to M$. There is a one to one correspondence between 
\begin{enumerate}
\item multiplicative distributions $\H\subset T\G$, 
\item sub-bundles $\g\subset A$ together with a Spencer operator $D$ on $A$ relative to the projection $A\to A/\g$.
\end{enumerate}
Moreover, $\H$ is of Pfaffian type if and only if $\g\subset A$ is a Lie subalgebroid. 
\end{theoremtwo}
 
 As we shall see the relation between $D$ and $\H$ is given by
 \begin{eqnarray*}
D_X\alpha=[\tilde X,\alpha^r]|_M\mod \H^s ,\quad \g=\H^s|_M,
\end{eqnarray*}
where $\tilde X\in \Gamma(\H)$ is any vector field which is $s$-projectable to $X$ and extends $u_*(X)$. \\

Then I go on with the notion of prolongation. The notions of abstract prolongation for Spencer operators (called compatible Spencer operators) is ``integrated'' to find the global counterpart of abstract prolongation for Pfaffian groupoids. I again have an integrability result which gives a 1-1 correspondence between (abstract) prolongations of Spencer operators on the Lie algebroid side, and (abstract) prolongations of Pfaffian groupoids on the Lie groupoid side. Explicitly,

\begin{theoremthree}
 Let $\tG$ and $\G$ be two Lie groupoids over $M$ with $\tG$ $s$-simply connected and $\G$ $s$-connected, with Lie algebroids $\tilde A$ and $A$ respectively. Let  $\theta\in\Omega^1(\G,t^*E)$ be a point-wise surjective multiplicative form and denote by $(D,l):A\to E$ the Spencer operator associated to $\theta$. There is a 1-1 correspondence between:
\begin{enumerate}
\item prolongations $p:(\tG,\tt)\to (\G,\theta)$ of $(\G,\theta)$, and
\item Spencer operators $(\tilde D,\tilde l):\tilde A\to A$ compatible with $(D,l):A\to E$.
\end{enumerate} 
In this correspondence, $\tilde D$ is the associated Lie-Spencer operator of $\tt$. 
\end{theoremthree}

Corollary 1 is a slight generalization of the previous theorem which gives a correspondence between towers of prolongations on the groupoid side (Cartan towers), with the ones on the algebroid side (Spencer towers). 

I also show that the compatibility conditions that are required for the notion of abstract prolongation of a Pfaffian groupoids are equivalent to Maurer-Cartan type equations: we consider two Pfaffian groupoids $(\tG,\tt)$ and $(\G,\theta)$ over $M$, with the property that $\tt$ takes values on the Lie algebroid $A$ of $\G$. For the the Spencer operator $D$ of $\theta$, we define
\begin{eqnarray*}
\frac{1}{2}\{\al,\be\}_D:=D_{\rho(\al)}(\be)-D_{\rho(\be)}(\al)-l[\al,\be]
\end{eqnarray*}
for $\al,\be\in\Gamma(A)$. On the other hand, we have the differential with respect to $D$ 
\begin{eqnarray*}
d_D\tt\in\Omega^2(\tG,t^*E)
\end{eqnarray*}
defined by the usual Koszul formula for the pull-back connection $t^*D:\Gamma(A)\to \Omega^1(\tG,E)$.

\begin{theoremfour}
Let $p:\tG\to\G$ be a Lie groupoid map which is also a surjective submersion. If  
 $$p:(\tG,\tt)\To(\G,\theta)$$
 is a Lie prolongation of $(\G,\theta)$ then $$d_{D}\tt-\frac{1}{2}\{\tt,\tt\}_{D}=0.$$
 
 If $\tG$ is source connected and $Lie(p)=\tt|_{Lie(\tG)}$, then the converse also holds. 
\end{theoremfour}

At the end of thesis we investigate the relation of our results and ideas to Poisson and relates geometries. There has been recent interest on understanding multiplicative foliations. Theorem 5 clarifies the infinitesimal conditions that ensure that the Pfaffian distributions $\H$ is Frobenius involutive. I consider the associated Spencer operator $D$ relative to the projection $A\to A/\g$, and look at its restriction to $\g$ 
$$\partial_D:\g\To\hom(TM,A/\g).$$ One notices that whenever $\partial_D$ vanishes, $D$ descends to an honest connection $\nabla$ on the quotient $A/\g$.

\begin{theoremfive}
A multiplicative distribution $\H \subset T\G$ is involutive if and only if $\partial_D$ vanishes
and the connection $\nabla$ on $A/\g$ is flat.
\end{theoremfive}

Finally I have an application to Jacobi structures and contact groupoids. The main contribution is to introduce an appropriate language to deal with these structures using a more conceptual and global approach. First of all, on the Jacobi side we allow for line bundles $L$ which are not necessarily trivial, with a Lie bracket on the space of sections. We define a Lie algebroid structure on $J^1L$ canonically associated to the Jacobi structure. 

\begin{corollarytwo}
Given a Jacobi structure $(L, \{\cdot, \cdot\})$ over $M$, if the associated Lie algebroid $\Jet^1L$ comes from an
$s$-simply connected Lie groupoid $\Sigma$, then $\Sigma$ carries a contact hyperfield $\H$ making it into a contact groupoid in the wide sense. 
\end{corollarytwo}

 \clearpage \pagestyle{plain}


\chapter{Preliminaries}\label{Preliminaries}
\pagestyle{fancy}
\fancyhead[CE]{Chapter \ref{Preliminaries}} 
\fancyhead[CO]{Preliminaries} 

\section{Some notions related to PDEs}

\subsection{Jets}\label{sec:jets}

Throughout this thesis we will use the language of jets, which we now briefly recall. As reference we cite \cite{russians}. Let $k\geq 0$ be an integer. Recall that, for a function
\[ f: \mathbb{R}^m \rmap \mathbb{R}^r ,\]
($m, r\geq 1$ integers), the $k$-jet of $f$ at $x\in\mathbb{R}^m$ is encoded by the collection of all partial derivatives up to order $k$ of $f$ at $x$. 
More concretely, one says that two such functions $f$ and $g$ have the same $k$-jet at $x$ if 
\[ \partial^{\alpha}_{x}f= \partial^{\alpha}_{x}g\]
for all $m$-multi-indices $\alpha= (\alpha_1, \ldots, \alpha_m)$, with $|\alpha|\leq k$, where
$|\alpha|= \alpha_1+ \ldots + \alpha_m$ and where
\[ \partial^{\alpha}_{x}f= \frac{\partial^{|\alpha|}f}{\partial x_{1}^{\alpha_1} \ldots \partial x_{m}^{\alpha_m}}(x) .\]
This defines an equivalence relation $\sim^{k}_{x}$ on the space of smooth maps $C^{\infty}(\mathbb{R}^m, \mathbb{R}^r)$. A point in the quotient has 
coordinates $p^{\alpha}_{i}$ with $\alpha$ as above and $i\in \{1, \ldots , r\}$, representing the partial derivatives. 
More generally, given two manifolds $R$ and $M$ (of dimensions $r$ and $m$ respectively), one
obtains an induced equivalence relation $\sim^{k}_{x}$ on the set $C^{\infty}(M, R)$ by representing maps $f\in C^{\infty}(M, R)$ in local coordinates. The space of $k$-jets at $x$ of functions from $R$ to $M$, denoted by $J^k(M, R)_x$, is the resulting quotient; for $f: R \to M$ smooth, we will denote the induced $k$-jet by $j^{k}_{x}f$. Hence,
\[ J^k(M, R)_x= \{ j^{k}_{x}f: x\in M\}.\]

\subsection{Jet bundles}\label{jet-bundles}
Assume now that we have a bundle (by which we mean a surjective submersion) 
\[ \pi: R\rmap M .\]
We denote by $\Gamma(R)$ the set of sections of $\pi$, i.e. the set of smooth maps $\sigma: M \to R$ satisfying $\pi\circ \sigma= \textrm{Id}_M$. We also consider the set
$\Gamma_{\textrm{loc}}(\pi)$ of local sections $\sigma$ of $\pi$ (each such $\sigma$ is defined over some
open in $M$, called the domain of $\sigma$ and denoted by $\textrm{Dom}(\sigma)$). 
For $k\geq 0$ an integer, the space of $k$-jets of sections of $\pi: R\to M$ is defined as
\[ \Jet^kR= \{j_{x}^{k}\sigma: x\in M, \ \sigma\in \Gamma_{\textrm{loc}}(\pi), x\in \textrm{Dom}(\sigma) \}.\]
This has a canonical manifold structure (as can be seen by looking at local coordinates), which fibers over $M$; the various jet bundles are related to each other via the obvious projection maps. In other words, we obtain a tower of bundles over $M$
\[ \ldots \rmap \Jet^2R \rmap \Jet^1R \rmap \Jet^0R= R.\]
In the limit, one obtains the infinite jet bundle $\Jet^{\infty}R$. To keep notation simpler, we will denote all projections between jet bundles by $pr$, and all bundle maps $\Jet^kR\to M$ by $\pi$. \\

\subsection{PDEs and the Cartan forms}\label{PDEs; the Cartan forms:} The language of jets is very well suited for a conceptual theory of PDE. See for example \cite{Gold2,russians}. Given a submersion $\pi$ as above, a PDE of order $k$ on $\pi$ is, by definition, a fibered submanifold
\[ R_k\subset J^kR .\]
A (local) solution of $R_k$ is then any (local) section $\sigma$ of $R$ with the property that 
\[ j^{k}_{x}\sigma\in R_k,\ \ \forall \ x\in \textrm{Dom}(\sigma).\]
In other words, $j^{k}\sigma$, a priori a (local) section of $J^kR$, must be a section of $R_k$. We denote the set of solutions by $\textrm{Sol}(R_k)$. To recognize which sections of $R_k$ are $k$-jets of sections of $R$, one makes use of the so-called Cartan form. For notational ease, let us assume here that $k= 1$ (a more general discussion can be found in example \ref{cartandist}). Then the Cartan form is a 1-form
\[ \theta\in \Omega^1(\Jet^1R, pr^*T^{\pi}R),\]
where $T^{\pi}R= T^{\textrm{vert}}R$ consisting of vectors tangent to the fibers of $\pi$ (a vector bundle over $R$) where $pr: \Jet^1R\to R$ is the projection.
To describe $\theta$, let $p= j^{1}_{x}\sigma\in \Jet^1R$ and $X_p\in T_p\Jet^1R$; then $\theta(X_p)\in T^{\textrm{vert}}_{\sigma(x)}R$ is given by the expression
\[ dpr(X_p)- d_x\sigma(d\pi)(X_p) \]
(a priori an element in $T_{\sigma(x)}R$, it is clearly in $\ker d\pi$). The main property of the Cartan form is recalled below.

\begin{lemma}\label{nomas} A section $\xi$ of $\Jet^1R\to  M$ is of type $j^{1}\sigma$ for some section $\sigma$ of $R$ if and only if $\xi^*(\theta)= 0$. 
\end{lemma} 

For arbitrary $k$, there is a version of the Cartan form (see example \ref{cartandist}),
\begin{equation}\label{eq-higher-Cartan} 
\theta\in \Omega^1(\Jet^kR, pr^*T^{\pi}\Jet^{k-1} R)
\end{equation}
and the analogue of lemma \ref{nomas} holds true. Hence, for a PDE $R_k\subset J^kR$, one has
\[ \textrm{Sol}(R_k) \cong \{\xi\in \Gamma(R_k): \xi^{*}\theta = 0\}.\]
Conceptually, this indicates that the main information is contained in $R_k$, viewed as a bundle over $M$ (and forgetting that it sits inside $\Jet^kR$), together with the restriction of the Cartan form to $R_k$. In other words, for the study of PDE's as above it may be enough to consider bundles $S\to M$ endowed with appropriate $1$-forms on them. This will be formalised in chapter \ref{Pfaffian bundles}.

\begin{remark}\rm \ 
\label{1-when working with jets}
Here is a slightly different description of $\Jet^1R$, which we will be using whenever we have to work more explicitly. 
Since the first jet $\jet^1_x\sigma$ of a section $\sigma$ at $x\in M$ is encoded by $r:= \sigma(x)$ and $d_x\sigma: T_xM\to T_r R$, we see that an element of $\Jet^1R \to M$ can be thought of as a splitting 
\begin{eqnarray*}
\xi=d_x\sigma:T_xM\To T_rR
\end{eqnarray*}
of the map $d_r\pi:T_rR\to T_xM$, with $r$ an element in $R$ and $x=\pi(r)$. Of course, $r$ is encoded by $\xi$, but we will often use the notation $\xi_r$ to indicate that our splitting takes values in $T_rR$. 
\end{remark}

\subsection{Jets of vector bundles; Spencer relative connections:}\label{sp-dc} Assume now that 
\[ \pi: E\rmap M\]
is a vector bundle over $M$. In this case, $J^kE$ is canonically a vector bundle over $M$, with addition determined by
\[ j^{k}_{x}\alpha+ j^{k}_x\beta= j^{k}_{x}(\alpha + \beta)\]
for $\alpha, \beta\in \Gamma(E)$.

\begin{remark}[The Spencer decomposition] \label{remark-J-decomposition}\rm \ 
Again, when $k= 1$, there is a very convenient way of representing sections of the first jet bundle $J^1E$
which will be used throughout this thesis under the name ``Spencer decomposition''. More precisely, one has a canonical isomorphism of vector spaces 
\begin{equation}\label{J-decomposition} 
\Gamma(\Jet^1E)\cong \Gamma(E)\oplus \Omega^1(M, E).
\end{equation}
This decomposition comes from the short exact sequence of vector bundles 
\[ 0\To \textrm{Hom}(TM, E)\stackrel{i}{\To} J^1(E) \stackrel{\pr}{\To} E\To 0,\]
where $\pr$ is the canonical projection $j^{1}_{x}s\mapsto s(x)$ and $i$ is determined by
\[ i(\d  f\otimes \alpha)= f\jet^1\alpha- \jet^1(f\alpha).\]
Although this sequence does not have a canonical splitting, at the level of sections it does: $\alpha\mapsto \jet^1\alpha$.
This gives the identification (\ref{J-decomposition}). In other words, any $\xi\in \Gamma(\Jet^1E)$ can be written uniquely as
\[  \xi= j^1\alpha+ i(\omega)\]
with $\al\in \Gamma(E)$, $\omega\in \Omega^1(M, E)$; we write $\xi= (\alpha, \omega)$. One should keep in mind, however, that the resulting $C^{\infty}(M)$-module structure becomes
\begin{equation}\label{strange-module-structure} 
f \cdot (\alpha, \omega)= (f\alpha, f\omega+ \d f\wedge \alpha).
\end{equation}
Using the decomposition (\ref{J-decomposition}), the projection on the second component induces an operator 
\[ D^{\textrm{clas}}: \Gamma(\Jet^1E) \To \Omega^1(M, E)\]
called \textbf{Spencer's relative connection} (or \textbf{the classical Spencer operator}). There is an extensive list of literature about the classical Spencer operator, see for example \cite{Veloso, Ngo, Ngo1, Spencerflatt, Gold3, GoldschmidtSpencer, Spencer}.
\end{remark}

As we shall show in this thesis, $D^{\textrm{clas}}$ is the infinitesimal counterpart of the Cartan form. This is already indicated by the fact that $D^{\textrm{clas}}$ has a property resembling that of $\theta$: a section $\xi$ of $\Jet^1E$ is the first jet of a section of $E$ if and only if $D^{\textrm{clas}}(\xi)= 0$. Moreover, as for Cartan forms, there are versions of the Spencer operator for higher jets:
\[ D^{\textrm{clas}}: \Gamma(J^kE) \To \Omega^1(M, J^{k-1}E).\]

\subsection{Tableaux and bundles of tableaux}\label{digression}


When dealing with PDE's one often encounters huge jet spaces, specially after the prolongation process. One would like to compare such (prolongation) spaces by looking at smaller, hopefully linear, spaces. This is one way to arrive at the notion of a tableau. Of course, their importance is deeper and they provide the framework for handling the intricate linear algebra behind PDE's. In this section we recall some of the basic definitions, and we allow a slight generalisation of the notion of tableau (which will be used later on). Our main reference for this part is \cite{BC,Gold2}.\\ 

Let $V$ and $W$ be finite dimensional vector spaces, let $S^kV^*$ be the $k$-th symmetric product of $V^*$, and let $S^kV^*\otimes W$ be the $W$-valued $k$-th symmetric multilinear maps $\phi:V^k\to W$.

\begin{definition} A \textbf{tableau} on $(V, W)$ is a linear subspace
\[ \mathfrak{g}\subset V^*\otimes W .\]
More generally, for $k\geq 1$ integer, a {\bf tableau of order $k$} on $(V, W)$ is a linear subspace
\begin{eqnarray*}
\g_k\subset S^{k}V^*\otimes W
\end{eqnarray*}
We will often omit ``on $(V, W)$'' from the terminology.
\end{definition}

In general, a tableau of order $k$ can be ``prolonged'' to tableaux of higher orders.

\begin{definition} The \textbf{first prolongation} of a tableau of order $k$
\[ \g_k\subset S^{k}V^*\otimes W\]
is the tableau of order $(k+1)$ 
\[ \g_{k}^{(1)} :=\{T\in S^{k+1}V^*\otimes W\mid T(v,\cdot,\ldots,\cdot)\in \g_k\text{ for all }v\in V\}.\]
The higher prolongations $\g_k^{(q)}$ are defined inductively by
\[ \g_{k}^{(q)}= (\g_{k}^{(q-1)})^{(1)} \subset S^{k+q}V^*\otimes W.\]
\end{definition}

Next, we recall the construction of the formal De Rham operator on a vector space $V$. Start with 
\begin{eqnarray}\label{fdo}
\partial:S^{k}V^*\To V^*\otimes S^{k-1}V^*
\end{eqnarray}
the linear map sending $T\in S^kV^*$ to 
\begin{eqnarray*}
\partial(T):V\To S^{k-1}V^*,\quad v\mapsto T(v,\cdot,\ldots,\cdot).
\end{eqnarray*}
We extend $\partial$ to a linear map 
\begin{eqnarray*}
\partial:\wedge^jV^*\otimes S^kV^*\To \wedge^{j+1}V^*\otimes S^{k-1}V*
\end{eqnarray*}
sending $\omega\otimes T$ to $(-1)^{j}\omega\wedge\partial(T)$. We form the resulting complex
\begin{eqnarray*}
0\longrightarrow S^kV^*\xrightarrow{\partial}{}V^*\otimes S^{k-1}V^*\xrightarrow{\partial}{}\wedge^2V^*\otimes S^{k-2}V^*\xrightarrow{\partial}{}\cdots\\
\cdots\xrightarrow{\partial}{}\wedge^{n}V^*\otimes S^{k-n}V^*\longrightarrow 0,
\end{eqnarray*}
where $S^lV^*=0$ for $l<0$. Of course, $\partial$ is just the restriction of the classical De Rham operator acting on $\Omega^{\bullet}(V)= C^{\infty}(V, \Lambda^{\bullet}V^*)$ to polynomial forms, where we identify 
$S^{p}V^*$ with the space of polynomial functions on $V$ of degree $p$. For that reason we will call $\partial$ {\bf formal differentiation}. In the presence of another vector space $W$, we will tensor the sequence above by $W$ and the operator $\partial$ by $\textrm{Id}_W$, keeping the same notation $\partial$. Note that, for a tableau $\g_k\subset S^{k}V^*\otimes W$, its first prolongation can also be described as
\[ \g_{k}^{(1)}= \{ T:V\to \g_k\ \textrm{linear}\mid \partial(Tv)u=\partial(Tu)v\ \forall \ u, v\in V\}.\]
Hence it can be interpreted as a tableau on $(V, \g_k)$. It is not difficult to check that the complex above tensored by $W$ contains the sequence 
\begin{equation}\label{Spencer-cohomology}\begin{aligned}
0\longrightarrow \mathfrak{g}_k^{(p)}\xrightarrow{\partial}{}&V^*\otimes\mathfrak{g}_k^{(p-1)}\xrightarrow{\partial}{}\wedge^2V^*\otimes\mathfrak{g}_k^{(p-2)}\xrightarrow{\partial}{}\cdots\\&
\cdots\xrightarrow{\partial}{}\wedge^{p}V^*\otimes \mathfrak{g}_k\xrightarrow{\partial}{} \wedge^{p+1}V^*\otimes S^{k-1}V^*\otimes W
\end{aligned}\end{equation}
 as a sub-complex.

\begin{definition} Given a tableau $\mathfrak{g}_k\subset S^{k}V^*\otimes W$ of order $k$, 
\mbox{}\begin{itemize}
\item The {\bf Spencer cohomology} of $\mathfrak{g}_k$ is given by the cohomology groups of the sequences (\ref{Spencer-cohomology}). More precisely, we denote by $H^{k+p-j,j}(\mathfrak{g}_k)$ (or simply $H^{k+p-j,j}$) the cohomology at $\wedge^jV^*\otimes_R\mathfrak{g}_k^{(p-j)}$. 
\item We say that $\g_k$ is {\bf involutive} if $H^{p,q}=0$ for all $p\geq k, q\geq 0.$
\item We call $\g_k$ {\bf $r$-acyclic} if $H^{k+p,j}= 0$ for  $p\geq 0,0\leq j\leq r$.
\end{itemize}
\end{definition}

 Spencer cohomology has been studied in the context of PDE's. See for example the work of Spencer \cite{Spencer,Spencerflatt}.


Although most of the time we will consider tableaux $\g$ of order $1$, we will need a slight generalization in which the inclusion $\g\hookrightarrow V^* \otimes W$ is replaced by an arbitrary linear map
 \begin{eqnarray*}
 \varphi:\g\To V^*\otimes W .
 \end{eqnarray*}

 \begin{definition}\label{1st-prolongation}
The {\bf first prolongation of $\g$ with respect to $\varphi$} (or simply of $\varphi$) is 
\begin{eqnarray*}
\g^{(1)}(\varphi):= \{T:V\to \g\ \textrm{\rm linear}\mid \varphi(Tu)v=\varphi(Tv)u\ \forall\ u,v\in V\}.\end{eqnarray*} 
\end{definition}

Interpreting $\g^{(1)}(\varphi)$ as a tableau on $(V, \g)$, one can prolong it repeatedly, giving rise to the higher prolongations
\[ \g^{(k)}(\varphi)\subset S^kV^*\otimes \g .\]
One also has a version of the Spencer sequences (\ref{Spencer-cohomology}). For that, we extend $\varphi$ to a linear map
\begin{eqnarray*}
 \partial_{\varphi}:\wedge^jV^*\otimes \g\To \wedge^{j+1}V^*\otimes W
 \end{eqnarray*}
 by mapping $\omega\otimes T$ to $(-1)^j\omega\wedge\varphi(T)$. In the next lemma, we use the notation $\g^{(k)}= \g^{(k)}(\varphi)$. 

\begin{lemma}
The Spencer complex of the tableau $\mathfrak{g}^{(1)}$ extends to the complex
\begin{eqnarray}\label{extension}
\begin{aligned}
0\longrightarrow \mathfrak{g}^{(k)}\xrightarrow{\partial}{}V^*\otimes&\mathfrak{g}^{(k-1)}\xrightarrow{\partial}{}\wedge^2V^*\otimes\mathfrak{g}^{(k-2)}\xrightarrow{\partial}{}\cdots\\
&\xrightarrow{\partial}{}\wedge^{k-1}V^*\otimes\mathfrak{g}^{(1)}\xrightarrow{\partial}{}\wedge^kV^*\otimes \mathfrak{g}\xrightarrow{\partial_{\varphi}}{} \wedge^{k+1}V^*\otimes W,
\end{aligned}
\end{eqnarray}
i.e. $\partial_{\varphi}\circ\partial=0,$ where $k\geq0$.
\end{lemma}

\begin{proof}
Notice that the sequence 
\begin{eqnarray*}
\begin{aligned}
0\To\g^{(1)}\xrightarrow{\partial} V^*\otimes \g\xrightarrow{\varphi} \wedge^2V^*\otimes W
\end{aligned}
\end{eqnarray*}
is exact and that for $\omega\in\wedge^jV^*$ and $\phi\in \g^{(1)}$, 
\begin{eqnarray*}\varphi\circ\partial(\omega\otimes \phi)=-\omega\wedge\varphi(\partial (\phi)).\end{eqnarray*}
\end{proof}

\begin{definition}\label{exten}
The {\bf $\varphi$-Spencer cohomology of $\g$} is the cohomology of the sequence \eqref{extension}.
\end{definition}


Next, we extend the previous discussion to vector bundles over a manifold $M$. So, assume now that $V$ and $W$ are smooth vector bundles over $M$. 

\begin{definition}
A \textbf{ bundle of tableaux} on $(V, W)$ (or \textbf{tableau bundle} over $M$) is any vector sub-bundle 
\begin{eqnarray*}
\g_k\subset S^kV^*\otimes W,
\end{eqnarray*} 
which is not necessarily of constant rank. 
\end{definition}

For a bundle of tableaux the notions of prolongation, involutivity, Spencer cohomology and so on are defined point-wise. Of course the prolongation $\g_k^{(p)}$ may fail to be smooth even if the starting $\g_k$ is. The following is a cohomological criteria for smoothness of the prolongation. We refer to \cite{Gold2} for the proof.

\begin{lemma}\label{smooth} Assume that $\g_k$ is the kernel of a morphism of smooth vector bundles $\Psi:S^{k}V^*\otimes W\to E$ over $M$. If $\g_k$ is $2$-acyclic and $\g_{k}^{(1)}$ is smooth, then all prolongations $\g_k^{(p)}$ are smooth.
\end{lemma}
 
 For the first prolongation of a vector bundle map $\varphi:\g\to V^*\otimes W$ we have a similar result, namely

\begin{lemma}\label{exact}
If $\g^{(1)}:=\mathfrak{g}^{(1)}(\varphi)$ is a vector bundle over $M$ and the sequence
\begin{eqnarray*}
0\longrightarrow \mathfrak{g}^{(m)}\xrightarrow{\partial}{}V^*\otimes\mathfrak{g}^{(m-1)}\xrightarrow{\partial}{}\wedge^2V^*\otimes\mathfrak{g}^{(m-2)}\xrightarrow{\partial}{}\wedge^3T^*\otimes_R\mathfrak{g}^{(m-3)}
\end{eqnarray*}
is exact for all $m\geq2$, then $\mathfrak{g}^{(k)}$ is smooth for $k\geq1$. Here $\g^{-1}$ denotes $W$.
\end{lemma}

\begin{proof}
The proof is an inductive argument and is basically the same as that of lemma 6.5 of \cite{Gold2}.
The exact sequence of vector bundles
\begin{eqnarray*}
V^*\otimes\mathfrak{g}^{(1)}\xrightarrow{\partial}{}\wedge^2V^*\otimes\mathfrak{g}\xrightarrow{\varphi}{}\wedge^3V^*\otimes W
\end{eqnarray*}
induces the exact sequence
\begin{eqnarray*}
0\longrightarrow\partial(V^*\otimes\mathfrak{g}^{(1)})\longrightarrow\wedge^2V^*\otimes\mathfrak{g}\xrightarrow{\varphi}{}\wedge^3V^*\otimes W.
\end{eqnarray*}
This shows that $\partial(V^*\otimes\mathfrak{g}^{(1)})$ is the kernel of a vector bundle map and therefore the function $M\ni x\mapsto\dim\partial(V^*\otimes\mathfrak{g}^{(1)})$ is upper semi-continuous.
On the other hand, by definition of $\mathfrak{g}^{(2)}$ (the first prolongation of $\mathfrak{g}^{(1)}$), one has that it is the kernel of the vector bundle map given by the composition
\begin{eqnarray*}
S^2V^*\otimes\mathfrak{g}\xrightarrow{\Delta_{1,1}}{}V^*\otimes V^*\otimes\mathfrak{g}\xrightarrow{id\otimes pr}{}V^*\otimes(V^*\otimes\mathfrak{g})/\mathfrak{g^{(1)}}
\end{eqnarray*}
which implies, once again, that $M\ni x\mapsto\dim\delta(\mathfrak{g}^{(2)})$ is upper semi-continuous.
Now, taking the Euler-Poincare characteristic of the exact sequence 
\begin{eqnarray*}
0\longrightarrow \mathfrak{g}^{(2)}\xrightarrow{\partial}{}V^*\otimes\mathfrak{g}^{(1)}\xrightarrow{\partial}{
}\partial(V^*\otimes\mathfrak{g}^{(1)})\xrightarrow{}{}0 
\end{eqnarray*}
one has that $\dim(\mathfrak{g}^{(2)})+\dim\partial(V^*\otimes\mathfrak{g}^{(1)})$ is locally constant since $\mathfrak{g}^{(1)}$ is a vector bundle, and therefore $\mathfrak{g}^{(2)}$ and $\partial(V^*\otimes\mathfrak{g}^{(1)})$ are vector bundles over $M.$

For $l>2$, consider the exact sequence 
\begin{eqnarray*}
0\longrightarrow \mathfrak{g}^{(l)}\xrightarrow{\partial}{}V^*\otimes\mathfrak{g}^{(l-1)}\xrightarrow{\partial}{}\wedge^2V^*\otimes\mathfrak{g}^{(l-2)}\xrightarrow{\partial}{}\wedge^3V^*\otimes\mathfrak{g}^{(l-3)}
\end{eqnarray*}
which shows by lemma 3.3 in \cite{Gold2} that $\mathfrak{g}^{(l)}$ is a vector bundle whenever $\mathfrak{g}^{(l-1)}$, $\mathfrak{g}^{(l-2)}$ and $\mathfrak{g}^{(l-3)}$ are vector bundles.
\end{proof}

Another important result is the so called $\partial$-Poincar\'{e} lemma (proved e.g. in \cite{quillen}), a version of which is:

\begin{lemma}\label{poincare}Let $\g_k\subset S^{k}V^*\otimes W$ be a bundle of tableaux.  
There exists an integer $p_0>0$, such that $\g_{k}^{(p)}$ is involutive for all $p\geq p_0$.
\end{lemma}

\subsection{Towers of tableaux}

Let $V$ and $W$ be two vector spaces. Our main reference for this part is \cite{SingerSternberg}.

\begin{definition}A {\bf tableaux tower} (on $(V, W)$) is a sequence 
\begin{eqnarray}\label{es}
\g^{\infty}= (\g^{1},\g^{2},\ldots,\g^{p},\ldots)\end{eqnarray} 
consisting of tableaux $\g^{p}\subset S^{p}V^*\otimes W$ for $p\geq 1$ such that
each $\g^{p+1}$ is inside the prolongation of $\g^{p}$:
\begin{eqnarray*}
\g^{p+1}\subset (\g^{p})^{(1)}.
\end{eqnarray*}
\end{definition}

Of course, one may want to consider towers starting with a tableaux of order $k\geq 1$
\[ (\g^k, \g^{k+1}, \ldots ),\]
but any such object can be completed to a tower in the previous sense by adding
\[ \g^i= S^iV^*\otimes W,\ \ \forall \ 1\leq i\leq k-1 .\]
We see that, in particular, any tableaux $\g_{k}\subset S^{k}V^*\otimes W$ defines a tower
\[ (V^*\otimes W, \ldots, S^{k-1}V^* \otimes W, \g_k, \g_{k}^{(1)}, \g_{k}^{(2)}, \ldots )\]
by prolongation. The Spencer cohomology of a tableau can then be viewed as an instance of the Spencer cohomology of a tower. More precisely, for a tower $\g^{\infty}$,
the formal differential 
\begin{eqnarray*}
\partial:\wedge^kV^*\otimes (\g^{p})^{(1)}\To \wedge^{k+1}V^*\otimes \g^{p}
\end{eqnarray*}
restricts to 
\begin{eqnarray*}
\partial:\wedge^kV^*\otimes \g^{p+1}\To \wedge^{k+1}V^*\otimes \g^{p}
\end{eqnarray*}
for any $p\geq0$, and this induces the Spencer complexes of the tower $\g^{\infty}$
given by
\begin{equation}\label{Spencer-cohomology2}\begin{aligned}
0\longrightarrow \mathfrak{g}^{k+p}\xrightarrow{\partial}{}&V^*\otimes\mathfrak{g}^{k+p-1}\xrightarrow{\partial}{}\wedge^2V^*\otimes\mathfrak{g}^{k+p-2}\xrightarrow{\partial}{}\cdots\\&
\cdots\xrightarrow{\partial}{}\wedge^{p-k}V^*\otimes \mathfrak{g}^k\xrightarrow{\partial}{} \wedge^{p+1}V^*\otimes S^{k-1}V^*\otimes W.
\end{aligned}\end{equation}

\begin{definition}
The \textbf{Spencer cohomology of the sequence of the tableaux tower $\g^{\infty}$} is the cohomology of the sequence \eqref{Spencer-cohomology2}. As before, denote by $H^{k+p-j,j}(\g^{\infty})$, the cohomology of the sequence \eqref{Spencer-cohomology2} at $\wedge^jV^*\otimes\g^{k+p-j}$.
\end{definition}

\begin{lemma}\label{lemma:prol}
For any tower of tableaux $\g^{\infty}$, there exists an integer $p_0$ such that, for all $p\geq p_0$ 
\begin{eqnarray*}
\g^{p+1}= (\g^{p})^{(1)}.
\end{eqnarray*}
\end{lemma} 

For the proof we refer to \cite{SingerSternberg}. Finally, we will also encounter bundles of towers of tableaux. Let $V$ and $W$ be two vector bundles over the manifold $M$.

\begin{definition}\label{tableaux tower} A {\bf bundle of tableaux towers over $M$} (on $(V, W)$) is a sequence of smooth vector bundles over $M$
\begin{eqnarray*}(\g^1,\g^{2},\ldots,\g^{k},\ldots), \end{eqnarray*}
with $\g^k\subset S^{k}V^*\otimes W$ and such that, for all $p\geq 1$
\begin{eqnarray*}
\g^{p+1}\subset(\g^{p})^{(1)}.
\end{eqnarray*}
\end{definition}

Here we state the analogue of lemma \ref{lemma:prol} in the smooth case. For the proof see for example \cite{Gold3}.

\begin{proposition}\label{stability}Let $(\g^1,\g^{2},\ldots,\g^{k},\ldots)$ be a bundle of tableaux towers over a connected manifold $M$. Then there exists an integer $p_0$ such that for all $p\geq p_0$
\begin{eqnarray*}
\g^{p+1}=(\g^{p})^{(1)}.
\end{eqnarray*}
\end{proposition}

\section{Lie groupoids and Lie algebroids}

We refer the reader to \cite{CrainicFernandes:lecture} for a more complete description of the theory of Lie groupoids and Lie algebroids.

\subsection{Lie groupoids}
\label{Lie groupoids}

Recall that a \textbf{groupoid} $\G$ is a (small) category in which every arrow is invertible. That means that we have a set $\G$ of arrows and a set $M$ of objects equipped with the following structure maps:
\begin{enumerate} 
\item the source and target
\[ s, t: \G\rmap M,\]
\item the composition
\[ m: \G_{2}\rmap \G, \]
where $\G_2$ is the set of composable arrows
\[ \G_2= \{(g, h): g, h\in \G, s(g)= t(h) \},\]
\item the unit 
\[ u: M\rmap \G ,\]
\item the inversion
\[ i: \G\rmap \G .\]
\end{enumerate}

We will identify the groupoid $\G$ with its set of arrows and we will say that $\G$ is a groupoid over $M$ or that $M$ is the base of the groupoid $\G$, using the graphic notation 
\[ \G \tto M .\]
The elements of $\G$ will be called arrows of the groupoid and will be denoted
by letters $g, h, \gamma, $ etc, while the elements of $M$ will be called points of the groupoid and will be denoted by letters $x, y, p, $ etc. An arrow $g$ from $x$ to $y$ is any arrow $g\in \G$ with $s(g)= x$, $t(g)= y$; in this case we use the
graphic notation
\[ g: x\rmap y, \ \ \textrm{or} \ \ x \stackrel{g}{\rmap} y \ \ \textrm{or}\ \ y \stackrel{g}{\lmap} x .\]
For the structure maps, we will use the notation
\[ m(g, h)= g\cdot h= gh, \ \ u(x)= 1_x, \ \ i(g)= g^{-1}.\]
With these, the groupoid axioms take the familiar form that reminds us of composition of functions:
\begin{enumerate}
\item if $z \stackrel{g}{\lmap} y \stackrel{h}{\lmap} x$, then $z \stackrel{gh}{\lmap} x$ and the composition is associative. 
\item for $x\in M$, $1_x: x\rmap x$. 
\item for any  $g: x\rmap y$, $g\cdot 1_x= 1_y\cdot g= g$.
\item if $g: x\rmap y$ then $g^{-1}: y\rmap x$ and 
\[ g^{-1}\cdot g= 1_x,\ \ g\cdot g^{-1}= 1_y.\]
\end{enumerate}

\begin{definition} A \textbf{smooth} (or \textbf{Lie}) groupoid $\G$ over $M$ is any groupoid $\G \tto M$ endowed with smooth (manifold) structures on $\G$ and $M$, such that
$s, t: \G\to M$ are submersions and all the other structure maps $m$, $u$ and $i$ are smooth.
\end{definition}

Henceforth, all objects in this thesis will be smooth unless otherwise stated. Note that the conditions on $s$ and $t$ imply that the set $\G_2$ of composable arrows is a smooth submanifold of $\G\times \G$; hence it makes sense to say that $m$ is smooth. Note also that, in the smooth case, $i$ is a diffeomorphism and $u$ is an embedding. Actually, we will often identify $M$ with the
submanifolds of units via 
\[ u: M\hookrightarrow \G .\]

One of the main notions used in this thesis is that of a bisection.

\begin{definition} \label{definition-bisections}
Given a Lie groupoid $\G$ over $M$, a \textbf{bisection} of $\G$ is any splitting $b: M\to \G$ of the source map with the property that $\phi_b:= t\circ b: M\to M$ is a diffeomorphism.
\end{definition}

The bisections of $\G$ form a group $\textrm{Bis}(\G)$ with multiplication and inverse given by
\[b_1\cdot b_2(x) = b_1(\phi_{b_2}(x))b_2(x),\ b^{-1}(x)=i\circ b\circ\phi_b^{-1}(x).\]

\begin{example}\rm \
A Lie group is a Lie groupoid over a point. The group of bisections coincides with the group itself.
\end{example}

\begin{example}\rm \  Bundles of Lie groups over $M$ can be seen as Lie groupoids over $M$ with $s= t$. In particular, any vector bundle $\pi: E\to M$ can be interpreted as a Lie groupoid over $M$ with $s= t:= \pi$
and with composition being the fiberwise vector bundle addition.
\end{example}

\begin{example}\rm \
For any manifold $M$, the pair groupoid of $M$ is 
\[ \Pi(M):= M\times M\]
with source and target
\[ s(x, y)= y, \ t(x, y)= x\]
and with composition 
\[ (x, y)\cdot (y, z)= (x, z).\]
For the group of bisections, we find that $\textrm{Bis}(\Pi(M))= \textrm{Diff}(M)$. 
\end{example}

\begin{example}\label{ex-action-groupoids} \rm \ For any Lie group $G$ acting on a manifold $M$ (say on the left) one forms the action groupoid $G\ltimes M$, whose space of arrows is the product $G\times M$, the base is $M$,  
the source and target maps are
\[ s(g, x)= x, \ t(g, x)= gx,\]
and multiplication is given by 
\[ (h, gx)\cdot (g, x)= (hg, x).\]
\end{example}

\begin{example}\rm \  For any vector field $X$ on a manifold $M$, the domain $\mathcal{D}(X)$ of the flow $\phi_{X}^{\epsilon}$ of $X$,
\[ \mathcal{D}(X)= \{ (\epsilon, x)\in \mathbb{R}\times M: \ \phi_{X}^{\epsilon} \ \textrm{is\ defined} \}\]
can be seen as a groupoid over $M$ with source and target
\[ s(\epsilon, x)= x,\ t(\epsilon, x)= \phi_{X}^{\epsilon}(x),\]
and composition
\[ (\epsilon', \phi_{X}^{\epsilon}(x))\cdot (\epsilon, x)= (\epsilon+ \epsilon', x).\]
Note that, when $X$ is complete, then $\mathcal{D}(X)$ is the action groupoid $\mathbb{R}\ltimes M$ (see the previous example) associated to the global flow of $X$ interpreted as an
action of $\mathbb{R}$ on $M$. 
\end{example}

\begin{example}\label{gauge-groupoids} \rm For any principal $G$-bundle $\pi: P\to M$ the gauge groupoid of $P$, denoted $\G\textrm{auge}(P)$,
is defined as the quotient of the pair groupoid $\Pi(P)$ modulo the (diagonal) action of $G$. Hence 
\[ \G\textrm{auge}(P)= (P\times P)/G,\]
is a Lie groupoid over $P/G= M$, with source and target
\[ s([p, q])= \pi(q),\ t([p, q])= \pi(p).\]
For the bisections, we find that $\textrm{Bis}(\G\textrm{auge}(P))$ is isomorphic to the automorphism group of $P$. 

Note that the gauge groupoid is transitive, in the sense that any two points of $M$ are related by an arrow of the groupoid. Conversely, any transitive groupoid $\G$ over $M$ must be of this type. Indeed, 
fixing a base point $x\in M$, it follows that
\[ G_x:= \{g\in \G: s(g)= t(g)= x\}\]
is a Lie group,
\[ P_x:= s^{-1}(x)= \{g\in \G: s(g)= x\}\]
is a principal $G_x$-bundle over $M$ with projection $t$, and 
\[ \G\textrm{auge}(P_x)\rmap \G,\ \ [h, k]\mapsto hk^{-1}\]
is an isomorphism of Lie groupoids. 
\end{example}

\begin{example}\label{GL-groupoids} \rm \ While the general linear group $GL(V)$ associated to a (finite dimensional) vector space $V$ is a Lie group, the similar object $GL(E)$ associated to a vector bundle $\pi: E\to M$ is a Lie groupoid over $M$. 
More precisely, an arrow of $GL(E)$ between two points $x, y\in M$ is a linear isomorphisms $E_x\stackrel{\sim}{\to} E_y$ and the multiplication is given by the usual composition of maps. There is a canonical smooth structure on 
$GL(E)$ which makes it into a Lie groupoid. Alternatively, one can realize $GL(E)$ as a gauge groupoid in the sense of the previous example. More precisely, consider the frame bundle $Fr(E)$ associated to $E$
\[ Fr(E)= \{ (x, u)\mid x\in M, u: \mathbb{R}^r \to E_x\ \ \textrm{linear\ isomorphism} \},\]
where $r$ is the rank of $\pi$. Then $Fr(E)$ is a principal $GL_r$-bundle and one has a simple isomorphism of Lie groupoids
\[ \G\textrm{auge}(Fr(E)) \rmap GL(E), \ [(u, v)] \mapsto u\circ v^{-1} .\]

Note that, for an arbitrary Lie groupoid $\G$ over $M$, a representation of $\G$ on $E$ (see definition \ref{def:representation}) is the same thing as a Lie groupoid homomorphism $\G\to \textrm{GL}(E)$.
\end{example}

There are other important examples arising from foliation theory or from Poisson geometry. For this thesis the most important examples are jet groupoids, whose simplest version is introduced below.

\begin{example}[cf. \cite{Veloso, KumperaSpencer}]\label{difeo}\rm \
Given a manifold $M$, consider the set $\textrm{Diff}_{\textrm{loc}}(M)$ of all diffeomorphisms $\phi: U\to V$ between two open sets $U, V\subset M$; $U$ will be called the domain of $\phi$ and will be denoted by $\textrm{Dom}(\phi)$.
For any integer $k\geq 0$, one forms the groupoid of $k$-jets of local diffeomorphisms of $M$:
\[ J^k(M, M):= \{ j^{k}_{x}\phi: \ x\in M,\ \ \phi\in \textrm{Diff}_{\textrm{loc}}(M), \ x\in \textrm{Dom}(\phi)\},\]
which is a Lie groupoid over $M$ with source and target 
\[ s(j^{k}_{x}\phi)= x, \ \ t(j^{k}_{x}\phi)= \phi(x),\]
and composition 
\[ j^{k}_{\psi(x)}\phi\cdot j^{k}_{x}\psi= j^{k}_{x}(\phi\circ \psi).\]
Note that $J^k(M, M)$ is a transitive Lie groupoid (see example \ref{gauge-groupoids}); hence, fixing a base-point $O\in M$, one can consider the Lie group 
\[ G^{k}:= \{j^{k}_{O}\phi: \phi\in \textrm{Diff}_{\textrm{loc}}(M), \ O\in \textrm{Dom}(\phi), \phi(O)= O \},\]
the $k$-th order frame bundle (with origin $O$) 
\[ F^{k}:= \{j^{k}_{O}\phi: \phi\in \textrm{Diff}_{\textrm{loc}}(M), \ O\in \textrm{Dom}(\phi)\}\]
and then $F^k$ is a principal $G^k$-bundle over $M$ whose associated gauge groupoid is precisely $J^k(M, M)$.

Note also that, for $k= 0$, $J^0(M, M)$ is just the pair groupoid of $M$. Also, for $k= 1$, $J^1(M, M)$ is just the general linear groupoid $GL(TM)$ (see example \ref{GL-groupoids}); equivalently, 
$G^1$ is isomorphic to $GL_n$ ($n= \textrm{dim}(M)$) and $F^1$ is isomorphic to the frame bundle of $M$. 
\end{example}

Another notion central to this thesis is that of a representation of a Lie groupoid. Recall that, given a Lie groupoid $\G$ over $M$ and a bundle $\mu: P\to M$, an action of $\G$ on $P$ (via $\mu$) associates to any arrow $g: x\to y$ of $\G$ a map
\[ g\cdot \  : \mu^{-1}(x)\rmap \mu^{-1}(y), \ \ p\mapsto g\cdot p= gp\]
such that the usual algebraic axioms for actions are satisfied ($(gh)\cdot p= g\cdot (h\cdot p)$ whenever $g$ and $h$ are composable, and $1_x\cdot p= p$ for $p\in \mu^{-1}(x)$), and such that the action is smooth, i.e. the map 
\[ \G \times_{s,\mu} P \longrightarrow P,\ \ (g, p)\mapsto g\cdot p \] 
defined on the space of pairs $(g, p)$ with $s(g)= \mu(p)$ (a smooth submanifold of $\G\times P$) is smooth.

\begin{definition}\label{def:representation} Let $\G$ be a Lie groupoid over $M$. A \textbf{representation} of $\G$ is a vector bundle $\mu: E \to M$ together with a linear action of $\G$ on $E$, i.e. an action with the property that, for each $g: x\to y$, the multiplication $E_x\to E_y, v\mapsto g\cdot v$ is linear. We denote by $ \textrm{\rm Rep}(\G)$ the set of representations of $\G$.
\end{definition}

\begin{example}\rm \ A representation of $\G$ on a vector bundle $E$ is the same thing as a Lie groupoid homomorphism $\G\to \textrm{GL}(E)$.
This notion generalizes the usual notion of representations of Lie groups. For an action groupoid 
$G\ltimes M$ (example \ref{ex-action-groupoids}), representations are the same thing as equivariant vector bundles over $M$. 
For a general Lie groupoid $\G$ over $M$, there are very few representations available ``for free''. Of course, there is the trivial representation $\mathbb{R}_M$ whose underlying vector bundle is the trivial line bundle and the action is the identity on the fibers. However, there is no analogue of the adjoint representation for Lie groups (see also below).
\end{example}

\begin{example}[\cite{GuilleminSternberg:deformation}]\label{adjoint-for-classical}\rm \ Consider the $k$-jet groupoids $\Jet^k(M, M)$ from example \ref{difeo}. For $k= 1$ it is clear that 
\[ TM\in \textrm{Rep}(\Jet^1(M, M)),\]
i.e. there is a natural action of $\Jet^1(M, M)$ on $TM$: for $g= \jet^{1}_{x}\phi$, the induced linear action is
\[ d_x\phi: T_xM\rmap T_{\phi(x)}M.\]
Similarly, for any $k$ one has
\[ \Jet^{k-1}TM\in \textrm{Rep}(\Jet^k(M, M)).\]
To describe the action, let $g= j^{k}_{x}\phi\in J^k(M, M)$. Due to the naturality of our objects, the bundle map $d\phi: TM\to TM$ (covering $\phi$) induces a map $j^{k-1}d\phi$ on $J^{k-1}TM$; explicitely, 
\[ j^{k-1}_xd\phi: J^{k-1}_{x} TM\rmap J^{k-1}_{\phi(x)} TM,\ \ j_{x}^{k-1}(X)\mapsto j^{k-1}_{\phi(x)}(\phi_*(X)),\]
where recall that the push-forward vector field is given by
\[ \phi_*(X)_y:= (d\phi)_{\phi^{-1}(y)}(X_{\phi^{-1}(y)}).\]
Of course, $j^{k-1}_xd\phi$ only depends on $j^{k}_{x}\phi= g$, and this defines the linear action of $g$. 
\end{example}

\subsection{Lie algebroids}\label{Lie algebroids}

Lie algebroids arise as the infinitesimal counterpart of Lie groupoids. Abstractly, they can be thought of as ``replacements of tangent bundles'' (which are related to the ordinary tangent bundle by the anchor map), which are better suited to
reflect the various geometric structures under consideration.

\begin{definition}
A \textbf{Lie algebroid} over a manifold $M$ is a vector bundle $A \to M$ endowed with a vector bundle map, called anchor, $\rho: A \to TM$, and a Lie bracket on the space $\Gamma(A)$ of sections of $A$ such that the following Leibniz identity
\[[\al, f\be] = f[\al,\be] + (\Lie_{\rho(\al)}f)\be\]
holds for all sections $\al,\be \in \Gamma(A)$ and smooth functions $f \in \mathrm{C}^{\infty}(M)$. 
\end{definition}

Let us first recall the construction of the Lie algebroid of a Lie groupoid $\G$ over $M$. This is completely analogous to the construction of the Lie algebra of a Lie group, as the tangent of the group at the unit element, or as the algebra of right invariant vector fields on the group. The main differences come from the fact that
\begin{enumerate}
\item there are many units -- one for each $x\in M$ -- hence one ends up with a vector bundle over $M$; 
\item right translations by an element $g\in \G$ are defined only along $s$-fibers: the formula $R_{g}(a)= ag$ defines a map
\[ R_g: s^{-1}(y)\rmap s^{-1}(x) ,\]
where $x= s(g)$, $y= t(g)$.
\end{enumerate}
Putting everything together, the relevant infinitesimal object will be the vector bundle over $M$ whose fiber over $x\in M$ is the tangent space at the unit $1_x$ of the $s$-fiber $s^{-1}(x)$. More formally, we consider the $s$-tangent bundle
\[ T^s\G:= \textrm{Ker}(ds) \]
and the Lie algebroid of $\G$ is, as a vector bundle over $M$, the restriction of $T^s\G$ to $M$ (the pull-back via the unit map $u: M\to \G$):
\[ A= (T^s\G)|_{M}.\]
The anchor of $A$ is simply the restriction of $dt: T\G\to TM$ to vectors that belong to $A$. For the bracket, we identify the space of sections $\Gamma(A)$ with the
right invariant vector fields on $\G$. More precisely, right translation by $g$ differentiates to give a linear map
\[ R_g: T^{s}_{a}\G \rmap T^{s}_{ag}\G\]
(for $a\in s^{-1}(y)$); with this, the space of right invariant vector fields on $\G$ is
\[ \X^{\textrm{inv}}(\G)= \{X\in \Gamma(T^s\G): R_g(X_a)= X_{ag} \ \forall\ a, g\in \G\ \textrm{composable} \}.\]
Finally, there is a 1-1 correspondence 
\[ \Gamma(A) \cong \X^{\textrm{inv}}(\G),\ \ \alpha\mapsto \alpha^r,\]
where 
\[ \alpha^{r}_{g}= R_g(\alpha_{t(g)}).\]
This induces the Lie bracket of $A$, characterized by:
\[ [\alpha, \beta]^r= [\alpha^r, \beta^r],\]
where the second bracket is the usual Lie bracket of vector fields on $\G$. 
To keep some formulas simpler, we will sometimes also use the notation
\begin{equation}
\label{notation-right-invariant} 
\stackrel{\rightarrow}{\alpha}= \alpha^{r}
\end{equation}
for right invariant vector fields.

\begin{example}\rm \ 
A Lie algebra is a Lie algebroid over a point.
\end{example}

\begin{example} \rm \ 
Bundles of Lie algebras can be viewed as Lie algebroids with $\rho= 0$.
\end{example}

\begin{example}[See \cite{Veloso}]\label{ex-0-jet-gpd} \rm \ 
The tangent bundle of any manifold $M$ is a Lie algebroid with $\rho= \textrm{Id}$ and the Lie bracket the usual Lie bracket of vector fields. As such, $TM$ coincides with the Lie algebroid of the pair groupoid $\Pi(M)$. To fix the notation for computations in local coordinates, let us describe this more explicitely in the case of the pair groupoid $\Pi$ of $\mathbb{R}^n$. As a convention, we denote the source coordinates by $x_a$, the target coordinates by $X_i$. Hence, in local coordinates, 
\[\Pi = \{(x_1, \ldots , x_n, X_1, \ldots , X_n)\}\]
with the source, target and multiplication
\[ s(x, X)= x, \ t(x, X)= X, \ (x, X)\cdot (x', x)= (x', X).\]
Denote the Lie algebroid of $\Pi$ by $A(\Pi)$, and let's calculate it explicitly by applying the definition. Since the unit of $\Pi$ at $x\in \mathbb{R}^n$ is $1_x= (x, x)$, the fiber of $A(\Pi)$ at $x$ is (by definition)
\[ A(\Pi)_x= T_{1_x} s^{-1}(x)= T_{(x, x)} \{(x, X): X-\textrm{variable} \}.\]
The canonical basis at $x$ is then
\[ \partial^{i}(x):= \frac{\partial}{\partial X_i} (x, x) \in A(\Pi)_x,\ 1\leq i\leq n .\]
In some examples (e.g. in low dimensions when the variables are denoted $x, y, z,$ etc), it is more natural to use the notation
\[  \partial^{X_i}:= \partial^{i}.\]
Hence, as a vector bundle, $A(\Pi)$ is spanned by
\[ \{ \partial^1, \ldots , \partial^n\}= \{\partial^{X_1}, \ldots , \partial^{X_n}\}.\]
The induced right invariant vector fields on $\Pi$ (tangent to $s$-fibers) are (in the notation (\ref{notation-right-invariant}))
\begin{equation}\label{right-inv-local} 
\stackrel{\rightarrow}{\partial^{i}}(x, X)= \frac{\partial}{\partial X_i} (x, X)\in T_{x, X}\Pi .
\end{equation}
For the anchor of $A(\Pi)$ we find
\[ \rho: A(\Pi)\rmap T\mathbb{R}^n,\ \rho(\partial^{i})= \frac{\partial}{\partial X_i}\]
and this provides an identification between $A(\Pi)$ and $T\mathbb{R}^n$. 
\end{example}

\begin{example}\label{inf-act}\rm \ Let $\mathfrak{g}$ be a Lie algebra and assume that we have given an infinitesimal action of $\mathfrak{g}$ on a manifold $M$, i.e. a Lie algebra map $a: \mathfrak{g}\to \X(M)$. Then one forms the action Lie algebroid
$\mathfrak{g}\ltimes M$ as follows. As a vector bundle, it is the trivial vector bundle $\mathfrak{g}_{M}$ over $M$ with fiber $\mathfrak{g}$. The anchor is precisely the infinitesimal action interpreted as a vector bundle map 
\[ \rho: \mathfrak{g}_M\rmap TM, \ \ (x, u)\mapsto a(u)_x .\]
For the bracket, we note that $\Gamma(\mathfrak{g}_{M})= C^{\infty}(M, \mathfrak{g})$ contains the constant sections $c_u$ with $u\in \mathfrak{g}$; with this, the bracket of $\mathfrak{g}\ltimes M$ is defined on the constant sections by 
\[ [c_u, c_v]= c_{[u, v]_{\mathfrak{g}}},\]
(where the bracket on the right hand side is the one on $\mathfrak{g}$) and extended to arbitrary sections using the Leibniz identity. For a global formula, using the canonical flat connection $\nabla^{\textrm{flat}}$ on $\mathfrak{g}_M$, we have:
\[ [\alpha, \beta]= [\alpha, \beta]_{\mathfrak{g}}+ \nabla^{\textrm{flat}}_{\rho(\alpha)}(\beta)- \nabla^{\textrm{flat}}_{\rho(\beta)}(\alpha).\]

Note that, if the infinitesimal action comes from a global action of a Lie group $G$ on $M$, then $\mathfrak{g}\ltimes M$ is precisely the Lie algebroid of the action groupoid $G\ltimes M$.
\end{example}

\begin{example}\label{ex-alg-J1}\rm \ In analogy with Example \ref{ex-0-jet-gpd}, for any manifold $M$, the Lie algebroid of the first jet groupoid $J^1(M, M)$ (Example \ref{difeo} for $k= 1$) is isomorphic to the bundle $J^1(TM)$ of first jets of vector fields on $M$. Let us start by recalling that $J^1(TM)$ has a canonical structure of Lie algebroid, 
with the anchor given by the canonical projection $l: J^1(TM)\rmap TM$ and the bracket uniquely determined by the Leibniz identity and the condition 
\[ [j^1(V), j^1(W)]= j^1([V, W])\]
for all vector field $V, W$. Recall (see remark \ref{remark-J-decomposition}) that $T^*M\otimes TM$ is identified with a subspace of $J^1(TM)$ (namely with the kernel of the anchor $l$) and this identification reads, at the level of elements, as follows:
\begin{equation}
\label{eq-ident-Hom} 
df\otimes V= fj^1(V)-  j^1(fV).
\end{equation}
A simple computation shows that the bracket of $J^1(TM)$ restricts to $T^*M\otimes TM$ to:
\[ [df\otimes V, dg\otimes W]= L_W(f) dg\otimes V- L_V(g) df\otimes W .\]
Equivalently, $T^*M\otimes TM= \textrm{Hom}(TM, TM)$ is endowed with the pointwise standard commutator bracket 
\[ [T, S]= T\circ S- S\circ T ,\]
making it into a bundle of Lie algebras (hence a Lie algebroid with zero anchor). 

With these, $J^1(TM)$ is canonically isomorphic to the Lie algebroid of the first jet groupoid $J^1(M, M)$. As before, in order to fix the notation for computations in local coordinates, we describe this identification more explicitely in the case of the first jet groupoid $J^1$ of $\mathbb{R}^n$. First, 
\[ J^1= J^1(\mathbb{R}^n,  \mathbb{R}^n)= \{(x_1, \ldots , x_n , X_1, \ldots , X_n, p): p= (p^{i}_{a})_{1\leq i, a\leq n}) \in GL_n\} ,\]
where the last equality indicates the notation that we use for the coordinates in $J^1$ ($p^{i}_{a}$ corresponds to the partial derivative $\frac{\partial X_i}{\partial x_a}$).
The source is the projection on $x$, the target is the projection on $X$, while multiplication is
\[ (x, X, p)\cdot (x', x, q)= (x', X, pq),\]
where $pq$ uses matrix multiplication. For the algebroid $A(J^1)$ of $J^1$, since the unit of $J^1$ at $x\in \mathbb{R}^n$ is
\[ 1_x= (x, x, 1)\in J^1(\mathbb{R}^n,  \mathbb{R}^n),\]
the fiber of $A(J^1)$ above $x\in \mathbb{R}^n$ is  (by definition)
\[ A(J^1)_x= T_{1_x} s^{-1}(x) = T_{(x, x, 1)} \{(x, X, p): X, p-\textrm{variables}\},\]
with canonical basis
\[ \partial^{i}(x):= \frac{\partial}{\partial X_i} (x, x, 1), \  \partial^{i}_{a}(x):= \frac{\partial}{\partial p^{i}_{a}} (x, x, 1)\ \ (1\leq i, a\leq n) .\]
Hence 
\[ A(J^1)= \textrm{Span} \{\partial^{i}, \partial^{i}_{a}: 1\leq i, a\leq n \}.\]
For the anchor of $A(J^1)$ we find
\[ \rho(\partial^{i})= \frac{\partial}{\partial X_i},\ \rho( \partial^{i}_{a})= 0,\]
To compute the Lie bracket of $A(J^1)$, one first has to compute the 
corresponding right invariant vector fields on $J^1$. For this we use right translations: associated to $g= (x, X, p)\in J^1$, we have
\[ R_g: s^{-1}(X)\rmap s^{-1}(x), \ R_g(X, \tilde{X}, q)= (x, \tilde{X}, qp),\]
and then we compute
\[ \stackrel{\rmap}{\partial^{i}}= \frac{\partial}{\partial X_i},\  \stackrel{\rmap}{\partial^{i}_{a}}= \sum_{u} p^{a}_{u} \frac{\partial}{\partial p^{i}_{u}}  \ \ (\textrm{at\ any}\ (x, X, p)\in J^1(T\mathbb{R}^n)).\]
We then find:
\[ [\partial^{i}, \partial^{j}]= 0,\ [\partial^{i}_{a}, \partial^{j}]= 0, \ [\partial^{i}_{a}, \partial^{j}_{b}]=
\delta^{i}_{b} \partial^{j}_{a}- \delta^{j}_{a} \partial^{i}_{b},\]
where $\delta$ is the Kronecker symbol. The canonical identification between $J^1(T\mathbb{R}^n)$ and $A(J^1)$ can be described as follows: the first jet at $X$ of a vector field $V= V_i \frac{\partial}{\partial X_i}$ is identified with 
\[J^1(T\mathbb{R}^n)\ni j^1_{X}(V) \longleftrightarrow V_i(X)\partial^{i} + \frac{\partial V_i}{\partial X_a}(X)  \partial^{i}_{a}\in A(J^1).\]
We see that, in terms of jets, the canonical frame of $A(J^1)= J^1(T\mathbb{R}^n)$ is
\[ \partial^{i}= j^1(\frac{\partial}{\partial X_i}),\  \partial^{i}_{a}= j^1(X_a\frac{\partial}{\partial X_i})-X_a j^1(\frac{\partial}{\partial X_i})= -dX_a\otimes \frac{\partial}{\partial X_i}.\]
where, for the last equality, we used (\ref{eq-ident-Hom}). 

As in example \ref{ex-0-jet-gpd}, it is sometimes more appropriate to use different notations for the variables $p^{i}_{a}$ and the 
basis of $A$:
\[ p^{X_i}_{x_a}:= p^{i}_{a},\ \ \partial^{X_i}:= \partial^i, \ \ \partial^{X_i}_{x_a}= \partial^{p^{i}_{a}}:= \partial^{i}_{a} .\]
For instance, when $n= 2$, we have coordinates $(x, y)$ in $\mathbb{R}^2$, coordinates 
\[ \Bigg(x, y, X, Y, \left( \begin{array}{ll}
p^{X}_{x} & p^{X}_{y}\\
p^{Y}_{x} & p^{Y}_{y}\\
\end{array}\right) \Bigg)\]
for $J^1$ (where the matrix is invertible) and $A(J^1)$ is spanned by $\partial^X, \partial^Y, \partial^{X}_{x}, \partial^{X}_{y}, \partial^{Y}_{x}, \partial^{Y}_{y}$. 
\end{example}

\begin{example}\label{Cartan-forms-on-jets}\rm \  Of course, the previous example has a version for higher jets and the Lie algebroid of the $k$-jet groupoid $\Jet^k(M, M)$ is isomorphic to $\Jet^kTM$. Note that this implies that the resulting Cartan forms (coming from (\ref{eq-higher-Cartan}) of subsection \ref{PDEs; the Cartan forms:}) has simpler coefficients. Let us be more precise. First of all, considering the fibration $s= \textrm{pr}_1: M\times M\to M$, we see that $\Jet^k(M, M)$ is open inside $\Jet^kR$ ($R:=M\times M,\pi:= s$). Restricting the Cartan form (\ref{eq-higher-Cartan}) of $\Jet^kR$, we obtain a 1-form with values in 
\[ \textrm{pr}^*T^{s}\Jet^{k-1}(M, M) .\]
But $T^{s}\Jet^{k-1}(M, M)$ is precisely the pull-back by the target map of the Lie algebroid of $\Jet^{k-1}(M, M)$, i.e. of $\Jet^{k-1}TM$, hence the Cartan form becomes
\begin{equation}\label{eq-Cartan-forms-on-jets} 
\theta\in \Omega^1( \Jet^k(M, M), t^*\Jet^{k-1}TM).
\end{equation}
See also \cite{GuilleminSternberg:deformation}.
\end{example}

\begin{example}[\cite{Cartan1904}, see also \cite{Ori}] \label{toy-example}\rm \ 
Here is an explicit example which, as we shall explain later, arises from a Lie pseudogroup.  Over the base manifold
\[ M: =\{ (x, y)\in \mathbb{R}^2: y\neq 0\} \]
we consider the $5$-dimensional Lie groupoid
\[ \G= \{ (x, y, X, Y, u)\in \mathbb{R}^5: y\neq 0, Y\neq 0 \}\]
with source, target and multiplication given by:
\[ s(x, y, X, Y, u)= (x, y), \ t(x, y, X, Y, u)= (X, Y),\]
\[ (x, y, X, Y, u)(x', y', x, y, v)= (x', y', X, Y, \frac{y'u+ Yv}{y}).\]
Let us compute the Lie algebroid $A(\G)$ in two different ways. We first apply directly the definition. Since the unit at a point $(x, y)$ is $(x, y, x, y, 0)$, we find that $A(\G)$ is the trivial vector bundle over $\mathbb{R}^2$ spanned by $\{e^X, e^Y, e^u\}$, where
\[ e^X(x,y)= \frac{\partial}{\partial X}(x, y, x, y, 0) \in T_{(x, y, x, y, 0)}^{s} \G \]
and similarly for $e^Y$, $e^u$. The anchor sends $e^X$ to $\frac{\partial}{\partial X}$, similarly for $e^Y$ and kills $e^u$. For the bracket, we need the induced right invariant vector fields. 
To compute them at some $g= (x, y, X, Y, u)\in J^1(\Gamma)$, we use right translations
\[ R_g: s^{-1}(X, Y)\rmap s^{-1}(x, y),\ (X, Y, \overline{X}, \overline{Y}, v)\mapsto (x, y, \overline{X}, \overline{Y}, \frac{yV+ \overline{Y}u}{Y})\]
whose differential at at the unit $(X, Y, X, Y, 0)$ gives us
\[ \stackrel{\rightarrow}{e^X}= \frac{\partial}{\partial X}, \ \stackrel{\rightarrow}{e^Y}= \frac{\partial}{\partial Y}+ \frac{u}{Y} \frac{\partial}{\partial u}, \ \stackrel{\rightarrow}{e^u}= \frac{y}{Y} \frac{\partial}{\partial u}.\]
Computing their brackets we find 
\[ [\stackrel{\rightarrow}{e^X}, \ \stackrel{\rightarrow}{e^Y}]= 0, \ [\stackrel{\rightarrow}{e^X}, \stackrel{\rightarrow}{e^u}]= 0, \
[\stackrel{\rightarrow}{e^Y}, \ \stackrel{\rightarrow}{e^u}]= -\frac{2y}{Y^2} \frac{\partial}{\partial u}= -\frac{2}{Y} \stackrel{\rightarrow}{e^u},\]
hence the Lie bracket of $A(\G)$ is given by
\[ [e^X, e^Y]= [e^X, e^u]= 0, \ [e^Y, e^u]= -\frac{2}{Y} e^u.\]

Here is an alternative way of computing $A(\G)$. The key remark is that $\G$ can be embedded as a Lie subgroupoid of the first jet groupoid $J^1(\mathbb{R}^2, \mathbb{R}^2)$:
\begin{equation}\label{ex1-inclusion} 
\G\ni (x, y, X, Y, u)\mapsto (x, y, X, Y, \left( \begin{array}{ll}
\frac{y}{Y} & 0\\
u & \frac{Y}{y}\\
\end{array}\right) )\in J^1(\mathbb{R}^2, \mathbb{R}^2).
\end{equation}
Indeed, the multiplication of $\G$ comes out of the matrix multiplication 
\[ \left( \begin{array}{ll}
\frac{y}{Y} & 0\\
u & \frac{Y}{y}\\
\end{array}\right) \left( \begin{array}{ll}
\frac{y'}{y} & 0\\
v & \frac{y}{y'}\\
\end{array}\right)=
\left( \begin{array}{ll}
\frac{y'}{Y} & 0\\
\frac{y'u+ Yv}{y} & \frac{Y}{y'}\\
\end{array}\right) .\]
Computing the map induced by this inclusion at the level of Lie algebroids, we find the inclusion 
\[ A(\G)\hookrightarrow J^1(T\mathbb{R}^2), \ \left\{ 
\begin{array}{lll}
e^X & \mapsto \partial^{X}\\
e^Y & \mapsto \partial^Y- \frac{1}{Y}\partial^{X}_{x}+ \frac{1}{Y} \partial^{Y}_{x} \\
e^u & \mapsto \partial^{Y}_{y} 
\end{array}\right.\]
Computing the Lie brackets of the vectors on the right hand side (using the formulas we already know in $J^1(T\mathbb{R}^2)$), we recover the Lie brackets of $A^1(\Gamma)$. 
\end{example}

\begin{remark}[Flows of sections]\label{flows of sections}\rm \  If $A$ is the Lie algebroid of a Lie groupoid $\G$, then the Lie algebra $\Gamma(A)$ plays (morally) the role of the Lie algebra of the group of bisections $\textrm{Bis}(\G)$ (see definition \ref{definition-bisections}). For instance, for the pair groupoid of $M$, this becomes the usual interpretation of the Lie algebra $\X(M)$ of vector fields on $M$ as the Lie algebra of the diffeomorphism group $\textrm{Diff}(M)$. Here are a few details, which will also allow us to fix some notation regarding flows that will be used throughout the thesis. 

Let us assume that $A$ is the Lie algebroid of $\G$. For $\alpha\in \Gamma(A)$, one defines the (local) flow of $\alpha$ by
\[ \phi_{\alpha}^{\epsilon}:= \varphi_{\alpha^r}^{\epsilon}|_{M}: M\To \G ,\]
where $\varphi_{\alpha^r}^{\epsilon}$ is the (local) flow of the right invariant vector field $\alpha^r$. As usual, we are sloppy with the precise notation for the domain of the flow.
From right invariance it follows that $\phi_{\alpha}^{\epsilon}$ is a bisection of $\G$ which determines
the entire flow $\varphi_{\alpha^r}^{\epsilon}$ ($\varphi_{\alpha^r}^{\epsilon}(g)= \phi_{\alpha}^{\epsilon}(t(g))g$). Note also that, in terms of multiplication of (local) bisections, the flow property for  $\varphi_{\alpha^r}^{\epsilon}$ translates into
\[ \phi_{\alpha}^{\epsilon}\cdot \phi_{\alpha}^{\epsilon'}= \phi_{\alpha}^{\epsilon+ \epsilon'}.\]
This shows that, indeed, $\Gamma(A)$ behaves like ``the Lie algebra of $\textrm{Bis}(\G)$''. 
\end{remark}

Next, we recall the notion of representation of Lie algebroids. Let $A$ be a Lie algebroid over $M$. An $A$-connection on a vector bundle over $M$ is an $\mathbb{R}$-bilinear operator 
\[ \nabla: \Gamma(A)\times \Gamma(E)\rmap \Gamma(E),\ \ (\alpha, e)\mapsto \nabla_{\alpha}(e)\]
with the property that, for all $\alpha\in \Gamma(A)$, $e\in \Gamma(E)$, $f\in C^{\infty}(M)$,
\[ \nabla_{f\alpha}(e)= f\nabla_{\alpha}(e),\ \ \nabla_{\alpha}(fe)= f\nabla_{\alpha}(e)+ L_{\rho(\alpha)}(f) e.\]
The curvature of the $A$-connection $\nabla$ is the tensor
\[ R_{\nabla}\in \textrm{Hom}(\Lambda^2A, \textrm{Hom}(E, E))\]
given by (for $\alpha, \beta\in \Gamma(A)$)
\[ R_{\nabla}(\alpha, \beta) = \nabla_{[\alpha, \beta]}- [\nabla_{\alpha}, \nabla_{\beta}].\]

\begin{definition} 
Let $A$ be a Lie algebroid over $M$. A \textbf{representation} (or \textbf{infinitesimal action}) of $A$ on a vector bundle $E$ over $M$ is an $A$-connection $\nabla$ on $E$ satisfying the flatness condition $R_{\nabla}= 0$.
\end{definition}

\begin{remark}\rm\label{rk-algbds} Similar to the groupoid case, a vector bundle $E$ has an associated Lie algebroid $\gl(E)$ whose sections are derivations on $E$, i.e. 
operators $P: \Gamma(E)\to \Gamma(E)$ with the property that they satisfy a Leibniz identity of type
\[ P(fs)= fP(s)+ L_{X_P}(f) s\]
for some vector field $X_P$ on $M$ (called the symbol of $P$). The bracket on $\gl(E)$ corresponds to commutators of operators, and the
anchor sends $P$ to $X_{P}$. Of course, $\gl(E)$ is just the Lie algebroid of $GL(E)$. With this, a representation of $A$ on $E$ is the same thing as a Lie algebroid homomorphism $A\to \gl(E)$. In particular, since Lie groupoid morphisms give rise, after differentiation, to Lie algebroid morphisms, it follows that any representation $E$ of a Lie groupoid $\G$ is canonically a representation of the Lie algebroid $A$ of $\G$. For the explicit formulas, see lemma \ref{derivating-representations} of the next subsection.
\end{remark}

\begin{example} \rm \ Using the previous remark and example \ref{adjoint-for-classical}, we obtain that the jet bundle $J^{k-1}TM$ is canonically a representation of the $k$-jet algebroid $J^kTM$. 
\end{example}

\subsection{The Lie functor}\label{Lie functor}

Throughout this thesis, the term ``Lie functor'' is used to indicate passing from global to infinitesimal objects and should be thought of as ``linearization''. The reverse process is coined as ``integration''. For ``integrability theorems'', one of the conditions that constantly appears as a necessary condition in the case of groupoids is the $s$-simply connectedness.

\begin{definition} A Lie groupoids $\G$ is called \textbf{$s$-connected} if all $s$-fibers $s^{-1}(x)$ are connected. It is called \textbf{$s$-simply connected} if, furthermore, all $s$-fibers $s^{-1}(x)$ are simply connected. 
\end{definition}

The first ``example'' of the Lie functor is the construction of the Lie algebroid $A= A(\G)$ of a Lie groupoid $\G$ that we have already mentioned in the previous section. 
For the reverse process, starting with a Lie algebroid $A$, one looks for a Lie groupoid $\G$ which integrates $A$, i.e. whose Lie algebroid is isomorphic to $A$; if such a $\G$ exists, one says that $A$ is integrable.  A Lie groupoid $\G$ can always be replaced by an $s$-connected one without changing the Lie algebroid: one considers the open subgroupoid $\G^{0}\subset \G$ made of the connected component of the identities in the $s$-fibers of $\G$. One can go further and replace $\G$ by an $s$-simply connected Lie groupoid. More precisely, one constructs $\widetilde{\G}$ by putting together the universal covers of the $s$-fibers of $\G$ with base points the units (see e.g. \cite{CrainicFernandes} for the general discussion). Note that the canonical projection 
\[p: \widetilde{\G} \To \G\]
is a groupoid map which is a local diffeomorphism onto the $s$-connected component $\G^0$; this immediately implies that $\widetilde{\G}$ has the same Lie algebroid as $\G$. 
As for Lie groups, there is a Lie II theorem (saying that a morphism of Lie algebroids can be integrated to one of Lie groupoids, provided the domain groupoid is $s$-simply connected); these imply the following basic result  in the theory of Lie groupoids:

\begin{proposition} 
  If the Lie algebroid $A$ is integrable then there exists and is unique (up to isomorphism) a Lie groupoid $\G$ which is $s$-simply connected and integrates $A$.
\end{proposition}

Back to the general discussion on the Lie functor, given a Lie groupoid $\G$ with Lie algebroid $A$, intuitively, the Lie functor takes structures on $\G$ and transforms them into structures on $A$. However, one should be aware that the outcome is not always obvious; also, the reverse process (integrability theorems) is usually more difficult and, as mentioned above, requires various connectedness assumptions on the $s$-fibers. 

A very good and simple example is the notion of representations. Remark \ref{rk-algbds} shows that any representation $E$ of a Lie groupoid $\G$ can be made into a representation of the Lie algebroid $A$ of $\G$. Writing out the explicit formulas, one finds:

\begin{lemma}\label{derivating-representations} Let $\G$ be a Lie groupoid with Lie algebroid $A$ and let $E$ be a representation of $\G$. Then $E$ is canonically a representation of $A$, with linear action defined as follows:
for $\al \in \Gamma(A)$, and $e \in \Gamma(E)$,
\begin{eqnarray}
\nabla_{\al}e(x) = \frac{\d}{\d \eps}\bigg\vert_{\eps = 0}g(\eps)^{-1}\cdot e\big(t(g(\eps))\big),
\end{eqnarray}
where $g(\eps)$ is any curve in $s^{-1}(x)$ with $g(0)=1_x$, $\frac{\d}{\d \eps}\vert_{\eps = 0}g(\eps) = \al(x)$. 
\end{lemma}

We see that the Lie functors takes representations of $\G$ into those of $A$. It is not difficult to check that, if $\G$ is $s$-connected then, for two representations $E$ and $F$ of $\G$, if $\textrm{Lie}(E)= \textrm{Lie}(F)$ as representations of $A$, then $E= F$. For the reverse process we have the following integrability theorem (itself a consequence of Lie II mentioned above):

\begin{proposition} If $\G$ is an $s$-simply connected Lie groupoid with Lie algebroid $A$, then any representation of $A$ comes from a representation of $\G$. 
\end{proposition}

Again, one should keep in mind that this is just one instance of an integrability theorem (for representations).

\subsection{Lie pseudogroups}\label{Lie pseudogroups}

Recall that, for a manifold $M$, $\textrm{Diff}_{\textrm{loc}}(M)$ stands for the set of diffeomorphisms $\phi: U\to V$ between open sets $U, V\subset M$. There is a lot in the literature about pseudogroups, some references are \cite{Cartan1904,GuilleminSternberg:deformation,Olver:MC,Niky,Shadwick,KumperaSpencer,Pommaret}

\begin{definition}
A \textbf{pseudogroup} on a manifold $M$ is a collection $\Gamma$ of diffeomorphisms between open sets in $M$, i.e. a subset $\Gamma \subset  \textrm{Diff}_{\textrm{loc}}(M)$, satisfying:
\begin{enumerate}
\item[1.] If $\phi\in \Gamma$, then $\phi^{-1}\in \Gamma$.
\item[2.] If $\phi, \psi \in \Gamma$, and $\phi\circ \psi$ is defined, then $\phi\circ \psi\in \Gamma$.
\item[3.] If $\phi\in \Gamma$ and $U$ is an open set contained in the domain of $\phi$, then $\phi|_{U}\in \Gamma$.
\item[4.] If $\phi: U\to V$ is a diffeomorphism and $U$ can be covered by a family of open sets $U_i$ such that $\phi|_{U_i}\in \Gamma$ for all $i$, then $\phi\in \Gamma$.
\end{enumerate}
\end{definition}

Roughly speaking, a Lie pseudogroup is a pseudogroup which is defined by (a system of) PDEs. Of course, there are various regularity conditions one may require;  unfortunately, this gives rise to several non-equivalent notions of Lie pseudogroups that one can find in the literature. The conditions that we will impose here are weaker than most of the conditions one finds; however, to avoid (even more) conflicts in terminology, we will call our objects ``smooth pseudogroups''. Let us start with the notion of order of a pseudogroup.

\begin{definition}
We say that a pseudogroup $\Gamma$ is of \textbf{order} $k$ if $k$ is the smallest number with the following property: any $\phi\in \textrm{Diff}_{\textrm{loc}}(M)$ with the property that for any $x\in \textrm{Dom}(\phi)$ there exists 
$\phi_x\in \Gamma$ such that $j^{k}_{x}\phi= j^{k}_{x}(\phi_x)$, must belong to $\Gamma$.
\end{definition}

In other words, the elements of $\Gamma$ are determined by their $k$-jets. This is best formalized using jet spaces and groupoids. More precisely, for any pseudogroup $\Gamma$, it is clear that the $k$-jets of elements of $\Gamma$ define a sub-groupoid 
\[ \Gamma^{(k)} \subset \Jet^k(M, M)\]
of the groupoid of $k$-jets of diffeomorphisms of $M$. Moreover, while the $k$-jet of any $\phi \in \textrm{Diff}_{\textrm{loc}}(M)$ can be viewed as a local bisection of $\Jet^k(M, M)$, the $k$-jet of an element $\phi\in \Gamma$ defines a bisection which takes values in $\Gamma^{(k)}$. The previous condition on $k$ says that, for $\phi \in \textrm{Diff}_{\textrm{loc}}(M)$:
\[ j^k\phi \in \textrm{Bis}_{\textrm{loc}}(\Gamma^{(k)}) \Longrightarrow \phi \in \Gamma .\]
The regularity condition that we will impose is the following:

\begin{definition} 
A pseudogroup $\Gamma$ on $M$ is called \textbf{smooth} if it is of finite order, all its $k$-jet groupoids $\Gamma^{(k)}$ are smooth subgroupoids of $\Jet^k(M, M)$ (for all $k\geq 0$), and the projection maps $\Gamma^{(k+1)}\to \Gamma^{(k)}$ are smooth surjective submersions.    
\end{definition}

Sometimes one requires the previous conditions only for $k$ greater or equal to the order of $\Gamma$, but one can always pass to the case of the definition above.

This is one way of making sense of the fact that ``$\Gamma$ is defined by (a system of) PDEs''. Note that in this case, indeed, $\Gamma^{(k)}$ is a PDE on the bundle $\textrm{pr}_1: M\times M\to M$ whose solutions correspond precisely to the elements of $\Gamma$.

Given a smooth pseudogroup $\Gamma$, one has induced sequence of Lie groupoids related to each other by surjective submersions (a tower of Lie groupoids)
\[ \cdots \rmap \Gamma^{(2)}\rmap \Gamma^{(1)} \rmap \Gamma^{(0)} .\]

Applying the Lie functor, one obtains a similar sequence of Lie algebroids 
\[ \cdots \rmap A^{(2)}(\Gamma)\rmap A^{(1)}(\Gamma) \rmap A^{(0)}(\Gamma).\]
In what follows, we will use the notation $A^{(k)}= A^{(k)}(\Gamma)$ for the Lie algebroid of $\Gamma^{(k)}$.\\

A large part of this thesis arises from our attempt of understanding the structures that governs these towers. Here are a few such concepts. As we know from the existing theory, one of the main ingredients is the Cartan forms; our aim will be to understand them more conceptually, to understand their key properties, etc. (and this will be done in the more general context of Pfaffian groupoids). In this setting, the Cartan forms become 1-forms
\[ \theta\in \Omega^1(\Gamma^{(k)}, t^*A^{(k-1)}),\]
called \textbf{the Cartan forms of the pseudogroup $\Gamma$}. 
Let us give here the direct description. To describe $\theta_g(V_g)$ for $g\in \Gamma^{(k)}$, $V_g\in T_g\Gamma^{(k)}$, write $g= j^{k}_{x}\phi$ with $\phi\in \Gamma$.  Using $l: \Gamma^{(k)}\to \Gamma^{(k-1)}$ and viewing $j^{k-1}\phi$ as a map (bisection) $\sigma: M\to \Gamma^{(k-1)}$, 
\[ dl(V_g)- d_x\sigma\circ d_gs(V_g) \in T_{l(g)}\Gamma^{(k-1)} \]
is killed by $ds$, hence it comes from an element in the fiber of $A^{(k-1)}$ at $t(l(g))= t(g)$; this is $\theta_g(V_g)$. Explicitely,
\[ \theta_g(V_g)= (dR_{l(g)})_{l(g)^{-1}} ( dl(V_g)- d_x\sigma \circ d_gs(V_g)) \in A^{(k-1)}_{t(g)}.\]
Of course, $\theta$ can also be seen as the restriction to $\Gamma^{(k)}$ of the Cartan form $\theta$ on the groupoid $\Jet^k(M, M)$ (i.e. equation (\ref{eq-Cartan-forms-on-jets}) from example \ref{Cartan-forms-on-jets}) and the previous discussion shows that, indeed, this restriction takes values in $A^{(k-1)}\subset \Jet^{k-1}TM$. In particular, the main property of $\theta$ gives, for $k$ greater or equal to the order of $\Gamma$, that the elements of $\Gamma$ correspond to bisections of $\Gamma^{(k)}$ which kill $\theta$.\\

As we shall see later in the thesis, the key property of $\theta$ (which gives rise to the entire theory!) is its compatibility with the groupoid structure (called multiplicativity later on). To make sense of this, one has to make sense of  $A^{(k-1)}$ as a representation of $\Gamma^{(k)}$:
\[ A^{(k-1)} \in \textrm{Rep}(\Gamma^{(k)}) .\]
This is immediately obtained by restricting the canonical action of $J^k(M, M)$ on $J^{k-1}TM$ (from example \ref{adjoint-for-classical}) to $\Gamma^{(k)}$; staring at the definition carefully one finds that, indeed, $A^{(k-1)}$ is invariant under the action of  $\Gamma^{(k)}$. Together with this action, $A^{(k-1)}$ will be called \textbf{the adjoint representation of $\Gamma^{(k)}$}. \\

Next, let us also indicate the appearance of tableaux towers in this discussion. They arise when comparing the various levels of the tower of Lie algebroids mentioned above. More precisely, we consider 
\[ \g^k(\Gamma):= \textrm{Ker}(A^{(k)}\rmap A^{(k-1)}) .\]
Interpreting again $A^{(k)}(\Gamma)$ as a sub-algebroid of $\Jet^kTM$, and using the fact that the kernel of the projection $\Jet^kTM\to \Jet^{k-1}TM$ is canonically isomorphic to $\textrm{Hom}(S^kTM, TM)$, one obtains a canonical inclusion 
\[ \g^k(\Gamma)\subset \textrm{Hom}(S^kTM, TM) . \]
Hence, for each $k$,  $\g^k(\Gamma)$  is a bundle of tableaux on $(TM, TM)$, called \textbf{the $k$-th order tableaux of $\Gamma$}. One can also show that  $\g^{k+1}(\Gamma)$ sits inside the prolongation of $\g^k(\Gamma)$, i.e. in the terminology of definition \ref{tableaux tower}
\[ \g^{\infty}(\Gamma): = (\g^1(\Gamma), \g^2(\Gamma), \ldots ) \]
is a bundle of tableaux towers over $M$, called \textbf{the tableaux tower associated to $\Gamma$}. 

Here are a few examples of pseudogroups and the computation of their first Cartan form.

\begin{example}\rm \ For the maximal pseudogroup $\Gamma= \textrm{Diff}_{\textrm{loc}}(M)$, one recovers the full jet groupoids $\Jet^k(M, M)$ and their Cartan forms (see example \ref{Cartan-forms-on-jets} and \cite{Gold3}). Let us use the previous description of $\theta$ to compute it for $J^1(\mathbb{R}^n, \mathbb{R}^n)$ (see also example \ref{ex-alg-J1}). 
For $g=j^{1}_{x}(\phi)= (x, X, p)\in J^1$, and $V_{g}\in T_{g}J^1$, we have to look at
\[(d_gl- d_x\sigma\circ d_gs)(V_g),\]
where $\sigma(x)= (x, \phi(x))$. We find:
\begin{itemize}
\item for $V= \frac{\partial}{\partial p^{i}_{a}}(g)$ it is $0$.
\item for $V= \frac{\partial}{\partial X_i}(g)$ it is $\frac{\partial}{\partial X_i}(x, X)= \stackrel{\rightarrow}{\partial}^{X_{i}}$.
\item for $V= \frac{\partial}{\partial x_a}(g)$ it is $- p^{i}_{a} \frac{\partial}{\partial X_i} (x, X)= - p^{i}_{a} \stackrel{\rightarrow}{\partial}^{X_{i}}$.
\end{itemize}
Hence we obtain the standard formula
\[ \omega_i= dX_i- p^{i}_{a} dx_a .\]
\end{example}

\begin{example}[\cite{Cartan1904}, see also \cite{Ori}]\rm \ On the manifold
\[ M= \{(x, y)\in \mathbb{R}^2: y\neq 0 \}\]
we consider the pseudogroup $\Gamma$ consisting of all transformations of type
\[ (x, y)\mapsto (f(x), \frac{y}{f'(x)}),\]
where $f$ is a smooth function with non-vanishing derivative defined on some open inside $\mathbb{R}$. One also writes:
\[ \Gamma: \ X= f(x),\ Y= \frac{y}{f'(x)}.\]
One can easily check that this is a pseudogroup. Playing with the partial derivatives of $X$ and $Y$, we see that $\Gamma$ can be described as the solutions of the system
\[ \frac{\partial X}{\partial x}= \frac{y}{Y},\ \  \frac{\partial X}{\partial y}= 0, \ \ \frac{\partial Y}{\partial y}= \frac{Y}{y}.\]
Note that $\frac{\partial Y}{\partial x}$ does not appear in this list (one can easily check that, fixing the values of $X$ and $Y$ at any point $(x_0, y_0)$, $\frac{\partial Y}{\partial x}(x_0, y_0)$ can be arbitrary).  This shows that $\Gamma$ is of order $1$ and 
\[ \Gamma^{(1)}= \{  (x, y, X, Y, \left( \begin{array}{ll}
\frac{y}{Y} & 0\\
u & \frac{Y}{y}\\
\end{array}\right): (x, y), (X, Y)\in M, u\in \mathbb{R}\} \subset  J^1(\mathbb{R}^2, \mathbb{R}^2),\]
which is precisely the groupoid discussed in example \ref{toy-example}. Using the formulas for the Cartan form from the previous example and pulling it back to $\G$, we see that
$\Gamma^{(1)}$ is isomorphic to the groupoid $\G$ from example \ref{toy-example} and the Cartan form becomes
\[ \theta\in \Omega^1(\G, \mathbb{R}^2), \ \ \theta_{1}= dX- \frac{y}{Y}dx,\ \ \theta_{2}= dY- udx- \frac{Y}{y} dy.\]
It is instructive to check directly that $\Gamma$ can be recovered from $(\G, \theta)$ (i.e. without using that they are related to the jet groupoids).
\end{example}

\begin{example}[\cite{Cartan1904}, see also \cite{Ori}] \label{SL-2-example} \rm \ On the manifold $M= \mathbb{R}$ we consider the pseudogroup
\[ \Gamma:\ \ X= \frac{ax+b}{cx+d}, \ \ ad-bc= 1,\]
i.e. consisting of all local transformations of type
\[ x\mapsto \frac{ax+b}{cx+d},\]
where $a, b, c, d$ are arbitrary real numbers satisfying $ad-bc= 1$. Let us first find the order of $\Gamma$.
Denoting
\[ u= \frac{\partial X}{\partial x}= \frac{1}{(cx+ d)^2}, \ \ v=  \frac{\partial^2 X}{\partial x^2}= - \frac{2c}{(cx+ d)^3},\]
we see that the resulting equations on $a, b, c, d$ always have the unique solution
\[ a= \frac{2u^2- vX}{2u\sqrt{u}}, \ \ b = \frac{2uX-2u^2x+ vxX}{2u\sqrt{u}}, \ \ c= -\frac{v}{2u\sqrt{u}}, \ \ d= \frac{2u+ vx}{2u\sqrt{u}},\]
hence 
\[ \Gamma^{(0)}= J^0(\mathbb{R}, \mathbb{R}), \ \ \Gamma^{(1)}= J^1(\mathbb{R}, \mathbb{R}),\ \ \Gamma^{(2)}= J^2(\mathbb{R}, \mathbb{R}).\]
Computing the third order derivatives of $X$ we find 
\[ \frac{\partial^3 X}{\partial x^3}= \frac{6c^2}{(cx+ d)^4}= \frac{3v^2}{2u}= \frac{3}{2}\frac{\frac{\partial^2 X}{\partial x^2}}{\frac{\partial X}{\partial x}}.\]
We see that $\Gamma$ is of order $3$, $\Gamma^{(3)}$ is four-dimensional, diffeomorphic to  
\[ \G:= \{ (x, X, u, v)\in \mathbb{R}^4: u\neq 0\},\]
with explicit diffeomorphism 
\[ \G \stackrel{\sim}{\rmap} \Gamma^{(3)}\subset  J^3(\mathbb{R}, \mathbb{R}),\ \ (x, X, u, v)\mapsto (x, X, u, v, \frac{3v^2}{2u}).\]
Note that the resulting composition on $\G$ is 
\[ (x, X, u, v)(x', x, u', v')= (x', X, uu', vu'^2+ uv').\]
The Lie algebroid $A^{(3)}$ is then the 3-dimensional vector space spanned by three vectors $e^X$, $e^u$ and $e^v$ (corresponding to the partial derivatives with respect to the indicated variables):
\[ A^{(3)}= \textrm{Span}\{e^X, e^Y, e^u\}.\]
 The  induced right invariant vector fields are computed applying the right translations: for $g= (x, X, u, v)$,
\[ R_g: s^{-1}(X)\rmap s^{-1}(x),\ \ (X, \overline{X}, p, q)\mapsto (x, \overline{X}, pu, qu^2+ pv);\]
one computes $dR_g$ at the unit element $1_x= (X, X, 1, 0)$ and one finds
\begin{equation}\label{right-invariant-sl2}
\stackrel{\rightarrow}{e}^X= \frac{\partial}{\partial X}, \ \ \stackrel{\rightarrow}{e}^u= u\frac{\partial}{\partial u}+ v\frac{\partial}{\partial v}, \ \
\stackrel{\rightarrow}{e}^v= u^2\frac{\partial}{\partial v}.
\end{equation}
Hence the Lie bracket of $A^3(\Gamma)$ is 
\[ [e^X, e^u]= [e^X, e^v]= 0, \ \ [e^u, e^v]= e^v.\]
Of course, one could also have used the inclusion into $J^3(\mathbb{R}, \mathbb{R})$. Alternatively, we could have noticed that the restriction of 
$l: J^3(\mathbb{R}, \mathbb{R})\to J^2(\mathbb{R}, \mathbb{R})$ to $\Gamma^{(3)}$ induces an isomorphism between $\Gamma^{(3)}$ and
$J^2(\mathbb{R}, \mathbb{R})$ (actually, this is implicitly present in the computations above).

The relevant Cartan form, 
\[ \theta\in \Omega^1(\Gamma^{(3)}, J^2(T\mathbb{R})),\]
takes values in the algebroid $J^2(T\mathbb{R})$ of $\Gamma^{(2)}= J^2(\mathbb{R}, \mathbb{R})$ (with coordinates $(x, X, u, v)$). As above, $J^2(T\mathbb{R})$ is spanned by 
three vectors $e^{X}, e^{u}, e^{v}$ which determine the right invariant vector fields given by the same formulas (\ref{right-invariant-sl2}). Note that, on $J^2(\mathbb{R}, \mathbb{R})$,
\[ \frac{\partial}{\partial X}= \stackrel{\rightarrow}{e}^{X},\  \ \frac{\partial}{\partial u}= \frac{1}{u}\stackrel{\rightarrow}{e}^u- \frac{v}{u^3}\stackrel{\rightarrow}{e}^v
, \ \   \frac{\partial}{\partial v}= \frac{1}{u^2} \stackrel{\rightarrow}{e}^v  .\]
Returning to $\theta$, to compute it at $g= (x, X, u, v)$ choose $\phi\in \Gamma$ with $j^{3}_{x}\phi= g$ in $\Gamma^{(3)}$. We have to consider
\[ \sigma= j^{2}\phi= (id, \phi, \phi', \phi''): \mathbb{R}\rmap \Gamma^{(2)}\]
and then to compute $\stackrel{\rightarrow}{\theta}_g= d_gl- d_x\sigma\circ d_gs$ in terms of the right invariant vector fields above. We find
\[\left\{ \begin{array}{llll}
\stackrel{\rmap}{\theta (\frac{\partial}{\partial X}) }& = \frac{\partial}{\partial X}= \stackrel{\rightarrow}{e}^{X}\\
\stackrel{\rmap}{\theta (\frac{\partial}{\partial u}) }& = \frac{\partial}{\partial u}= \frac{1}{u}\stackrel{\rightarrow}{e}^u- \frac{v}{u^3}\stackrel{\rightarrow}{e}^v\\
\stackrel{\rmap}{\theta (\frac{\partial}{\partial v}) }& = \frac{\partial}{\partial v}= \frac{1}{u^2} \stackrel{\rightarrow}{e}^v\\
\stackrel{\rmap}{\theta (\frac{\partial}{\partial x}) }& = -(\frac{\partial}{\partial x}+ u\frac{\partial}{\partial X}+ v \frac{\partial}{\partial u}+ \frac{3v^2}{2u}\frac{\partial}{\partial v})= 
- (u \stackrel{\rightarrow}{e}^X+ \frac{v}{u} \stackrel{\rightarrow}{e}^u+ \frac{v^2}{2u^3}\stackrel{\rightarrow}{e}^v)
\end{array}\right.\]

Hence the components of $\theta$ (the coefficients of $\stackrel{\rightarrow}{e}^X, \stackrel{\rightarrow}{e}^u, \stackrel{\rightarrow}{e}^v$) are
\[ \theta_1= dX- udx,\ \ \theta_2= \frac{1}{u}du- \frac{v}{u}dx,\ \ \theta_3= \frac{1}{u^2}dv- \frac{v}{u^3} du- \frac{v^2}{2u^3} dx.\]
Note again that $\Gamma$ can be recovered from the groupoid $\G$ and the form $\theta\in \Omega^1(\G, \mathbb{R}^3)$ without any reference to jet groupoids, by looking at the bisection of $\G$ that kill $\theta$.
\end{example}

\section{Generalized pseudogroups}\label{Generalized pseudogroups}

In this section we recall the construction of the jet groupoids $J^k\G$ associated to an arbitrary Lie groupoid $\G$, of the jet algebroids $J^kA$ associated to an arbitrary Lie algebroid $A$, the associated adjoint representations and Cartan forms; these are rather straightforward generalizations of the $k$-jet groupoids and algebroids that we already discussed and which are recovered when $\G$ is a pair groupoid. Finally, these will serve as the start of a theory of ``generalized pseudogroups'' that we are proposing (for motivations, see subsection \ref{Generalized Lie pseudogroups}).

\subsection{Jet groupoids and algebroids}\label{Jet groupoids and algebroids}

In this section we recall the jet construction applied to general Lie groupoids and Lie algebroids. 

First, we consider Lie groupoids. The concept of bisections of $\G$ (see definition \ref{definition-bisections}) extends easily to that of local bisection, the difference being that the latter are defined only over some open $U\subset M$ (and then $\phi_b$ is a diffeomorphism from $U$ to
$\phi_b(U)$); if $b_1$ is defined on $U_1$ and $b_2$ on $U_2$, then $b_1\cdot b_2$ is a local bisection defined on $\phi_{b_2}^{-1}( U_1)\cap U_2$. 
We denote by $\textrm{Bis}_{\textrm{loc}}(\G)$ the set of local bisections of $\G$. 

With these, for any integer $k\geq 0$, one defines the $k$-jet groupoid 
\[ J^k\G:= \{\jet^{k}_{x}b:\ \ b\in \textrm{Bis}_{\textrm{loc}}(\G), \ x\in \textrm{Dom}(b)\},\]
with groupoid structure given by:
\begin{eqnarray*}
s(\jet^k_xb)= x,\ t(\jet^k_xb)= \phi_b(x),
\end{eqnarray*}
\begin{eqnarray*}
\jet^k_{\phi_{b_2}(x)}b_1\cdot\jet^k_xb_2=\jet^k_x(b_1\cdot b_2),\text{ and }(\jet^k_xb)^{-1}=\jet_{\phi_{b}(x)}^k(b^{-1}).
\end{eqnarray*}
Note that $J^k\G$ is an open subspace of the manifold of $k$-jets of sections of the source map, hence it carries a natural smooth structure and then it is not difficult to check that 
$J^k\G$ becomes a Lie groupoid. 

\begin{definition} $J^k\G$, with the Lie groupoid structure described above, is called the \textbf{$k$-jet groupoid associated to $\G$}.
\end{definition}

\begin{example} \rm \ Of course, when $\G= \Pi(M)$ is the pair groupoid of $M$, we recover the $k$-jet groupoids $J^k(M, M)$. 
\end{example}

As in the previous discussion, to any Lie algebroid $A$ and for any integer $k\geq 0$, one associates a new Lie algebroid $J^kA$. The underlying vector bundle is just the
vector bundle of $k$-jets of sections of $A$; the anchor is the composition of the anchor of $A$ with the projection $J^kA\to A$ and the main property of the bracket is that
\[ [j^k(\alpha), j^k(\beta)]= j^k([\alpha, \beta]) \]
(for all $\alpha, \beta\in \Gamma(A)$). Of course, since the space of sections of $J^kA$ is generated, as a $C^{\infty}(M)$-module, by the holonomic sections (i.e. those of type $j^k\alpha$), the Leibniz identity implies that the previous condition determines the bracket of $J^kA$ uniquely; of course, one proves separately that such a bracket actually exists. For that, one could use for instance the explicit formulas in therms of the relative Spencer connection, given in the next subsection. Note also that the projections $J^kA\to J^{k-1}A$ become Lie algebroid morphisms. Finally, one also shows, as in the case of pair groupoids, that for any Lie groupoid $\G$ with Lie algebroid $A$, $J^kA$ is (isomorphic to) the Lie algebroid of $J^k\G$.

\begin{remark}[Working with $J^1$]\rm \ \label{when working with jets}\rm \ 
According to remark \ref{1-when working with jets}, for $k= 1$, there is a slightly different description of $\Jet^1\G$, which we will be using whenever we have to work more explicitly with $J^1\G$. According to the remark, we identify $\jet^1_xb$ with $d_xb$ and then $\Jet^1\G\tto M$ is interpreted as the the space of splittings
\begin{eqnarray*}
\sigma_g  :T_xM\To T_g\G
\end{eqnarray*}
of the map $d_gs:T_g\G\to T_xM$, with $g$ an element in $\G$ and $x= s(g)$, with the property that
\begin{eqnarray}
\lambda_{\sigma_g}:= \d t\circ\sigma:T_xM\to T_{t(g)}M \text{ is an isomorphism.}
\end{eqnarray}
The groupoid structure of $J^1(\G)$ becomes:
\begin{eqnarray*}
s(\sigma_g)= s(g),\ t(\sigma_g)= t(g), 
\end{eqnarray*}
\begin{equation}\label{mult-i-J1}
 \sigma_g\cdot \sigma_h (u)=(\d m)_{g, h}(\sigma_g(\lambda_{\sigma_h}(u)), \sigma_h(u)) .
\end{equation}

A similar discussion applies at the Lie algebroid level, when working with $J^1A$. In this case one can use the Spencer decomposition of remark \ref{remark-J-decomposition} to represent the sections of $J^1A$ as pairs $(\alpha, \omega)$ with $\alpha$ a section of $A$ and $\omega\in \Omega^1(M, A)$ (but be aware of the module structure (\ref{strange-module-structure})!). The Lie bracket can then be written as
\[ [(\alpha, \omega), (\alpha', \omega')]= ([\alpha, \alpha'], L_{\alpha}(\omega')- L_{\alpha'}(\omega)+ [\omega, \omega']_{\rho}),\]
where 
\[ [\omega, \omega']_{\rho}(X)= \omega(\rho(\omega'(X)))-  \omega'(\rho(\omega(X))),\]
and the Lie derivative $L_{\alpha}(\omega')$ is
\[ L_{\alpha}(\omega')(X)= [\alpha, \omega'(X)]- \omega'([\rho(\alpha), X]).\]
Another way of writing (and explaining) these formulas will be given in the next subsection. 
\end{remark}

%

\subsection{The adjoint representations; the Cartan forms}\label{The adjoint representation}

Given a Lie groupoid $\G$ over $M$ with Lie algebroid denoted by $A$, the groupoid $\Jet^k\G$ has canonical representations on the vector bundles $J^{k-1}TM$ and $J^{k-1}A$, known under the name of  ``adjoint representations''. This is done in complete analogy with the case of the pair groupoid, explained in example \ref{adjoint-for-classical}. Let us repeat the construction but, for notational simplicity, let us assume that $k= 1$. Then:
\begin{enumerate}
\item the linear action of $J^1\G$ on $TM$ associates to an element  $\sigma= \jet^{1}_{x}b \in J^1\G$ the isomorphism $\lambda_{\sigma}:= d_x\phi_b: T_{s(\sigma)}M \to T_{t(\sigma)}M$. 
\item the linear action of $J^1\G$ on $A$ (the adjoint action) is defined as follows. A bisection $b$ of $\G$ acts on $\G$ by conjugation
\[C_b(g) = b(t(g))\cdot g \cdot b(s(g))^{-1}.\]
It is clear that this action maps units to units, and source fibers to source fibers. Moreover, the differential of $C_b$ at a unit $x$ depends only on $\jet^1_xb$. We define
\[\Ad_{\jet^1_xb}\al = \d_{x}C_b(\al)\]
for all $\al \in A_x$.
\end{enumerate}

Using the description of $\Jet^1\G$ given in remark \ref{when working with jets}, the adjoint representations can be described as follows. If $\al \in A_x$, and $\eps \mapsto h_{\eps}$ is any curve in $s^{-1}(x)$ such that 
\[h_0 = 1_x, \ \text{ and } \ \frac{\d}{\d\eps}\big{|}_{\eps = 0}h_{\eps} = \al,\]
then
\begin{equation}\begin{aligned}\label{eq: Ad}
\Ad_{\sigma_g}\al & = \frac{\d}{\d\eps}\big{|}_{\eps = 0}m(m(b(t(h_{\eps})),h_\eps), b(x))\\ 
&=\d_gR_{g^{-1}} \circ \d_{(g, s(g))}m(\sigma_g(\rho(\al)), \al).
\end{aligned}
\end{equation}

Similarly, $\Jet^1A$ has a canonical adjoint representations on $TM$ and on $A$ (both denoted by $\ad$) determined by the Leibniz  identities and the conditions
\[\ad_{\jet^1\al}X = [\rho(\al),X], \ \text{ and }  \ \ad_{\jet^1\al}\be = [\al,\be].\]
Using the Spencer decomposition to represent the sections of $J^1A$ as pairs $(\alpha, \omega)$ as at the end of remark \ref{when working with jets}, one has the general formulas
\[ \ad_{(\alpha, \omega)}(X)= [\rho(\alpha), X]+ \rho(\omega(X)),  \ \ad_{(\alpha, \omega)}(\beta)= [\alpha, \beta]+ \omega(\rho(\beta))\]
(see also below). Again, if $A$ is the Lie algebroid of a Lie groupoid $\G$, then these representations correspond (via lemma \ref{derivating-representations}) to the canonical representations of $\Jet^1\G$ on $TM$ and $A$ respectively (mentioned in the groupoid version of our discussion).

What is probably less known is the fact the the Lie bracket of $J^1A$ and the adjoint actions can be described explicitely using Spencer's relative connection $D^\clas$ (see remark \ref{remark-J-decomposition}). For the actions,
\begin{eqnarray}\label{actions} 
\ad_{\xi}(X)= [\rho(\xi), X]+ \rho D^\clas_{X}(\xi), \ \  \ad_{\xi}(\al)= [pr(\xi),\al]+ D^\clas_{\rho(\al)}(\xi).
\end{eqnarray}
Using these, each $\xi\in \Gamma(J^1A)$ induces a Lie derivative $\Lie_{\xi}$ on $\Omega^1(M, A)$ by
\begin{eqnarray}\label{adjunta} 
L_{\xi}(\omega)(X)= \ad_{\xi}(\omega(X))- \omega([\rho(\xi),X])
\end{eqnarray}
and then, for the bracket of $J^1A$, one finds 
\begin{eqnarray}\label{bracket}
[\xi, \eta]= j^1([pr(\xi), pr(\eta)])+ L_{\xi}(D^\clas(\eta))- L_{\eta}(D^\clas(\xi)).
\end{eqnarray}
Note several of the advantages of this point of view:
\begin{itemize}
\item (\ref{bracket}) can be taken as the definition of the bracket of $J^1A$. The defining properties of the bracket (that it satisfies Leibniz and the bracket of $1$-jets is the $1$-jet of the bracket are easy checks);
\item (\ref{actions}) can be taken as the definition of the adjoint actions; 
\item both formulas work for the higher jet algebroids $J^kA$ (and this will be our definition for the adjoint representations of $J^kA$); 
\item more conceptually, it indicates that the key ingredient for the entire theory are the Spencer operators $D^\clas$. This will become more clear in the later parts of the thesis. One of the first questions to answer is: which are the key properties of $D^\clas$ that make the theory work? (e.g. what makes the bracket satisfy the Jacobi identity?; what ensures that the adjoint actions are actual actions?). The answer will be given later on, when discussing abstract Spencer operators. 
\end{itemize}

Finally, we move to the global counterpart of the Spencer operators $D^\clas$, i.e. to the Cartan forms. Again, the precise meaning of ``global counterpart" will be made more clear later on, where we will actually see that the Cartan form and the Spencer operator are, modulo the Lie functor, the same object. Given a Lie groupoid $\G$ with Lie algebroid $A$, the associated Cartan forms will be $1$-forms
\[ \theta\in \Omega^1(\Jet^k\G, t^*J^{k-1}A).\]
They are induced by the Cartan forms associated to the bundle $s: \G\to M$ (see subsection \ref{PDEs; the Cartan forms:}) by restricting them to 
$\Jet^k\G\subset \Jet^k(\G\stackrel{s}{\to} M)$. As such, they take values in 
\[ \pi^* T^{s} \Jet^{k-1}\G \cong \pi^* t^* \Jet^{k-1}A= t^* \Jet^{k-1}A .\]
The main property of the Cartan forms discussed before becomes:

\begin{lemma} 
A bisection $\xi$ of $\Jet^k\G$ is of type $\jet^kb$ for some bisection $b$ of $\G$ if and only if $\xi^*(\theta)= 0$.
\end{lemma}

One can also give an explicit description of $\theta$, completely similar to the one of the Cartan forms associated to pseudogroups (subsection \ref{Lie pseudogroups}). Here is the outcome when $k= 1$, i.e. for 
\[ \theta\in  \Omega^1(\Jet^1\G, t^{\ast}A).\]
Let $\pr:\Jet^1\G\to \G$ be the canonical projection and let $\xi$ be a vector tangent to $\Jet^1\G$ at some point
$\jet^1_xb\in \Jet^1\G$. Then the difference
\[\d_{\jet^1_xb}\pr(\xi)-\d_x b\circ \d_{\jet^1_xb}s(\xi) \in T_g\G\] 
lies in $ \ker \d s$, hence it comes from an element in $A_{t(g)}$:
\[\theta_{\mathrm{can}}(\xi)=R_{b(x)^{-1}}(\d_{\jet^1_xb}pr(\xi)- \d_x b\circ \d_{\jet^1_xb}s(\xi))\in A_{t(g)}.\]

\subsection{Generalized Lie pseudogroups}\label{Generalized Lie pseudogroups}

One of the main problems with the various notions of ``Lie pseudogroups'' that one finds in the literature is that, in most cases, the pseudogroups are assumed to be transitive (recall that a pseudogroup $\Gamma$ acting on $M$ is transitive if for any $x, y\in M$, there exists $\phi\in \Gamma$ such that $\phi(x)= y$; equivalently, $\Gamma^{(0)}$ is the pair groupoid of $M$). Moreover, the literature on ``the intransitive case'' refers to a rather mild relaxation of transitivity as it still requires the orbits of $\Gamma$ to have constant dimension (usually required to the be fibers of a submersion $I: M\to B$). Of course, this excludes some very simple examples. Here we indicate how, using arbitrary Lie groupoids instead of the pair groupoid (and bisections instead of local diffeomorphisms), there is a straightforward generalization of the notion of pseudogroups, for which the intransitivity is built in (because one can start with an intransitive Lie groupoid in the first place), and for which the standard theory can be extended without much trouble (that is at least the way we organized the exposition).

So, let $\G$ be a Lie groupoid over a manifold $M$. When $\G= \Pi(M)$ we will recover the classical theory. In general, we 
consider the set $\Bis_{\text{\rm loc}}(\G)$ of local bisections of $\G$, where the composition of elements of $\Diff_{\text{\rm loc}}(M)$, is replaced by the product of local bisections
(see definition \ref{Lie groupoids} and the previous subsection).

\begin{definition}
A {\bf generalized Lie pseudogroup} is a Lie groupoid $\G$ together with a subset $\Gamma\subset\Bis_{\text{\rm loc}}(\G)$, satisfying:
\begin{enumerate}
\item If $b\in \Gamma$, then $b^{-1}\in\Gamma$.
\item If $b_1,b_2\in \Gamma$ and $b_1\cdot b_2$ is defined, then $b_1\cdot b_2\in \Gamma$.
\item If $b\in\Gamma$ and $U$ is an open set contained in the domain of $b$, then $b|_U\in\Gamma$.
\item If $b:U\to \G$ is a bisection and $U$ can be covered be a family of open sets $U_i$ such that $b|_{U_i}\in\Gamma$ for all $i$, then $b\in\Gamma$.
\end{enumerate}
We also say that $\Gamma$ is supported by $\G$, or that $\Gamma$ is a $\G$-pseudogroup.
\end{definition} 

\begin{example}[The classical shadow] \rm \ Since any bisection $b$ of $\G$ induces a local diffeomorphism $\phi_b:= t\circ b$ on $M$, any generalized pseudogroup $\Gamma$ induces a classical pseudogroup 
\[ \Gamma^{\textrm{cl}}:= \{ \phi_b: b\in \Gamma \},\]
which we call the {\bf classical shadow} of $\Gamma$; 
however, in many case $\Gamma^{\textrm{cl}}$ is badly behaved although $\Gamma$ is not. 
\end{example}

As in the classical case (see subsection \ref{Lie pseudogroups}), 

\begin{itemize}
\item the $k$-jets of elements of $\Gamma$ define a subgroupoid
\begin{eqnarray*}
\Gamma^{(k)}\subset J^k\G;
\end{eqnarray*}
\item the order $k$ of $\Gamma$ is determined by the condition that, for a bisection $b$ of $\G$:
\[ j^{k}_{x}b\in \Gamma^{(k)} \ \ \forall \ \ x\in \textrm{Dom}(b) \Longrightarrow b\in \Gamma ;\]
\item we say that $\Gamma$ is smooth if each $\Gamma^{(k)}$ is smooth (as a subgroupoid of $J^k\G$);
\item in the smooth case (which we tacitly assume from now on), we obtain a tower of Lie groupoids
\[ \ldots \rmap \Gamma^{(2)}\rmap \Gamma^{(1)}\rmap \Gamma^{(0)} \]
and, passing to Lie algebroids, a tower of Lie algebroids
\[ \ldots \rmap A^{(2)} \rmap A^{(1)}\rmap A^{(0)} .\]
Here, each $A^{(k)}$ is a Lie sub-algebroid of $J^kA$, where $A$ is the Lie algebroid of $\G$; 
\item the adjoint actions of $J^k\G$ (see the previous subsections) restrict to give the adjoint representations of $\Gamma$:
\[ A^{(k-1)}, J^{k-1}TM \in \textrm{Rep}(\Gamma^{(k)});\]
\item the Cartan form on $J^k\G$ restricts to give the Cartan forms of $\Gamma$:
\[ \theta\in \Omega^1(\Gamma^{(k)}; A^{(k-1)}).\]
Similarly for the classical Spencer relative connections:
\[ D^{\textrm{clas}}: \Gamma(A^{(k)})\rmap \Omega^1(M,  A^{(k-1)});\]
\item for each $k$, one considers 
\[ \mathfrak{g}^{k}(\Gamma):= \textrm{Ker} (A^{(k)}\rmap A^{(k-1)}).\]
Since the kernel of the map $J^kA\rmap J^{k-1}A$ is canonically identified with $\textrm{Hom}(S^kTM, A)$, we have canonical inclusions
\[ \mathfrak{g}^{k}(\Gamma)\subset \textrm{Hom}(S^kTM, A),\]
hence $\mathfrak{g}^k(\Gamma)$ is a tableaux (bundle). Again, $\mathfrak{g}^{k+1}(\Gamma)$ sits inside the prolongation of $\mathfrak{g}^k(\Gamma)$
hence, in the terminology of definition \ref{tableaux tower}
\[ \g^{\infty}(\Gamma): = (\g^1(\Gamma), \g^2(\Gamma), \ldots ) \]
is a bundle of tableaux towers over $M$.
\end{itemize}

What is maybe less obvious is the notion that corresponds to the transitivity from the classical case. To avoid confusions, we will use the term ``full''. More precisely, given a generalized 
$\Gamma$ supported by $\G$, we say that $\Gamma$ is a \textbf{full generalized pseudogroup} if $\Gamma^{(0)}= \G$. We expect that a large part of the classical theory of transitive Lie pseudogroups can be generalized to such full pseudogroups.

Here are some examples of generalized pseudogroups. We start with some trivial ones, in order to illustrate the concept.

\begin{example}[The ``intransitive case'' from the literature] \rm \ For any submersion $p: M\to B$ between two manifolds, one can form the subgroupoid $\G:= M\times_{B}M$ of the pair groupoid of $M$, consisting of pairs $(x, X)$ with $p(x)= p(X)$. In this case, generalized pseudogroups $\Gamma$ supported by $\G$ are the same thing as classical pseudogroups on $M$ with the property that 
$p\circ f= p$ --  a property that is usually formulated as: ``$p$ is an invariant of the pseudogroup''. As we have mentioned above, a large part of the literature on ``intransitive pseudogroups'' refers to this rather superficial relaxation of transitivity: one requires that the orbits of $\Gamma$ are the fibers of a submersion $p: M\to B$. In other words, one is looking at full generalized pseudogroups supported by groupoids of type $M\times_{B}M$. Of course, the main point in this case is that, since $M\times_{B}M\subset M\times M$, one is still dealing with classical pseudogroups. 
\end{example}

\begin{example}[Rotations of the plane] \rm \ On $M=\mathbb{R}^2$ (with coordinates denoted by $(x, y)$) we consider the pseudogroup $\Gamma^{\textrm{naive}}$ consisting of all the transformations $(X, Y)$ of type 
\[ \Gamma^{\textrm{naive}}: X= x\cdot \textrm{cos}(\alpha)+ y\cdot \textrm{sin}(\alpha), \ Y= - x\cdot \textrm{sin}(\alpha)+ y\cdot \textrm{cos}(\alpha),\]
with $\alpha \in \mathbb{R}$ variable. Intuitively, $\Gamma^{\textrm{naive}}$ is a Lie pseudogroup, because its elements can be described as the solutions of the system
\[ \frac{\partial X}{\partial x}= \frac{\partial Y}{\partial y}, \ \ \ \frac{\partial X}{\partial y}= - \frac{\partial Y}{\partial x}, \ \ \ X^2+ Y^2= x^2+ y^2 .\]
However, this groupoid has orbits the concentric circles around the origin, plus the origin itself, hence it fails the usual regularity assumptions; the origin poses smoothness problems in the associated jet groupoids (the starting components $(x, y, X, Y)$ are ``free'' except for $x= y= 0$, when one must have $X= Y= 0$). 

There is also a more philosophical criticism on $\Gamma^{\textrm{naive}}$: while this is clearly about the standard action of $S^1$ on $\mathbb{R}^2$ by rotations, this geometric fact is not really reflected in the objects associated to $\Gamma^{\textrm{naive}}$.

Our point is to look at this example slightly differently and to encode the geometric situation in the associated groupoid. More precisely, consider the groupoid associated to the action of $S^1$ on $\mathbb{R}^2$ (see example \ref{ex-action-groupoids}):
\[ \G:= S^1\ltimes \mathbb{R}^2 .\]
Bisections of $\G$ correspond to functions $Z: \mathbb{R}^2\to S^1$ with the property that $(x, y)\mapsto Z\cdot (x, y)$ is a local diffeomorphism. Clearly, the $\G$-pseudogroup $\Gamma$ that is naturally associated to our problem is the one defined by the condition that $Z= \textrm{const}$:
\[ \Gamma: \frac{\partial Z}{\partial x}= \frac{\partial Z}{\partial y}= 0.\]
As it can be easily seen, $\Gamma$ is a smooth $\G$-pseudogroup of order one. Actually $\Gamma^{(0)}= \G$ and each $\Gamma^{(k)}$ with $k\geq 1$ is a subgroupoid of $J^k\G$ isomorphic to $\G$.

Of course, the classical pseudogroup associated to $\Gamma$, $\Gamma^{\textrm{cl}}$, is precisely the naive version $\Gamma^{\textrm{naive}}$.
\end{example}

\begin{example}[The $SL_2$-example] \rm \ There are interesting cases of classical pseudogroups for which, although they are well behaved, it is interesting to realize them as the ``classical shadow'' of a generalized pseudogroup. For instance, the pseudogroup from example \ref{SL-2-example},
\[ \Gamma:\ \ X= \frac{ax+b}{cx+d}, \ \ ad-bc= 1,\]
is clearly about the action of $SL_2(\mathbb{R})$ on $\mathbb{R}$ (a fact that was not really reflected by the discussion in the example that we mentioned). The full description of the geometric situation is very similar to the one from the previous example: one has the action groupoid $\G= SL_2(\mathbb{R})\ltimes \mathbb{R}$ and a $\G$-pseudogroup $\widetilde{\Gamma}$ characterized by $Z= \textrm{const}\in SL_2(\mathbb{R})$ (i.e. bisections of type $(x, y)\mapsto (A, x, y)$ with $A\in SL_2(\mathbb{R})$). Again, $\Gamma$ is just the classical shadow of $\widetilde{\Gamma}$.
\end{example}

\begin{example}[Action groupoids] \rm \ Of course, the previous two examples fit into a more general class of examples, associated to an action of a Lie group $G$ on a manifold $M$, to which one associates a classical pseudogroup $\Gamma$ generated by multiplications by elements of $G$, as well as a $\G$-pseudogroup $\widetilde{\Gamma}$, where $\G= G\ltimes M$ and $\Gamma$ consists of the local bisections which are constant in the $G$-argument. Note that $\G$ admits a 1-form similar to the Cartan form, $\xi \in \Omega^1(\G, \mathfrak{g})$, with values in the Lie algebra of $G$ and such that $\widetilde{\Gamma}$ is characterized by $b^*(\xi)= 0$. 

Of course, in general, while $\widetilde{\Gamma}$ is smooth, $\Gamma$ has little chance to be well-behaved. 
\end{example}

\begin{example}[The pseudogroup of Lagrangian bisections of a Poisson manifold] \rm  Let $(M, \pi)$ be a Poisson manifold, i.e. a manifold endowed with a bivector $\pi$ with the property that the bracket
\[ \{f, g\}:= df\wedge dg (\pi)= \pi(df, dg)\]
makes $C^{\infty}(M)$ into a Lie algebra. It is well-known that (integrable) Poisson manifolds correspond to symplectic groupoids (see e.g. \cite{Zung}). 
Starting with $(M, \pi)$, the corresponding groupoid is most easily seen via the corresponding infinitesimal object- its Lie algebroid. More precisely, associated to $\pi$ there is a
Lie algebroid
\[ A_{\pi}:= T^{*}M,\]
with anchor $\pi^{\sharp}$ -- that is $\pi$ interpreted as a linear map $T^*M\to TM$ (hence it is determined by $\xi(\pi^{\sharp}(\eta))= \pi(\eta, \xi)$) and with the bracket uniquely determined by the Leibniz identity and the requirement that, on exact $1$-forms, 
\[ [df, dg]= d\{f, g\} .\]
On general elements $\xi, \eta\in \Omega^1(M)$, 
\[ [\eta, \xi]_{\pi}= L_{\pi^{\sharp}(\eta)}(\xi)- L_{\pi^{\sharp}(\xi)}(\eta)- d\pi(\eta, \xi)).\]
One says that $(M, \pi)$ is integrable if the algebroid $A_{\pi}$ is integrable (note that this property is equivalent to the more geometric -- and ``groupoid-free'' -- condition on the existence of complete symplectic realizations- see \cite{CrainicFernandes:jdd}). In this case, one denotes the unique $s$-simply connected Lie groupoid integrating $A_\pi$ by $\Sigma(M, \pi)$. It then follows that $\Sigma(M, \pi)$ carries a unique symplectic form $\omega$, which is multiplicative (see definition \ref{def-mult-form} for the general notion of multiplicativity) and has the property that the source map is a Poisson map. This statement is certainly non-trivial, as it involves an integrability theorem. Actually, that is precisely the integrability theorem that we will generalize in chapter \ref{The integrability theorem for multiplicative forms} (theorem \ref{t1}) to arbitrary forms (for the case of trivial coefficients that are relevant here, see subsection \ref{sec: trivial coefficients} and example \ref{sympl-gpds-from-t1}). 

Using the groupoid $\Sigma= \Sigma(M, \pi)$ and its symplectic form, there is a natural $\Sigma$-pseudogroup associated to $(M, \pi)$: the pseudogroup of Lagrangian bisections:
\[ \Gamma:= \{ b\in \Bis_{\text{\rm loc}}(\Sigma): b^*(\omega)= 0\}.\]
More geometrically, identifying a bisection $b$ with its graph $L_b\subset \Sigma$, $\Gamma$ consists of Lagrangian subspaces $L\subset \Sigma$ with the property that $s|_{L}$ and $t|_{L}$ are diffeomorphism into opens inside $M$. 
\end{example}

\begin{example}[The pseudogroup of Legendrian bisections of a Jacobi manifold] \rm \ There are several classes of examples that are similar to the previous one, obtained for various Poisson-related geometric structures. For Dirac structures, the discussion is completely similar (see e.g. \cite{BCWZ} for the relevant notions).  The case of Jacobi structures is more interesting, and the relevant notions will be recalled later in the thesis, in subsection \ref{Contact groupoids}. Let us only mention here that the ``basic theory'' for Jacobi manifolds and their associated (contact) groupoids is rather unsatisfactory in the existing literature. E.g., in \cite{Jacobi}, one works very indirectly by passing to Poisson manifolds and symplectic groupoids (by the symplectization functor). A rather unexpected application of our integrability theorem (theorem \ref{t1}) is that it allows us to understand the basics of Jacobi structures directly, shedding light even on the very definition of contact groupoids. The bottom line is that, associated to any integrable Jacobi structure $(L, \{\cdot, \cdot\})$ on $M$ (see subsection \ref{Contact groupoids}), there is a contact groupoid $(\Sigma, \theta)$, where $\theta\in \Omega^1(\Sigma, L)$ is a contact form with coefficients in a line bundle, which is multiplicative. Hence one has a natural $\Sigma$-pseudogroup, consisting of bisections that kill $\theta$ (Legendrian bisections).
\end{example}

\begin{example}[The generalized pseudogroup of jets] \rm \ Starting with any classical pseudogroup $\Gamma$ on $M$, for any $k\geq 0$, the set of bisections of $J^k(M, M)$:
\[ \Gamma^{(k)}:= \{ j^kb: b\in \Gamma \}\]
form a generalized pseudogroup living on $J^k(M, M)$, which has $\Gamma$ as its classical shadow (in some sense, the two are isomorphic). Note that this example is similar to the ones from the previous two examples since it has a description of type:
\[ \Gamma^{(k)}= \{ b\in \Bis_{\text{\rm loc}}(\G): b^*(\theta)= 0 \}\]
where, this time, $\G= J^k(M, M)$ and $\theta$ is the Cartan form. 
\end{example}

\begin{example}[Right translations] \rm \ On the other hand, to each generalized pseudogroup $\Gamma$ one can associate a classical pseudogroup. More precisely, given any bisection $b$ of a Lie groupoid $\G$, one has an induced right translation by $b$,
\[ R_b: \G \rmap \G, \ \ R_b(g)= g \cdot \overline{b}(s(g)),\]
where $\overline{b}$ is the section of the target map $t: \G\to M$ induced by $b$:
\[ \overline{b}(y):= b(\phi_{b}^{-1}(y)) \ \ \ (\phi_b= t\circ b).\]
Of course, this makes sense also for local bisections. 
In particular, if $\Gamma$ is a $\G$-pseudogroup, then one has an induced classical pseudogroup induced by right
translations, acting on $\G$, consisting of all transformations of type $R_b$ with $b\in \Gamma$. 
\end{example}

\begin{example}[Groups] \rm \ Here is a more ``linguistic example''. In general, Lie groups $G$ are interpreted as Lie pseudogroups by letting them act on $G$ by right (or left) translations. This defines a pseudogroup $\Gamma_G$ on $G$. With our terminology, we can say that $G$ can be viewed as a generalized $G$-pseudogroup, and $\Gamma_G$ is the classical pseudogroup induced by right
translations, in the sense of the previous example. 
\end{example}

\section{Outline of the main objects of this thesis}

Here is a short outline of the main key-words of the thesis. Note that each one of them corresponds to a different chapter, but the order is different (here we present them in the more natural order). The ultimate concept is that of Pfaffian groupoid. One aim is to study it using Lie's infinitesimal methods, and that gives rise to Spencer operators on Lie algebroids. However, before we start with the multiplicative (or ``oid") story, we first go through the case of bundles (without any multiplication involved). In all these cases, we will carry out some of the standard theory from the geometry of PDEs (solutions, prolongations, Spencer cohomology), attempting to attain a more conceptual formulation or understanding of the objects involved (e.g. in the case of prolongations we will point out their universal property, while for prolongations of Pfaffian groupoids we will show that the compatibility condition of the Cartan forms involved is equivalent to a Maurer-Cartan equation). \\

 \hspace*{-.3in} \textbf{Pfaffian bundles:} A Pfaffian bundle
 $$\pi:R\To M,\quad H\subset TR$$
 is a  bundle $\pi$, together with a distribution $H\subset TR$ satisfying some compatibility conditions with $\pi$ (see definition \ref{def: pfaffian distributions}). One of the main notions of Pfaffian bundles is that of a {\bf solution}. These are sections $\sigma\in\Gamma(R)$ with the property that 
 \begin{eqnarray*}
 d_x\sigma(T_xM)\subset H_{\sigma(x)},\quad\text{for all }x\in M.
 \end{eqnarray*}
 Of course one also has the dual picture with one forms. The two points of view will be present in this thesis. 
 
 We remark that our main interest is Pfaffian bundles along the source map of a Lie groupoid. \\
 
 \hspace*{-.3in} \textbf{Relative connections:} Relative connections are linear Pfaffian bundles (this will require a precise formulation and a proof). They can be described using simpler data: some operator
 \begin{eqnarray*}
 \X(M)\times\Gamma(F)\To\Gamma(E),
 \end{eqnarray*}
satisfying connection-like properties along a vector bundle map $l:F\to E$ (see definition \ref{definitions}).  These also arise as the linearization (along solutions) of general Pfaffian bundles. \\

\hspace*{-.3in} \textbf{Pfaffian groupoids:} Pfaffian groupoids are the multiplicative version of Pfaffian bundles; they can actually be viewed as Pfaffian bundles if we consider the source map $s:\G\to M$. More precisely, a Pfaffian groupoid is Lie groupoid $\G$, together with a (multiplicative) distribution $$\H\subset T\G$$ satisfying some compatibility with the groupoid structure (see definition \ref{def: linear distributions}). The dual point of view consists of (multiplicative) one forms on $\G$ with values in some representation, which are compatible with the groupoid structure (see definition \ref{def-mult-form}). Pfaffian groupoids provide the main framework for all the examples of groupoids that arise from (generalized) Lie pseudogroups.\\ 

\hspace*{-.3in} \textbf{Spencer operators:} Spencer operators arise as the infinitesimal counterpart of Pfaffian groupoids. They are actually the linearization of a Pfaffian groupoid along the unit map. In particular they are relative connections
\begin{eqnarray*}
\X(M)\times \Gamma(A)\To \Gamma(E)
\end{eqnarray*}
on a Lie algebroid $A$, satisfying some compatibility conditions with the algebroid structure (see definition \ref{def1}).
It is remarkable that the Pfaffian groupoid  itself can be recovered from its Spencer operator (see theorem \ref{t2}).\\

\hspace*{-.3in} \textbf{Multiplicative forms and Spencer operators of degree $k$:} More generally we will consider multiplicative forms and Spencer operators of arbitrary degree. This generalization turns out to be related also to Poisson and related geometries (we will actually present an application that clarifies some basic aspects of Jacobi manifolds in subsection \ref{Contact groupoids}). \\

\chapter{Relative connections}\label{Relative connections}
\pagestyle{fancy}
\fancyhead[CE]{Chapter \ref{Relative connections}} 
\fancyhead[CO]{Relative connections} 

Motivated by our attempt to understand linear Pfaffian bundles, we are brought to the notion of relative connection. They also appear as the complete data of linear distribution in the sense of \ref{def: linear distributions}. An example one should keep in mind of a relative connection is the classical Spencer operator described in subsection \ref{sp-dc}.\\

Throughout this chapter let $K,E,F_k$ denote the total spaces of smooth vector bundles over a smooth manifold $M$, for any integer $k\geq-1$. We denote by $\pi:J^kF\to M$ the $k$-jet (vector) bundle of $F$, and by $pr:J^{k+1}F\to J^kF$ the canonical projection. See section \ref{jet-bundles} and remark \ref{remark-J-decomposition} for notation.
For ease of notation $T$ and $T^*$ denote the tangent and cotangent bundles of the manifold $M$ respectively. Given $\pi : F \to M$, let $T^\pi F$ denote the subbundle of $TF$ given by $\ker d\pi$, the bundle of vectors tangent to the fibers of $\pi$. 
We will be repeatedly using the Spencer decomposition
\begin{eqnarray}\label{Spencer decomposition}
\Gamma(J^1E)\simeq \Gamma(E)\oplus \Omega^1(M,E)
\end{eqnarray} 
 and the classical Spencer operator relative to the projection $pr:J^1E\to E$
\begin{eqnarray}\label{Spencer operator}
D^\clas:\Gamma(J^1E)\To\Omega^1(M,E)
\end{eqnarray}
See remark \ref{remark-J-decomposition}.

\section{Preliminaries, definitions and basic properties}

\subsection{Relative connections: definitions and basic properties}
 
 Let $l:F\to E$ be a surjective vector bundle map of vector bundles over $M$.
 
\begin{definition}\label{definitions} 
\mbox{}
\begin{itemize}
\item An {\bf $l$-connection on $F$ (or a connection on $F$ relative to $l$)}, is any linear operator
\begin{equation*}
D:\Gamma(F)\To\Omega^1(M,E)\quad s\mapsto D_X(s):=D(s)(X)
\end{equation*}
satisfying the Leibniz identity relative to $l$:
\begin{equation*}
D_X(fs)=fD_Xs+L_X(f)l(s)
\end{equation*}
for all $s\in\Gamma(F)$, $X\in\X(M)$ and $f\in C^\infty (M)$. We use the notation $(D,l):F\to E$; we also say that $(F,D)$ is a relative connection.
\item We say that $s\in\Gamma(F)$ is a {\bf solution of $(F,D)$} if 
$D_X(s)=0$
for all $X\in\X(M)$. We denote the set of solutions of $(F,D)$ by 
\begin{eqnarray*}\text{\rm Sol}(F,D)\subset \Gamma(F)\end{eqnarray*}
\item  {The \bf symbol space of $D$}, denoted by \begin{eqnarray*}\g(F,D)\subset F\end{eqnarray*} or simply by $\g$, is the subbundle given by the kernel of $l$.
\item We call the {\bf symbol map of $D$} the linear map over $M$
\begin{eqnarray*}
\partial_D:\g\To\hom(TM,E)
\end{eqnarray*}
 defined by 
$\partial_D(v)(X)=D_X(v)$
for $v\in \g_x$ and $X\in T_xM.$
\item The \textbf{$k$-prolongation of the symbol map $\partial_D$}, in the sense of definition \eqref{1st-prolongation}, is denoted  
by \begin{eqnarray*}\g^{(k)}(F,D)\subset S^kT^*\otimes \g\end{eqnarray*} or simply by $\g^{(k)}$. 
\item We say that an $l$-connection $D$ is \textbf{standard} if $\partial_D$ is injective.
\end{itemize}
\end{definition}

\begin{remark}\label{rk-convenient'}\rm
An $l$-connection can be interpreted as a vector bundle map 
\begin{eqnarray*}
j_D:F\To J^1E
\end{eqnarray*}
given at the level of sections by \begin{eqnarray*}j_D(s)=(l(s),D(s))\in\Gamma(E)\oplus\Omega^1(M,E),\end{eqnarray*} where we used the Spencer decomposition \eqref{Spencer decomposition}. Note that the restriction of $j_D=(l,D)$ to the symbol space $\g$ takes values in the subbundle $\hom(TM,E)$ of $J^1E$. Indeed, if $s$ is a section of $\g$, then 
\begin{eqnarray*}
pr(j_D(s))=pr((l(s),D(s))=pr(0,D(s))=0
\end{eqnarray*}
i.e. $j_D(s)\in\ker pr=\hom(TM,E)$, where $pr:J^1E\to E$ is the projection map.
Note that 
\begin{eqnarray*}
\partial_D=j_D|_{\g}.
\end{eqnarray*}
\end{remark}

\begin{lemma}\label{D-diff} Given $(D,l):F\to E$, there exists a unique linear operator 
\begin{eqnarray*}\bar{D}:\Omega^{*}(M,F)\to\Omega^{*+1}(M,E)
\end{eqnarray*}
such that
\begin{itemize}
\item on $\Gamma(F)=\Omega^0(M,F)$ is equal to $D$, and
\item satisfies the Leibniz identity relative to $l$, i.e. \begin{eqnarray}\label{unique}\bar{D}(\omega\otimes s)=d\omega\otimes l(s)+(-1)^{k}\omega\wedge D(s)\end{eqnarray} for any $k$-form $\omega\in\Omega^k(M)$ and any section $s\in\Gamma(F)$.
\end{itemize}
Moreover, $\bar{D}$ is given by the Koszul formula
\begin{equation}\label{D-differential}\begin{aligned}
\bar{D}_{(X_0,\ldots,X_k)}&(\theta)=\sum_{i}(-1)^{i}D_{X_i}(\theta(X_0,\ldots,\hat{X_i},\ldots,X_k))\\&+\sum_{i<j}(-1)^{i+j}l(\theta([X_i,X_j],X_0,\ldots,\hat{X_i},\ldots,\hat{X_j},\ldots,X_k))
\end{aligned}\end{equation}
for any $\theta\in\Omega^k(M,F).$
\end{lemma}

\begin{proof}
Uniqueness follows from the fact that locally any form $\theta\in\Omega^k(M,F)$, can be written as 
\begin{eqnarray*}
\theta=\sum_i\omega^{i}\otimes s_i
\end{eqnarray*}
where $\omega^{i}\in\Omega^k(M)$ and $s_i\in\Gamma(F)$. So 
\begin{eqnarray*}
\bar{D}(\theta)=\sum_i\bar{D}(\omega^{i}\otimes s_i)=\sum_i(d\omega^{i}\otimes l(s_i)+(-1)^{k}\omega^{i}\wedge D(s_i)).
\end{eqnarray*} 
On the other hand, $\bar{D}$ defined by equation \eqref{D-differential} clearly extends $D$ and satisfies the Leibniz identity relative to $l$. Indeed, 
\[\begin{aligned}
\bar{D}_{(X_0,\ldots,X_k)}&(\omega\otimes s)=\sum_{i}(-1)^{i}D_{X_i}(\omega(X_0,\ldots,\hat{X_i},\ldots,X_k)s)\\&+\sum_{i<j}(-1)^{i+j}l(\omega([X_i,X_j],X_0,\ldots,\hat{X_i},\ldots,\hat{X_j},\ldots,X_k)s)\\
=&\sum_{i}(-1)^{i}\omega(X_0,\ldots,\hat{X_i},\ldots,X_k)D_{X_i}(s)\\
&+\sum_{i}(-1)^{i}L_{X_i}(\omega(X_0,\ldots,\hat{X_i},\ldots,X_k))l(s)\\
&+\sum_{i<j}(-1)^{i+j}\omega([X_i,X_j],X_0,\ldots,\hat{X_i},\ldots,\hat{X_j},\ldots,X_k)l(s)\\
=& d\omega\otimes l(s)+(-1)^k\omega\wedge D(s).
\end{aligned}\]
\end{proof}

\begin{remark}\label{delta}\rm
The restriction of $\bar{D}:\Omega^k(M,F)\to\Omega^{k+1}(M,E)$ to $\Omega^k(M,\g)$ is a $C^{\infty}(M)$-linear map and therefore it induces a vector bundle map over $M$
\begin{eqnarray*}
\bar\partial_{D}:\wedge^{k}T^*\otimes \g\To \wedge^{k+1}T^*\otimes E,\quad 
\bar\partial_D(\omega\otimes s)=(-1)^k\omega\wedge \partial_D(s)
\end{eqnarray*}
for any $\omega\in\wedge^kT^*$ and $s\in\g$. The proof is an easy consequence of formula \eqref{unique}.
\end{remark}

\begin{lemma}\label{standard} An $l$-connection $(D,l):F\to E$ is standard if and only if 
\begin{equation}
j_D:F\to J^1E
\end{equation}
is injective. In this case $D=D^\clas\circ j_D$, where $D^\clas$ is the classical Spencer operator \eqref{Spencer operator} of $E$.
\end{lemma}

\begin{proof}
For $s,s'$ sections of $F$, $$j_D(s-s')=(l(s-s'),D(s-s')).$$ So, if $j_D(s-s')(x)=0$ for some $x\in M$ then $l(s(x))=l(s'(x))$, i.e. $s(x)-s'(x)\in\g$, and therefore we can assume that the section $s-s'$ belongs to $\Gamma(\g)$. With this, $j_D(s-s')(x)=0$ if and only if $D(s)(x)=D(s')(x)$, and  
\begin{equation}\label{eq:2}
\partial_D(s(x)-s'(x))=D(s-s')(x)=D(s)(x)-D(s')(x)=0.
\end{equation}
That $D=D^{\text{clas}}\circ j_D$ follows from the fact that any $l$-connection $D$ can be recovered from $D^{\text {clas}}$ by $D=D^{\text {clas}}\circ j_D$.
\end{proof}

\subsection{First examples and operations}\label{first examples and operations}

This subsection has the first basic examples you encounter in the literature and some operations on relative connections which we will use later on.

\begin{example}\label{linear connections}\rm
A linear connection 
\begin{eqnarray*}\nabla:\X(M)\times\Gamma(F)\To\Gamma(F)\end{eqnarray*} 
is a standard connection relative to the identity map $id:F\to F$. It is well-known that linear connections $\nabla$ correspond to horizontal linear distributions $H\subset TF$. This correspondence is given by
\begin{eqnarray*}
\nabla_X(s)=[\tilde X,\tilde s]
\end{eqnarray*} 
where $X\in\X(M)$, $\tilde X\in\Gamma(H)\subset \X(M)$ is the lift of $X$ to $H$, and $\tilde s\in\Gamma(T^\pi F)\subset\X (F)$ is the extension of $s$ to the vector field constant on each fiber of $F$.

It is natural to have a similar interpretation for $l$-connections. This will be done in corollary \ref{t:linear distributions}.
\end{example} 

\begin{example}[The classical Spencer operator on higher jets, see \cite{Spencer,Veloso, Ngo, Spencerflatt,Ngo1}]\label{Spencer tower} \rm
On higher jets, we also have a version of the classical Spencer operator (see remark \ref{remark-J-decomposition} for $k=1$): On $\pi:J^{k}F\to M$, $D^{k\text{\rm-clas}}$ is the unique relative connection
\begin{eqnarray*}
D^{k\text{\rm-clas}}:\Gamma(J^{k}F)\To\Omega^1(M,J^{k-1}F)
\end{eqnarray*}
satisfying the Leibniz identity relative to the projection $pr:J^{k}F\to J^{k-1}F$, and the holonomicity condition 
\begin{eqnarray*}
D^{k\text{-clas}}(j^{k}s)=0.
\end{eqnarray*}
Alternatively, $D^{k\text{\rm-clas}}$ is the restriction to $\Gamma(J^{k}F)\subset \Gamma(J^1(J^{k-1}F))$ of the classical Spencer operator
\begin{eqnarray*}
D^{\text{\rm clas}}:\Gamma(J^1(J^{k-1}F))\To\Omega^1(M,J^{k-1}F)
\end{eqnarray*}
associated to the vector bundle $J^{k-1}F \to M$.
In this case, the symbol space $\g_k=\ker pr$ is equal to $S^{k}T^*\otimes F$. This comes from the short exact sequence of vector bundles over $M$
\begin{eqnarray*}
0\To S^{k}T^*\otimes F\xrightarrow{i}J^{k}F\xrightarrow{pr}J^{k-1}F\To 0
\end{eqnarray*} 
where $i$ at the point $x$ is determined by
\begin{eqnarray*}
i((df_1\otimes\cdots\otimes df_{k})\otimes s)=j^{k}_x(f_1\cdots f_{k}s)
\end{eqnarray*}
for any $f_1,\ldots,f_{k}\in C^{\infty}(M)$ such that $f_1(x)=\cdots=f_{k}(x)=0$. Note also that $D^{k{\text-\clas}}$ is a standard $pr$-connection as $j_{D^{\text{clas}}}=id$.

\end{example}

Now we concentrate on operations of relative connections. For this purpose let's fix some notation. Let $(D^1,l^1):F_1\to E_1$ and $(D^2,l^2):F_2\to E_2$ be relative connections and let $\psi: P\to M$ be a smooth map.

\subsubsection{Composition}
Let 
\begin{eqnarray*}
K\xrightarrow{(\tilde D,\tilde l)} F\xrightarrow{(D,l)}E
\end{eqnarray*}
be two relative connections. We can compose $(K,\tilde D)$ with $(F,D)$ to obtain connections
\begin{eqnarray*}
D^{i}:\Gamma(K)\To \Omega^1(M,E),\quad i=1,2
\end{eqnarray*}
relative to the composition $l\circ \tilde l:K\to E$ in two different ways, namely
\[\begin{aligned}
D^1(\al)&=l\circ \tilde D(\al)\\
D^2(\al)&= D(\tilde l(\al)),
\end{aligned}\]
where $\al\in \Gamma(K)$. Note that $D^1=D^2$ if and only if $D\circ \tilde l-l\circ \tilde D=0$.

\subsubsection{Direct sum} The direct sum
\begin{eqnarray*}
D^1\oplus D^2:\Gamma(F_1\oplus F_2)\To\Omega^1(M,E_1\oplus E_2)
\end{eqnarray*}
is the connection relative to $l^1\oplus l^2:F_1\oplus F_2\to E_1\oplus E_2$ defined by the equation
\begin{eqnarray*}
D^1\oplus D^1(s_1,s_2):=(D^1(s_1),D^2(s_2))
\end{eqnarray*}
for any $(s_1,s_2)\in \Gamma(F_1\oplus F_2)$, where we canonically identify $\Gamma(F_1\oplus F_2)$ with $\Gamma(F_1)\oplus\Gamma(F_2)$ as $C^\infty(M)$-modules.

\subsubsection{Tensor product}
The tensor product 
\begin{eqnarray*}
D^1\otimes D^2:\Gamma(F_1\otimes F_2)\To\Omega^1(M,E_1\otimes E_2)
\end{eqnarray*}
is the connection relative to $l^1\otimes l^2:F_1\otimes F_2\to E_1\otimes E_2$ defined by
\begin{eqnarray*}
D^1\otimes D^1(s_1\otimes s_2):=D^1(s_1)\otimes l^2(s_2)+l^1(s_1)\otimes D^2(s_2)
\end{eqnarray*}
for any $s_1\otimes s_2\in\Gamma(F_1\otimes F_2)$, where we canonically identify $\Gamma(F_1\otimes F_2)$ with $\Gamma(F_1)\otimes\Gamma(F_2)$ as $C^\infty(M)$-modules.

\subsubsection{Pull-back}

The pull-back connection 
\begin{eqnarray*}
\psi^*D:\Gamma(\psi^*F)\To \Omega^1(P,\psi^* E)
\end{eqnarray*}
of the map $\psi:P\to M$, is a connection relative to $\psi^*l:\psi^* F\to \psi^*E$. It is given on sections of the form
\begin{eqnarray*}
\psi^*s:=\psi\circ s,\quad s\in \Gamma(F)
\end{eqnarray*}
by the formula
\begin{eqnarray}\label{q}
(\psi^*D)_X(\psi^*s):=\psi^*(D_{d\psi(X)}(s))
\end{eqnarray}
where $X\in TP$. Formula \eqref{q} is extended to $\Gamma(\psi^* F)$ by imposing the Leibniz identity relative to  $\psi^*l:\psi^* F\to \psi^*E$
\begin{eqnarray*}
(\psi^*D)_X(f\psi^*s)=f(\psi^*D)_X(\psi^*s)+L_X(f)\psi^*l(\psi^*s),\quad f\in C^{\infty}(P).
\end{eqnarray*}

\section{Prolongations of relative connections}\label{section:prolongation}

Roughly speaking, a prolongation of a relative connection $D$ is a relative connection sitting above $D$, defined on ``the first order approximation of solutions of $D$''. The main property of a prolongation is that it gives an injection between solutions of the prolongation and solutions of the original relative connection. We will have existence results on solutions on relative connections under some homological and smoothness conditions on the process of prolonging. \\

Throughout this section $(D,l):F\to E$ is an $l$-connection.

\subsection{Compatible connections (abstract prolongations)}\label{compatible connections}

\subsubsection{Definition and basic properties}

\begin{definition}\label{first definitions}
Let 
\begin{eqnarray*}\label{seq}
K\xrightarrow{(\tilde D,\tilde l)}F\xrightarrow{(D,l)}E
\end{eqnarray*}
be relative connections. 
\mbox{}\begin{itemize}
\item We say that $(D,\tilde D)$ are {\bf compatible connections} if the following two conditions are satisfied:
\begin{eqnarray}
\label{compatibility1}& D\circ \tilde l-l\circ \tilde D=0\\
\label{compatibility2}& D_X\tilde D_Y-D_Y\tilde D_X-l\circ \tilde D_{[X,Y]}=0
\end{eqnarray}
for any $X,Y\in\X(M)$.
\item A {\bf morphism} $\Psi$ from a pair of compatible connections $K\xrightarrow{(\tilde D,\tilde l)}F\xrightarrow{(D,l)}E$ to another $\hat{K}\xrightarrow{(\hat{\tilde D},\hat{\tilde l})}\hat{F}\xrightarrow{(\hat{D},\hat{l})}\hat{E}$
is a commutative diagram of vector bundles
\begin{eqnarray*}
\xymatrix{
K \ar[r]^{\tilde l} \ar[d]^{\Psi_2} & F \ar[r]^l \ar[d]^{\Psi_1} & E \ar[d]^{\Psi_0}\\
\hat K \ar[r]^{\hat {\tilde {l}}}  & \hat F \ar[r]^{\hat l}  & \hat E 
}\end{eqnarray*}
such that 
\begin{eqnarray*}\label{conmuta}
\hat {\tilde {D}}\circ \Psi_2=\Psi_1\circ\tilde{D} ,\quad\quad\hat{D}\circ \Psi_1=\Psi_0\circ D
\end{eqnarray*} 
We say that $\Psi$ is injective if $\Psi_0,\Psi_1$ and $\Psi_2$ are injective.
\item We say that $(B,\tilde D)$ is a {\bf prolongation} of $(F,D)$ if $(D,\tilde D)$ are compatible connections and $\tilde D$ is standard.
\end{itemize}
\end{definition} 

Compatible relative connections already appear in \cite{Gold3,Ngo1}, in the more familiar form described by the following lemma

\begin{lemma}\label{complex}
Let 
\begin{eqnarray*}
K\xrightarrow{(\tilde D,\tilde l)}F\xrightarrow{(D,l)}E
\end{eqnarray*}
by relative connections. If $\dim M>0$, then $(D,\tilde D)$ are compatible connections if and only if the composition
\begin{eqnarray*}
\Omega^*(M,K)\xrightarrow{\bar{\tilde D}}\Omega^{*+1}(M,F)\xrightarrow{\bar{D}}\Omega^{*+2}(M,E)
\end{eqnarray*}
is zero.
\end{lemma}

\begin{proof}
Let $s\in\Gamma(K)$ and $\omega\in\Omega^k(M)$. Then
\begin{equation}\label{ax}
\bar{D}\circ \bar{\tilde D}(s)(X,Y)=D_X(\tilde D_Y(s))-D_Y(\tilde D_X(s))-l(\tilde D_{[X,Y]}(s))
\end{equation}
where $X,Y\in\X(M)$, and by formula \eqref{unique}, we have that 
\begin{equation}\label{bx}
\begin{aligned}
\bar{D}\circ \bar{\tilde D}(\omega\otimes s)=&D(d\omega\otimes \tilde l(s)+(-1)^k\omega\wedge \tilde D(s))\\
=&(-1)^{k+1}d\omega\wedge D(\tilde l(s))+(-1)^kd\omega \wedge l(\tilde D(s))\\&+(-1)^k\omega\wedge \bar{D}\circ\bar{D}(s)
\end{aligned}
\end{equation}
Hence, expression \eqref{ax} is zero for all $s\in\Gamma(F)$ if and only if condition \eqref{compatibility2} for compatible connections is satisfied, and under this assumption expression \eqref{bx} is zero for all $\omega\in\Omega^k(M)$ if and only if condition \eqref{compatibility1} is fulfilled.
\end{proof}

For 
\begin{eqnarray*}
K\xrightarrow{(\tilde D,\tilde l)}F\xrightarrow{(D,l)}E
\end{eqnarray*}
compatible connections, let $\g$ and $\g'$ be the symbol spaces of $D$ and $\tilde D$ respectively and consider the restriction map 
\begin{eqnarray*}
\g'\xrightarrow{\partial_{\tilde D}}\hom(TM,F)
\end{eqnarray*} 
of $j_{\tilde D}:K\to J^1F$.

\begin{lemma}\label{values}
The map 
\begin{eqnarray*}\partial_{\tilde D}:\g'\to \hom(TM,F)
\end{eqnarray*}
 takes values in $\hom (TM,\g)$. 
\end{lemma}

\begin{proof}
Recall that at the level of sections 
\begin{equation}
\partial_{\tilde D}(s)=(\tilde l(s),\tilde D(s))=(0,\tilde D(s))=\tilde D(s)
\end{equation}
for any $s\in\Gamma(\g')$. On the other hand, by condition \eqref{compatibility1},
\begin{equation}
l\circ \tilde D(s)=D\circ \tilde l(s)=0.
\end{equation}
Hence, $l\circ \partial_{\tilde D}(s)=0$, or equivalently $\partial_{\tilde D}(s)\in\hom(TM,\g).$
\end{proof}

The following corollary is a consequence of the previous lemma.

\begin{corollary}\label{secuencia}
Let
\begin{eqnarray*}
K\xrightarrow{(\tilde D,\tilde l)}F\xrightarrow{(D,l)}E
\end{eqnarray*}
be compatible connections. Then
for each integer $k$ the sequence of vector bundles
\begin{eqnarray*}
\wedge^kT^*\otimes\g'\xrightarrow{\bar\partial_{\tilde D}}\wedge^{k+1}T^*\otimes\g\xrightarrow{\bar\partial_{D}}\wedge^{k+2}T^*\otimes E
\end{eqnarray*}
is a complex, i.e. $\bar\partial_{D}\circ\bar\partial_{\tilde D}=0$, where $\g'$ and $\g$ are the symbols of $\tilde D$ and $D$ respectively. See remark \ref{delta} for the definition of $\bar\partial_D$.
\end{corollary}

\begin{lemma}\label{cohomologia}Let 
\begin{eqnarray*}
K\xrightarrow{(\tilde D,\tilde l)}F\xrightarrow{(D,l)}E,\quad\quad \hat K\xrightarrow{(\hat{\tilde{D}},\hat{\tilde{l}})}\hat F\xrightarrow{(\hat D,\hat l)}\hat E
\end{eqnarray*}
be compatible connections. Then, if $\Psi$ is a morphism as in definition \ref{first definitions}, then $\Psi$ induces a morphism $\Psi_*$ of complexes 
\begin{eqnarray*}
\xymatrix{\Omega^*(M,K)\ar[r]^{{\tilde D}} \ar[d]^{\circ\Psi_2} &\Omega^{*+1}(M,F)\ar[r]^{{D}} \ar[d]^{\circ\Psi_1} &\Omega^{*+2}(M,E) \ar[d]^{\circ\Psi_0}\\
\Omega^*(M,\hat K) \ar[r]^{{\hat{\tilde D}}}  &\Omega^{*+1}(M,\hat F) \ar[r]^{{\hat D}}& \Omega^{*+2}(M,E) 
}\end{eqnarray*}
by pre-composing with $\Psi$, which restricts to a morphism of vector bundle complexes
\begin{eqnarray*}
\xymatrix{
\wedge^kT^*\otimes\g' \ar[r]^{\partial_{\tilde D}} \ar[d]^{\circ\Psi_2} & \wedge^{k+1}T^*\otimes\g \ar[r]^{\partial_{D}} \ar[d]^{\circ\Psi_1} & \wedge^{k+2}T^*\otimes E \ar[d]^{\circ\Psi_0}\\
\wedge^kT^*\otimes\hat\g' \ar[r]^{\partial_{\hat{\tilde D}}}  & \wedge^{k+1}T^*\otimes\hat\g \ar[r]^{\partial_{\hat D}}  & \wedge^{k+2}T^*\otimes \hat E
}
\end{eqnarray*}
where $\hat\g'$ and $\hat\g$ are the symbols of $\hat{\tilde{D}}$ and $\hat D$ respectively. Moreover, if $\Psi$ is injective then $\Psi_*$ is injective in cohomology, i.e.
\begin{eqnarray*}
\frac{\ker(\Omega^{*+1}(M,F)\xrightarrow{\bar D}{}\Omega^{*+2}(M,E))}{\Im(\Omega^{*}(M,K)\xrightarrow{\bar \tilde D}{}\Omega^{*+1}(M,F))}\xrightarrow{\circ\Psi_1}\frac{\ker(\Omega^{*+1}(M,\hat F)\xrightarrow{\bar{\hat D}}{}\Omega^{*+2}(M,\hat E))}{\Im(\Omega^{*}(M,\hat K)\xrightarrow{\bar{\hat{\tilde{D}}}}{}\Omega^{*+1}(M,\hat F))}
\end{eqnarray*}
is injective.
\end{lemma}

\begin{proof}[Proof of lemma \ref{cohomologia}] 
This is a standard argument for morphisms of complexes using equalities \eqref{conmuta}. The proof is left to the 
reader. 
\end{proof}

\subsubsection{Solutions of compatible connections}

Throughout this section let  
\begin{eqnarray*}
K\xrightarrow{(\tilde D,\tilde l)}F\xrightarrow{(D,l)}E
\end{eqnarray*}
be compatible connections, and let $\g$ and $\g'$ be the symbol spaces of $D$ and $\tilde D$ respectively.
Consider 
\begin{eqnarray}\label{solution-map}
\text{\rm Sol}(K,\tilde D)\xrightarrow{\tilde l}\text{\rm Sol}(F,D),
\end{eqnarray} 
the $C^{\infty}(M)$-linear map given by the restriction of $\Gamma(K)\xrightarrow{\tilde l}\Gamma(F), \tilde l(s):=\tilde l\circ s$ to $\text{\rm Sol}(K,\tilde D)$. The fact that $\tilde l(s)$ belongs to $\text{\rm Sol}(F,D)$ for any $s\in\text{\rm Sol}(B,\tilde D)$ follows from condition \eqref{compatibility1}, i.e. $D\circ \tilde l(s)=l\circ \tilde D(s)=0$.

In this subsection we will construct a map 
\begin{eqnarray*}
S:\text{Sol}(F,D)\To H^{0,1}(\g)
\end{eqnarray*}
where 
\begin{eqnarray*}
H^{0,1}(\g):=\frac{\ker\{T^*\otimes\g\xrightarrow{\partial_D}\wedge^2T^*\otimes E\}}{\Im\{\partial_{\tilde D}:\g'\to T^*\otimes\g\}}
\end{eqnarray*}
and such that the following proposition holds.

\begin{proposition}\label{S}
The sequence 
\begin{eqnarray*}
\text{\rm Sol}(K,\tilde D)\xrightarrow{\tilde l}\text{\rm Sol}(F,D)\xrightarrow{S}H^{0,1}(\g)
\end{eqnarray*}
is exact. Moreover, if $\tilde D$ is a standard $\tilde l$-connection then $\tilde l$ is injective. 
\end{proposition}

\begin{remark}\rm Note that if $H^{0,1}(\g)=0$, then $\tilde l$ on the above proposition is surjective. If, moreover, $\tilde D$ is standard then $\tilde l$ is a bijection.
\end{remark}

The complete proof of the above proposition will be given at the end of the subsection.

\begin{remark}\rm Notice that if $(K,\tilde D)$ is a prolongation of $(F,D)$, i.e. $\tilde D$ is standard, the map \eqref{solution-map} is injective. Indeed, if $s,\al\in\text{\rm Sol}(K,\tilde D)$ are such that $\tilde l(s)=\tilde l(\al)$, then $s-\al\in\Gamma(\g')$ and 
\begin{eqnarray*}
\partial_{\tilde D}(s-\al)=\tilde D(s-\al)=\tilde D(s)-\tilde D(\al)=0-0=0,
\end{eqnarray*}
i.e. $s-\al\in\ker(\partial_{\tilde D})$ which is zero as $\tilde D$ is standard. Thus $s=\al.$\end{remark}

Using the compatibility of the pair $(D,\tilde D)$, we construct the map 
\begin{equation*}
S:{\text{\rm Sol}}(F,D)\To \frac{\hom (TM,\g)}{\Im\{\partial_{\tilde D}:\g'\to T^*\otimes\g\}}
\end{equation*}
as follows: for $s$ a solution of $(F,D)$, by surjectivity of $\tilde l:K\to F$, we find a section $\al\in \Gamma(K)$ (not necessarily in $\text{\rm Sol}(K,\tilde D)$) such that $\tilde l\circ \al=s$. $\tilde D(\al):TM\to F$ has the property that its image lies inside $\g$, as $l\circ \tilde D(\al)=D\circ \tilde l(\al)=0$. We define $S(s)$ by the class of $\tilde D(\al):TM\to\g$.

\begin{lemma}
\begin{equation*}
S:{\text{\rm Sol}}(F,D)\To \frac{\hom (TM,\g)}{\Im\{\partial_{\tilde D}:\g'\to T^*\otimes\g\}}
\end{equation*}
is a well-defined map whose image lies in the subbundle $H^{0,1}(\g)$
\end{lemma} 

\begin{proof}
Let's first prove that $S$ is well-defined. To this end take two section $s^1$ and $\al$ of $K$ such that $\tilde l\circ \al=\tilde l\circ s^1=s$, and $D(\tilde l(\al))=D(\tilde l(s^1))=D(s)=0$. This means that $s^1-\al$ is a section of $\g'$ and 
\begin{eqnarray*}
\partial_{\tilde D}(s^1-\al)=\tilde D(s^1-\al)=\tilde D(s^1)-\tilde D(\al)=0
\end{eqnarray*}
In other words, the class of $S(s)$ does not depend on the representative. Moreover, for $X,Y\in\X(M)$
\[\begin{aligned}
\partial_D(\tilde D(s^1))(X,Y)&=D_X\tilde D_Y(s^1)-D_Y\tilde D_X(s^1)\\
&=l\circ \tilde D_{[X,Y]}(s^1)=D_{[X,Y]}(\tilde l(s^1))=0
\end{aligned}\]
where in the second and third equality we use conditions \eqref{compatibility2} and \eqref{compatibility1} respectively. This proves that $S$ takes values where we claimed.
\end{proof}
 
\begin{proof}[Proof of proposition \ref{S}]
We already noticed that if $\tilde D$ is standard then $$\tilde l:\Sol(K,\tilde D)\To \Sol(F,D)$$ is injective. Let's concentrate on the exactness of the sequence.

Let $\al\in\Gamma(F)$ be a solution of $(K,\tilde D)$, i.e. $\tilde D(\al)=0$. Then $S (\tilde l(\al))$ is the class of the map
\begin{eqnarray*}
\tilde D(\al):TM\To \g
\end{eqnarray*} 
which is identically zero. Conversely, if $s\in\text{\rm Sol}(F,D)$ is such that $S(s)=0$ and $\al\in\Gamma(K)$ is such that $\tilde l(\al)=s$, then there exists $\beta\in \Gamma(g')$ such that 
\begin{equation*}
\tilde D(\al)=\partial_{\tilde D}(\beta)=\tilde D(\beta).
\end{equation*}
Therefore, the section $\al-\beta\in \Gamma(K)$ is such that $\tilde l(\al-\beta)=\tilde l(\al)=s$ and $\tilde D(\al-\beta)=0$. Equivalently, $\al-\beta\in\text{\rm Sol}(K,\tilde D)$ is such that $\tilde l(\al-\beta)=s$.
\end{proof}  

\subsection{The partial prolongation of a relative connection}\label{basic notions}

The partial prolongation of a fixed $l$-connection $(D,l):F\to E$, is our first attempt to find a prolongation of $(F,D)$. Although its relative connection may fail to be compatible with $D$, it does satisfy compatibility condition \eqref{compatibility2} as we will see later on.
 
\begin{definition}\label{partial prolongation} {\bf The partial prolongation of $F$ with respect to $D$}, denoted by $J_D^1F$,
is the smooth vector subbundle of $\pi:J^1F\to M$ defined by
\begin{eqnarray*}
(J^1_DF)_x:=\{j^1_xs\mid D_X(s)(x)=0\text{ for all }X\in\X(M)\}
\end{eqnarray*}
for $x\in M.$
\end{definition}

\begin{remark}\label{cocyclea}\rm As
$D$ is a differential operator of order $1$, it induces a vector bundle map 
\begin{eqnarray*}
a:J^1F\To\hom(TM,E), \quad j^1_xs\mapsto D(s)(x)
\end{eqnarray*}
with $J^1_DF$ as the kernel of $a$. With the Spencer decomposition \eqref{Spencer decomposition} in mind, we have that $a$ at the level of sections is given by
\begin{eqnarray*}
a(\al,\omega)=D(\al)-l\circ \omega,
\end{eqnarray*}
for $(\al,\omega)\in\Gamma(F)\oplus \Omega^1(M,F)$, and therefore
\begin{eqnarray*}
\Gamma(J^1_DF):=\{(\al,\omega)\in\Gamma(F)\oplus\Omega^1(M,F)\mid D(\al)=l\circ \omega\}
\end{eqnarray*}
With this description is clear that $J^1_DF$ is a smooth vector subbundle as it is the kernel of a surjective vector bundle map and that the restriction $pr:J^1_DF\to F$ is surjective. Indeed, if $f:TM\to E$ is any $C^{\infty}(M)$-linear map, $a(0,\sigma\circ f)=f$ for any splitting $\sigma$ of $l:F\to E$. Also, for $e\in F_x$ one can always find a section $\al\in\Gamma(F)$ with $\al(x)=e$, and then consider the form $\omega\in\Omega^1(M,E)$ given by $\omega=\sigma\circ D(\al)$. So, $(\al,\omega)\in\Gamma(J^1_DF)$ and $pr((\al,\omega)(x))=\al(x).$
\end{remark}

\begin{remark}\rm Using the restriction to $J^1_DF$ of the classical Spencer operator $D^{\clas}$ associated to $F$ (see \eqref{Spencer operator}), we call $(J^1_DF,D^{\clas})$ {\bf the partial prolongation of $(F,D)$}.

Note that the relative connections
\begin{eqnarray*}
J^1_DF\xrightarrow{D^{\clas}}F\overset{D}{\To}E
\end{eqnarray*}
are not necessarily compatible. Nevertheless, the compatibility condition \eqref{compatibility1} holds as it is equivalent to the fact that $D(\al)=l\circ\omega$ for sections $(\al,\omega)\in\Gamma(J^1F)$.
\end{remark}

\begin{remark}\rm\label{inj} The surjective vector bundle map 
\begin{eqnarray*}
pr:J^1_DF\To F
\end{eqnarray*}
is a bijection if and only if $l$ is bijective. This is clear from the description of $\Gamma(J^1_DF)$ given in remark \ref{cocyclea}: if $l$ is a bijection then for any $\al\in \Gamma(F)$ there is only one choice of $\omega\in \Omega^1(M,F)$, namely $\omega=l^{-1}(D(\al))$, with the property that $(\al,\omega)\in \Gamma(J^1_DF)$.\end{remark}

\subsection{The classical prolongation of a relative connection}

The purpose of this subsection is to describe and give a universal characterization of the classical prolongation of a relative connection. The classical prolongation, as its name suggests, is a prolongation of $(F,D)$ (under some smoothness assumptions) and may be thought of as the complete infinitesimal data of solutions of $(F,D)$. It is also characterized by the universal property stated in proposition \ref{equiv} below. See \cite{Gold3} for prolongations on linear differential equations.

\begin{definition}\label{def:curvature}
The {\bf curvature of $D$} is the vector bundle map over $M$
\begin{eqnarray*}
\varkappa_D:J^1_DF\To\hom(\wedge^2TM,E)
\end{eqnarray*} 
given on sections $(\al,\omega)\in \Gamma(J^1_DF)$ by 
\begin{eqnarray*}\label{curvature}
D_X(\omega(Y))-D_Y(\omega(X))-l\circ\omega[X,Y].
\end{eqnarray*}
where $X,Y\in\X(M)$.
The {\bf classical prolongation space with respect to $D$}, denoted by \begin{eqnarray*}P_D(F)\subset J^1_DF\subset J^1F\end{eqnarray*} is set to be the kernel of $\varkappa_D$. We say that it is \textbf{smooth} if $\ker\varkappa_D$ is a smooth sub-bundle of $J^1_DF.$
\end{definition}

In the above definition, the Leibniz identity relative to $l$ and the fact that $D(\al)=l\circ \omega$, implies that $\varkappa_D$ is indeed a vector bundle map. The proof is an easy computation and is left to the reader.

\begin{remark}\label{maximality}\rm
Note that, by construction, $P_D(F)$ is the maximal subspace of $J^1F$ where the relative connections
\begin{eqnarray*}
J^1F\xrightarrow{D^\clas}F\xrightarrow{D} E
\end{eqnarray*}
are compatible. Indeed, if $(\al,\omega)\in\Gamma(J^1F)$, the compatibility condition \eqref{compatibility1} holds if and only if $(\al,\omega)\in\Gamma(J^1_DA)$. As for the compatibility condition \eqref{compatibility2}, we find that 
\begin{equation*}
D_XD^{\clas}_Y(\al,\omega)-D_YD_X^{\clas}(\al,\omega)-l\circ D^{\clas}_{[X,Y]}(\al,\omega)=\varkappa(\al,\omega)(X,Y)
\end{equation*}
where $X,Y\in\X(M)$.
\end{remark}

\begin{definition}\label{smoothly-defined}Let $(F,D)$ be a relative $l$-connection. We say that $P_D(F)\subset J^1F$ is {\bf smoothly defined} if 
\begin{enumerate}[label=(\roman*), ref=\roman*]
\item $P_D(F)$ is smooth, and 
\item $pr:P_D(F)\to F$ is surjective.
\end{enumerate}
In this case,
\begin{eqnarray*}
(P_D(F),D^{(1)}):P_D(F)\xrightarrow{(D^{(1)}.pr)}F
\end{eqnarray*}
(a relative connection!) is called the {\bf first classical prolongation of $(F,D)$}, where $D^{(1)}$ is the restriction of the Spencer operator associated to $F$.
\end{definition}

\begin{remark}\label{invol}\rm In the case that $l:E \to F$ is a vector bundle isomorphism, then $D$ can be interpreted as a classical connection and $P_D(F)$ measures its failure in being flat in the following sense. By remark \ref{inj} 
\begin{eqnarray}\label{pr}
pr:P_D(F)\To F
\end{eqnarray} 
is an injection. Under the identification of $F$ with $E$ via $l$, $D$ is a flat linear connection if and only if \eqref{pr} is bijective. This follows from the general setting of Pfaffian bundles (see corollary \ref{corollaryB}).\end{remark}

The universal property of the classical prolongation of $(F,D)$ is the following:

\begin{proposition}\label{equiv} 
$(D^{(1)},pr):P_D(F)\to F$ is a prolongation of $(F,D)$ which is universal in the sense that for any other prolongation
\begin{eqnarray*}
K\xrightarrow{(\tilde D,\tilde l)} F\xrightarrow{(D,l)} E
\end{eqnarray*}
there exists a unique vector bundle map $j:K\to P_D(F)$ such that 
\begin{eqnarray}\label{co1}
\tilde D=D^{(1)}\circ j.
\end{eqnarray} 
Moreover, $j$ is injective.
\end{proposition} 

\begin{remark}\rm
The universal property given by proposition \ref{equiv} of the classical prolongation of $(F,D)$ can be rephrased as follows: for any prolongation $(\tilde D,\tilde l):K\to F$ of $(D,l)$, there exists a unique vector bundle map $j:K\to P_D(F)$ such that the diagram 
\begin{eqnarray*}
\xymatrix{
K \ar[r]^{\tilde D,\tilde l} \ar@{.>}[d]_{j} & F \ar[r]^{D,l} & E \\
P_D(F) \ar[ur]_{D^{(1)},pr}
}
\end{eqnarray*}
is commutative.
\end{remark}

\begin{proof}[Proof of proposition \ref{equiv}]
That $(P_D(F),D^{(1)})$ is a prolongation of $(F,D)$, whenever $pr:P_D(F)\to R$ is a surjective vector bundle map of smooth vector bundles, was proved in remark \ref{maximality}.

Let's prove now that $j:=j_{\tilde D}$ satisfies \eqref{co1}, whenever $(K,\tilde D)$ is a prolongation of $(F,D)$. Let $\al$ be a section of $K$. Then 
\begin{eqnarray*}
j_{\tilde D}(\al)=(\tilde D(\al),\tilde l(\al))\in \Omega^1(M,F)\oplus\Gamma(F)
\end{eqnarray*} 
and therefore $j_{\tilde D}(\al)\in P_D(F)$ if and only if 
\begin{equation}\label{1909}\begin{aligned}
&D(\tilde l(\al))=l(\tilde D(\al))\\
\varkappa_D(\tilde D(\al),\tilde l(\al))(X,Y)=D_X(&\tilde D_Y(\al))-D_Y(\tilde D_X(\al))-l\circ \tilde D_{[X,Y]}(\al)=0
\end{aligned}\end{equation}
for any $X,Y\in \X(M).$
We see then that equations \eqref{1909} are conditions \eqref{compatibility1} and \eqref{compatibility2} for compatible connections. For the uniqueness let $j:K\to P_D(F)$ be such that equation \eqref{co1} holds. Then for $j_1:\Gamma(K)\to\Gamma(F)$ and $j_2:\Gamma(K)\to \Omega^1(M,F)$ such that 
\begin{eqnarray*}
j(\al)=(j_1(\al),j_2(\al))
\end{eqnarray*} 
for $\al\in\Gamma(K)$, one has that equation \eqref{co1} implies that $j_2(\al)=\tilde D(\al)$. Now, as $j$ is a vector bundle map then for any $f\in C^{\infty}(M)$
\begin{eqnarray*}
(j_1(f\al),\tilde D(f\al))=(fj_1(\al),df\otimes j_1(\al)+f\tilde D(\al))
\end{eqnarray*}
Since $\tilde D$ satisfies the Leibniz identity relative to $\tilde l$, then $j_1(\al)=\tilde l(\al)$.
\end{proof}

The following is another remarkable property of the classical prolongation of $(F,D)$

\begin{proposition}\label{2218}
\item The map 
\begin{eqnarray*}
pr:\text{\rm Sol}(P_D(F),D^\text{\rm clas})\overset{}{\To}\text{\rm Sol}(F,D)
\end{eqnarray*}
is a bijection between solution with inverse given by the map
\begin{eqnarray*}
j^1:\text{\rm Sol}(F,D)\overset{}{\To}\text{\rm Sol}(P_D(F),D^\text{\rm clas})
\end{eqnarray*}
sending $s\in \text{\rm Sol}(F,D)$ to $j^1s$.
\end{proposition}

\begin{proof}
Take $s\in\Gamma(F)$ a solution of $(F,D)$, i.e., $D(s)=0$. Using the Spencer decomposition \eqref{Spencer decomposition}, we write the holonomic section $j^1s\in\Gamma(J^1F)$ as $(s,0)\in\Gamma(F)\oplus\Omega^1(M,F)$. Obviously, $D(s)=l\circ 0$ and $\varkappa_D(s,0)=0$. This means that $j^1s$ is a section of $P_D(F)$. Moreover,
\begin{eqnarray*}
D^\text{clas}(s,0)=0
\end{eqnarray*}  
which says that $j^1s=(s,0)$ is a solution of $(P_D(F),D^\text{clas})$. It is now straightforward to check that $j^1\circ pr=id$ and $pr\circ j^1=id$. 
\end{proof}

\begin{remark}\rm\label{smooth-prol}Whenever the bundle map $pr:P_D(F)\to F$ (of possible non-smooth vector bundles) is surjective, the short exact sequence of vector bundles over $M$
\begin{eqnarray*}
0\To T^*\otimes F\To J^1F\overset{pr}{\To}F\To 0
\end{eqnarray*}
restricts to a short exact sequence 
\begin{eqnarray}\label{cassi}
0\To \g^{(1)}(F,D)\To P_D(F)\overset{pr}{\To}F\To 0
\end{eqnarray}
where $\g^{(1)}(F,D)$ is the kernel of $pr:P_D(F)\to F$. Hence, if $P_D(F)$ is smoothly defined, \begin{eqnarray*}
\g(P_D(F),D^{(1)})=\g^{(1)}(F,D)
\end{eqnarray*} 
Indeed, to show \eqref{cassi} we use the Spencer decomposition \eqref{Spencer decomposition}. With this, a section of $(0,\omega)\in\Gamma(F)\oplus\Omega^1(M,F)$ belongs to $P_D(F)$ at $x$ if 
\begin{eqnarray*}
l\circ \omega_x=0 \ \quad \text {and } \ \quad 0=\varkappa_D(0,\omega_x)(X,Y)=D_X(\omega_x(Y))-D_Y(\omega_x(X)),
\end{eqnarray*}
for any $X,Y\in T_xM$.
This is: $\omega_x\in\g^{(1)}(F,D)_x$.
\end{remark}

\begin{lemma}\label{workable}
The classical prolongation space $P_D(F)$ is smoothly defined if and only if $\g^{(1)}(F,D)$ is a vector bundle over $M$ of constant rank, and $pr:P_D(F)\to F$ is surjective.
\end{lemma}  

\begin{proof}
From the short exact sequence \eqref{cassi}, we have that $P_D(F)$ is smooth if and only if $g^{(1)}(F,D)$ has constant rank.
\end{proof}

For the smooth relative connection $(P_D(F),D^{(1)})$, its classical prolongation space, denoted by $$P_D^2(F)$$ is called the classical $2$-prolongation space of $(F,D)$. On $P^2_D(F)$ we have the operator 
\begin{eqnarray*}
D^{(2)}:=(D^{(1)}):\Gamma(P^2_D(F))\To\Omega^1(M,P^1_D(F)).
\end{eqnarray*} 
As we shall see in corollary \ref{k-prolongations}, $P_D^2(F)\subset J^2F$ and $D^{(2)}$
 is just the restriction of the classical Spencer operator $D^{2\text{-clas}}$ from example \ref{Spencer tower}. We define the classical $k$-prolongation of $(F,D),$
 \begin{eqnarray*}
 (P_D^{k}(F),D^{(k)})
 \end{eqnarray*}
 inductively.
 
 \begin{definition}\label{k-prolongation-space}
 Let $(F,D)$ be a relative $l$-connection. We say that the {\bf classical $k$-prolongation space} $P_D^{k}(F)$ is {\bf smooth} if 
 \begin{enumerate}
 \item $(P_D(F),D^{(1)}),\ldots (P_D^{k-1}(F),D^{(k-1)})$ are smoothly defined as relative connections, and
 \item the classical prolongation space of $(P^{k-1}_D(F),D^{(k-1)})$:
 \begin{eqnarray*}
 P_D^k(F):=P_{D^{k-1}}(P_D^{k-1}(F))
 \end{eqnarray*}
 is smooth.
 \end{enumerate}
 In this case, we define the {\bf $k$-prolongation of $D$}:
 \begin{eqnarray*}
 D^{(k)}:=(D^{(k-1)})^{(1)}:\Gamma(P^{k}_D(F))\To \Omega^1(M,P_D^{k-1}(F)).
 \end{eqnarray*}
 We say that the classical $k$-prolongation of $(F,D)$ is {\bf smoothly defined} if, moreover,
 \begin{eqnarray*}
 pr:P^{k}_D(F)\To P^{k-1}_D(F)
 \end{eqnarray*}
 is surjective. In this case, 
 \begin{eqnarray*}
 (P_D^k(F),D^{(k)}): P_D^k(F)\xrightarrow{(D^{(k)},pr)}{} P_D^{k-1}(F)
 \end{eqnarray*}
 (a relative connection!) is called the {\bf classical $k$-prolongation of $(F,D)$}.
  \end{definition}

 \begin{definition}\label{def: formally inegrable}
A relative $l$-connection $(F,D)$ is called {\bf formally integrable} if all the classical prolongations spaces
\begin{eqnarray*}
P_D(F),P^2_D(F),\ldots, P^k_D(F),\ldots
\end{eqnarray*}
are smoothly defined.
\end{definition}

For a formally integrable relative connection $(F,D)$, we have a sequence of surjective vector bundle maps
\begin{eqnarray}
\cdots\To P^k_D(F)\xrightarrow{D^{(k)},pr}{}\cdots\To P^2_D(F)\xrightarrow{D^{(2)},pr}{} P^1_D(F)\xrightarrow{D^{(1)},pr} F\xrightarrow{D,l}{} E
\end{eqnarray}
in which any pair of consecutive relative connections are compatible.
This is the basic example of standard tower (to be defined in subsection \ref{subsection:prolongation and Spencer towers}).

\begin{proposition}\label{newremark} Let $(F,D)$ be a relative connection and let $k$ be an integer. If 
\begin{enumerate}
\item the classical $k$-prolongation space $P^{k}_D(F)$ is smoothly defined, and
\item the (possibly non-smooth) vector bundle map $pr:P_D^{k+1}(F)\to P_D^kF$ is surjective.
\end{enumerate}
Then the short exact sequence of vector bundles over $M$
\begin{eqnarray*}
0\To S^{k+1}T^*\otimes F\To J^{k+1}F\overset{pr}{\To} J^kF\To 0
\end{eqnarray*}
restricts to the short exact sequence of (possibly non-smooth) vector bundle over $M$
\begin{eqnarray}\label{sec2}
0\To\g^{(k+1)}(F,D)\To P_D^{k+1}(F)\overset{pr}{\To} P^k_DF\To 0.
\end{eqnarray}
\end{proposition}

The previous proposition will be proved at the end of this subsection. A direct corollary from the definitions is the following:

\begin{corollary} With the assumption of the previous proposition.
If the classical $k$-prolongation space $P^{k}_D(F)$ is smoothly defined then 
the symbol space of the classical $k$-prolongation of $(F,D)$ is equal to the $k$-prolongation of the symbol map $\partial_D$
\begin{eqnarray*}
\g(P_D^{k}(F),D^{(k)})=\g^{(k)}(F,D).
\end{eqnarray*}
\end{corollary}

\begin{corollary}\label{corollary2}
Let $(F,D)$ be a relative connection. If $P_D^{k_0}(F)$ is smoothly defined, then for any integer $0\leq k\leq k_0$, there is a one to one correspondence between $\text{\rm Sol}(P^k_D(F),D^{(k)})$ and $\text{\rm Sol}(F,D)$ given by the surjective vector bundle map
\begin{eqnarray*}
pr^k_0:P^k_D(F)\To F
\end{eqnarray*} 
where $pr^k_0$ is the composition $pr\circ pr^2\circ\cdots\circ pr^k$.
\end{corollary}

\begin{proof}
Use proposition \ref{2218} each time you prolong.  
\end{proof}

As an example, let's look at the classical Spencer operator on the $k$-jet bundle $pr:J^kF\to J^{k-1}F$. See example \ref{Spencer tower}
 
\begin{proposition}\label{jet-prolongation}
Let $k\geq 1$ be an integer. The classical prolongation of $(J^kF,D^{k\text{\rm-clas}})$ is 
\begin{eqnarray*}
(J^{k+1}F,D^{(k+1)\text{\rm-clas}}).
\end{eqnarray*}
\end{proposition}

{\it Warning!} The proofs of the previous proposition and the following lemma make use of results that will be given later on in this thesis... So the reader may want to skip the proofs and come back to them later.

\begin{proof}
We use the description of $J^1(J^kF)$ as the set of pairs $(p,\zeta)$ where $p\in J^k_xF$ and 
\begin{eqnarray*}
\zeta:T_xM\To T_pJ^kF
\end{eqnarray*}
is a splitting of $d_p\pi:T_pJ^kF\to T_xM$. See remark \ref{1-when working with jets} in the case $R=J^kF$.
With this description, we see that the section $(\al,\omega)\in\Gamma(J^kF)\oplus\Omega^1(M,J^kF)$ of $J^1(J^kF)$ at the point $x\in M$ corresponds to the pair $(p,\zeta)$ with $p=\al(x)$ and $\zeta$ equal to
\begin{eqnarray*}
d_x\al-\omega:T_xM\To T_pJ^kF,
\end{eqnarray*}
where we regard $J^k_xF\subset T_pJ^kF$ as the vertical part $T^{\pi}_pJ^kF$.
On the other hand, it is well-known that the subset $J^{k+1}F\subset J^1(J^kF)$ is given by pairs $(p,\zeta)$ as before with the properties that
\begin{eqnarray}\label{condition1}
\theta\circ\zeta=0 \ \text{ and } \  d_\nabla\theta(\zeta(\cdot),\zeta(\cdot))=0,
\end{eqnarray}
where $\theta\in \Omega^1(J^kF,J^{k-1}F)$ is the linear Cartan form defined in example \ref{cartandist} (in the case where $R=F$ is a vector bundle and therefore $T^\pi J^{k-1}F$ is canonically identified with $pr^* J^{k-1}F$), 
and $\nabla:\X(J^kF)\times\Gamma(J^{k-1}F)\to\Gamma(J^{k-1}F)$ is any connection on the pullback bundle $\pi^*:J^{k-1}F\to J^kF$ (see e.g \cite{Bocharov, Krasil, Olver}). By example \ref{compa} we have that the relative connection associated to $\theta$ given by corollary \ref{cor: linear one forms} is precisely\begin{eqnarray*}
D^{k\text{-clas}}:\Gamma(J^kF)\To \Omega^1(M,J^{k-1}F)
\end{eqnarray*}

With the above description of $J^1(J^kF)$ in mind, we apply the following lemma in the case where $F=J^kF$, $E=J^{k-1}F$, $D=D^{k\text{-clas}}$ and $l=pr$ to obtain that 
\begin{eqnarray*}
\theta\circ\zeta=D^{k\text{-clas}}(\al)(x)-pr^k(\omega_x).
\end{eqnarray*}
And, whenever $\theta\circ\zeta=0$, we also have that 
\begin{eqnarray*}
d_{\nabla}\theta(\zeta(X),\zeta(Y))=\varkappa_{D^{k\text{-clas}}}(\al,\omega)(X,Y)
\end{eqnarray*}
for $X,Y\in T_xM$. These last two equalities complete the proof.
\end{proof}

For the following lemma, see definition \ref{lin-form} and corollary \ref{cor: linear one forms}

\begin{lemma}\label{lema2}
Let $p:F\to M$ and $E\to M$ be vector bundles. Let $\theta\in\Omega^1(F,p^*E)$ be a fiber-wise surjective linear form with corresponding $l$-connection $D:\Gamma(F)\to \Omega^1(M,E)$. Then 
for any $j^1_xs\in J^1F$  
\begin{eqnarray}\label{AA}
\theta\circ d_xs=D(\al)(x)-l\circ\omega_x
\end{eqnarray} 
where, using the Spencer decomposition \eqref{Spencer decomposition}, $(\al,\omega)\in \Gamma(F)\oplus\Omega^1(M,F)$ are such that $(\al,\omega)_x=j^1_xs$.
Moreover, if either term in the above equality vanishes, then
\begin{eqnarray}\label{BB}
d_\nabla\theta(d_xs(\cdot),d_xs(\cdot))=\varkappa_D(\al,\omega)(x)
\end{eqnarray}
where $\nabla:\X(F)\times\Gamma(p^*E)\to \Gamma(p^*E)$ is any linear connection.
\end{lemma}

\begin{proof} 
By the description of $D$ in terms of $\theta$ given by corollary \ref{cor: linear one forms} we have that $\theta\circ d_xs=D(s)(x)$. Therefore, by remark \ref{cocyclea} the map
\begin{eqnarray*}
J^1F\To \hom(TM,E),\quad j^1_xs\mapsto\theta\circ d_xs
\end{eqnarray*} 
at the level of sections and using the Spencer decomposition is given by $D(\al)-l(\omega)$, for $(\al,\omega)\in\Gamma(A)\oplus\Omega^1(M,A)$. From this equation \ref{AA} follows.

For the second part we regard the vector bundle $F\to M$ as a Lie groupoid with multiplication given by fiber-wise addition, and we regard $E\to M$ as the trivial representation of $F$. In this case, linearity of $\theta\in\Omega^1(F,E)$ translates into multiplicativity of this form. Notice that, 
\begin{eqnarray*}
d_\nabla\theta(d_xs(X),d_xs(Y))=c_1(\theta)(j^1_xs)(X,Y)
\end{eqnarray*}
where $c_1$ is the curvature map defined on remark \ref{el-caso-forma}. By the same remark we know that $c_1$ is a cocycle with linearization equal to $\varkappa_D$, but in the linear case a cocycle is just a vector bundle map and therefore its linearization is itself. This proves equation \eqref{BB}.
\end{proof}

The following corollary allows us to define the classical prolongations without smoothness assumptions.

\begin{corollary}\label{k-prolongations} Let $(F,D)$ be a relative connection and let $k_0$ be an integer. If the classical $k_0$-prolongation space $P^{k_0}_D(F)$ is smoothly defined then, 
for any integer $0\leq k\leq k_0+1$,
\begin{eqnarray}\label{general prolongation}
P^k_D(F)=J^{k-1}(P_D(F))\cap J^kF
\end{eqnarray}
and the operator $D^{(k)}$ is the restriction of the Spencer operator $D^{k\text{-\clas}}:\Gamma(J^kF)\to \Omega^1(M,J^{k-1}F)$ to $P_D^k(F)$.
\end{corollary}

\begin{proof}By construction, the relative connection $D^{(1)}:\Gamma(P_D(F))\to\Omega^1(M,F)$ is the restriction of the Spencer operator $D^\clas$ associated to $F$, therefore its prolongation $(P^2_D(F),D^{(2)})$ is given by 
\begin{eqnarray*}
P^2_D(F)=J^1(P_D(F))\cap P_{D^{\clas}}(J^1F)=J^1(P_D(F))\cap J^2F,
\end{eqnarray*}  
with connection $D^{(2)}$ given by the restriction to $P^2_D(F)$ of the classical Spencer operator $D^\clas:\Gamma(J^1(P_D(F)))\to\Omega^1(M,P_D(F))$. Moreover, as $P_D(F)\subset J^1F$, the connection on $P^2_D(F)\subset J^2F$ is actually the restriction of the classical Spencer operator $D^\clas:\Gamma(J^1(J^1F))\to\Omega^1(M,J^1F)$. But this last operator restrict to $D^{2\text{-clas}}:\Gamma(J^2F)\to \Omega^1(M,J^1F)$ on $J^2F$, which proves the case $k=2$. For $k>2$ the results follows by an inductive argument. 
\end{proof}

\begin{remark}\label{extender}\rm
Note that to define $P^k_D(F)$ in the sense of definition \ref{general prolongation} 
 we do not use the smoothness of the space $P^{k-1}_D(F).$
\end{remark}

\begin{proof}[Proof of proposition \ref{newremark}] The case $k=0$ was proved in remark \ref{smooth-prol}. Let's check the case $k=1$ (the arbitrary case follows by an inductive argument). By corollary \ref{k-prolongations},
\begin{eqnarray*}
P^2_D(F)=J^1(P_D(F))\cap J^2F,
\end{eqnarray*} 
Consider now the commutative diagram
\begin{eqnarray*}
\xymatrix{
0  \ar[r]  & T^*\otimes P_D(F) \ar@{^{(}->}[d]_i  \ar[r] & J^1(P_D(F))  \ar@{^{(}->}[d] \ar[r] & P_D(F) \ar[r] \ar@{^{(}->}[d] & 0\\
0 \ar[r] & T^*\otimes J^1F \ar[r] & J^1(J^1F) \ar[r] & J^1F \ar[r] & 0 \\
0  \ar[r]  & S^2T^*\otimes F \ar@{_{(}->}[u]^j  \ar[r] & J^2F  \ar@{_{(}->}[u] \ar[r] & J^1F \ar[r] \ar@{_{(}->}[u] & 0
}
\end{eqnarray*}
Hence, an element $\phi\in\ker(pr:P^2_D(F)\to P_D(F))$ should belong to the intersection 
\begin{eqnarray*}
j(S^2T^*\otimes F)\cap i(T^*\otimes P_D(F)),
\end{eqnarray*}
where $i$ is given by tensoring with the identity on $T^*M$ with the natural inclusion $P_D(F)\hookrightarrow J^1F$, \begin{eqnarray*}j(\phi)(X)=\phi(X,\cdot):T_xM\To F \in \ker (pr:J^1_xF\to F_x)
\end{eqnarray*}
for any  $\phi:S^2T_x\to F_x$ and any $X\in T_xM.$ Therefore $j(\phi)\in (pr:P^2_D(F)\to P_D(F))$ if and only if 
\begin{eqnarray*}
j(\phi)(X)\in P_D(F)\cap\ker (pr:J^1_xF\to F_x)=\ker(pr:P_D(F)\to F)=\g^{(1)}(F,D)
\end{eqnarray*}
where in the last equality we used \eqref{cassi}. This is, $\phi\in \ker(pr:P^2_D(F)\to P_D(F))$ if and only if $\phi\in \g^{(2)}(F,D)$ by definition of the prolongations of $\g$ (see \ref{digression}).
\end{proof}

\subsection{Towers (of prolongations)}\label{subsection:prolongation and Spencer towers}

Let
\begin{eqnarray}\label{sequence}
\cdots \To F_{k}\xrightarrow{(D^{k},l^{k})}{}F_{k-1}\To\cdots\To F_1\xrightarrow{(D^1,l^1)}{} F
\end{eqnarray}
be an infinite sequence of relative connections. For ease of notation, instead of $(D^k,l^k):F_k\to F_{k-1}$ we write $(D,l):F_k\to F_{k-1}$ if there is no risk of confusion.
 
\begin{definition}
\mbox{}
\begin{itemize}
\item A {\bf standard tower} is a sequence \eqref{sequence} in which any two consecutive connections are compatible and all of them are standard. We denote a standard tower by $(F^\infty,D^\infty,l^\infty)$.
\item Let $(D,l):F\to E$ be a relative connection. A {\bf standard resolution of $(D,l)$} is a standard tower \eqref{sequence} such that $(F_1,D^1)$ is a prolongation of $(F,D).$ We use the notation 
\begin{eqnarray*}
F^\infty \xrightarrow{(D^\infty,l^\infty)}{} F\xrightarrow{(D,l)}{}E.
\end{eqnarray*}
\end{itemize}
\end{definition}  

\begin{remark}\rm
In the previous definition, one can give up the condition that \ref{sequence} is standard. We arise to the notions of {\bf tower} and {\bf resolution}.
\end{remark}

\begin{definition}
 A {\bf morphism $\Psi$ of towers}, denoted by \begin{eqnarray*}\Psi:(F^\infty,D^\infty,l^\infty)\To (\tilde F^\infty,\tilde D^\infty,\tilde l^\infty),\end{eqnarray*} is a commutative diagram of vector bundles over $M$
\begin{eqnarray*}
\xymatrix{
\cdots \ar[r] & F^{k+1} \ar[r]^{l^{k+1}} \ar[d]_{\Psi_{k+1}} & F^k \ar[r] \ar[d]^{\Psi_k} & \cdots \ar[r] & F_1 \ar[d]_{\Psi^1} \ar[r]^{l^1} & F \ar[d]^{\Psi_0} \\
\cdots \ar[r] & \tilde F^{k+1} \ar[r]^{\tilde l^{k+1}}  & \tilde F^k \ar[r]  & \cdots \ar[r] & \tilde F^1  \ar[r]^{\tilde l^1} & \tilde F
}\end{eqnarray*} 
such that $\Psi^k\circ D^{k+1}=\tilde D^{k+1}\circ\Psi^{k+1}$ for all $k\geq 0.$
We say that $\Psi$ is {\bf injective} if each $\Psi_k$ is injective.
\end{definition}

\begin{remark}\label{ext.prol}\rm Consider a tower $(F^\infty,D^\infty,l^\infty)$. By applying lemma \ref{complex} inductively, we get a complex
\begin{eqnarray*}
\begin{aligned}
\Gamma(F_k)\overset{ D^k}{\To}\Omega^1(M,F_{k-1})\overset{D^{k-1}}{\To}&\Omega^2(M,F_{k-1})\overset{}{\To}\cdots\\
&\cdots \overset{}{\To}\Omega^{k-1}(M,F_{1})\overset{ D^1}{\To}\Omega^k(M,F_0)\end{aligned}
\end{eqnarray*}  
for any $k\geq1$,
which restricts to a vector bundle complex over $M$
\begin{equation}\label{prol}\begin{aligned}
0\longrightarrow \mathfrak{g}_k\xrightarrow{\partial_{D^k}}{}T^*\otimes\mathfrak{g}_{k-1}\xrightarrow{\partial_{D^{k-1}}}{}&\wedge^2T^*\otimes\mathfrak{g}_{k-2}\xrightarrow{\partial_{D^{k-2}}}{}\cdots\\&
\cdots\xrightarrow{\partial_{D^2}}{}\wedge^{k-1}T^*\otimes \mathfrak{g}_1\xrightarrow{\partial_{D^1}}{} \wedge^{k}T^*\otimes F_0,
\end{aligned}
\end{equation}
where $\g_j$ is the symbol space of $D^j$. 

Note that the map
\begin{eqnarray*}
\partial_{D^{k+1}}:\g_{k+1}\To\hom(TM ,\g_k)
\end{eqnarray*}
takes values in the first prolongation $\g_k^{(1)}(\partial_{D^k})$ of the map $\partial_{D^k}:\g_k\to T^*\otimes \g_{k-1}$, as $\partial_{D^{k+1}}\circ\partial_{D^k}=0$ and hence, if all the connections are standard, one can identify $\g_{k+1}$ with its image in $g_k^{(1)}$. 
\end{remark}

\begin{lemma}\label{inclusion}Let $(F^\infty, D^\infty,l^\infty)$ be a standard tower, and let $\g_k$ be the symbol space of $D^k$, for $k\geq 1$. Then
\begin{eqnarray*}
\g_{k+1}\subset \g_k^{(1)}
\end{eqnarray*}
via the inclusion $\partial_{D^{k+1}}$. Moreover, equality holds if and only if 
\begin{eqnarray*}
F_{k+1}=P_{D^k}(F_k),
\end{eqnarray*}
 where we identify $F_{k+1}$ with the image of $j_{D^{k+1}}:F_{k+1}\to J^1F_k.$ \end{lemma}

\begin{proof}
That $\g_{k+1}\subset \g_k^{(1)}$ whenever the connections are standard was already observed at the end of remark \ref{ext.prol}. Let's prove then that equality holds if and only if $F_{k+1}=P_{D^k}(F_k)$. The only facts that we will use in the proof are that $D^{k+1}$ is standard and $(D^{k+1},D^k)$ are compatible.  

Take the commutative diagram of vector bundles over $M$
\begin{eqnarray*}
\xymatrix{
0 \ar[r] &T^*\otimes F_k \ar[r] & J^1F_k \ar[r] & F_k \ar[r] & 0\\
0 \ar[r] &\g_{k+1} \ar[u]^{\partial_{D^{k+1}}} \ar[r] & F_{k+1} \ar[u]^{j_{D^{k+1}}} \ar[r]^{l^{k+1}} & F_k \ar[u]^{id} \ar[r] & 0
}
\end{eqnarray*}
Since the image of $j_{D^{k+1}}:F_{k+1}\to J^1F_k$ lies inside $P_{D^k}(F_k)$, as $(D^k,D^{k+1})$ are compatible (see proposition \ref{equiv}), we have that $j_{D^{k+1}}(F_{k+1})=P_{D^k}(F_{k})$ if and only if 
\begin{eqnarray*}
\partial_{D^{k+1}}(\g_{k+1})=\ker(pr:P_{D^k}(F_k)\to F_k),
\end{eqnarray*} 
but by remark \ref{smooth-prol} we know that $\ker(pr:P_{D^k}(F_k)\to F_k)=\g_k^{(1)}$, which completes our proof.
\end{proof}

\begin{example}[The classical Spencer tower]\label{classicalSpencertower}\rm The sequence 
\[\begin{aligned}
(J^\infty F,D^{\infty\text{-\clas}}):\cdots \To J^{k+1}F\overset{D^\clas}{\To}J^kF\overset{D^\clas}{\To} \cdots\To J^1F\overset{D^\clas}{\To}F
\end{aligned}\]
is an example of a standard tower (see cf. \cite{quillen}). 
For each $k>0$, the induced complex \eqref{prol} is the classical Spencer complex
\begin{eqnarray*}
0\longrightarrow S^kT^*\otimes F\overset{\partial}{\longrightarrow}T^*\otimes S^{k-1}V^*\otimes F\overset{\partial}{\longrightarrow}\wedge^2T^*\otimes S^{k-2}T^*\otimes F\overset{\partial}{\longrightarrow}\cdots\\
\cdots\overset{\partial}{\longrightarrow}\wedge^{n}T^*\otimes S^{k-n}T^*\otimes F\longrightarrow 0.
\end{eqnarray*}
\end{example}

\begin{example}\rm Taking the Lie function of a smooth Lie pseudogroup $\Gamma\subset\Bis(\G)$, one has the standard tower of Lie algebroids as in subsection \ref{Lie pseudogroups}
$$\cdots \To A^{(k)}(\Gamma)\overset{D^{(k)}}{\To} \cdots\To A^{(1)}(\Gamma)\overset{D^{(1)}}{\To}A^{(0)}(\Gamma)
$$
where $D^{(k)}$ is the restriction of the Spencer operator $(D^\clas,pr):J^kA\to J^{k-1}A$, $A=Lie(\G).$  For each $k>0$, the induced complex \eqref{prol} is the complex of the associated bundle of tableaux of $\Gamma$ 
\begin{eqnarray*}\begin{split}
0\longrightarrow \g^{k}(\Gamma)\overset{\partial}{\longrightarrow}T^*\otimes \g^{k-1}(\Gamma)\overset{\partial}{\longrightarrow}\wedge^2T^*\otimes \g^{k-2}(\Gamma)\overset{\partial}{\longrightarrow}
\cdots\overset{\partial}{\longrightarrow}\wedge^{n}T^*\otimes\g^{1}(\Gamma)\end{split}
\end{eqnarray*}
\end{example}

\begin{example}[The classical resolution]\rm
 If $(D,l):F\to E$ is formally integrable, we obtain the classical standard resolution
\[\begin{aligned}\label{pro}
(P^\infty_D(F),D^{(\infty)}): \cdots\To P_D^{k}(F)\To\cdots\To P_D(F)\overset{D^{(1)}}{\To}F\overset{D}{\To}E.
\end{aligned}\]
For each integer $k>0$, the induced complex \eqref{prol} is the extension complex of the tableau bundle $\g^{(1)}(F,D)$ in the sense of definition \ref{exten}, given by
\begin{eqnarray*}\begin{split}
0\longrightarrow \g^{(k)}\overset{\partial}{\longrightarrow}T^*\otimes \g^{(k-1)}\overset{\partial}{\longrightarrow}&\wedge^2T^*\otimes \g^{(k-2)}\overset{\partial}{\longrightarrow}\cdots\\&
\cdots\overset{\partial}{\longrightarrow}\wedge^{n}T^*\otimes \g^{(1)}\overset{\partial_D}{\To}\wedge^{n+1}T^*\otimes E
\end{split}
\end{eqnarray*}
Note that \eqref{pro} is a subtower of $(J^\infty F,D^{\infty\text{-\clas}})$. Using corollary \ref{k-prolongations} and remark \ref{extender}, this sequence still makes sense even if $(F,D)$ is not formally integral, but this involves non-smooth subbundles.
\end{example}

The following is the universal property of (the possible non-smooth) classical standard resolution.

\begin{theorem}\label{clasical}
Let $(D,l):F\to E$ be an $l$-connection not necessarily formally integrable.
The (possibly non-smooth) subtower $(P_D^\infty(F),D^{(\infty)})$ of $(J^\infty F,D^\infty)$ is universal among the resolutions of $(D,l)$ in the sense that for any other resolution 
\begin{eqnarray*}
F^\infty\xrightarrow{D^\infty,l^\infty}{}F\xrightarrow{D,l}{}E,
\end{eqnarray*}
there exists a unique morphism $\Psi:(F^\infty,D^{\infty})\to (J^\infty F,D^{\infty\text{-\clas}})$ of towers such that $\Psi^0:F\to F$ is the identity map, and for $k\geq 1$
\begin{eqnarray*}
\Psi^k(F_k)\subset P_D^k(F).
\end{eqnarray*}
Moreover, if $(F^\infty,D^\infty)$ is a standard resolution of $(F,D)$ then $\Psi$ is injective.
\end{theorem}

\begin{proof}
We will construct the vector bundle maps $\Psi_k:F_k\to J^kF$ with the desired properties inductively.
For $k=1$, we define $\Psi_1$ to be 
\begin{eqnarray*}
\Psi_1:=j_{D^1}:F_1\To J^1F.
\end{eqnarray*}
By construction,
\begin{eqnarray*}
pr\circ\Psi_1(\al)=pr(l^1(\al),D^1(\al))=l^1(\al)
\end{eqnarray*}
for $\al\in\Gamma(F_1)$, and hence the first square of the diagram commutes. On the other hand, as $(D,D^1)$ are compatible connections then by (the proof of) proposition \ref{equiv} we get that the image of $\Psi_1=j_{D^1}$ lies inside of $P_D(F)$, $D^1=D\circ\Psi_1$, and that $\Psi_1$ is injective.

Assume now that for a $k$ fixed, we constructed $\Psi_l:F_l\to J^lF$ for $1\leq l\leq k$ with the desired properties. Let's construct $\Psi_{k+1}:F_{k+1}\to J^{k+1}F$. Let $prol(\Psi_k):J^1(F_k)\to J^1(J^kF)$ be the vector bundle map defined at the level of sections by
\begin{eqnarray*}
prol(\Psi_k)(\al,\omega)=(\Psi_k\circ \al,\Psi_k\circ\omega)
\end{eqnarray*}  
for $(\al,\omega)\in\Gamma(F_k)\oplus\Omega^1(M,F_k)\simeq\Gamma(J^1(F_k))$. We define $\Psi_{k+1}$ to be the composition
\begin{eqnarray*}
\Psi_{k+1}:F_{k+1}\xrightarrow{j_{D^{k+1}}}{}J^1(F_k)\xrightarrow{prol(\Psi_k)}{}J^1(J^kF).
\end{eqnarray*}
For the commutativity, take $\al\in \Gamma(F_{k+1})$. Then
\[\begin{aligned}
pr^{k+1}\circ\Psi_{k+1}(\al)&=pr^{k+1}\circ prol(\Psi_k)\circ j_{D^{k+1}}(\al)\\
&=pr^{k+1}\circ prol(\Psi_k)(l^{k+1}(\al),D^{k+1}(\al))\\&=pr^{k+1}(\Psi_k\circ l^{k+1}(\al),\Psi_k\circ D^{k+1}(\al))\\&=\Psi_k\circ l^{k+1}(\al).
\end{aligned}\]
Let's show that $\Psi_{k+1}(F_{k+1})\subset P^{k+1}_DF$. As $(D^k,D^{k+1})$ are compatible connections then $j_{D^{k+1}}$ lies inside $P_{D^{k}}(F_k)$. On the other hand, $\Psi^{k-1}\circ D^{k}=D^{k\text{-clas}}\circ \Psi_k$ implies that 
\begin{eqnarray*}
prol(\Psi_k)(P_{D^k}F_k)\subset P_{D^{\text{clas}}}(J^kF)=J^{k+1}F,
\end{eqnarray*}
where in the last equality we used proposition \ref{jet-prolongation}.\\
Indeed, let $(\al,\omega)\in \Gamma(F_k)\oplus \Omega^1(M,F_k)\simeq\Gamma(J^1F^k)$ be such that at the point $x\in M$
\begin{eqnarray*}
D^k(\al)(x)=l\circ\omega_x\quad\text{and}\quad \varkappa_{D^k}(\al,\omega)(x)=0;
\end{eqnarray*}
in other words, $(\al,\omega)(x)\in P_{D^k}(F_k)$. Then, 
\begin{eqnarray*}
prol(\Psi_k)(\al,\omega)(x)=(\Psi_k\circ\al,\Psi_k\circ\omega)(x).
\end{eqnarray*}
Hence,
\[\begin{aligned}
D^{k\text{-clas}}(\Psi_k\circ\al)(x)-pr^k(\Psi_k\circ\omega)(x)&=\Psi_{k-1}\circ D^k(\al)(x)-\Psi_{k-1}\circ l^k\circ\omega_x\\&=\Psi_{k-1}(D^k(\al)(x)- l^k\circ\omega_x)=0
\end{aligned}\]
and 
\[\begin{aligned}
\varkappa_{D^{k\text{-clas}}}(\Psi_k\circ\al,\Psi_k\circ\omega)&(X,Y)(x)=
D^{k\text{-clas}}_X(\Psi_k\circ\omega(Y))(x)\\&-D^{k\text{-clas}}_Y(\Psi_k\circ\omega(X))(x)-pr^k(\Psi_k\circ\omega[X,Y])(x)\\&=
\Psi_{k-1}\circ D^k_X(\omega(Y))(x)-\Psi_{k-1}\circ D^k_Y(\omega(X))(x)\\&-\Psi_{k-1}\circ l^k\omega[X,Y](x)=\Psi_{k-1}(\varkappa_D^k(\al,\omega)(x))=0.
\end{aligned}\]
By the inductive hypothesis we also know that $\Psi_k(F_k)\subset P_D^k(F)$, thus \begin{eqnarray*}prol(\Psi_k)(J^1F^k)\subset J^1(P_D^k(F)).\end{eqnarray*} 
From the last inclusions,  
\[\begin{aligned}
\Psi_{k+1}(F_{k+1})\subset prol(\Psi_k)\circ j_{D^{k+1}}(F_{k+1})&\subset  prol(\Psi_k)(P_{D^{k}}(F_k))\subset \\&\subset J^1(P_D^k(F))\cap J^{k+1}F=P^{k+1}_D(F). 
\end{aligned}\]
Finally let's show that $\Psi^k\circ D^{k+1}=D^{(k+1)\text{-clas}}\circ\Psi_{k+1}$. Let $\al\in \Gamma(F_{k+1})$, then
\[\begin{aligned}
D^{(k+1)\text{-clas}}\circ\Psi_{k+1}(\al)&=D^{(k+1)\text{-clas}}\circ prol(\Psi_k)(j_{D^{k+1}}(\al))\\&=D^{(k+1)\text{-clas}}\circ prol(\Psi_k)(l^{k+1}(\al),D^{k+1}(\al))\\&=D^{(k+1)\text{-clas}}(\Psi_k\circ l^{k+1}(\al),\Psi_k\circ D^{k+1}(\al))\\&=\Psi_k\circ D^{k+1}(\al).
\end{aligned}\]
Note also that if $D^{k+1}$ is standard and $\Psi_k$ is injective, then $\Psi_{k+1}$ is injective by construction. The uniqueness is left to the reader.
\end{proof}

\begin{remark}\label{ab-cla}\rm  Let $(F^\infty, D^\infty)$ be a standard resolution of $(F,D)$ and consider a morphism of towers $\Psi:(F^\infty,D^\infty)\to (J^\infty F,D^{\infty\text{-\clas}})$. By proposition \ref{cohomologia}, for any $k\geq1$ we have an induced morphism of complexes \begin{eqnarray*}
\xymatrix{
\g_k \ar[r]^{\partial_{D}} \ar[d]^{\Psi_k} & T^*\otimes \g_{k-1} \ar[r]^{\partial_{D}} \ar[d]^{\Psi_{k-1}} &\cdots \wedge^{k-1}T^*\otimes \g_1 \ar[r]^{\partial_{D}} \ar[d]^{\Psi_1} & \wedge^kT^*\otimes F \ar[d]^{\Psi_0}\\
S^kT^*\otimes F \ar[r]^{\partial} & T^*\otimes S^{k-1}T^*\otimes F \ar[r]^{\partial} & \cdots  \wedge^{k-1}T^*\otimes T^*\otimes F \ar[r]^{\partial} & \wedge^k T^*\otimes F 
}
\end{eqnarray*}
where $\g_k$ is the symbol space of $D^k$. As $\Psi$ is injective and all the relative connections $D^k$ are standard, we have that the sequence 
\begin{eqnarray*}
(\Psi_0(\partial _{D^1}(\g_1)),\Psi_1(\partial_{D^2}(\g_2)),\ldots)
\end{eqnarray*}
is a tower of tableaux as in definition \ref{tableaux tower}. This follows from lemma \ref{inclusion} and the fact that injectivity of $\Psi$ implies that 
\begin{eqnarray*}
\Psi_k(\partial_{D^k}(\g_{k+1}))^{(1)}=\Psi_{k}(\g_k^{(1)}(\partial_D))
\end{eqnarray*}
where $\Psi_k(\partial_{D^k}(\g_{k+1}))^{(1)}\subset T^*\otimes \Psi_k(\g_k)$ is the usual prolongation of a tableau (see subsection \ref{digression}).
\end{remark}

\begin{lemma}
Let $(F^\infty,D^\infty,l^\infty)$ be a standard tower. Then, there exists an integer $p$ such that for all $k\geq p$,
\begin{eqnarray*}
F_{k+1}=P_{D^k}(F_k),
\end{eqnarray*} 
where we regard $F_{k+1}$ as a subset of $J^1F_k$ via the inclusion $j_{D^{k+1}}:F_{k+1}\to J^1F_k$, and 
\begin{eqnarray*}
D^{k+1}=D^{\clas}\circ j_{D^{k+1}}.
\end{eqnarray*}
\end{lemma}

\begin{proof}
From the proof of theorem \ref{clasical}, we know that there exists an injective morphism $\Psi:(F^\infty,D^\infty,l^\infty)\to (J^\infty F,D^{\infty\text{-\clas}},pr^\infty)$. As in remark \ref{ab-cla}, we have a tableau tower
\begin{eqnarray*}
(\Psi_0(\partial _{D}(\g_1)),\Psi_1(\partial_{D}(\g_2)),\ldots).
\end{eqnarray*}
By proposition \ref{stability}, there exists $p$ such that for all $k\geq p$
\begin{eqnarray*}
\Psi_{k}(\partial_{D}(\g_{k+1}))=\Psi_{k-1}(	\partial_D(\g_k))^{(1)}.
\end{eqnarray*}
Using that $\Psi_{k-1}(\partial_D(\g_k))^{(1)}=\Psi_{k-1}(\g_k^{(1)}(\partial_D))$ and that $\Psi$ is injective, we get that \begin{eqnarray*}\partial_{D}(\g_{k+1})=\g_k^{(1)}(\partial_D)\end{eqnarray*} for all $k\geq p$. Now we apply lemma \ref{inclusion} to get that $F_{k+1}=P_{D^k}(F_k)$. For the equality $D^{k+1}=D^\clas\circ j_{D^k}$ see remark \ref{standard}.
\end{proof}

\subsection{Formal integrability of relative connections}\label{classical resolution}

Here we concentrate on formal integrability (defined in \ref{def: formally inegrable}). For $(F,D)$ formally integrable, the tower 
\[\begin{aligned}
\cdots\To P_D^{k}(F)&\overset{D^{(k)}}{\To}P_D^{k-1}(F)\To\cdots\To P_D(F)\overset{D^{(1)}}{\To}F\overset{D}{\To}E
\end{aligned}\]
is a standard tower by construction, with the universal property of resolutions given by theorem \ref{clasical}. We call this tower {\bf the classical resolution} of $(F,D)$. 
One of the main properties of formal integrability is that one has the following existence result in the analytic case. The proof is explained in the more general setting of Pfaffian bundles (see theorem \ref{anal}).

\begin{theorem}
Let $(F,D)$ be an analytic relative connection which is formally integrable. Then, given $p\in P^k_D(F)$ with $\pi(p)=x\in M$, there exists an analytic solution $s\in\text{\rm Sol}(F,D)$ over a neighborhood of $x$ such that $j^k_xs=p.$
\end{theorem}

The following result is a workable criteria for formal integrability; the rest of the section will be devoted to its proof. A more general theorem (for Pfaffian bundles) will be presented in chapter \ref{Pfaffian bundles}. Similar results can be found in \cite{Gold3,Gold2}.

\begin{theorem}\label{formally integrable}
Let $(D,l):F\to E$ be a relative connection and let $\g$ be its symbol space. Assume that 
\begin{enumerate}
\item $pr:P_D(F)\to F$ is surjective,
\item $\g^{(1)}$ is a smooth vector bundle over $M$, and 
\item $H^{2,k}(\g)=0$ for $k\geq0$.
\end{enumerate}
Then, $(F,D)$ is formally integrable.
\end{theorem}

We will now introduce some notions concerning the smoothness of the prolongations and the surjectivity of $pr:P^k_D(F)\to P^{k-1}_D(F)$. \\

The following result, generalizing lemma \ref{workable}, is a workable criteria for the smoothness of the classical prolongation spaces. For this purpose we consider the possible non-smooth bundles $P^k_D(F)$ given by
\begin{eqnarray*}J^{k-1}(P_D(F))\cap J^kF\end{eqnarray*}
See corollary \ref{k-prolongations} and remark \ref{extender}. Recall that, if the classical $k$-prolongation space is smoothly defined, we have a short exact sequence of (possibly non-smooth) vector bundles over $M$
 \begin{eqnarray}\label{ses}
0\To \g^{(k+1)}\To P_D^{k+1}(F)\overset{pr}{\To}P_D^{k}(F)\To 0,
\end{eqnarray}
where $\g^{(k)}:=\g^{(k)}(F,D)$ is the $k$-prolongation of the map $\partial_D$ (see proposition \ref{newremark}).

\begin{proposition}\label{l}
Let $k$ be a natural number. If  
\begin{enumerate}
\item the classical $k$-prolongation space $P_D^{k}(F)$ is smoothly defined,
\item $\g^{(k+1)}$ is a vector bundle over $M$, and 
\item $pr:P^{k+1}_D(F)\to P^k_D(F)$ is surjective for all $0\leq k\leq k_0$
\end{enumerate}
Then $P^{k+1}_D(F)\subset J^{k+1}F$ is smoothly defined. 
\end{proposition}

See \cite{Gold2,Gold3} for similar results.

\begin{proof} This is a direct consequence of proposition \ref{newremark}: by the short exact sequence \eqref{sec2}, $\g^{(k+1)}$ is smooth if and only if $P^{k+1}_D(F)$ is smooth.
\end{proof}

\subsubsection{Higher curvatures}

Now we will prove theorem \ref{formally integrable}. This will be done inductively by showing:

\begin{proposition}Let $(F,D)$ be a relative connection and let $k$ be an integer. If the classical $k$-prolongation space $P^k_D(F)$ is smoothly defined, then there exists an exact sequence of vector bundles over $M$:
\begin{eqnarray*}
P^{k+1}_D(F)\overset{pr}{\To} P^k_D(F)\xrightarrow{\kappa_{k+1}}{}H^{(2,k-1)}(\g).
\end{eqnarray*}
\end{proposition}

It is clear now that by an inductive argument, the previous proposition together with proposition \ref{l} proves theorem \ref{formally integrable}.

Throughout this subsection let's assume that for a fixed integer $k_0$, proposition \ref{l} holds. Hence,
\begin{eqnarray*}
P_D^{k_0+1}\overset{pr}{\To}P_D^{k_0}\overset{pr}{\To}\cdots \To P_D(F)\overset{pr}{\To}F
\end{eqnarray*}
is a sequence of surjective smooth vector bundle maps.

 We will construct, for each natural number $0\leq k\leq k_0+1$, the {\bf $(k+1)$-reduced curvature map} of $(F,D)$
\begin{eqnarray*}
\kappa_{k+1}:P^k_D(F)\longrightarrow C^{2,k-1}\end{eqnarray*}
where $C^{2,k-1}$ is the bundle of vector spaces over $M$
\begin{eqnarray*}
C^{2,k-1}:= \frac{\wedge^2T^*\otimes\mathfrak{g}^{(k-1)}}{\partial(T^*\otimes\mathfrak{g}^{(k)})},
\end{eqnarray*}
Here we set $P^0_D(F)=F$, $\g^{(0)}=\g$ and $\g^{(-1)}=E.$ In \cite{Gold3} one can also find the construction of the reduced curvature maps for differential equations, and the following results regarding them, but this is done from an algebraic point of view.

We begin by showing the construction of 
\begin{eqnarray*}
\kappa_1:F\To C^{2,-1}:= \frac{\wedge^2T^*\otimes E}{\partial_D(T^*\otimes\g)}
\end{eqnarray*}
using the vector bundle map
\begin{eqnarray*}
\varkappa_D:J_D^1F\to \wedge^2T^*\otimes E
\end{eqnarray*}
which has $P_D(F)$ as its kernel (see subsection \ref{basic notions}). Recall that $\varkappa_D(\al,\omega)$, where $(\al,\omega)\in \Gamma(F)\oplus \Omega^1(M,F)$ is such that $D(\al)=l\circ\omega$, is given by
\begin{eqnarray}\label{equation:si}
\varkappa_D(\al,\omega)(X,Y)=D_X\omega(Y)- D_Y\omega(X)-l\circ\omega[X,Y]
\end{eqnarray}
for $X,Y\in\X(M)$, and therefore if $\omega\in \Omega^1(M,\g)$ 
\begin{eqnarray*}
\varkappa_D(0,\omega)=\partial_D(\omega),\quad\text{\rm for }\partial_D:T^*\otimes \g\to \wedge^2 T^*\otimes E.
\end{eqnarray*}
On the other hand, the short exact sequence of vector bundles over $M$ (for a proof see remark \ref{remark1})
\begin{eqnarray*}
0\To T^*\otimes \g\To J^1_DF\overset{pr}{\To}F\To 0,
\end{eqnarray*}
shows that 
\begin{eqnarray*}
F\simeq J^1_DF/(T^*\otimes \g),
\end{eqnarray*}
and that 
\begin{eqnarray}\label{1-curvature}
\kappa_1:\frac{J^1_DF}{T^*\otimes \g}\To C^{2,-1},\quad [p]\mapsto[\varkappa_D(p)]
\end{eqnarray}
is well defined.

\begin{definition}
The {\bf $1$-reduced curvature map} 
\begin{eqnarray*}
\kappa_1:F\To C^{2,-1}
\end{eqnarray*}
is defined by equation \eqref{1-curvature}.
\end{definition}

\begin{lemma}
The sequence 
\begin{eqnarray*}
P_D(F)\overset{pr}{\To} F\overset{\kappa_1}{\To}C^{2,-1}
\end{eqnarray*}
of bundles over $M$ is exact.
\end{lemma}

\begin{proof}
That $\kappa_1\circ pr=0$ follows from the fact that $P_D(F)=\ker\varkappa_D$. Now, 
\begin{eqnarray*}[(\al,\omega)(x)]\in F\simeq J^1_D(F)/(T^*\otimes \g)
\end{eqnarray*} 
is such that $\kappa_1([(\al,\omega)]_x)=0$ if and only if we find $\phi_x:T_xM\to \g_x$ such that 
\begin{eqnarray*}
\begin{aligned}
\varkappa_D(\al,\omega)(X,Y)(x)&=\partial_D(\phi(Y))(X)(x)-\partial_D(\phi(X))(Y)(x)\\
&=D_X(\phi(Y))(x)-D_Y(\phi(X))(x) 
\end{aligned}
\end{eqnarray*}
for all $X,Y\in\X(M)$, and where $\phi\in\Gamma(T^*\otimes\g)$ is any extension of $\phi_x$.
Since $\g=\ker l$, and $\varkappa_D(\al,\omega)(X,Y)(x)$ is given by formula \eqref{equation:si}, the last equality shows that $(\al,\omega-\phi)(x)\in P_D(F)$. And of course $pr((\al,\omega-\phi)(x))=[(\al,\omega)(x)].$
\end{proof}

For $1\leq k\leq k_0+1$, the map
\begin{eqnarray*}
\kappa_{k+1}:P^k_D(F)\To C^{2,k-1}
\end{eqnarray*} 
is just the 1-reduced curvature map of $D^{(k)}:\Gamma(P_D^k(F))\to \Omega^1(M,P^{k-1}_D(F)).$
Having in mind the short exact sequence 
\begin{eqnarray*}
0\To T^*\otimes \g^{(k)}\To J^1_{D^{(k)}}(P^{k}_D(F))\overset{pr}{\To} P^k_D(F)\To 0
\end{eqnarray*} 
and that $P^{k+1}_D(F)$ is the prolongation of $(P^{k}_D(F),D^{(k)})$, we have the following definition and lemma.

\begin{definition}\label{curv(D)}For $1\leq k\leq k_0+1$, we define the {\bf $(k+1)$-reduced curvature map} by 
\begin{eqnarray*}
\kappa_{k+1}:P^k_D(F)\To C^{2,k-1}, \ \
\kappa_{k+1}([p])=[\varkappa_{D^{(k)}}(p)]
\end{eqnarray*}
for $[p]\in P^k_D(F)\simeq J^1_{D^{(k)}}(P^{k}_D(F))/T^*\otimes \g^{(k)}.$
\end{definition}

\begin{lemma}
The sequence 
\begin{eqnarray*}
P_D^{k+1}(F)\overset{pr}{\To} P^k_D(F)\overset{\kappa_{k+1}}{\To}C^{2,k-1}\end{eqnarray*}
of bundles over $M$ is exact.
\end{lemma}

\begin{lemma}\label{39}
For $1\leq k\leq k_0+1$, the image of $\varkappa_{D^{(k)}}$ lies in the family of subspaces
\begin{eqnarray*}
\ker\{\partial:\wedge^2T^*\otimes \g^{(k-1)}\To \wedge^3T^*\otimes\g^{(k-2)} \}.
\end{eqnarray*}
Hence, $\kappa_{k+1}$ takes values in
\begin{eqnarray*}
H^{2,k-1}(\g):=\frac{\ker\{\partial:\wedge^2T^*\otimes \g^{(k-1)}\to \wedge^3T^*\otimes\g^{(k-2)} \}}{\Im\{\partial:T^*\otimes\mathfrak{g}^{(k)}\to \wedge^2T^*\otimes \g^{(k-1)}\}}
\end{eqnarray*}
\end{lemma}

\begin{proof} Let $(\al,\Omega)\in \Gamma(P^k_D(F))\oplus \Omega^1(M,P^k_D(F))$ be such that 
\begin{eqnarray*}D^{k\text{-\clas}}\al=pr\omega
\end{eqnarray*} 
or, in other words, $(\al,\Omega)$ is a section of $J^1_{D^{(k)}}(P^k_D(F))$. Then,
\begin{eqnarray*}
\begin{aligned}
pr(\varkappa_{D^{(k)}}(\al,\Omega)(X,Y))
=pr(D_X^{k\text{-\clas}}\Omega(Y)-D^{(k)}_Y\Omega(X)-pr\Omega[X,Y])
\end{aligned}
\end{eqnarray*}
for all $X,Y\in\X(M)$.
Since $pr\circ D^{(k)}=D^{(k-1)}\circ pr$ because $(D^{(k-1)},D^{(k)})$ are compatible, and $D^{(k)}\al=pr\Omega$ we get that 
\begin{eqnarray*}
pr(\varkappa_{D^{(k)}}(\al,\Omega)(X,Y))=\bar{D}^{(k-1)}\circ \bar{D}^{(k)}(\al)=0.
\end{eqnarray*}
for $\bar{D}^{(k-1)}\circ \bar{D}^{(k)}:\Gamma(P^k_D(F))\to \Omega^2(M,P^{k-1}_D(F))$ as in lemma \ref{complex}. 
This shows, by the short exact sequence \eqref{ses}, that 
\begin{eqnarray*}
\varkappa_{D^{(k)}}(\al,\Omega)(X,Y)\in\g^{(k-1)}(F,D).
\end{eqnarray*}
On the other hand,
\begin{eqnarray}\label{a}
\begin{aligned}
\partial(&\varkappa_{D^{(k)}}(\al,\Omega))(X,Y,Z)=\partial(\varkappa_{D^{(k)}}(\al,\Omega)(X,Y))(Z)\\&+\partial(\varkappa_{D^{(k)}}(\al,\Omega)(Y,Z))(X)+\partial(\varkappa_{D^{(k)}}(\al,\Omega)(Z,X))(Y)\\& =D^{(k-1)}_Z(\varkappa_{D^{(k)}}(\al,\Omega)(X,Y))+D^{(k-1)}_X(\varkappa_{D^{(k)}}(\al,\Omega)(Y,Z))\\&+D^{(k-1)}_Y\varkappa_{D^{(k)}}(\al,\Omega)(Z,X))
\end{aligned}
\end{eqnarray}
for $Z\in\X(M)$. To prove that this expression vanishes let's look closer at the expression on the right hand side: as $\al$ is a section of $P^k_D(F)$, we can write it as $(\beta,\omega)\in\Gamma (P_D^{k-1}(F))\oplus \Omega^1(M,P^{k-1}_D(F))$ with
\begin{eqnarray}\label{b}
\begin{split}
&\omega=D^{(k)}(\al)=pr\Omega \\
D^{(k-1)}_X\omega(Y)&-D^{(k-1)}_Y\omega(X)-pr\omega[X,Y]=0
\end{split}
\end{eqnarray}
for all $X,Y\in \X(M)$.
Note that $\omega\in\ker\{pr:P^{k-1}_D(F)\to P^{k-2}_D(F)\}$, which is equal to $\g^{k-1}$, by the exact sequence \eqref{ses}. Since $\g^{k-1}$ is the first prolongation of $\g^{k-2}$ we have that 
\begin{eqnarray}\label{h}
\partial(\omega(X))(Y)-\partial(\omega(Y))(X)=D^{(k-1)}_Y\omega(X)-D^{(k-1)}_X\omega(Y)=0.
\end{eqnarray}

Now, $\Omega(X)$ is also a section of $P^k_D(F)$ and therefore, using \eqref{b}, we can also write it as 
\begin{eqnarray*}
(\omega(X),\zeta_X)\in \Gamma(P^{k-1}_D(F))\oplus \Omega^1(M,P^{k-1}_D(F))
\end{eqnarray*}
 where 
 \begin{eqnarray}\label{cc}\zeta_X=D^{(k)}\Omega(X),
 \end{eqnarray}
  and such that
\begin{eqnarray}\label{d}
\begin{split}
&D^{(k-1)}\omega(X)=pr\zeta_X\\
D^{(k-1)}_Y\zeta_X(Z&)-D^{(k-1)}_Z\zeta_X(Y)-pr\zeta_X[Y,Z]=0.
\end{split}
\end{eqnarray}
From \eqref{b} and \eqref{cc}, we write $D^{(k-1)}_Z(\varkappa_{D^{(k)}}(\al,\Omega)(X,Y))$ as
\begin{eqnarray*}
D^{(k-1)}_Z(D_X^{(k)}\Omega(Y)-D^{(k)}_Y\Omega(X)-pr\Omega[X,Y])=\\
D^{(k-1)}_Z(\zeta_Y(X)-\zeta_X(Y)-\omega[X,Y])=\\
D^{(k-1)}_Z\zeta_Y(X)-D^{(k-1)}_Z\zeta_X(Y)-D^{(k-1)}_Z\omega[X,Y],
\end{eqnarray*}
and the right hand side of equation \eqref{a} becomes
\begin{eqnarray}\label{f}
D^{(k-1)}_Z\zeta_Y(X)-D^{(k-1)}_Z\zeta_X(Y)-D^{(k-1)}_Z\omega[X,Y]\\
+D^{(k-1)}_X\zeta_Z(Y)-D^{(k-1)}_X\zeta_Y(Z)-D^{(k-1)}_X\omega[Y,Z]\\
+D^{(k-1)}_Y\zeta_X(Z)-D^{(k-1)}_Y\zeta_Z(X)-D^{(k-1)}_Y\omega[Z,X].
\end{eqnarray} 
Using \eqref{d}, expression \eqref{f} reduces to
\begin{eqnarray*}
D^{(k-1)}_{[Z,X]}\omega(Y)+D^{(k-1)}_{[X,Y]}\omega(Z)+D^{(k-1)}_{[Y,Z]}\omega(X)\\
-D^{(k-1)}_Z\omega[X,Y]-D^{(k-1)}_X\omega[Y,Z]-D^{(k-1)}_Y\omega[Z,X]
\end{eqnarray*}
which is zero by \eqref{h}.
\end{proof}

\section{Relative connections of finite type}\label{finite type}

\begin{definition}We say that the $l$-connections $(F,D)$ is of {\bf finite type} if there exists an integer $k$ such that $\g^{(k)}(F,D)=0$. The smallest $k$ with this property is called the {\bf order of $D$}.
\end{definition}
 
 For a vector space $V$, denote by $V_M$ the trivial vector bundle over $M$ with fiber $V$. It comes equipped with the obvious flat connection, denoted by $\nabla^{\text{\rm flat}}$. One has 
 \begin{eqnarray*}
 \text{Sol}(V_M,\nabla^{\text{\rm flat}})\simeq V,
 \end{eqnarray*} 
where $v\in V$ is identified with the constant section of $V_M$.
 
In this section we will prove the following result

\begin{theorem}\label{finitecase}Let $D$ be a relative connection of finite type over a simply connected manifold, and let $k\geq 1$ be the order of $D$. If 
\begin{enumerate}
\item $P^{k}_D(F)$ is smoothly defined, and
\item $pr:P^{k+1}_D(F)\to P^{k}_D(F)$ is surjective 
\end{enumerate}
then $\Sol(F,D)$ is a finite dimensional vector space of dimension 
   \begin{eqnarray*}
r:=\rank F+\rank\g^{(1)}+\rank\g^{(2)}+\cdots+\rank\g^{(k-1)}.
\end{eqnarray*}
More precisely, choosing $V=\Rr^r$, there exists a morphism $p:(V_M,\nabla^{\text{\rm flat}})\to (F,D)$ of relative connections, inducing a bijection on the space of solutions.  
\end{theorem}

\begin{remark}\rm In the previous theorem, a morphism of relative connections means that $p:V_M\to F$ is a vector bundle map such that 
\begin{eqnarray}\label{formula11}
D\circ p=(l\circ p)\circ\nabla^{\text{\rm flat}}
\end{eqnarray}
The message here is the following: a relative operator satisfying the hypothesis of theorem \ref{finitecase} is the quotient of a trivial bundle with the canonical flat connection $\nabla^{\text{\rm flat}}$. Intuitively the situation is as follows 
\begin{eqnarray*}
\xymatrix{
V_M \ar[r]^{\nabla^{\text{\rm flat}}} \ar[d]_p & V_M \ar[d]^{l\circ p}\\
F \ar[r]^{D} & E.
}
\end{eqnarray*}
Although the above diagram is not precise, it helps to illustrate our situation.
\end{remark}

Theorem \ref{finitecase} follows from the following lemma.

 \begin{lemma}\label{lema}With the hypotheses of theorem \ref{finitecase}, one has that 
 \begin{enumerate}
 \item $P^{k}_D(F)$ is isomorphic to the trivial bundle $V_M$,
\item $pr:P^{k}_D(F)\to P^{k-1}_D(F)$ is an isomorphism of vector bundles over $M$.
 \end{enumerate}
 Moreover, under the identification given by $pr$, $D^{(k)}$ becomes the trivial connection $\nabla^{\text{\rm flat}}$.
 \end{lemma}
 
 \begin{proof}
By proposition \ref{newremark} we have a short exact sequence 
 \begin{eqnarray}\label{si}
 0\To \g^{(l+1)}\To P^{l+1}_D(F)\overset{pr}{\To} P^{l}(F)\To 0
 \end{eqnarray}
 of vector bundles over $M$ for any $0\leq l\leq k$. As $\g^{(k)}=0$ we have that $P^{k}_D(F)$ is isomorphic to $P^{k-1}_D(F)$ via $pr$, and under this identification we get that the connection $D^{(k)}:\Gamma(P^{k}(F))\to \Omega^1(M,P^{k-1}(F))$ is a classical linear connection on $P^{k}_D(F)$. We claim that this connection is flat. Indeed since $\g^{(k)}=0$ then its prolongation $\g^{(k+1)}=0$, and therefore the sequence \eqref{si} for $l=k$ implies that $pr:P^{k+1}_D(F)\to P^{k}_D(F)$ is a bijection. From remark \ref{invol} it follows that $D^{(k)}$ is flat.
 We are now in the situation of a vector bundle $L$ over a simply connected manifold $M$, which admits a flat linear connection $\nabla:\X(M)\times\Gamma(L)\to \Gamma(L)$. It is a well-known fact that in this case, $L$ is isomorphic to the trivial bundle $V_M$, where $V$ is the finite dimensional vector space $V\subset\Gamma(L)$ given by parallel sections of $\nabla$. The isomorphism is given by
\begin{eqnarray*}
V_M\overset{\psi}{\To} L, \
(s,x)\mapsto s(x)
\end{eqnarray*}
From this isomorphism, it is clear that $\text{Sol}(V_M,\nabla^{\text{\rm flat}})\simeq V$ is mapped onto $\text{Sol}(L,\nabla)$. From the Leibniz identity for $\nabla$ and $\nabla^{\text{\rm flat}}$, and the fact that the set $V\subset \Gamma(V_M)$ generates the space $\Gamma(V_M)$ as a $C^{\infty}(M)$-module, we have that $\nabla$ becomes $\nabla^{\text{\rm flat}}$, i.e.
\begin{eqnarray*}
\psi\circ\nabla^{\text{\rm flat}}=\nabla\circ \psi
\end{eqnarray*}
In our case take $L=P^{k}_D(F)$ and $\nabla=D^{(k)}$. To show that the dimension of $V$ is equal to $r$, we count the rank of $P^{k}_D(F)$ by the recursive formula
\begin{eqnarray*}
\begin{aligned}
&\rank P_D(F)=\rank F+\rank \g^{(1)}\\
&\rank P^l_D(F)=\rank P^{l-1}_D(F)+\rank\g^{(l)},\quad\text{for }l\geq2
\end{aligned}
\end{eqnarray*}
which is true by the exact sequence \eqref{si}.
 \end{proof}
 
 \begin{proof}[Proof of theorem \ref{finitecase}] From lemma \ref{lema}, one takes $V_M=P^{k}_D(F)$ and $p=pr^{k}_0$. Since $pr^{k}_0$ gives a bijection between $\text{Sol}(P^{k}_D(F),D^{(k)})$ and $\text{Sol}(F,D)$, then formula \eqref{formula11} is well-defined. Moreover, as $p$ is surjective and in this case $\Gamma(P^{k}_D(F))$ is generated by $\text{Sol}(P^{k}_D(F),D^{(k)})$ as a $C^{\infty}(M)$-module, it follows that $D$ is determined by formula \eqref{formula11}. Moreover, 
  \begin{eqnarray*}
r=\rank P^{k}_D(F)=\dim \text{\rm Sol}(P^{k}_D(F),\nabla^{\text{flat}})=\dim \text{\rm Sol}(P^{k}_D(F),D^{(k)}),
 \end{eqnarray*}
and from corollary \ref{corollary2} we have that $pr^k_0$ gives a linear bijection between $$\text{\rm Sol}(P^{k}_D(F),D^{(k)})$$ and $\text{\rm Sol}(F,D)$.   
 \end{proof}

\begin{corollary}
Let $(F,D)$ be a relative connection of finite type over a simply connected manifold $M$, and let $k\geq 1$ be the order of $D$. If 
\begin{enumerate}
\item $pr:P_D(F)\to F$ is surjective, 
\item $\g^{(1)}$ is a smooth vector bundle over $M$, and 
\item $H^{2,l}(\g)=0$ for $0\leq l\leq k-1$,
\end{enumerate}
then $\Sol(F,D)$ is a finite dimensional vector space of dimension 
   \begin{eqnarray*}
r:=\rank F+\rank\g^{(1)}+\rank\g^{(2)}+\cdots+\rank\g^{(k-1)}.
\end{eqnarray*}
More precisely, choosing $V=\Rr^r$, there exists a morphism $p:(V_M,\nabla^{\text{\rm flat}})\to (F,D)$ of relative connections, inducing a bijection on the space of solutions.  
\end{corollary}

\begin{proof}
 Note that $\g^{(k)}=0$ implies that $\g^{(l)}=0$ for all $l\geq k$, and therefore $H^{2,l}=0$ for all $l\geq 1$. Hence, we can apply theorem \ref{formally integrable} in this case to get that $(F,D)$ is formally integrable, and therefore the hypotheses of theorem \ref{finitecase} are fulfilled.
 \end{proof}

\chapter{Pfaffian bundles}\label{Pfaffian bundles}
\pagestyle{fancy}
\fancyhead[CE]{Chapter \ref{Pfaffian bundles}} 
\fancyhead[CO]{Pfaffian bundles}

In this chapter we will define two equivalent notions of Pfaffian bundles, one using distributions $H\subset TR$ and one using surjective one forms on $R$ with values in some vector bundle. The connection between the two notions is the kernel of the one form which is a honest smooth distribution thanks to the surjectivity of the one form under consideration. These two equivalent languages allow us to develop the theory of Pfaffian bundles using either distributions or forms, depending on which is more convenient at each specific step. However, the approach using one forms allows us to have a slightly more general version of Pfaffian bundles when one does not require the surjectivity assumption. Bearing this in mind, throughout this chapter we will point out which properties and results still hold in the more general case of a one form which is not necessarily surjective.  \\
 
Through this chapter $\pi:R\to M$ denotes a surjective submersion. By $T^\pi R$ we denote the subbundle of $TR$ given by $\ker\pi$, the vectors tangent to the fibers of $R$. By $\pi:J^kR\to M,\pi(j^k_xs)=x$ we denote the smooth bundle over $M$ which consists of $k$-jets $j_x^ks$ of local sections of $\pi:R\to M$, and we set $pr:J^kR\to J^{k-1}R$ to be the surjective submersion mapping $j^k_xs$ to $j^{k-1}_xs.$ See subsection \ref{jet-bundles}.

For ease of notation $T$ and $T^*$ denote the tangent and cotangent bundles of a manifold $M$ respectively. Also we will regard the vector bundles over $M$ as vector bundles over $R$ via the pullback of $\pi$, unless otherwise stated.

\section{Pfaffian bundles and Pfaffian forms}

Let $H\subset TR$ be a distribution. By $E$ we denote the vector bundle over $R$ given by the quotient
\begin{eqnarray*}
E:=TR/H
\end{eqnarray*}
We denote by $\g:=\g(H)$ the bundle (not necessarily of constant rank) over $R$ given by \begin{eqnarray*}\g=H^\pi:=H\cap T^{\pi}R\end{eqnarray*} 
We call $\g$ {\bf the symbol space of $H$.}

\begin{definition}\label{def: pfaffian distributions}\mbox{}
\begin{itemize}
\item We say that $H$ is {\bf $\pi$-transversal} if 
\begin{eqnarray*}
TR=H+T^\pi R.
\end{eqnarray*}
\item We call $H$ {\bf $\pi$-involutive} if $H^\pi$ is closed under the Lie bracket of vector fields on $\X(R)$. 
\item A {\bf Pfaffian distribution} $H$ is a distribution which is both $\pi$-transversal and $\pi$-involutive.
\item A {\bf Pfaffian bundle} is a surjective submersion $\pi:R\to M$ together with a Pfaffian distribution $H\subset TR$. In this case we use the notation $\pi:(R,H)\to M$. 
\item A {\bf solution of $(R,H)$} is a section $s:M\to R$ with the property that  
\begin{eqnarray*}d_xs(T_xM)\subset H_{s(x)}
\end{eqnarray*}
for all $x\in M.$ We denote the set of solutions by 
\begin{eqnarray*}
\Sol(R,H)\subset \Gamma(R).
\end{eqnarray*}
\item A {\bf partial integral element of $H$}  is any linear subspace $V\subset T_pR$ of dimension equal to the dimension of $M$, such that  
\begin{eqnarray*}
V\subset H_p \qquad\text{and}\qquad T_pR=V\oplus T^\pi_p R.
\end{eqnarray*}
The set of partial integral elements, denoted by $J^1_HR$, is called {\bf the partial prolongation with respect to $H$}.
\end{itemize}
\end{definition}

\begin{remark}\rm Some comments about notions and names that appear in the literature. In \cite{BC}, a Pfaffian system was an exterior differential system $\mathcal{I}\subset \Omega^*(R)$ with independence condition generated as an exterior differential ideal in degree one. This data was encoded in sub-bundles
\begin{eqnarray*}
I\subset J\subset T^*R
\end{eqnarray*}  
which, in analogy with our approach, is given by 
\begin{eqnarray*}
J=(TR/H^\pi)^*, \ \text{ and } \ \ I=(TR/H)^*,
\end{eqnarray*}
where $H\subset TR$ is a $\pi$-transversal distribution. What they called a linear Pfaffian system in \cite{BC} is in our language equivalent to $H$ being $\pi$-involutive. We warn that linearity in the sense of \cite{BC} is a different notion of what we will call linear Pfaffian bundles.
\end{remark}

\begin{remark}\label{transversal-properties}\rm \mbox{}
\begin{itemize}
\item When $H$ is $\pi$-transversal, then $\g$ is a smooth vector subbundle of $TR$ as it has constant rank and 
\begin{eqnarray*}
H/\g=\pi^*TM.
\end{eqnarray*}
 Indeed, for any $p\in R$,
\begin{eqnarray*}
\dim\g_p=\rank TR-\rank T^\pi R-\rank H.
\end{eqnarray*}
On the other hand, 
\begin{eqnarray*}
\pi^*TM\simeq TR/T^\pi R\simeq (T^\pi R+H)/T^\pi R\simeq H/\g.
\end{eqnarray*}
\item If $H$ is $\pi$-transversal then the restriction of $d_p\pi$ to $H_p$ 
\begin{eqnarray*}
d_p\pi:H_p\To T_{\pi(p)}M
\end{eqnarray*}
is surjective for any $p\in R.$ For a partial integral element $V\subset H_p$ of $\pi:(R,H)\to M$ we have that the restriction
\begin{eqnarray}\label{iso}
d_p\pi:V\overset{\simeq}{\To} T_{\pi(p)}M
\end{eqnarray}
is an isomorphism as $V$ is transversal to $T^\pi_pR$ and $\dim V=\dim M.$
\item When $H$ is $\pi$-transversal, the set of partial integral elements $J^1_HR$ can be regarded as a sub-bundle of $J^1R$ in the following way. The fact that \eqref{iso} is an isomorphism for a partial integral element $V\subset H_p$ means that $V$ is the image of a splitting $\sigma_p:T_{\pi(p)}M\to T_pR$ of $d_p\pi:T_pR\to T_{\pi(p)}M$. But the set of splittings $\sigma_p$, where $p\in R$, is identified with $J^1 R$.
\end{itemize}
\end{remark}

The involutivity of $H$ is measured by the so called curvature map:

\begin{definition}\label{curvature-map} The {\bf curvature map} of $H$ is the $C^{\infty}(R)$-bilinear map, defined by
\begin{eqnarray*}
c:=c(H):H\times H\To E, \quad c(u,v)=[U,V]_p\;\mod H,
\end{eqnarray*}
where $U,V\in\X(R)$ are vector fields tangent to $H$ such that $U_p=u$ and $V_p=v$.
\end{definition}

Note that the Leibniz identity implies that $c$ is a well-defined bilinear map.

\begin{remark}\label{involutive-properties}\rm\mbox{}
\begin{itemize}
\item $H$ is involutive if and only if $c=0$.
\item $H$ is $\pi$-involutive if and only if $c_{\g\times \g}=0$. If $H$ is also $\pi$-transversal we have an induced vector bundle map over $R$
\begin{eqnarray*}
\g\times H/\g\To E,\quad (v,[X])\mapsto c(v,X)
\end{eqnarray*}
or analogously the dual vector bundle map over $R$
\begin{eqnarray}\label{sm}
\partial_H:\g\To\hom(TM,E),
\end{eqnarray}
where we are using the vector bundle isomorphism $\pi^*TM\simeq H/\g$. We call the map \eqref{sm} {\bf the symbol map of $H$}, and we extend it to the linear map $\partial_H:\wedge^kT^*\otimes \g\to \wedge^{k+1}T^*\otimes E$ by the usual formula. We set 
\begin{eqnarray*}
\g^{(k)}:=\g^{(k)}(R,H)\subset S^k T^*\otimes \g
\end{eqnarray*} the bundle over $R$ (not necessarily of constant rank) equal to the $k$-prolongation of the map $\partial_H$. See definition \ref{1st-prolongation}.
\end{itemize}
\end{remark}

Associated to any distribution $H\subset TR$ we have a canonical surjective form $\theta_H\in\Omega(R,TR/H)$ which is nothing more than the projection map
\begin{eqnarray}\label{quotient-form}
\theta_H:TR\To TR/H,\quad U\mapsto[U].
\end{eqnarray} 
Moreover, any surjective form $\theta\in\Omega^1(R,E')$ arises in this way (up to isomorphism of $E'$):
 using the identification
 \begin{eqnarray*}
 TR/ H\overset{\simeq}{\To}E',\quad[U]\mapsto \theta(U)
 \end{eqnarray*}
 of vector bundles over $R$ with $H=\ker\theta$, $\theta$ becomes the canonical projection $TR\to TR/\ker\theta$. 

\begin{definition} Let $\theta\in\Omega^1(R,E')$ be a one form with values on the vector bundle $E'\to R$
\mbox{}\begin{itemize}
\item {\bf The symbol space of $\theta$} is the bundle over $R$ (not necessarily of constant rank) given by
\begin{eqnarray*}
\g=\g(\theta):=\ker\theta\cap T^{\pi}R.
\end{eqnarray*}
\item We say that $\theta$ is {\bf regular} if it surjective and its kernel is $\pi$-transversal, i.e.
\begin{eqnarray*}
TR= \ker\theta+TR^\pi.
\end{eqnarray*}
\item The 1-form $\theta$ is {\bf $\pi$-involutive} if for any $p\in R$
\begin{eqnarray*}
\theta_p([U,V])=0
\end{eqnarray*}
for any vector fields $U,V\in\X^\pi(R)$ such that $U,V\in\Gamma(\ker\theta)$.
\item $\theta$ is called a {\bf Pfaffian form} if it is regular an $\pi$-involutive.
\item A {\bf Pfaffian bundle} is a surjective submersion $\pi:R\to M$ together with a Pfaffian form $\theta$. In this case we use the notation $\pi:(R,\theta)\to M$.
\item A {\bf solution of $(R,\theta)$} is a section $s:M\to R$ with the property that $s^*\theta=0$. The set of solutions of $(R,\theta)$ is denoted by
\begin{eqnarray*}
\Sol(R,\theta)\subset \Gamma(R).
\end{eqnarray*} 

\item A {\bf partial integral element of $(R,\theta)$} is a linear space $V\subset T_pR$ of dimension equal to the dimension of $M$, and such that 
\begin{eqnarray*}
\theta_p(V)=0\quad{and}\quad V\oplus T^\pi_pR=T_pR.
\end{eqnarray*}
The set of partial integral elements, denoted by $J^1_\theta R$, is called {\bf the partial prolongation with respect to $\theta$}.
\end{itemize}\end{definition}

Again for a one form $\theta\in\Omega^1(M,E')$, the involutivity of $H:=\ker\theta$ is measured by the so-called partial differentiation of $\theta$:

\begin{definition}\label{partial-diff} For a point-wise surjective $1$-form $\theta\in \Omega^1(R,E')$, {\bf the partial differential of $\theta$} is the $C^{\infty}(R)$-bilinear map over $R$ defined by 
\begin{eqnarray*}\label{deltatheta}
\delta\theta:H\times H\To E',\quad \delta\theta(u,v):=\theta_p([U,V])
\end{eqnarray*}
where $U,V\in\X(R)$ are vector fields such that $\theta(U)=\theta(V)=0$, with $U_p=u$ and $V_p=v$.
\end{definition} 

The Leibniz identity implies again that $\delta\theta$ is a well-defined map.

\begin{remark}\rm Note that $\theta$ is $\pi$-involutive if $\delta\theta_{\g\times\g}=0$. So for a Pfaffian form $\theta$ we have an induced vector bundle map over $R$
\begin{eqnarray}\label{partialtheta}
\partial_\theta:\g\To\hom(TM,E'),\quad \partial_\theta(v)(X)=\delta\theta(v,[X]),
\end{eqnarray}
where again we are using the vector bundle isomorphism $H/\g\simeq \pi^*TM$. The map \eqref{partialtheta} is called {\bf the symbol map of $\theta$}. We extend $\partial_\theta$ to the linear map $\partial_\theta:\wedge^kT^*\otimes \g\to \wedge^{k+1}T^*\otimes E$ by the usual formula.
The bundle over $R$ given by the $k$-prolongation of $\partial_\theta$ is denoted by 
\begin{eqnarray*}\g^{(k)}:=\g^{(k)}(R,\theta)\subset S^kT^*\otimes \g.\end{eqnarray*}
 
When $\theta$ is not surjective but is both $\pi$-involutive and $\pi$-transversal, the bundle map \eqref{partialtheta} still makes sense as a $C^\infty(R)$-linear map but we deal with a possibly non-smooth bundle $\g(\theta).$
\end{remark}

From now on we should always keep in mind the following result which will allow us to work either with Pfaffian distributions or with the equivalent dual picture of Pfaffian forms.

\begin{lemma}\label{lema1} Let $H\subset TR$ be a distribution and $\theta:=\theta_H\in\Omega^1(R,E)$ its associated canonical form where $E=TR/H$. Then,
\begin{enumerate}
\item $H$ is $\pi$-transversal if and only if $\theta$ is regular,
\item $H$ is $\pi$-involutive if and only if $\theta$ is $\pi$-involutive, 
\item a section $s\in\Gamma(R)$ is a solution of $(R,H)$ if and only if it is a solution of $(R,\theta)$,
\item $V\subset T_pR$ is partial integral element of $(R,H)$ if and only if it is a partial integral element of $(R,\theta)$, i.e. $J^1_HR=J^1_\theta R,$
\item the curvature map $c(H)$ of $H$ and the partial differential $\delta\theta$ of $\theta$ coincide, i.e, for any $u,v\in H$ 
\begin{eqnarray*}
c(u,v)=\delta\theta(u,v).
\end{eqnarray*}
\end{enumerate}
\end{lemma}

The proof of the above lemma is a direct consequence of the definitions.

\begin{remark}\rm
From item 5 of lemma \ref{lema1} we get that the symbol maps $\partial_H,\partial_\theta:\g\to T^*\otimes E$ of $H$ and $\theta$ respectively coincide, therefore all their $k$-prolongations do as well. That is 
\begin{eqnarray*}
\g^{(k)}(R,H)=\g^{(k)}(R,\theta).
\end{eqnarray*}
\end{remark}

\begin{example}\rm
Assume that our Pfaffian bundle $\pi:(F, \theta)\to M$ is linear in the sense that $\pi:F\to M$ is a vector bundle and the distribution $H\subset TF$ is a sub-bundle of the vector bundle $d\pi:TF\to TM$, where the structural maps of $TF$ are given by the differential of the structural maps of $F$. In the dual picture we get a linear one form $\theta_H\in\Omega^1(M,\pi^*E')$, where $E'$ is now a vector bundle over $M$ defined by
\begin{eqnarray*}
E':=TF/H|_{M}.
\end{eqnarray*}
The linearity condition of $\theta_H$ is given by the equation
\begin{eqnarray*}
a^*\theta=p_1^*\theta+p_2^*\theta,
\end{eqnarray*}
where $a,p_1,p_2:F\times_MF\to F$ are the fiber-wise addition, the projection to the first and second components respectively. In this setting the bundles under consideration become pullbacks of vector bundles over $M$, for example the symbol space $\g(H)\subset H$ is now the pullback via $\pi$ of the vector bundle  $H^\pi|_M\to M$, and the bundle maps $c$ and $\partial_\theta$ are now pullbacks of vector bundle maps over $M$. Another very useful fact is that in this case one can associate to $H$ (or rather to $\theta$) a canonical connection $D$ relative to the projection $F\to F/(H^\pi|_M)$, and therefore we can apply all the theory developed in chapter \ref{Relative connections}. 
This be fully treated in section \ref{linear Pfaffian bundles}.
\end{example}

\begin{example}[Jet bundles and Cartan distributions and forms, see \cite{russians, Gold2}]\label{cartandist}\rm Continuing with the description of jet bundles given in subsection \ref{jet-bundles}, recall that
for any surjective submersion $\pi:R\to M$, and $k\geq 1$ integer, the $k$-jet bundle $\pi:J^kR\to M$ carries a canonical Pfaffian distribution $C_k$, knows as the Cartan distribution. It is designed to detect holonomic sections in the sense that a section $\al:M\to J^kR$ is of the form $\al=j^ks$ for some $s\in \Gamma(R)$ if and only if $\al$ is a solution of $(J^kR,C^k):$
\begin{eqnarray*}
j^k:\Gamma(R)\overset{\simeq}{\To}\Sol(J^kR,C_k).
\end{eqnarray*}
 Recall the definition of $C_k$. As a set, $C_k$ is generated as a $C^{\infty}(J^kR)$-module by the image of the differential of holonomic sections of $\pi:J^kR\to M$, and therefore it is transversal to $\pi$. Its vertical part $C_k^{\pi}$ is the image of the inclusion $i:S^kT^*\otimes T^{\pi}R\to T^{\pi}J^kR$ given by the short exact sequence of vector bundles over $J^kR$
\begin{eqnarray}\label{sesj}
0\To S^kT^*\otimes T^\pi R\overset{i}{\To} TJ^kR\overset{dpr}{\To} TJ^{k-1}R\To 0.
\end{eqnarray}
This sequence also implies that $C_k$ is $\pi$-involutive since it is the kernel of the differential of a surjective submersion. From this we have that \begin{eqnarray*}pr:(J^kR,C_k)\To M\end{eqnarray*} is a Pfaffian bundle.  See \cite{Gold2} for the sequence \eqref{sesj}. 

In the dual picture we recover the so called Cartan forms \begin{eqnarray*}\theta^k\in\Omega^1(J^kR,T^{\pi}J^{k-1}R).\end{eqnarray*}
Their main importance is that they detect holonomic section: $\xi\in \Gamma(J^kR)$ is of the form $j^ks$ for some $s\in\Gamma(R)$ if
$$\xi^*\theta^k=0.$$
At the point $p=j^k_xs$, $\theta_p$ is given by the formula
\begin{eqnarray*}
\theta_p(X):=dpr(X)-d_x(j^{k-1}s)\circ d\pi(X),
\end{eqnarray*}
where $X\in T_pJ^kR$. That $\theta_p(X)$ is indeed an element of $T^\pi J^{k-1}R$ follows from the fact that $\pi\circ pr=\pi$ and $\pi\circ j^{k-1}s=id_M$.
As the kernel of $\theta^k$ is precisely the Cartan distribution $C_k\subset TJ^kR$ (see e.g \cite{Bocharov, Krasil, Olver}), $pr:(j^kR,\theta^k)\to M$ is a Pfaffian bundle.\\

The following remark is for later use: the symbol map $\partial_k$ of $C_k$ actually takes values in 
\begin{eqnarray*}
T^*\otimes(S^{k-1}T^*\otimes T^\pi R)\subset T^*\otimes  TJ^{k-1}R,
\end{eqnarray*} 
where we are using the identification of the quotient $TJ^kR/C_k\simeq TJ^{k-1}R$ given by sequence \ref{sesj}. Moreover,
\begin{eqnarray*}
\partial_k:S^kT^*\otimes T^\pi R\To T^*\otimes (S^{k-1}T^*\otimes T^\pi R)
\end{eqnarray*}
is the formal differentiation described in \eqref{fdo}.

\begin{remark}\rm  There are two questions for a Pfaffian bundle $\pi:(R, H)\to M$ that are equivalent (modulo a topological condition):
\begin{enumerate}
\item\label{A1} When can one find a bundle $\tilde R$ over $M$ and an immersion $i: R\hookrightarrow J^1(\tilde R)$ such that
$\theta_H$ is the pull-back via $i$ of the Cartan form $\theta^1$ on $J^1\tilde R$?
\item\label{A2} When is the symbol map $\partial_H: H^\pi\to \hom(\pi^*TM, TR/H)$ injective?
\end{enumerate}
In the less interesting direction \ref{A1} implies \ref{A2}. One remarks that $\partial_H$ is the restriction of the differential of $i$ to $H^\pi$. For the converse, 
recall that $H^\pi$ is an involutive distribution on $R$. Take the leaf space of this foliation and call
it $\tilde R$. It is here that the topological condition comes in. We require that this quotient is a manifold (so that, strictly speaking,
\ref{A1} is equivalent to \ref{A2} under the assumption that the foliation $H^\pi$ on $R$ is simple). Our inclusion $i$ is the canonical map from $R$ to $J^1\tilde R$ sending $p\in R_x$ to 
\begin{eqnarray*}
[\sigma_p]:T_xM\To T_{[p]}\tilde R\simeq T_pR/H^\pi_p
\end{eqnarray*}
where $\sigma_p:T_xM\to T_pR$ is any splitting of $d\pi$ with the property that its image lies inside $H$.
\end{remark}
\end{example}

\section{Prolongations of Pfaffian bundles}\label{prolongations of Pfaffian bundles}

\subsection{The partial prolongation}

Recall from remark \ref{transversal-properties}, that for a distribution $H\subset TR$ transversal to $\pi:R\to M$, the partial prolongation $J^1_HR\subset J^1R$ is given by 
\begin{eqnarray*}\label{part-prol}
J^1_HR=\{j^1_xs\mid d_xs(T_xM)\subset H_{s(x)}\}.
\end{eqnarray*}
In the dual picture, for a one form $\theta\in\Omega^1(M,E)$, the partial prolongation is 
\begin{eqnarray*}
J^1_\theta R=\{j^1_xs\mid s^*\theta_x=0\}
\end{eqnarray*}
Note that for $\theta=\theta_H$, $J^1_HR=J^1_\theta R.$

\begin{lemma}\label{smothness} Let $H\subset TR$ be $\pi$-transversal.
The space $J^1_HR$ of partial integral elements is a smooth subbundle of $pr:J^1R\to R$, i.e. $J^1_HR$ is a smooth submanifold of $J^1R$ and the restriction $pr:J^1_HR\to R$ is a surjective submersion.
\end{lemma}

\begin{proof}
$J^1_HR$ is the kernel of the smooth bundle map over $R$ given by
\begin{eqnarray*}
ev:J^1R\To T^*\otimes E,\quad j^1_xs\mapsto d_xs(\cdot)\;\mod H_{s(x)}.
\end{eqnarray*} 
Using the exact sequence of vector bundles over $J^1R$ (see proposition 5.2 of \cite{Gold2})
\begin{eqnarray*}
0\To T^*\otimes T^\pi R\To TJ^1R\overset{dpr}{\To} TR\To 0,
\end{eqnarray*}
one has 
\begin{eqnarray*}
\ker dev\cap\ker dpr=T^*\otimes \g,
\end{eqnarray*}
which implies that $ev$ is of constant rank. On the other hand, the condition that  $H$ is $\pi$-transversal ensures the existence of partial integral elements at any $p\in R$, and therefore $z(R)\subset ev(J^1R)$, where $z:R\to T^*\otimes E$ is the zero section. So, we are left in the situation of a bundle map $ev$ of constant rank of two fibered manifolds $ev:X\to Y$ over $R$ with the property that $z(X)\subset ev(X)$. Hence we can apply proposition 2.1 of \cite {Gold2} to ensure that $\ker_zev$ is a fibered submanifold of $X\to R$, which in our case says that $pr:J^1_HR\to R$ is a smooth subbundle of $pr:J^1R\to R$ .  
\end{proof}

\begin{remark}\label{remark1}\rm
From the above proof we have that for a distribution $H\subset TR$ $\pi$-transversal, the sequence of vector bundles over $J^1_HR$
\begin{eqnarray*}
0\To T^*\otimes\g\To TJ^1_HR\overset{dpr}{\To} TR\To 0
\end{eqnarray*}
is exact, where $\g=H\cap T^\pi R$. 
\end{remark}

\begin{corollary}\label{corollary1}
Let $H\subset TR$ be a transversal distribution. If $\dim M>0$ then $H$ is an Ehresmann connection, i.e. 
\begin{eqnarray*}
TR=H\oplus T^\pi R
\end{eqnarray*}
if and only if $pr:J^1_HR\to R$ is a bijection.
\end{corollary}

\begin{proof}
If $TR=H\oplus T^\pi R$, then $\g:=H\cap T^\pi R=0$. So, if $j^1_xs,j^1_xu\in J^1_HR$ are such that $u(x)=s(x)$ then 
\begin{eqnarray*}
d_xs-d_xu:T_xM\To T_{s(x)}R
\end{eqnarray*}
has image inside $\g_{s(x)}=0$. Therefore $j^1_xs=j^1_xu$, which proves that $pr$ is a bijection. Conversely, if $pr:J^1_HR\to R$ is a bijection then by the short exact sequence of remark \ref{remark1} we have that $T^*\otimes\g=0$. This happens only if $\g=0$, or in other words, if $H$ is an Ehresmann connection. 
\end{proof}

\begin{remark}\rm
For a regular one form $\theta\in\Omega^1(R,E')$ we have analogous versions of lemma \ref{smothness}, remark \ref{remark1} and corollary \ref{corollary1}, using $H=\ker\theta$.
\end{remark}

\begin{definition} Let $H\subset TR$ be a $\pi$-transversal distribution and let $\theta\in\Omega^1(R,E')$ be a regular form.
\mbox{}
\begin{itemize} 
\item {\bf The partial prolongation of $(R,H)$} is the Pfaffian bundle $J^1_HR$ endowed with the distribution
\begin{eqnarray*}
H^{(1)}=C_1\cap TJ^1_HR,
\end{eqnarray*}
where $C_1\subset J^1R$ is the Cartan distribution.
\item {\bf The partial prolongation of $(R,\theta)$} is the Pfaffian bundle $J^1_\theta R$ endowed with the Pfaffian form\begin{eqnarray*}
\theta^{(1)}=\theta^1|_{ TJ^1_\theta R},
\end{eqnarray*}
where $\theta^1\in\Omega^1(J^1R,pr^*T^\pi R)$ is the Cartan form.
\end{itemize}
\end{definition}

From example \ref{cartandist} we have the following proposition.

\begin{proposition}Let $H\subset TR$ be a $\pi$-transversal distribution, and let $\theta:=\theta_H$ be its associated regular one form. Then
\begin{eqnarray*}
J^1_HR=J^1_\theta R\qquad\text{and}\qquad
H^{(1)}=\ker\theta^{(1)}
\end{eqnarray*}
where $H^{(1)}:=C_1\cap TJ^1_HR$ and $\theta^{(1)}=\theta^1|_{J^1_\theta R}$.
\end{proposition}

One of the main properties of the partial prolongation is:

\begin{proposition}\label{proposition1} Let $H\subset TR$ be a $\pi$-transversal distribution and let $\theta\in\Omega^1(M,E')$ be a regular form. The partial prolongations
$(J^1_HR,H^{(1)})$ and $(J^1_\theta R,\theta^{(1)})$ are indeed Pfaffian bundles. Moreover the map 
\begin{eqnarray*}
pr:\Sol(J^1_HR, H^{(1)})\overset{}{\To}\Sol(R,H)
\end{eqnarray*}
is a bijection with inverse $j^1:\Sol(R,H)\to \Sol (J^1_HR,H^{(1)}).$ Analogously,
\begin{eqnarray}\label{ese'}
pr:\Sol(J^1_\theta R, \theta^{(1)})\overset{}{\To}\Sol(R,\theta)
\end{eqnarray}
is a bijection with inverse $j^1:\Sol(R,\theta)\to \Sol (J^1_\theta R,\theta^{(1)})$.
\end{proposition}

This result is inspired by similar results of \cite{Gold3} in the setting of partial differential equations.

\begin{proof} We will prove that $(J^1_\theta R,\theta^{(1)})$ is a Pfaffian bundle. The case of Pfaffian distributions follows from the case of Pfaffian forms taking $\theta=\theta_H$. Let's first show that for any $\sigma\in J^1_\theta R$,
\begin{eqnarray*}
\theta^{(1)}:T_\sigma J^1_{\theta}R\To T_{pr(\sigma)}^\pi R
\end{eqnarray*}
is surjective. Note that $\theta^{(1)}$ restricted to $T^\pi_\sigma J^1_\sigma R$ is equal to the projection $dpr:T^\pi_\sigma J^1_\theta R\to T_{pr(\sigma)}^\pi R$. By the analogous version of remark \ref{remark1} we have that that $dpr$ restricted to $T^\pi J^1_\theta R$ is surjective. This also shows that 
\begin{eqnarray}\label{dimension}\begin{split}
\dim H^{(1)}_\sigma&=\dim T_\sigma J^1_\theta R-\dim T^\pi_{pr(\sigma)}R\\
&=\dim M+\dim T^\pi_\sigma J^1_\theta R-\dim T^\pi_{pr(\sigma)}R\\
&=\dim M+\dim(T^*_{\pi(\sigma)}\otimes \g_{pr(\sigma)}),
\end{split}
\end{eqnarray}  
where $\g=\ker\theta\cap T^\pi R$, and that  
\begin{eqnarray}\label{dimension2}
\ker\theta\cap T^\pi J^1_\theta R=T^*\otimes \g
\end{eqnarray}
as vector bundles over $J^1_\theta R$. Using \eqref{dimension} and \eqref{dimension2} we can show that $\theta^{(1)}$ is regular with respect to $\pi_1:J^1_\theta R\to M$ as follows: let $H^{(1)}:=\ker\theta^{(1)}$, then 
\begin{eqnarray*}
\dim(H_\sigma^{(1)}+T^\pi_\sigma J^1_\theta R)&=&\dim H_\sigma^{(1)}+\dim T^\pi_\sigma J^1_\theta R\\&&-\dim H_\sigma^{(1)}\cap T^\pi_\sigma J^1_\theta R\\&=&\dim M+\dim(T^*_{\pi(\sigma)}\otimes \g_{pr(\sigma)})+\dim T^\pi_\sigma J^1_\theta R\\&&-\dim(T^*_{\pi(\sigma)}\otimes \g_{pr(\sigma)})\\&=&\dim M+\dim T^\pi_\sigma J^1_\theta R\\&=&\dim T_\sigma J^1_\theta R.
\end{eqnarray*} 
That $\theta^{(1)}$ is $\pi$-closed is a consequence of the fact that the Cartan form $\theta^1$ is $\pi$-closed.

 As for the second part, let
 $\al:M\to J^1_HR$ be a solution of $(J^1_HR, H^{(1)})$, then for all $x\in M$, $d\al(T_xM)\subset C_1$, which implies that $\al$ is of the form $j^1s$ for $s=pr(\al)$. On the other hand, since $\al(x)=j^1_xs\in J^1_HR$, then $ds(T_xM)\subset H_{s(x)}$. This means that $s=pr(\al)$ is a solution of $(R,H)$. It is clear that if $s$ is a solution of $(R,H)$, then the holonomic section $j^1s:M\to J^1R$ has its image in $J^1_HR$ and it is a solution of $(J^1_HR, H^{(1)}).$ The part with Pfaffian forms is analogous to this case.
\end{proof}

\begin{remark}\rm
For a general one form $\theta\in\Omega^1(M,E')$, the map \eqref{ese'} still makes sense whenever $J^1_\theta R$ is smooth. In fact, the map \eqref{ese'} is a bijection with inverse given by $j^1.$
\end{remark}

\begin{remark}\rm Let $H\subset TR$ be a $\pi$-transversal distribution. A key property of $J^1_HR$ is that the differential of the projection $pr:J^1_HR\to R$ is such that 
\begin{eqnarray*}
\theta\circ dpr=\theta\circ\theta^{(1)}.
\end{eqnarray*}
Indeed, for $X\in T^\pi J^1_HR$ this is trivially true since $\theta^{(1)}(X)=dpr(X)$. For $X\in H^{(1)}_{j^1_xs}$ (i.e. $�\theta^{(1)}(X)=0$), with $j^1_xs\in J^1_HR$, one has that 
\begin{eqnarray*}
dpr(X)-d_xs(d\pi(X))=0,
\end{eqnarray*}
and therefore
\begin{eqnarray*}
\theta(dpr(X))=\theta(d_xs(d\pi(X)))=0.
\end{eqnarray*} 
\end{remark}

\subsection{The classical prolongation}\label{the-classical-prolongation}

The classical prolongation of $H\subset TR$ is a Pfaffian bundle (under some smoothness conditions) sitting above $R$, and may be thought as the complete infinitesimal data of solutions of $(R,H)$. See \cite{Gold2} for prolongations of PDE's.

\begin{definition}\label{primer}
Let $H\subset TR$ be a $\pi$-transversal distribution and let $\theta\in\Omega^1(R,E')$ be a regular one form.
 \begin{itemize}
 \item The {\bf 1-curvature map $c_1:=c_1(H)$} is the bundle map over $R$
\begin{eqnarray*}
c_1:J^1_HR&\To \wedge^2T^*\otimes E,\quad
j^1_xs\mapsto c(d_xs(\cdot),d_xs(\cdot)),
\end{eqnarray*}
where $c$ is the curvature map of $H$ (definition \ref{curvature-map}).
The {\bf classical prolongation space with respect to $H$}, denoted by \begin{eqnarray*}P_H(R)\subset J^1_HR\subset J^1R,\end{eqnarray*} is set to be $ \ker c_1$. We say that $P_H(R)$ is {\bf smooth} if it is a smooth submanifold of $J^1R$, and that is {\bf smoothly defined} if, moreover, $pr:P_H(R)\to R$ is a surjective submersion.
\item The {\bf 1-curvature map $c_1:=c_1(\theta)$} is the bundle map over $R$ 
\begin{eqnarray*}
c_1:J^1_\theta R\To \wedge^2T^*\otimes E,\quad j^1_xs\mapsto \delta\theta(d_xs(\cdot),d_xs(\cdot)).
\end{eqnarray*}
where $\delta\theta$
The {\bf classical prolongation space with respect to $\theta$}, denoted by 
\begin{eqnarray*}
P_\theta(R)\subset J^1_\theta R,
\end{eqnarray*} is set to be $\ker c_1.$
\end{itemize}
\end{definition}

\begin{remark}\rm
For a one form $\theta\in\Omega^1(R,E')$, $c_1(\theta)$ in definition \ref{primer} still makes sense even though the objects under consideration are not necessarily smooth bundles.
\end{remark}
 
 \begin{proposition}\label{compare}Let $H\subset TR$ be a Pfaffian distribution and let $\theta:=\theta_H$ be its associated Pfaffian form. Then the 1-curvature map $c_1(H)$ of $H$ and  the 1-curvature map $c_1(\theta)$ of $\theta$ coincide. Hence,
\begin{eqnarray*}
P_H(R)=P_\theta(R).
\end{eqnarray*}
Moreover, if their classical prolongations are smooth, then 
\begin{eqnarray*}
H^{(1)}=\ker\theta^{(1)},
\end{eqnarray*}
where $H^{(1)}:=C_1\cap TP_H(R)$ and $\theta^{(1)}=\theta^1|_{P_\theta(R)}.$
\end{proposition}

\begin{proof}
From lemma \ref{lema1} we have that $J^1_\theta R=J^1_HR$ and $c(H)=\delta\theta$. Therefore $c_1(\theta)=c_1(H)$, which implies that $P_H(R)=P_\theta(R)$. That $H^{(1)}=\ker\theta^{(1)}$ is a consequence of remark \ref{remark1}.
\end{proof}

\begin{proposition}\label{propositionA} For a Pfaffian distribution $H\subset TR$ the following are equivalent:
\begin{enumerate}
\item $H$ is involutive,
\item $P_H(R)=J^1_HR$ and $\partial_H$ vanishes.
\end{enumerate}
\end{proposition}

\begin{proof}
Assume that 
$P_H(R)=J^1_HR$, or equivalently $c_1=0$, and that $\partial_H$ vanishes. Now, for $U,V\in H_p$, let $s\in \Gamma(A)$ be such that $s(\pi(p))=p$ and $j^1_xs\in J^1_HR$. Using $d_xs$ we write 
\begin{eqnarray*}
U=d_xs(d\pi(U))+u,\qquad V=d_xs(d\pi(V))+v,
\end{eqnarray*}
where $u,v\in \g$ are defined by $u=U-d_xs(d\pi(U))$ and $v=V-d_xs(d\pi(V))$. Hence
\begin{eqnarray*}
\begin{split}
c(U,V)=&c(d_xs(d\pi(U)),d_xs(d\pi(V)))+c(d_xs(d\pi(U)),v)
\\&+c(u,d_xs(d\pi(V)))+c(u,v)
\\=&c_1(j^1_xs)(d\pi(U),d\pi(V))-\partial_H(v)(d\pi(U))+\partial_H(u)(d\pi(V))\\=&0,
\end{split}
\end{eqnarray*}
where we use the $\pi$-involutivity of $H$ passing from the second to the third line. This shows that $c$ vanishes, and therefore $H$ is involutive. Conversely, if $H$ is involutive then the map $c$ is identically zero, and therefore $c_1=0$ and $\partial_H=0$ by definition.
\end{proof}
 
 As for Ehresmann connections $H\subset TR$, one gets the following corollary.
 
 \begin{corollary}\label{corollaryB} For any Ehresmann connection $H$,
 \begin{enumerate}
 \item $pr:P_H(R)\to R$ is injective, and 
 \item $H$ is involutive if and only if $pr:P_H(R)\to R$ is a bijection.
 \end{enumerate}
 \end{corollary}
 
 \begin{proof}
 Form corollary \ref{corollary1} one has that $H$ is an Ehresmann connection if $pr:J^1_HR\to R$ is a bijection. Since $P_H(R)$ is a subset of $J^1_HR$ then the injective map $pr:P_H(R)\to R$ is a bijection if and only if $J^1_H(R)=P_H(R)$. Now we apply the previous proposition in our case to have that $H$ is involutive if and only if $J^1_HR=P_H(R)$ as $\partial_H$ is zero ($\g=0$), i.e. if and only if $pr:P_H(R)\to R$ is a bijection.
 \end{proof}
 
 Regarding smoothness of the partial prolongation $P_H(R)$ we have the following two results:
 
 \begin{proposition}\label{model}Let $H\subset TR$ be a Pfaffian distribution.
If $P_H(R)\subset J^1R$ is smoothly defined, then 
there is a commutative diagram over $P_H(R)$ with short exact rows  
\begin{eqnarray}\label{use}
\xymatrix{
0 \ar[r] & pr^*\g^{(1)}(R,H) \ar@{^{(}->}[d] \ar[r] & T^\pi P_H(R) \ar@{^{(}->}[d] \ar[r]^{dpr}  & T^\pi R \ar[r] \ar[d]^{=} & 0\\
0 \ar[r] & T^*\otimes T^\pi R \ar[r] & T^\pi J^1R \ar[r]^{dpr} & T^\pi R \ar[r] & 0.
}\end{eqnarray}  
In particular,
\begin{eqnarray*}
pr^*\mathfrak{g}^{(1)}(R,H)\simeq T^\pi P_H(R)\cap(T^*\otimes T^\pi R).
\end{eqnarray*}
\end{proposition}

\begin{proof}
First of all, for $\sigma=j^1_xs\in J^1R$, the inclusion
\begin{eqnarray*}
T^*_{x}\otimes T_{s(x)}^\pi R\hookrightarrow T_\sigma^\pi J^1R
\end{eqnarray*}
is given by
\begin{eqnarray*}
\phi:T_{x}M\to T_p^vR\mapsto \frac{d}{dt}\left(d_xs+t\phi:T_{x}M\to T_{s(x)}R\right)_{|t=0}\in T_\sigma^\pi J^1R.
\end{eqnarray*}
So, if $j^1_xs\in P_H(R)\subset J^1R$ and $\phi\in \mathfrak{g}^{(1)}_{s(x)}$, i.e.
\begin{eqnarray*}
d_xs(T_{x}M)\subset H_{s(x)}\quad\text{  and  }\quad c_1(j^1_xs)=0,
\end{eqnarray*}
and for every $X,Y\in T_xM$
\begin{eqnarray*}
\phi(T_{x}M)\subset \g_{s(x)}\quad\text{  and  }\quad \partial_H(\phi(X))(Y)-\partial_H(\phi(Y))(X)=0,
\end{eqnarray*}
then one has that for $t\in\mathbb{R}$, $d_xs+t\phi:T_{x}M\to T_{s(x)}R$ belongs to $P_H(R)$ since $(d_xs+t\phi)(T_{x}M)\subset H_{s(x)}$. In other words $d_xs+t\phi\in J^1_HR$,  and
\begin{eqnarray*}
c(d_xs+t\phi(X),d_xs+t\phi(Y))&=&c(d_xs(X),d_xs(Y))+c(\phi(X),\phi(Y))\\
&& +c(d_xs(X),\phi(Y))+c(\phi(X),d_xs(Y))\\
&=&c_1(j_x^1s)(X,Y)+c(d_xs(X),\phi(Y))\\
&& +c(\phi(X),d_xs(Y))\\
&=&-\partial_H(\phi(Y))(X)+\partial_H(\phi(Y))(X)\\
&=&0,
\end{eqnarray*}
where in the second equality we used that $c(\phi(X),\phi(Y))=0$ since $\g$ is involutive, and in the third equality we used that $c(d_xs(X),\phi(Y))=-\partial_H(\phi(Y))(X)$ by definition of $\partial_H$. Hence $pr^*\mathfrak{g}^{(1)}\subset (T^*M\otimes T^\pi R)|_{P_H(R)}\cap T^\pi(P_H(R))$.

To show the other inclusion, note that since $pr:P_H(R)\to R$ is a submersion then
\begin{eqnarray*}
(T_x^*M\otimes T_{s(x)}^\pi R)\cap T_\sigma^\pi(P_H(R))=\ker(d_\sigma(pr_{|P_H(R)}))=T_\sigma pr_{|P_H(R)}^{-1}(s(x))
\end{eqnarray*}
 So, an element of $T_\sigma pr_{|P_H(R)}^{-1}(s(x))$ can be represented as the derivative of a path
\begin{eqnarray*}
\Phi(t):T_{x}M\to T_{s(x)}R
\end{eqnarray*}
with the properties that $\Phi(0)=d_xs$ and that for all $t\in I$
\begin{eqnarray*}
\Phi(t)(T_{x}M)\subset H_s(x)\quad\text{ and }\quad c_1(\Phi(t))=c(\Phi(t)(\cdot),\Phi(t)(\cdot))=0.
\end{eqnarray*}
This implies that the image of $\Phi(t)-\Phi(0)$ lies in $\g_{s(x)}$, and also that 
\begin{eqnarray*}
\phi:=\dot\Phi(0):T_{x}M\to \g_{s(x)}
\end{eqnarray*} 
satisfies the equation
\begin{eqnarray*}
\partial_H(\phi(X))(Y)-\partial_H(\phi(Y))(X)=0,
\end{eqnarray*}
for any $X,Y\in T_xM$, i.e. $\phi\in\g^{(1)}$. 
This last formula follows since $c$ is a bilinear map and therefore
\begin{eqnarray*}
0&=&\frac{d}{dt}c(\Phi(t)(X),\Phi(t)(Y))_{|t=0}\\
&=&c(\frac{d}{dt}\Phi(t)(X)_{|t=0},\Phi(0)(Y))+c(\Phi(0)(X),\frac{d}{dt}\Phi(t)(Y)_{|t=0})\\
&=&\partial_H(\phi(X))(Y)-\partial_H(\phi(Y))(X).
\end{eqnarray*}
\end{proof}

 But even more interestingly:

\begin{proposition}\label{integrable1} Let $H\subset TR$ be a Pfaffian distribution. Then
the following statements are equivalent:
\begin{enumerate}
\item $\g^{(1)}(R,H)$ is a smooth vector bundle over $R$ and the map $pr:P_H(R)\to R$ is surjective,
\item $P_H(R)$ is smoothly defined.
\end{enumerate}
\end{proposition}

\begin{proof}
We will prove that 1 implies 2. From lemma \ref{smothness} one has that 
$pr:J^1_HR\to R$ is a fibered submanifold of $pr:J^1R\to R$. Moreover, it is easy to check that the short exact sequence \eqref{sesj2} for $k=1$ restricts to the short exact sequence over $J^1_HR$
\begin{eqnarray*}
0\To T^*\otimes \g\To TJ^1_HR\overset{dpr}{\To} TR\To 0.
\end{eqnarray*}
Now, $P_H(R)\subset J^1_HR$ is the kernel of the fiber bundle map over R,
\begin{eqnarray*}
c_1:J^1_HR \To \wedge^2T^*\otimes E, \quad
\sigma\mapsto c(\sigma(\cdot),\sigma(\cdot)).
\end{eqnarray*}
As in the proof of proposition \ref{model} one can prove  that  
\begin{eqnarray*}
\ker d_\sigma c_1\cap\ker d_\sigma pr=\g^{(1)}_{pr(\sigma)}
\end{eqnarray*} 
for $\sigma\in \ker c_1$, and therefore $c_1$ in this case is of constant rank as $\g^{(1)}$ is of constant rank by assumption. On the other hand, the surjectivity of  $pr:P_H(R)\to R$ implies that $z(R)\subset c_1(J^1_HR)$, where $z:R\to \wedge^2T^*\otimes E$ is the zero-section.
 We are left again in the situation of a fibered bundle map $c_1$ of constant rank of two fibered manifolds $a:X\to Y$ over $R$ with the property that $z(X)\subset c_1(X)$. We apply proposition 2.1 of \cite {Gold2} to ensure that $\ker_zc_1$ is a fibered submanifold of $X\to R$, which in our case says that $pr:P_H(R)\to R$ is a smooth subbundle of $pr:J^1_H(R)\to R$.
\end{proof}

\begin{remark}\rm
By proposition \ref{compare}, we have analogous versions of propositions \ref{model} and \ref{integrable1} for the classical prolongation $P_\theta(R)$ of a Pfaffian form $\theta\in\Omega^1(M,E')$.  
\end{remark}

 \begin{definition}\label{prolongationdistributions}Let $H\subset TR$ be a distribution and let $\theta\in\Omega^1(R,E')$ be a one form, and assume that $P_H(R)$ and $P_\theta(R)$ are smoothly defined.
 \begin{itemize}
 \item {\bf The classical prolongation of $(R,H)$} is $P_H(R)$ endowed with the Pfaffian distribution\begin{eqnarray*}
H^{(1)}=C_1\cap TP_H(R),
\end{eqnarray*} 
where $C_1\subset TJ^1R$ is the Cartan distribution.
\item {\bf The classical prolongation of $(R,\theta)$} is $P_\theta(R)$ endowed with the Pfaffian form
\begin{eqnarray*}
\theta^{(1)}:=\theta^1|_{TP_\theta(R)}\in\Omega^1(P_\theta(R),T^{\pi}R),
\end{eqnarray*} 
where $\theta^1\in\Omega^1(J^1R,T^\pi R)$ is the Cartan form. 
\end{itemize}
 \end{definition}
 
 
 \begin{proposition}\label{bundle-prop} Let $(R,H)$ be a Pfaffian bundle and suppose that $P_H(R)$ is smoothly defined. Then, 
 \begin{eqnarray*}
 pr:\Sol(P_H(R),H^{(1)})\To \Sol(R,H)
 \end{eqnarray*}
 is a bijection with inverse $j^1:\Sol(R,H)\to \Sol(P_H(R),H^{(1)})$.
 \end{proposition}
 
 \begin{proof}
 By example \ref{cartandist} and the definition of the classical prolongation, we have that if $\xi\in\Sol(P_H(R),H^{(1)})$, then $\xi=j^1s$ for some $s\in\Gamma(R)$. Since for any $x\in M$, $j^1_xs\in P_H(R)$, then it follows that $s\in \Sol(R,H)$. On the other hand, if $s\in\Sol(R,H)$, then it is easy to show that $j^1s\in\Sol(P_H(R),H^{(1)})$
 \end{proof}

\begin{lemma}\label{lemaese}Let $H\subset TR$ be a distribution and let $\theta\in\Omega^1(R,E')$ be a one form. If $P_H(R)$ is smoothly defined, then  
\begin{eqnarray*}
\pi:(P_H(R),H^{(1)})\To M
\end{eqnarray*}
is indeed a Pfaffian bundle. The same holds for ($P_\theta(R),\theta^{(1)})$ whenever $P_\theta(R)$ is smoothly defined.\end{lemma}

See \cite{Gold3} for similar results in the setting of partial differential equations.

\begin{proof} By proposition \ref{compare} and proposition \ref{compare} it is enough to prove the result for a $\pi$-transversal distribution $H\subset TR$.
Let's show that $H^{(1)}\subset TP_H(R)$ has constant rank equal to $\dim M+\rank K^{(1)}$, where $K^{(1)}$ is the vector bundle over $P_H(R)$ given by the kernel of the point-wise surjective map
\begin{eqnarray}\label{uses}
dpr:T^\pi P_H(R)\To T^\pi R.
\end{eqnarray} 
Let $\theta^1\in \Omega^1(J^1R;T^\pi R)$ be the Cartan form. Since $H^{(1)}=\ker\theta^1|_{TP_H(R)}$, it's enough to show that $\theta^{(1)}:T_\sigma P_H(R)\to T_{pr(\sigma)}^\pi R$ is surjective at every point $\sigma\in P_H(R)$. For this, notice that the restriction $\theta^{(1)}:T_
\sigma^\pi P_H(R)\to T^\pi_{pr(\sigma)}R$ is equal to $dpr:T_\sigma^\pi P_H(R)\to T_{pr(\sigma)}^\pi R$ on the one hand, and on the other hand $dpr:T_\sigma^\pi P_H(R)\to T_{pr(\sigma)}^\pi R$ is surjective since $pr:P_H(R)\to R$ is a surjective submersion and the commutativity of the diagram
\begin{eqnarray*}
\xymatrix{
P_H(R) \ar[rr]^{pr} \ar[dr]_{\pi} & & R \ar[dl]^\pi \\
& M. &
}
\end{eqnarray*}
This also shows that 
\begin{eqnarray*}
\dim H^{(1)}_\sigma&=&\dim T_\sigma P_H(R)-\dim T^\pi_{pr(\sigma)}R\\
&=&\dim M+\dim T_\sigma^\pi P_H(R)-\dim T_{pr(\sigma)}^\pi R\\
&=&\dim M+\dim K^{(1)}_\sigma,
\end{eqnarray*}
where in the third equality we used the surjectivity of \eqref{uses}.

To show that $H^{(1)}$ is transversal to the fibers of $\pi:P_H(R)\to M$ we see that 
\begin{eqnarray*}
\dim(H^{(1)}_\sigma+\ker d_\sigma(\pi|_{P_H(R)}))&=&\dim H^{(1)}_\sigma+\dim\ker d_\sigma(\pi|_{P_H(R)})\\&&
-\dim(H^{(1)}_\sigma\cap\ker d_\sigma(\pi|_{P_H(R)}))\\
&=&\dim M+\dim K^{(1)}_\sigma+\dim\ker d_\sigma(\pi|_{P_D(R)})\\
&&-\dim K^{(1)}_\sigma\\
&=&\dim M+\dim T_\sigma^\pi P_H(R)\\
&=&\dim T_\sigma P_H(R),
\end{eqnarray*}
where in the second equality we used that 
\begin{eqnarray*}\begin{split}
(H^{(1)})^{\pi}&=H^{(1)}\cap\ker d(\pi|_{P_H(R)})=(C_1)^\pi\cap T^\pi P_H(R)\\&=(T^*\otimes T^\pi R)|_{P_H(R)}\cap T^\pi P_H(R)=K^{(1)}\end{split}
\end{eqnarray*}
again by the surjectivity of \eqref{uses}. Now as $(H^{(1)})^{\pi}$ is equal to $(C_1)^\pi\cap T^\pi P_H(R)$, and $(C_1)^\pi$ and $T^\pi P_H(R)$ are both involutive, this implies that $(H^{(1)})^\pi$ is involutive. Therefore $\pi:(P_H(R),H^{(1)})\to M$ is a Pfaffian fibration.
\end{proof}

\begin{remark}[The key properties of $(P_H(R),H^{(1)})$]\label{key-properties}\rm Two remarkable properties that the pair $(H,H^{(1)})$, satisfies are the following:
\begin{enumerate}
\item The differential $dpr:TP_H(R)\to TR$ is such that
\begin{eqnarray*}
dpr(H^{(1)})\subset H.
\end{eqnarray*}
\item For any $X,Y\in H^{(1)}_\sigma\subset T_\sigma P_H(R)$,
\begin{eqnarray*}
c_H(dpr(X),dpr(Y))=0.
\end{eqnarray*}
where $c_H$ is the Lie bracket modulo $H$.
\end{enumerate}
Indeed, the first property was already shown for the partial prolongation $(J^1_HR,H^{(1)})$. For the second property, take $X,Y\in H^{(1)}_\sigma$, i.e. 
\begin{eqnarray*}
dpr(X)=\sigma(d\pi (X))\quad\text{and}\quad dpr(Y)=\sigma(d\pi(Y)),
\end{eqnarray*}
and 
\begin{eqnarray*}
c_H(dpr(X),dpr(Y))=c_H(\sigma(d\pi(X)),\sigma(d\pi(Y)))=0.
\end{eqnarray*}

Paraphrasing the above properties in the dual picture of a point-wise surjective form $\theta\in \Omega^1(R,E')$, one finds that the classical prolongation $(P_\theta(R),\theta^{(1)})$ satisfies the following:
\begin{enumerate}
\item The following diagram commutes:
\begin{eqnarray*}
\xymatrix{
TP_\theta(R) \ar[r]^{dpr} \ar[d]_{\theta^{(1)}} & TR \ar[d]^{\theta} \\
T^\pi R \ar[r]^{\theta} & E'.
}
\end{eqnarray*} 
\item For any $X,Y\in\ker\theta^{(1)}_\sigma$,
\begin{eqnarray*}
\delta\theta(dpr(X),dpr(Y))=0.
\end{eqnarray*}
\end{enumerate}
\end{remark}

\begin{remark}\rm
The previous properties can serve as an inspiration for introducing the notions of morphism and abstract prolongations of Pfaffian bundles. We will do this in the case of Pfaffian groupoids (see chapter \ref{Pfaffian groupoids}).\end{remark}

For instance, the higher jet bundles $J^kR$ of a surjective submersion $\pi:R\to M$, endowed with the Cartan distributions $C_k\subset TJ^kR$ or rather with the Cartan forms $\theta^k\in\Omega^1(J^kR,T^\pi J^{k-1}R)$ (see example \ref{cartandist}), are ``compatible under prolongations''. More precisely,
 
\begin{proposition}\label{prol-jet}
For $k\geq 1$ a natural number, 
\begin{eqnarray*}
P_{\theta^k}(J^kR)=P_{C_k}(J^kR)=J^{k+1}R\subset J^1(J^kR)\end{eqnarray*} 
and 
\begin{eqnarray*}
C_k^{(1)}=C_{k+1},\qquad (\theta^{k})^{(1)}=\theta^{k+1}.
\end{eqnarray*} \end{proposition}

\begin{proof}  That $P_{\theta^k}(R)=J^{k+1}R\subset J^1(J^kR)$ its a well-known fact (see e.g \cite{Bocharov, Krasil, Olver}). That $(\theta^k)^{(1)}=\theta^{k+1}$ is true since the restriction of the Cartan form $\theta\in\Omega^1(J^1(J^kR), T^{\pi}J^kR)$ to $J^{k+1}R$ coincides precisely with $\theta^{k+1}$.
\end{proof}

\section{Formal integrability of Pfaffian bundles}

Throughout this subsection $\pi:(R,H)\to M$ is a Pfaffian bundle with symbol space given by $\g$, symbol map denoted by $\partial_H:\g\to T^*\otimes E$, and curvature map denoted by $c:H\times H\to E$, where $E:=TR/H$. In the rest of this section the equivalent languages of Pfaffian distributions and Pfaffian forms will be used interchangeably according to which suits our purposes better.

The reason to consider formally integrable Pfaffian bundles is that in some cases (e.g. for analytic Pfaffian bundles) they have local solutions at any point of the base manifold. This will be discussed later on in this subsection.

\subsection{The classical prolongations of Pfaffian bundles}\label{sec:class-prol-pfaff}

We now define the classical prolongations of a Pfaffian bundle $\pi:(R,H)\to M$ inductively: 

\begin{definition}\label{higher-prol} Let $(R,H)$ be a Pfaffian bundle. We say that the {classical \bf $k$-prolongation space $P^k_H(R)$} is {\bf smooth} if 
\begin{enumerate}
\item $(P_H(R), H^{(1)}),\dots,(P_H^{k-1}(R),H^{(k-1)})$ are smoothly defined, and
\item the classical prolongation space of $(P_H^{k-1}(R),H^{(k-1)})$
\begin{eqnarray*}
P^k_H(R):=P_{H^{(k-1)}}(P_H^{k-1})
\end{eqnarray*}
is smooth.
\end{enumerate}
In this case, we define the {\bf $k$-prolongation of $H$}:
\begin{eqnarray*}
H^{(k)}:=(H^{(k-1)})^{(1)}\subset TP^k_H(R).
\end{eqnarray*}
We say that the classical prolongation space $P_H^k(R)$ is {\bf smoothly defined} if, moreover,
\begin{eqnarray*}
pr: P^k_H(R)\to P_H^{k-1}(R)
\end{eqnarray*}
is a surjective submersion. In this case,
\begin{eqnarray*}
\pi:(P_H^k(R),H^{(k)})\To M
\end{eqnarray*}
(a Pfaffian bundle: see proposition \ref{compare}) is called the {\bf classical $k$-prolongation of $(R,H)$.}
\end{definition}

 \begin{proposition}\label{bundle-prop2} Let $(R,H)$ be a Pfaffian bundle and suppose that $P^k_H(R)$ is smoothly  defined. Then, 
 \begin{eqnarray*}
 pr^k_0:\Sol(P^k_H(R),H^{(k)})\To \Sol(R,H)
 \end{eqnarray*}
 is a bijection with inverse $j^k:\Sol(R,H)\to \Sol(P^k_H(R),H^{(k)})$.
 \end{proposition}
 
 \begin{proof}
 Apply proposition \ref{bundle-prop} each time you prolong.
 \end{proof}
 
For now let us study the smoothness of the prolongations spaces as it is one of the conditions on definition \ref{fi}. The following proposition allows us to define higher classical prolongations of $(R,H)$ with no smoothness assumptions on the preceding prolongation.

\begin{proposition}\label{jet-prol}
Let $(R,H)$ be a Pfaffian bundle and let $k_0$ be an integer. If the classical $k_0$-prolongation space $P^{k_0}_H(R)$ is smoothly defined, then for any $0\leq k\leq k_0+1$,
\begin{eqnarray}\label{yay}
P^k_H(R)=J^{k-1}(P_H(R))\cap J^kR
\end{eqnarray}
and $H^{(k)}$ is the intersection of the Cartan distribution $C_k\subset TJ^kR$ with $TP^k_H(R)$, the tangent bundle of $P^k_H(R)$.
\end{proposition}

\begin{proof}
For $k=2$, we have that $P^2_H(R)$ is by definition the set of jets $j^1_xs\in J^1(P_H(R))\subset J^1(J^1(R))$ such that 
\begin{eqnarray*}
d_xs(T_xM)\subset H^{(1)}\qquad{\text{and}}\qquad c_1(H^{(1)})(j^1_xs)=0.
\end{eqnarray*}
As $H^{(1)}\subset C_1$, then $j^1_xs\in P^2_H(R)$ if and only if $j^1_xs\in J^1(P_H(R))$ and 
\begin{eqnarray}\label{312}
d_xs(T_xM)\subset C_1\qquad{\text{and}}\qquad c_1(C_1)(j^1_xs)=0.
\end{eqnarray}
By proposition \ref{prol-jet}, conditions \eqref{yay} mean that $j^1_xs\in J^2R$. In other words,
\begin{eqnarray*}
P^2_H(R)=J^1(P_H(R))\cap J^2R
\end{eqnarray*}
The reader may verify with an inductive argument that the equality \eqref{yay} is valid for $k\geq 2$.

Now, the Cartan form $\theta^1\in \Omega^1(J^1(J^{k-1}R), T^\pi J^{k-1}R)$ associated to the fiber bundle $\pi:J^{k-1}R\to M$, restricts to the sub-bundle $J^{k}R\subset J^1(J^{k-1}R)$ to the Cartan form $\theta^k\in\Omega^1(J^kR,T^\pi J^{k-1}R)$, and therefore the associated Pfaffian form $\theta^{(k)}$ of $P^k_H(R)$ is
\begin{eqnarray*}
\theta^1|_{J^1(P^{k-1}_H(R))\cap J^kR}=\theta^k|_{J^1(P^{k-1}_H(R))\cap J^kR}=\theta^k|_{P^k_H(R)}.
\end{eqnarray*}
Since $H^{(k)}$ is the kernel of $\theta^{(k)}$, by the above equality it follows that $H^{(k)}=C_k\cap TP^k_H(R).$
\end{proof}

For what follows let's keep in mind the exact sequence of vector bundles over $J^kR$
\begin{eqnarray}\label{sesj2}
0\To S^kT^*\otimes T^\pi R\overset{i}{\To} T^\pi J^kR\overset{dpr}{\To} T^\pi J^{k-1}R\To 0.
\end{eqnarray}
When working with the first jet $J^1R$ we may consider the elements of $J^1R$ as pairs $(p,\sigma)$, where $p\in R$ and $\sigma:T_{\pi(p)}M\to T_pR$ is a splitting of $d\pi:T_pR\to T_{\pi(p)}M$. \\

The next proposition puts together proposition \ref{integrable1} and proposition \ref{compare}, and explains the compatibility between higher prolongations. More precisely, 

\begin{proposition}\label{integrable}
The following statements are equivalent:
\begin{enumerate}
\item $\g^{(1)}(R,H)$ is a smooth vector bundle over $R$ and the map $pr:P_H(R)\to R$ is surjective,
\item $P_H(R)$ is smoothly defined.
\end{enumerate}
Moreover, if either of the above two conditions holds, then  
\begin{eqnarray*}
\pi:(P_H(R),H^{(1)})\To M
\end{eqnarray*}
is a Pfaffian bundle, and the first classical prolongation space of $(P_H(R),H^{(1)})$ is equal to the classical $2$-prolongation space of $(R,H)$. 
\end{proposition}

See again \cite{Gold3} for similar results in the setting of partial differential equations.

\begin{proof}It remains to prove that the first prolongation of $(P_H(R),H^{(1)})$ is equal to the classical $2$-prolongation of $(R,H)$. 
One has that the first classical prolongation of $\pi:(P_H(R),H^{(1)})\to M$ is the subset of $J^1(P_H(R))\subset J^1(J^1R)$, given by splittings $\xi:T_{\pi(\sigma)}M\to T_\sigma P_H(R)\subset T_\sigma J^1R$ of $d\pi:T_\sigma P_H(R)\to T_{\pi(\sigma)}M$ such that
\begin{eqnarray*}
\xi(T_{\pi(\zeta)}M)\subset H^{(1)}\subset C_1\quad\quad\text{and }\quad\quad c(\xi(X),\xi(Y))=0
\end{eqnarray*}
for all $X,Y\in T_{\pi^1(\sigma)}M$, where $c:H^{(1)}\times H^{(1)}\to TP_H(R)/H^{(1)}\simeq TJ^1R/C_1$ is the curvature of $H^{(1)}$, which coincides with the curvature of $C_1$ restricted to $H^{(1)}$. These two conditions are equivalent to the fact that $\xi$ is actually an element of $J^2R$ (see e.g \cite{Bocharov, Krasil, Olver}). Therefore, the first classical prolongation of $\pi:(P_H(R),H^{(1)})\to M$ is defined on 
\begin{eqnarray*}
J^1(P_H(R))\cap J^2R=P_H^2(R).
\end{eqnarray*}
\end{proof}

For the next proposition we choose to define the classical prolongations spaces of the Pfaffian bundle $\pi:(R,H)\to M$ by equation \eqref{yay}, where no smoothness assumptions are needed. See also remark \ref{involutive-properties} for the definition of $\g^{(k)}$.

\begin{proposition}\label{affine}
Let $\pi:(R,H)\to M$ be a Pfaffian bundle and let $k$ be an integer. If 
\begin{enumerate}
\item $P_H^k(R)$ is smoothly defined,
\item $\mathfrak{g}^{(k+1)}(R,H)$ is a vector bundle over $R$, and
\item $pr:P_H^{k+1}(R)\to P_H^k(R)$ is surjective for $0\leq k\leq l$,
\end{enumerate} 
then
\begin{enumerate}
\item $P_H^{k+1}(R)$ is smoothly defined, and  
\item the sequence \eqref{sesj2} restricts to the short exact sequence 
\begin{eqnarray}\label{sesj3}
0\To pr^*\mathfrak{g}^{(k+1)}\To T^\pi P^{k+1}_H(R)\overset{dpr}{\To}T^\pi P^k_H(R)\To 0
\end{eqnarray}
of vector bundles over $P^{k+1}_H(R)$ for any $0\leq k\leq l$.
\end{enumerate} 
Moreover, $\pi:(P^{k+1}_H(R),H^{(k+1)})\to M$ is a Pfaffian bundle, and its classical $l$-prolongation space is the same as the classical $l+k$-prolongation space of $(R,H)$ 
\end{proposition}

\begin{proof}
We proceed by induction. For $k=0$, the conclusion holds by proposition \ref{integrable}. For $l\geq 1$ we apply proposition 7.2 of \cite{Gold2} to the partial differential equation $R_1=P_H(R)\subset J^1R.$
\end{proof}

\begin{remark}\label{partevertical}\rm In proposition \ref{affine}, the vector bundle $pr^*\g^{(k)}$ over $P^{k}_H(R)$ is equal to the vertical part of $H^{(k)}$, i.e.
\begin{eqnarray*}
\g(P^k_H(R), H^{(k)})\simeq pr^*\g^{(k)}.
\end{eqnarray*} 
This is clear since $C_{k}^{\pi}=S^kT^*\otimes T^\pi R$ by example \ref{cartandist}, and $pr^*\g^{(k)}=S^kT^*\otimes T^\pi R\cap TP_H^k(R).$ 
\end{remark}

\subsection{Formal integrability}

 \begin{definition}\label{fi}
The Pfaffian bundle $(R,H)$ is called {\bf formally integrable} if 
all the classical prolongation spaces 
\begin{eqnarray*}
P_H(R),P_H^2(R),\ldots,P^k_H(R),\ldots
\end{eqnarray*}
are smoothly defined.
\end{definition}

The next corollary follows by the definitions and proposition \ref{jet-prol}, and is stated here for future use.

\begin{corollary}\label{337}
If $\pi:(R,H)\to M$ is a formally integrable Pfaffian bundle, then for each $k\geq0$, $pr:P_H^{k+1}(R)\to P_H^k(R)$ is a smooth fiber subbundle of $pr:J^{k+1}R\mid_{P^{k}_H(R)}\to P^k_H(R)$.
\end{corollary}

Hence, for a formally integrable Pfaffian bundle $(R,H)$,  we have a sequence of Pfaffian bundles over $M$
 \begin{eqnarray*}
 \cdots\To (P^k_H(R),H^{(k)})\overset{pr}{\To}\cdots \To (P_H(R),H^{(1)})\overset{pr}{\To} (R,H)
 \end{eqnarray*}
  called {\bf the classical resolution of $(R,H)$}.

The following existence result in the analytic case shows the importance of formally integrable Pfaffian bundles. See \cite{Gold3} for the same result in the setting of partial differential equations.

\begin{theorem}\label{anal}
Suppose that $\pi:(R,H)\to M$ is a formally integrable analytic Pfaffian bundle. Then, given $p\in P^l_H(R)$ with $\pi(p)=x\in M,$ there exists an analytic solution $s$ of $\pi:(R,H)\to M$ over a neighborhood of $x$ such that $j^l_xs=p.$ 
\end{theorem}

\begin{proof}
We know that if $(R,H)$ is formally integrable, then $P_H(R)\subset J^1R$ is a formally integrable differential equation in the sense of \cite{Gold2}. So, if $P_H(R)\subset J^1R$ is an analytic submanifold we can apply theorem 9.1 of \cite{Gold2} to $P_H(R)$ to obtain the result (since any solution of the differential equation $P_H(R)$ is a solution of $\pi:(R,H)\to M$). So, we only need to check that $P_H(R)\subset J^1R$ is an analytic submanifold. This follows from the fact that $P_H(R)$ is the kernel of the analytic morphism
\begin{eqnarray*}
c_1:J^1_HR \To& \hom_R(\wedge^2TM;TR/H),\quad j^1_xs\mapsto c(d_xs(\cdot),d_xs(\cdot)),
\end{eqnarray*}
and of course, $J^1_HR$ is itself analytic as it is the kernel of the analytic morphism
\begin{eqnarray*}
J^1R\To\hom_R(TM;TR/H),\quad
j^1_xs\mapsto pr\circ d_xs:T_{x}M\to T_{s(x)}R/H_{s(x)}.
\end{eqnarray*}
\end{proof}

One of the aims of this subsection is to prove following workable criteria for formal integrability of Pfaffian bundles. This is a slight generalization of results in \cite{Gold3} in the setting or partial differential equations.

\begin{theorem}\label{workable-pfaffian-bundles}
Let $\pi:(R,H)\to M$ be a Pfaffian bundle such that:
\begin{enumerate}
\item $pr:P_H(R)\to R$ is surjective,
\item $\mathfrak{g}^{(1)}(R,H)$ is a vector bundle over $R$, and
\item $H^{2,l}(\g)=0$ for $l\geq0.$
\end{enumerate}
Then, $(R,H)$ is formally integrable.
\end{theorem}

In the previous theorem we are considering the $\partial_H$-Spencer cohomology of $\g^{(1)}$. See definition \ref{exten}.

\subsection{Higher curvatures}

Now we concentrate on the proof of theorem \ref{workable-pfaffian-bundles}; this will be achieved using induction. One of the main ingredients of the proof is given by proposition \ref{affine} together with the following.

\begin{proposition}Let $(R,H)$ be a Pfaffian bundle and let $k$ be an integer. If the classical $k$-prolongation space $P^k_H(F)$ is smoothly defined, then there exists an exact sequence of bundles over $R$:
\begin{eqnarray*}
P^{k+1}_H(R)\overset{pr}{\To} P^k_H(R)\xrightarrow{\bar c_{k+1}}{}H^{(2,k-1)}(\g).
\end{eqnarray*}
\end{proposition}

In the previous proposition, by exactness of the non-linear sequence we mean that 
$$\Im(pr) = \ker(\bar{c}_{k+1}),$$ where $Z(\bar{c}_{k+1})$ is the zero set of $\bar{c}_{k+1}$.\\

 Assume that for an integer $k_0$, proposition \ref{affine} holds. Hence we have a sequence of Pfaffian bundles 
 \begin{eqnarray*}
 (P_H^{k_0+1}(R),H^{(k_0+1)})\overset{pr}{\To}(P_H^{k_0}(R),H^{(k_0)})\overset{pr}{\To}\cdots \To(P_H(R),H^{(1)})\overset{pr}{\To}(R,H)
 \end{eqnarray*}
  each consecutive pair satisfying the properties \ref{key-properties}.
 
We will construct, for each natural number $0\leq k\leq k_0+1$, the {\bf $(k+1)$-reduced curvature} map of $\pi:(R,H)\to M$,
\begin{eqnarray*}
\bar c_{k+1}:P^k_H(R)\longrightarrow C^{2,k-1},
\end{eqnarray*}
where $C^{2,k-1}$ is the family of vector spaces over $R$
\begin{eqnarray}\label{c}
C^{2,k-1}:=\frac{\wedge^2T^*\otimes\mathfrak{g}^{(k-1)}}{\partial(T^*\otimes\mathfrak{g}^{(k)})}.
\end{eqnarray}
 Here
 we set $P^0_H(R)=R$, $P_H^{-1}(R)=M$ and $\g^{(-1)}=TR/H$. The definition of reduced curvature map is given in \cite{Gold3} in the setting of partial differential equation, from an algebraic point of view. Similar results regarding them can be also found in \cite{Gold3}.
 
 We begin by constructing
 \begin{eqnarray*}
\bar c_{1}:R\longrightarrow C^{2,-1}=\frac{\wedge^2T^*\otimes TR/H}{\partial_H(T^*\otimes \g)}
\end{eqnarray*} 
using the bundle map over $R$ 
\begin{eqnarray*}
c_1:J^1_HR\longrightarrow \wedge^2T^*\otimes TR/H,\quad j^1_xs\mapsto c(d_xs(\cdot),d_xs(\cdot))
\end{eqnarray*}
whose kernel is precisely $P_H(R)$. Note that if $j^1_xs,j^1_x\al\in J^1_HR$ are such that $s(x)=u(x)$, then the image of the map $\Phi:=d_xs-d_xu$ lies in $\g$, and 
\begin{eqnarray*}\begin{split}
c_1(j^1_xs)(X,&Y)-c_1(j^1_xu)(X,Y)=c(d_xs(X),d_xs(Y))-c(d_xu(X),d_xu(Y))\\&
=c(d_xs(X)-d_xu(X),d_xs(Y))+c(d_xu(X),d_xs(Y)-d_xu(Y))\\&
=\partial_H(\Phi(X))(Y)-\partial_H(\Phi(Y))(X).
\end{split}
\end{eqnarray*}
So the map 
\begin{eqnarray}\label{defi}
R\ni s(x)\mapsto c_1(j^1_xs)\; \mod \partial_H(T^*\otimes \g)
\end{eqnarray}
does not depend on the element $j^1_xs\in J^1_HR$. 

\begin{remark}\rm
Computing the vertical derivative of $c_1$, i.e. the derivative of $c_1$ restricted to $T^*\otimes\mathfrak{g}$ (see remark \ref{remark1}), we see that its value at $\Phi:T_{\pi(p)}M\to \mathfrak{g}_p$ is equal to $\partial_H(\Phi).$
\end{remark}

\begin{definition}\label{curvature-map1}
Let $\pi:(R,H)\to M$ be a Pfaffian fibration. The {\bf $1$-reduced curvature map}
\begin{eqnarray*}
\bar c_1:R \To  C^{2,-1}
\end{eqnarray*}
is a well-defined map given by \eqref{defi}.
\end{definition} 

\begin{lemma}
The sequence
\begin{eqnarray*}
P_H(R)\overset{pr}{\longrightarrow}R\overset{\bar c_1}{\longrightarrow}C^{2,-1}
\end{eqnarray*}
of bundles over $R$ is exact.
\end{lemma}

\begin{proof}
If $j^1_xs\in P_H(R)$, of course $\bar c_1(s(x))=0$ as $c_1(j^1_xs)=0$. On the other hand, if $\bar c_1(s(x))=0$ then there exists $j^1_xs\in J^1_HR$ and $\Phi:T_xM\to \g_{s(x)}$ such that
\begin{eqnarray*}\begin{split}
c(d_xs(X),d_xs(Y))&=\partial_H(\Phi(X))(Y)-\partial_H(\Phi(Y))(X)\\&=c(\Phi(X),d_xs(Y))-c(\Phi(Y),d_xs(X)).
\end{split}
\end{eqnarray*}
Then, as $\g$ is involutive, $\sigma=d_xs-\Phi:T_xM\to T_{s(x)}R$ belongs to $P_H(R)$ and $pr(\sigma)=s(x)$. This shows the exactness of the sequence.
\end{proof}

For $0\leq k\leq k_0$, the map 
\begin{eqnarray*}
\bar c_{k+1}:P^k(H)\longrightarrow C^{2,k-1}
\end{eqnarray*}
is the first reduced curvature map of $(P_H^k(R))$: we consider the curvature map over $P^k_H(R)$
\begin{eqnarray}\label{aca}
c_1(H^{(k)}):J^1_{H^{(k)}}(P^k_H(R))\longrightarrow \wedge^2 T^*\otimes T^\pi P_H^{k-1}(R),
\end{eqnarray} 
where we are using the exact sequence \eqref{sesj3} to identify $T^\pi P_H^{k-1}(R)$ with $T^\pi P^k_H(R)/pr^*\g^{(k)}=TP^k_HR/H^{(k)}$.

\begin{definition}\label{curv(H)} For $1\leq k\leq k_0+1$, we define the $(k+1)${\bf-reduced curvature map}
\begin{eqnarray*}
\bar c_{k+1}:P_H^k(R)\To C^{2,k-1}
\end{eqnarray*}
by 
\begin{eqnarray*}
\bar c_{k+1}(s(x))=c_1(H^{(k)})(j^1_xs)\; \mod \partial(T^*\otimes \g^{(k)}),
\end{eqnarray*}
where $j^1_xs\in J^1_{H^{(k)}}(P^k_H(R)).$ 
\end{definition}
  
  The following lemma is immediate by construction:
  
\begin{lemma}\label{surjectivity} The sequence
\begin{eqnarray*}
P_H^{k+1}(R)\overset{pr}{\longrightarrow}P^k_H(R)\overset{\bar c_{k+1}}{\longrightarrow}C^{2,k-1}
\end{eqnarray*}
of bundles over $R$ is exact.
\end{lemma}

\begin{lemma}\label{lemma-curvature}
For $1\leq k\leq l+1$, the image of $c_1(H^{(k)})$ lies in the family of subspaces 
\begin{eqnarray*}
\ker\left\{\partial:\wedge^2T^*\otimes\mathfrak{g}^{(k-1)}\longrightarrow \wedge^3T^*\otimes\mathfrak{g}^{(k-2)}\right\}.
\end{eqnarray*}
 Hence, $\bar c_{k+1}$ takes values in
\begin{eqnarray*}
H^{2,k-1}(\g):=\frac{\ker\left\{\partial:\wedge^2T^*\otimes\mathfrak{g}^{(k-1)}\to \wedge^3T^*\otimes\mathfrak{g}^{(k-2)}\right\}}{\Im\{\partial:T^*\otimes \g^{(k)}\to\wedge^2T^*\otimes\g^{(k-1)}\}}.
\end{eqnarray*}
\end{lemma}  

To prove the previous lemma we will need the following result which will be used in other contexts:

\begin{lemma}\label{delta-commutes} Let $\theta\in\Omega^1(R,E')$ be a one form of constant rank and let $p:P\to R$ be a submersion, then 
\begin{eqnarray*}
\delta p^*\theta=p^*\delta \theta.
\end{eqnarray*}
Explicitly, for any $X,Y\in dp^{-1}\ker\theta$,
\begin{eqnarray*}
\delta p^*\theta(X,Y)=\delta\theta(dp(X),dp(Y)).
\end{eqnarray*}
\end{lemma}

\begin{proof}
It is a straightforward computation using projectable vector fields.
\end{proof}

\begin{remark}\rm For a distribution $H\subset TR$, the analogous version of lemma \ref{delta-commutes} says that for a submersion $p:P\to R$,
\begin{eqnarray*}
c_H(dp(X),dp(Y))=c_{dp^{-1}H}(X,Y)
\end{eqnarray*}
for $X,Y\in dp^{-1}H.$
\end{remark}

\begin{proof}[Proof of lemma \ref{lemma-curvature}] We will check the case $k=1$. The case $k>1$ follows by an inductive argument. Consider $\theta\in\Omega^1(R,TR/H)$ and $\theta^{(1)}$, the Cartan form on $P_H(R)$. Let's first see that for any $j^1_xs\in J^1_{H^{(1)}}(P_H(R))$ (i.e. $d_xs:T_xM\to T_{b(x)}P_HR$ has its image in $H^{(1)}_{b(x)}$), and any $X,Y\in T_xM$, 
\begin{eqnarray*}
c^1(H^{(k)})(j^1_xs)(X,Y)\in\g.
\end{eqnarray*}
Take $\tilde X,\tilde Y\in \X(P_HR)$ to be extensions of $d_xs(X)$ and $d_xs(Y)$ respectively, then 
\begin{eqnarray*}
\begin{split}
\theta&(c^1(H^{(k)})(j^1_xs)(X,Y))=\theta(\delta\theta^{(1)}(d_xs(X),d_xs(Y)))=\theta(\theta^{(1)}[\tilde X,\tilde Y])=\\&\theta(dpr[\tilde X,\tilde Y])=\delta pr^*\theta(d_xs(X),d_xs(Y))=\delta\theta(dpr(d_xs(X)),dpr(d_xs(Y)))=0,
\end{split}
\end{eqnarray*}
where we used the key property 1 of remark \ref{key-properties} in the third equality, and lemma \ref{delta-commutes} for the last equality.

To prove that $\partial_H(c_1(j^1_xs))(X,Y,Z)=0$ for any $X,Y,Z\in T_xM$, let $\tilde X,\tilde Y,\tilde Z\in \X(J^1_{H^{(1)}}(P_H(R)))$ be projectable vector fields along the submersion \begin{eqnarray*}pr\circ pr^1:J^1_{H^{(1)}}(P_H(R))\To R\end{eqnarray*} 
such that 
\begin{eqnarray*}
&\tilde X_{j^1_xs},\tilde Y_{j^1_xs},\tilde Z_{j^1_xs}\in H^{(2)}_{j^1_xs},\\
&d\pi^2(\tilde X_{j^1_xs})=X,\quad d\pi^2(\tilde Y_{j^1_xs})=Y\quad\text{and}\quad d\pi^2(\tilde Z_{j^1_xs})=Z.
\end{eqnarray*}
The first line means that 
\begin{eqnarray*}
dpr^1(\tilde X_{j^1_xs})=d_xs(X),\quad dpr^1(\tilde Y_{j^1_xs})=d_xs(Y)\quad\text{and}\quad dpr^1(\tilde Z_{j^1_xs})=d_xs(Z).
\end{eqnarray*}
Therefore, for $\tilde s:=pr(s(x))$,
\begin{eqnarray*}\begin{split}
&-\partial_H(c_1(j^1_xs)(Y,Z))(X)=
\delta\theta(dpr\circ dpr^1(\tilde X_{j^1_xs}),\delta\theta^{(1)}(d_xs(Y),d_xs(Z)))\\&=\delta\theta(dpr\circ dpr^1(\tilde X_{j^1_xs}),\theta^{(1)}_{s(x)}(dpr^1[\tilde Y,\tilde Z]))\\&=
\delta\theta(dpr\circ dpr^1(\tilde X_{j^1_xs}),[dpr(dpr^1(\tilde Y)),dpr(dpr^1(\tilde Z))]_{\tilde s}-dpr\circ d_xs(d\pi^1[\tilde Y,\tilde Z])\\
&=\delta\theta(dpr\circ dpr^1(\tilde X_{j^1_xs}),[dpr(dpr^1(\tilde Y)),dpr(dpr^1(\tilde Z))]_{\tilde s})\\&\quad-\delta\theta(dpr\circ dpr^1(\tilde X_{j^1_xs}),dpr\circ d_xs(d\pi^1[\tilde Y,\tilde Z])\\
&=\delta\theta(dpr\circ dpr^1(\tilde X_{j^1_xs}),d_{s(x)}pr[dpr^1(\tilde Y),dpr^1(\tilde Z)]),
\end{split}
\end{eqnarray*}
where in the passage to the last line we use the key property 2 for prolongations given in remark \ref{key-properties}.
Using $\tilde X,\tilde Y,\tilde Z$ for the computations we then have that 
\begin{eqnarray*}\begin{split}
\partial_H(c_1(j^1_xs))(X,Y,Z)&=\delta\theta(dpr\circ dpr^1(\tilde Z_{j^1_xs}),[dpr(dpr^1(\tilde X)),dpr(dpr^1(\tilde Y))]_{\tilde s})\\&\quad+\delta\theta(dpr\circ dpr^1(\tilde X_{j^1_xs}),[dpr(dpr^1(\tilde Y)),dpr(dpr^1(\tilde Z))]_{\tilde s})\\
&\quad+\delta\theta(dpr\circ dpr^1(\tilde Y_{j^1_xs}),[dpr(dpr^1(\tilde Z)),dpr(dpr^1(\tilde X))]_{\tilde s})
\end{split}
\end{eqnarray*}
is zero by the Jacobi identity.
\end{proof}

Now we have all the ingredients to prove theorem \ref{workable-pfaffian-bundles}. This proof is based on the argument given in \cite{Gold2} for theorem 8.1.

\begin{proof}[Proof of theorem \ref{workable-pfaffian-bundles}]$H^{2,l}=0$ for $l\geq 0$ is equivalent to the fact that the sequences in lemma \ref{exact} are exact. Since $\mathfrak{g}^{(1)}$ is a vector bundle, lemma \ref{exact} applies and therefore $\mathfrak{g}^{(l)}$ is a vector bundle over $R$ for $l\geq 1$. We now proceed by induction on $l$. Assume that for $l\geq0$, $pr:P^{m+1}_H(R)\to P^m_H(R)$ is surjective for $0\leq m\leq l$. Since $\mathfrak{g}^{(m+1)}$ are vector bundles, we are under the hypothesis of proposition \ref{affine}, and then we can define the curvature map
\begin{eqnarray*}
  \bar c_{l+2}:P^{l+1}_H(R)\To H^{2,l}.
\end{eqnarray*}
By lemma \ref{surjectivity} for $k=l$, we get that 
\begin{eqnarray*}
P_H^{l+2}(R)\overset{pr}{\To}P_H^{l+1}(R)\overset{\bar c_{l+2}}{\longrightarrow}H^{2,l}
\end{eqnarray*}
is exact. By condition 3 in the statement of the theorem, we have that $\bar c_{l+2}$ is identically zero and therefore $pr^{}:P_H^{l+2}(R)\To P_H^{l+1}(R)$ is surjective.
\end{proof}

\subsection{Abstract resolutions of Pfaffian bundles}

\begin{definition}
Let $\pi:(R,H)\to M$ be a Pfaffian bundle. A {\bf standard resolution of $(R,H)$} is a sequence $\{R_k\mid k\geq 0\}$ of fibered manifolds over $M$ with $R_0=R$, together with immersions $i_k:R_k\to J^1R_{k-1}$ for $k\geq0$, with the following properties:
\begin{itemize}
\item The map $\bar i_k:R_k\to R_{k-1}$ is a surjective submersion for $k\geq 0$, where $\bar i_k$ is the composition
\begin{eqnarray*}
R_k\overset{i_k}{\To}J^1R_{k-1}\overset{pr}{\To}R_{k-1}.
\end{eqnarray*}
\item $i_1(R_1)\subset P_H(R)$ and for $k\geq1$, $i_{k+1}(R_{k+1})\subset P^{k+1}_H(R).$
\end{itemize}
\end{definition}

The following result is a consequence of the Cartan-K\"ahler theorem and theorem 3.2 of \cite{BC}, adapted to our setting of Pfaffian bundles, and will be of use for us in the case of Pfaffian groupoids as we will see later on.

\begin{theorem}\label{resolution}
Let $\pi:(R,H)\to M$ be an analytic Pfaffian bundle. If $(R,H)$ admits a standard resolution, then for every point $p\in R$ there exists real analytic solutions of $(R,H)$ passing through $p$. 
\end{theorem}

\section{Linear Pfaffian bundles ($=$ linear connections)}\label{linear Pfaffian bundles}

Linear Pfaffian bundles are Pfaffian bundles where all objects are linear. One of the main advantages of linear Pfaffian bundles is that not only they are easier to handle, but each of them is endowed with a relative connection canonically associated to the linear distribution (or to the linear form), allowing us to apply the whole theory developed in chapter \ref{Relative connections}.\\  

\subsection{Definitions and basic properties}

The aim of this section is to establish the usual duality between distributions and forms in the linear case. We will prove:

\begin{theorem}\label{linear-one-form}Let $\pi:F\to M$ and $E'\to M$ be vector bundles. For any regular linear 1-form $\theta\in\Omega^1(M,\pi^*E')$
\begin{eqnarray*}
H_\theta:=\ker\theta\subset TF
\end{eqnarray*}
is a linear distribution. Conversely, any linear distribution arises in this way.
\end{theorem}

We start by explaining the definitions.
Recall that the vector bundle structure of $\pi:F\to M$ induces a vector bundle structure on $d\pi:TF\to TM$; its structure maps are the differentials of the structure maps of $F$. 

\begin{definition}\label{def: linear distributions}
A {\bf linear distribution} $H$ on $F$ is any distribution $H\subset TF$ which is also a subbundle of $TF\to TM$ (with the same base $TM$).
\end{definition}

\begin{lemma}\label{linear-transversal}
Linear distributions are Pfaffian distributions in the sense of definition \ref{def: pfaffian distributions}. More precisely, if $H\subset TF$ is linear then it is $\pi$-transversal and $\pi$-involutive.
\end{lemma}

\begin{proof}
The fact that $H$ is a subbundle of $TF\to TM$ implies that $H$ is $\pi$-transversal. To see this let's first show that $H^\pi$ is of constant rank equal to $\rank H-\dim M$. Indeed, as $dz(TM)\subset H$, where $z:M\to F$ is the zero-section, one has that 
\begin{eqnarray*}
\dim H^\pi_{z(x)}=\rank H-n
\end{eqnarray*}
is independent of the point $x\in M$, where $n=\dim M$. Moreover, for $e\in F_x$, the map
\begin{eqnarray*}
T^\pi_{z(x)}F\ni U\mapsto d_{(z(x),e)}a(U,0)\in T_e^\pi F
\end{eqnarray*}
is an isomorphism sending $H^\pi_{z(x)}$ onto $H^\pi_e$, where $a:F\times_MF\to F$ is the addition and $0\in T_{z(x)}F$ is the zero vector, since $H$ is closed under the differential of the addition. Thus, $H^\pi$ has constant rank. On the other hand,
\[\begin{aligned}
\dim (H_e+T^\pi_eF)&=\dim H_e+\dim T_e^\pi F-\dim H^\pi_e\\&=\rank H+\rank F-(\rank H-n)=\dim T_eF,
\end{aligned}\]
which shows that $H$ is transversal to the fibers of $F$. To show that $H$ is $\pi$-involutive notice that $\Gamma(\g)$ is generated as a $C^{\infty}(F)$-module by vector fields $X\in\Gamma(\g)$ of $F$, constant along the fibers of $\pi$. Since the Lie bracket of any such two vector fields $X$ and $Y$ is zero (as they are tangent and constant along the fibers of $\pi$), then for $g,f\in C^{\infty}(F)$
\begin{eqnarray*}
[gX,fY]=gL_X(f)Y-fL_Y(g)X\in\Gamma(\g)
\end{eqnarray*}
which completes the proof of our claim.
\end{proof}

The linearity of $H$ implies that the isomorphism of vector bundles $T:T^\pi F\to \pi^*F$ given by translation of vector tangent to the fibers of $\pi$ induces an isomorphism 
\begin{eqnarray*}
\g(H)\simeq \pi^*\g(H)|_M,
\end{eqnarray*}
where $\g(H)=H^\pi$. Hence, $T$ descends to the quotient to an isomorphism
of $TF/H\simeq T^\pi F/H^\pi$ with the pullback via $\pi$ of the vector bundle over $M$ given by 
\begin{eqnarray*}
E:=T^\pi F/H^\pi|_M.
\end{eqnarray*} 
Therefore the associated Pfaffian form $\theta_H$ (see \ref{quotient-form}) becomes a one form with values on the pullback bundle $\pi^*E$. In fact, $\theta_H$ is a regular linear form in the following sense:

\begin{definition}\label{lin-form} Let $E'$ be a vector bundle over $M$.
 A  {\bf linear form $\theta\in\Omega^k(F,\pi^*E')$} is a differential form such that
\begin{eqnarray*}
a^*\theta=pr_1^*\theta+pr_2^*\theta,
\end{eqnarray*}  
where $a,pr_1,pr_2:F\times_\pi F\to F$ are the fiber-wise addition, the projection to the first component and the projection to the second component respectively.
\end{definition}

\begin{remark}\label{points}\rm
If we regard a one form $\theta\in \Omega^1(F,\pi^*E')$ as a vector bundle map
\begin{eqnarray}\label{sum}
\theta:TF\To E'
\end{eqnarray}
over the map $\pi:F\to M$,
saying that $\theta$ is linear is the same as requiring that \eqref{sum} is also a vector bundle map over $TM\to M$, where the vector bundle structure of $d\pi:TF\to TM$ is the one given by the differentials of the structural maps of the vector bundle $\pi:F\to E'.$
Note that for a linear form $\theta$, $\theta(dz(X))=0$ where $X\in\X(M)$ and $z:M\to F$ is the zero section, as $dz(X)=da(dz(X),dz(X)).$
\end{remark}

\begin{proof}[Proof of theorem \ref{linear-one-form}]This fact follows from lemmas \ref{from-theta-H} and \ref{lemma-from-H-to-theta} regarding $F$ as a Lie groupoid with multiplication given by fiber-wise addition.
\end{proof}

\subsection{Equivalence between linear Pfaffian bundles and relative connections}

Another point of view of linear Pfaffian bundles is that of relative connections. More precisely, 

\begin{theorem}\label{cor: linear one forms}Let $\pi:F\to M$ and $E\to M$ be two vector bundle maps. There is a one to one correspondence between
\begin{enumerate}
\item point-wise surjective linear one forms $\theta\in\Omega^1(F,\pi^*E)$, and 
\item relative connections $(D,l):F\to E$
\end{enumerate}
In this correspondence
\begin{eqnarray}\label{for}
D(s)=s^*\theta\quad\quad\text{and}\quad\quad l(v)=\theta(v)
\end{eqnarray}
where $s\in\Gamma(F)$ and $v\in F_x\simeq T^\pi_{z(x)}F$. 
\end{theorem}

\begin{proof}
As remarked in \ref{ex: linear forms} a linear form $\theta$ is a multiplicative form on the Lie groupoid $F\to M$ (multiplication here is given by fiber-wise addition) with values in the trivial representation $E$ of $F$. Applying theorem \ref{t1} we get a one to one correspondence between linear one forms and $E$-valued Spencer operators of order $1$. Note that in this case the Lie algebroid of $F\to M$ is itself with trivial Lie bracket and zero anchor, and therefore the Spencer operator associated to a regular linear one form is nothing more than a relative connection.
Formulas \eqref{for} follow from remark \ref{points}, equations \eqref{eq: explicit formula}, and the fact that in this case
\begin{eqnarray*}
d\phi^\epsilon_s(\cdot)=da(dz(\cdot),\epsilon ds(\cdot)),
\end{eqnarray*}
as $\phi^\epsilon_s(x)=z(x)+\epsilon s(x)$.
\end{proof}

An analogous result regarding multiplicative forms and relative connections is the following:

\begin{theorem}\label{t:linear distributions}
Let $F\to M$ be a vector bundle. There is a one to one correspondence between
\begin{enumerate}
\item linear distributions $H\subset TF$,
\item subbundles $\g\subset F$ together with connections relative to the quotient map $F\to F/\g$.
\end{enumerate}
In this correspondence, $\g$ is the symbol space of $H$ and 
\begin{eqnarray*}
D_Xs(x)=[\tilde X,\tilde s]_x\; \mod H^\pi_{z(x)},
\end{eqnarray*}
where $\tilde X\in\Gamma(H)\subset \X(F)$ is any vector field $\pi$-projectable to $X$ and extending $dz(X)$, and $\tilde s\in\Gamma(T^\pi F)\subset\X (F)$ is the extension of $s$ to the vector field constant on each fiber of $F$.
\end{theorem}

The previous result is a special instance of theorem \ref{t2}.

\begin{proof}
We regard $\pi:F\to M$ as the Lie groupoid with source and target maps equal to $\pi$, and multiplication given by fiberwise addition. In this case, the Lie algebroid of $F$ is $\pi:F\to M$ itself with trivial Lie bracket and zero anchor.
In this sense, a linear distribution $H\subset TF$ is the same as a multiplicative distribution of the Lie groupoid $F$ in the sense of definition \ref{def-pf-syst}. Moreover, regarding $F$ as the Lie algebroid with trivial bracket and zero anchor, a Spencer operator on $F$ relative to a subbundle $\g\subset F$ in the sense of definition \ref{def1}, is the same as a connection relative to the quotient map $F\to F/\g$, since conditions \eqref{horizontal} and \eqref{horizontal2} are trivially satisfied ($``0=0"$).
In this case, we see then that theorem \ref{t2} translates into corollary \ref{t:linear distributions}.
\end{proof}

The next corollary shows us that the relative connection associated to a linear distribution $H$ is the equal to the one associated to $\theta_H.$ 

\begin{corollary}
Let $D_H$ be the relative connection associated to $H$, and $D_{\theta_H}$ the relative connection associated to $\theta_H$. Then
\begin{eqnarray*}
D_H=D_{\theta_H}.
\end{eqnarray*}
\end{corollary}

\begin{proof}
Let $s\in\Gamma(F)$. We must show that for any $X\in\X(M)$
\begin{eqnarray*}
s^*\theta_H(X)=[\tilde X,\tilde s]|_M\; \mod \g
\end{eqnarray*}
or in other words,
\begin{eqnarray}\label{eq30009}
s^*\theta_H(X)=\theta_H([\tilde X,\tilde s]|_M),
\end{eqnarray}
where $\tilde X\in\Gamma(H)\subset \X(F)$ is any vector field that is $\pi$-projectable to $X$ and extending $dz(X)$, and $\tilde s\in\Gamma(T^\pi F)\subset\X (F)$ is the extension of $s$ to the vector field constant on each fiber of $F$. Equation \eqref{eq30009} follows from lemma \ref{commutator} in the following way: one regards $F$ as a Lie groupoid with multiplication given by the fiber-wise addition (and with trivial representation on $E$), and therefore $\tilde s$ becomes a right invariant vector field with flow equal to 
\begin{eqnarray*}
\varphi_{\tilde s}^\epsilon:F\To F,\quad e_x\mapsto e_x+\epsilon s(x).
\end{eqnarray*}
On the other hand $\tilde X\in\ker\theta$, and therefore 
\begin{eqnarray*}
\begin{split}
s^*\theta_H(X)=\frac{d}{d\epsilon}|_{\epsilon=0}&(\theta_H(dz(X))+\epsilon \theta_H(ds(X)))\\&=\frac{d}{d\epsilon}|_{\epsilon=0}\theta_H(da(dz(X),\epsilon ds(X)))=(L_{s}\theta_H)(dz(X))\\&=([i_{\tilde X},L_s]\theta_H)|_M=\theta_H([\tilde X,\tilde s]|_M).
\end{split}
\end{eqnarray*}
\end{proof}

\subsection{Equivalence between the theory of Linear Pfaffian bundles and that of relative connections}

Theorems \ref{cor: linear one forms} and \ref{t:linear distributions} allow us to make an explicit connection between the theory of relative connections (see chapter \ref{Relative connections}) and that of linear Pfaffian bundles. Here we explain that the two theories are equivalent. First of all, one of the most relevant notions of both theories is that of a solution; as expected the two notions coincide. More precisely, if the relative connection $D$, the linear 1-form $\theta$ and the linear distribution $H$ are all related by theorems \ref{cor: linear one forms} and \ref{t:linear distributions}, then
\begin{eqnarray*}
\Sol(F,D)=\Sol(F,\theta)=\Sol(F,H).
\end{eqnarray*}
This is clear by definition \ref{definitions}. As a direct consequence we get that  
\begin{eqnarray*}J^1_HF=J^1_\theta F=J^1_DF.\end{eqnarray*}
See definition \ref{partial prolongation}.

\begin{example}\label{compa}\rm For an integer $k>0$,
the associated relative connection associated to the Cartan form $\theta^k\in\Omega^1(J^kF,J^{k-1}F)$, or equivalently to the Cartan distribution $C_k\subset J^kF$, is the classical Spencer operator
\begin{eqnarray*}
D^\clas:\Gamma(J^kF)\To \Omega^1(M,J^{k-1}F)
\end{eqnarray*}
relative to the projection $pr:J^kF\to J^{k-1}F$. See subsection \ref{example: jet groupoids}, where $\G=F$ is interpreted as a Lie groupoid with multiplication given by fiber-wise addition and its Lie algebroid is $F$ with trivial Lie bracket and zero anchor.
\end{example}

From the previous example we have the following two results:

\begin{proposition}\label{corazon}
Let $H\subset TF$ be a linear distribution with $D$ and $\theta$ the associated relative connection and linear form respectively. Then
\begin{eqnarray}\label{equation:star}
P_D(F)=P_\theta(F)=P_H(F).
\end{eqnarray}
Moreover, if one of the previous spaces are smoothly defined then
\begin{eqnarray*}
H^{(1)}:=C_1\cap TP_H(F)
\end{eqnarray*}
is a linear distribution with associated relative connection $D^{(1)}:\Gamma(P_D(F))\to \Omega^1(M,F)$ and linear form $\theta^{(1)}=\theta^1|_{P_\theta(F)}$.
\end{proposition}

\begin{proof}
Equality \eqref{equation:star} follows from the next lemma since $P_D(F)=\ker\varkappa_D$ and $P_H(F)=P_\theta(F)=\ker c_1$. 

That $\theta^{(1)}$ is linear follows from its definition. Since $H^{(1)}=\ker\theta^{(1)}$, the linearity of $H^{(1)}$ follows. As for the correspondence of $H^{(1)}$ with $D^{(1)}$, recall that $D^{(1)}$ is the restriction to $P_D(F)$ of the classical Spencer operator $D^\clas:\Gamma(J^1F)\to\Omega^1(M,F)$. The correspondence then is clear from example \ref{compa}.
\end{proof}

\begin{remark}\label{high}\rm As higher prolongations are defined inductively, the previous proposition implies that the same holds for the $k$-classical prolongation of $D$, $\theta$ and $H$, whenever they are smooth. This is, under smoothness assumptions,
\begin{eqnarray*}
P_D^k(F)=P_\theta^k(F)=P_H^k(F)
\end{eqnarray*}
and the associated relative connection of $H^{(k)}$ (or rather of $\theta^{(k)}$) is $D^{(k)}$. See corollary \ref{k-prolongations}.
\end{remark}

\begin{lemma}\label{curv-map}
The curvature map of $\theta$ (or rather of $H$)
\begin{eqnarray*}
c_1(\theta):J^1_\theta F\To \pi^*(\wedge^2T^*\otimes E)
\end{eqnarray*}
coincides with the curvature map of $D$
\begin{eqnarray*}
\varkappa_D:J^1_DF\to\wedge^2T^*\otimes E.
\end{eqnarray*}
Explicitly, for any $j^1_xs\in J^1_\theta F=J^1_DF,$ and $X,Y\in T_xM$
\begin{eqnarray*}
c_1(j^1_xs)(X,Y)=\varkappa_D(X,Y)\in E_x.
\end{eqnarray*}
\end{lemma}

\begin{proof} This follows from lemma \ref{lema2}, as 
\begin{eqnarray*}
d_\nabla\theta=\delta\theta
\end{eqnarray*}
when restricting $d_\nabla\theta$ to the $H$, the kernel of $\theta$.
\end{proof}

\begin{remark}\label{lin-pf}\rm
\mbox{}
\begin{itemize}
\item The symbol map $\partial_H:\g\to\pi^*(T^*\otimes E)$ of $H$ is the pullback of the symbol map $\partial_D:\g_M\to T^*\otimes E$, i.e. for $v\in\g_p=(\g_M)_{\pi(p)}$ and $X\in T_{\pi(p)}M$
\begin{eqnarray*}
\partial_H(v)(X)=\partial_D(v)(X)\in E_{\pi(p)}.
\end{eqnarray*}
This follows directly from the definition of $\partial_H$, and the definition of $D$ in terms of $H$.
\item We have isomorphisms of vector bundles over $F$
\begin{eqnarray*}
\g^{(k)}(F,H)\simeq \pi^*\g^{(k)}(F,D),
\end{eqnarray*}
where $\g^{(k)}(F,H)\subset\pi^*(S^kT^*\otimes E)$ and $\g^{(k)}(F,D)\subset S^kT^*\otimes E$ are the $k$-prolongation of the symbol maps $\partial_H$ and $\partial_D$ respectively.
\item The pullback of the cohomology groups of the $\partial_D$-Spencer cohomology are isomorphic as bundles over $F$ with the ones of the $\partial_H$- Spencer cohomology:
\begin{eqnarray*}
H^{(p,k)}(\g(F,H))\simeq \pi^*H^{(p,k)}(\g(F,D)).
\end{eqnarray*}
See definition \ref{exten}
\item From lemma \ref{curv-map} and remark \ref{high} it follows that the higher curvatures of $D$ and $H$, whenever they are well-defined, are equal. That is 
\begin{eqnarray*}
\kappa^{k+1}(D)=\bar c^{k+1}(H)
\end{eqnarray*}
for $\kappa^{k+1}(D):P_D^k(R)\to C^{2,k-1}$ and $\bar c^{k+1}(H):P_H^k(R)\to C^{2,k-1}$. See definitions \ref{curv(D)} and \ref{curv(H)}.
\end{itemize}
\end{remark}
  
\section{Linearization along solutions}  

\subsection{Linearization of distributions}\label{linearization of distributions}

Let $\pi:R\to M$ be a surjective submersion and let $H\subset TR$ be a $\pi$-transversal distribution with $\g:=H\cap T^\pi R$ the symbol of $H$. See \cite{Cullen, Hermann} for related work.\\

Out of a solution $\sigma:M\to R$ of $(R,H)$ one constructs a canonical linear Pfaffian bundle over $M$, denoted by $(L_\sigma(R),H^\sigma)$, in the following way. Let 
\begin{eqnarray*}
L_\sigma(R):=\sigma^*T^\pi R
\end{eqnarray*} 
together with the connection
\begin{eqnarray*}
D^\sigma:\Gamma(L_\sigma(R))\To\Omega^1(M,E^\sigma)
\end{eqnarray*}
relative to the projection map 
\begin{eqnarray*}
p^\sigma:L_\sigma(R)\To \sigma^*(T^\pi R/\g)=:E^\sigma
\end{eqnarray*}
given by the formula
\begin{eqnarray}\label{check}
D^\sigma_X(s):=[X^\sigma,s^\sigma]|_{\sigma(M)}\; \mod H,
\end{eqnarray}
where $X^\sigma\in\Gamma(H)$ is any $\pi$-projectable vector field to $X$ and extending $d\sigma(X)$, and $s^\sigma\in\Gamma(T^\pi R)$ is any vertical extension of $s\in\Gamma(L_\sigma(R))$.

\begin{definition} The {\bf linearization along $\sigma$} is the relative connection (or the associated linear Pfaffian bundle: see theorem \ref{t:linear distributions})
\begin{eqnarray*}
(D^\sigma,p^\sigma):L_\sigma(R)\To E^\sigma.
\end{eqnarray*}
\end{definition}

\begin{lemma}
The formula \eqref{check} is well-defined and it defines a connection relative to the projection $L_\sigma(R)\to E^\sigma$.
\end{lemma}

\begin{proof}
Let's first check that $D^\sigma_X(s)$ does not depend on the choice of $s^\sigma\in\Gamma(T^\pi R)$. For this it suffices to verify that if $s=0$, then $D^\sigma_X(s)=0$. Without loss of generality let $s^\sigma\in\Gamma(T^\pi R)$ be of the form
\begin{eqnarray*}
s^\sigma=f\beta
\end{eqnarray*}
with $\beta\in\Gamma(T^\pi R)$ and $f\in C^{\infty}(R)$ is such that $f|_{\sigma(M)}\equiv0$. Then,
\begin{eqnarray*}
\begin{split}
[X^\sigma,f\beta]|_{\sigma(M)}&=f|_{\sigma(M)}[X^\sigma,\beta]|_{\sigma(M)}+L_{X^\sigma}(f)|_{\sigma(M)}\beta|_{\sigma(M)}\\&=L_{d\sigma(X)}(f|_{\sigma(M)})\beta|_{\sigma(M)}=0,
\end{split}
\end{eqnarray*}
where in the last equation we used that $X^{\sigma}|_{\sigma(M)}=d\sigma(X)\in\X(\sigma(M))$.

Let's now show that $D^\sigma_X(s)$ does not depend on $X^\sigma$. For this it is enough to check $D^\sigma_X(s)=0$ whenever $X=0$. Again, without loss of generality let $X^\sigma$ be of the form
\begin{eqnarray*}
X^{\sigma}=f\beta
\end{eqnarray*}
with $\beta\in\Gamma(\g)$ and $f\in C^{\infty}(R)$ is such that $f|_{\sigma(M)}\equiv0$. Then,
\begin{eqnarray*}
\begin{split}
[f\beta,s^\sigma]|_{\sigma(M)}&=f|_{\sigma(M)}[\beta,s^\sigma]|_{\sigma(M)}-L_{s^\sigma}(f)\beta|_{\sigma(M)}
\\&=-L_{s^\sigma}(f)\beta|_{\sigma(M)}=0\; \mod H.
\end{split}
\end{eqnarray*}

To show that $D^\sigma(s)$ is indeed a form (i.e. it is $C^{\infty}(M)$-linear) and that $D^\sigma$ satisfies the Leibniz identity relative the projection $L_\sigma(R)\to E^\sigma$ follows from the Leibniz identity on vector field of $R$. This is left to the reader.
\end{proof}

 From the construction the symbol space $\g_\sigma$ of $D^\sigma$ (see definition \ref{definitions}) satisfies the following properties.
 
\begin{proposition} Let $H\subset TR$ be a $\pi$-transversal distribution and let $\sigma:M\to R$ be a solution of $(R,H)$. 
\begin{itemize}
\item The symbol space $\g_\sigma$ of $D^\sigma$ is given by the pullback
\begin{eqnarray*}
\g_\sigma=\sigma^*\g\simeq \g|_{\sigma(M)}.
\end{eqnarray*} 
\item If $H$ is $\pi$-involutive, the symbol map $\partial_{D^\sigma}:\g_\sigma\to T*\otimes E^\sigma$  of $D^\sigma$ is 
\begin{eqnarray*}
\partial_{D^\sigma}=\partial_H|_{\sigma(M)},
\end{eqnarray*}
where 
$\partial_H:\g\to T^*\otimes E$ is the symbol map of $H$.
\item If $H$ is $\pi$-involutive,
\begin{eqnarray*}
\g_\sigma^{(k)}=\sigma^*\g^{(k)},
\end{eqnarray*}
where $\g_\sigma^{(k)}\subset S^kT^*\otimes E^\sigma$ and $\g^{(k)}\subset \pi^*(S^kT^*)\otimes E$ are the $k$-prolongations of the symbol maps $\partial_{D^\sigma}$ and $\partial_{H}$ respectively.
\end{itemize}
\end{proposition} 

\begin{remark}\rm One can define
\begin{eqnarray*}
L_\xi(R,H)
\end{eqnarray*}
for any $\xi \in \Gamma(J^1_HR)$ so that for $\xi=j^1\sigma$, $\sigma\in \Sol(R,H)$
\begin{eqnarray*}
L_{\sigma}(R,H)=L_{j^1\sigma}(R,H).
\end{eqnarray*}
Many of the properties of $L_\sigma(R,H)$ will extend to $L_\xi$ provided $\xi\in \Gamma(J^1_HR)$. For the precise definition, write $\sigma_\xi\in\Gamma(R)$ as the projection of $\xi$, so $\xi$ is viewed as a map
\begin{eqnarray*}
\xi_x:T_xM\To T_{\sigma_\xi(x)}R.
\end{eqnarray*}
See remark \ref{1-when working with jets}. The condition that $\xi\in \Gamma(J^1_HR)$ means that $\xi_x$ takes values in $H_{\sigma_\xi(x)}$. With this, $D^\xi$ is defined exactly like before, just that, this time, $X^\xi\in\Gamma(H)$ is required to extend $\xi(X)$.\\
{\bf Conclusion:} To compute $\g^{(k)}(R,H)_r$, $r\in R$, choose any $\xi\in J^1_HR$ around $x:=\pi(r)$, with $\sigma_\xi(x)=r$. Consider the linearization of $(L_\xi(R),D^\xi)$, and compute $\g^{(k)}(D^\xi)_x$. The result is:
\begin{eqnarray*}
\g^{(k)}(D^\xi)_x=\g^{(k)}(R,H)_r.
\end{eqnarray*}
\end{remark}

\subsection{Heuristics of the linearization of one forms}

Using one forms $\theta\in\Omega^1(R,E)$ we have an alternative description of the linearization of $(R,\theta)$ along a solution $\sigma\in\Sol(R,\theta)$, which will give us more insight into the linearization along solutions and will allow us to treat the slightly more general case of a non-surjective one form $\theta$. We will give an idea of the procedure without proofs. \\

The infinite dimensional picture realizes $\Sol(R,\theta)$ as zeros of a section:
\begin{itemize}
\item the base manifold is $\mathcal{M}=\Gamma(R)$,
\item the vector bundle $\mathcal{E}$ over $\mathcal{M}$ has fiber over $\sigma\in\mathcal{M}$:
\begin{eqnarray*}
\mathcal{E}_\sigma:=\Omega^1(M,\sigma^*E),
\end{eqnarray*}
\item the section 
\begin{eqnarray*}
Eq:\mathcal{M}\To\mathcal{E},\quad\sigma\mapsto \sigma^*\theta.
\end{eqnarray*}
\end{itemize}
With this,
\begin{eqnarray*}
\Sol(R,\theta):=\{\sigma:Eq(\sigma)=0\}.
\end{eqnarray*}
So, naturally, the linearization of $(R,\theta)$ along a solution $\sigma$ would be the usual vertical differential of a section at a zero:
\begin{eqnarray*}
d^{\text{v}}_\sigma Eq:T_\sigma\mathcal{M}\To \mathcal{E}_\sigma.
\end{eqnarray*}
Now:
\begin{eqnarray}\label{star}
T_\sigma\mathcal{M}\simeq \Gamma(\sigma^*T^\pi R).
\end{eqnarray}
Hence, the linearization appears as an operator
\begin{eqnarray*}
d^{\text{v}}_\sigma Eq:\Gamma(\sigma^*T^\pi R)\To\Omega^1(M,\sigma^*E)
\end{eqnarray*}
which actually will be a relative connection on $L_\sigma(R):=\sigma^*T^\pi R$ with values on $E^\sigma:=\sigma^*E$. For explicit formulas we have to be more precise with the identifications (e.g. with \eqref{star}).
First of all, a section $s\in\Gamma(L_\sigma(R))$ is given by the derivative at $t=0$ of a variation of $\sigma$. This is a family $\sigma^t\in\Gamma(R)$ with the property that $\sigma^0=\sigma$ and such that for any $x\in M$
\begin{eqnarray*}
\frac{d}{dt}\sigma^t(x)|_{t=0}=s(x)\in T^\pi_{\sigma(x)}R.
\end{eqnarray*} 
The associated connection $D^\sigma:\Gamma(L_\sigma(R))\to \Omega^1(M,E^\sigma)$ relative to the not necessarily surjective vector bundle $l^\sigma:L_\sigma(R)\to E^\sigma,U_{\sigma(x)}\mapsto\theta_{\sigma(x)}(U)$ is defined at $D^\sigma_X(s)(x)$ by the projection to the second component of 
\begin{eqnarray}\label{w-d}
\frac{d}{dt}\theta\circ d_x\sigma^t(X)\in T_{z(\sigma(x))}E\simeq T_{\sigma(x)}R\oplus E_{\sigma(x)}.
\end{eqnarray}
Note that in \eqref{w-d} we are using that $\sigma$ is a solution of $(R,\theta)$ as $\theta\circ d\sigma=z,$ $z:R\to E$ being the zero-section of $E$.\\

 To see some properties of this construction let's examine it further. The derivative at $t=0$ of the variation 
\begin{eqnarray*}
j^1\sigma^t:M\To J^1R
\end{eqnarray*}
of $j^1\sigma\in\Sol(J^1R,\theta^1)$, where $\theta^1$ is the Cartan form is, by construction, a section $\phi:M\to L_{j^1\sigma}(J^1R)$ of the vector bundle over $M$ given by $L_{j^1\sigma}(J^1R):=(j^1\sigma)^*T^\pi(J^1R)$. The section $\phi$ can be identified with the first jet of the section
\begin{eqnarray*}
\frac{d}{dt}\sigma^t|_{t=0}:M\To L_\sigma(R).
\end{eqnarray*}
Hence, we have an identification
\begin{eqnarray}\label{jetes}
L_{j^1\sigma}(J^1R)\simeq J^1(L_\sigma(R))
\end{eqnarray}
of vector bundles over $M$. Moreover, under this identification, one has:
\begin{itemize}
\item The linearization along $j^1\sigma$ of the bundle map over $R$
\begin{eqnarray*}
J^1R\To T^*\otimes E,\quad j^1_xb\mapsto \theta\circ d_xb
\end{eqnarray*}
is equal to the vector bundle map over $M$ given by
\begin{eqnarray*}
J^1(L_\sigma(R))\To T^*\otimes E^\sigma,\quad j^1_xs\mapsto D^\sigma(s)(x).
\end{eqnarray*}
\item The partial prolongation $J^1_{D^\sigma}(L_\sigma(R))$ of $(L_\sigma(R),D^\sigma)$ has the property that 
\begin{eqnarray*}
J^1_{D^\sigma}(L_\sigma(R))\simeq L_{j^1\sigma}(J^1_\theta R).
\end{eqnarray*}
\item The curvature map of $D^\sigma$
\begin{eqnarray*}
\varkappa_{D^\sigma}:J^1_{D^\sigma}(L_\sigma(R))\To \wedge^2T^*\otimes E^\sigma
\end{eqnarray*}
is equal to the linearization along $j^1\sigma$ of the curvature map
\begin{eqnarray*}
c_1(\theta):J^1_\theta R\To \wedge^2T^*\otimes E.
\end{eqnarray*}
\item The relative connection $D^{j^1\sigma}$ obtained by the linearization of $\theta^1$ along $j^1\sigma$ is such that 
\begin{eqnarray*}
D^{j^1\sigma}=D^\clas,
\end{eqnarray*}
where $D^\clas:\Gamma(J^1(L_\sigma(R)))\to\Omega^1(M,L_\sigma(R))$ is the Spencer operator associated to the vector bundle $L_\sigma(R)\to M$ (see definition \ref{Spencer operator}).
\end{itemize}
From the previous discussion we want to state the following lemma, again without proof.

\begin{lemma}Let $\theta\in\Omega^1(R,E)$ be a one form and let $\sigma\in\Gamma(R)$ be a solution of $(R,\theta)$. The identification \eqref{jetes} restricts to an identification
\begin{eqnarray*}
P_{D^\sigma}(L_\sigma(R))\simeq L_{j^1\sigma}(P_\theta(R))
\end{eqnarray*}
of vector bundles over $M$. Moreover, if these spaces are smooth then 
\begin{eqnarray*}
D^{(1)}=D^{j^1\sigma},
\end{eqnarray*}
where $D^{(1)}:\Gamma(P_{D^\sigma}(L_\sigma(R)))\to \Omega^1(M,L_\sigma(R))$ is the restriction of the Spencer operator associated to $L_{\sigma}(R)$, and $D^{j^1\sigma}:\Gamma(L_{j^1\sigma}(P_\theta(R)))\to (M,L_\sigma(R))$ is the linearization along $j^1\sigma$ of $\theta^{(1)}\in \Omega^1(P_\theta(R),T^\pi R).$
\end{lemma}

\chapter{The integrability theorem for multiplicative forms}\label{The integrability theorem for multiplicative forms}
\pagestyle{fancy}
\fancyhead[CE]{Chapter \ref{The integrability theorem for multiplicative forms}} 
\fancyhead[CO]{The integrability theorem for multiplicative forms}

In the process of understanding Pfaffian bundles in the setting of Lie groupoids, we are brought back to the question of understanding the linearization of multiplicative forms with coefficients and the corresponding integrability problem.
We prove an integrability result, in the spirit of Lie, which allows us to pass from the (more interesting) global picture to the (easier-to-handle) linear picture. 
As an outcome we find that the associated infinitesimal data are certain ``connection-like operators", which we call `Spencer operators', and which are the analogue of relative connections in the context of Lie algebroids. 
The main example is the classical Cartan system (see example \ref{difeo}) on the groupoids of jets of diffeomorphisms
\cite{Cartan1904, Kuranishi, GuilleminSternberg:deformation, Olver:MC}. On the infinitesimal side we recover the classical Spencer operator \cite{Spencer, KumperaSpencer, Gold3,  GuilleminSternberg:deformation}. See subsection \ref{sp-dc}. Hence, using 
Lie groupoids, we learn that the classical Cartan forms and the classical Spencer operators are the same thing, modulo the Lie functor.


Working with forms of arbitrary degree is natural from Cartan's point of view (exterior differential systems). However, the main reason for us to allow general forms is the fact that
multiplicative $2$-forms are central to Poisson and related geometries. Moreover, while multiplicative $2$-forms with trivial
coefficients (!) are well understood, the question of passing from trivial to arbitrary coefficients has been around for a while. 
Surprisingly enough, even the statement of an integrability theorem for multiplicative forms with non-trivial coefficients was completely unclear.
This shows, we believe, that even the case of trivial coefficients was still hiding phenomena that we did not understand. The fact that our work related to Pfaffian groupoids clarifies this point
came as a nice surprise to us and, looking back, we can now say what was missing from the existing picture in multiplicative $2$-forms: Spencer operators.\\

And here are a few connections with the existing literature. On one hand, there is the literature related to Poisson geometry. Symplectic groupoids were discovered as the global counterparts to Poisson manifolds, thought of as infinitesimal (i.e. Lie algebroids) objects \cite{Weinstein}, while Ping Xu realised the relevance of the multiplicativity condition \cite{Xu}. Multiplicative $0$-forms (1-cocycles) were studied in the context of quantization \cite{WeinsteinXu} (see also our subsection \ref{1-cocyles and relations to the van Est map}). 
Motivated by Dirac geometry and the theory of Lie group-valued moment maps, the case of closed multiplicative $2$-forms was analyzed in \cite{BCWZ}. The case of multiplicative $1$-forms (not necessarily closed) appeared in \cite{prequantization} in the context of pre-quantization. General multiplicative forms, i.e. which are not necessarily closed and are of arbitrary degree (but still with trivial coefficients!) were understood in \cite{BursztynCabrera, CrainicArias}. \\

The work in this chapter is largely based on the preprint \cite{Maria}.

\section{Multiplicative forms of degree $k$}

\begin{definition}\label{def-mult-form}\rm
Given a Lie groupoid $\G$ and a representation $E$ of $\G$, an \textbf{$E$-valued multiplicative $k$-form on $\G$} is any form $\theta \in \Omega^k(\G, t^{\ast}E)$ satisfying
\begin{equation}\label{eq: multiplicative}
(m^{\ast}\theta)_{(g,h)} = \pr_1^{\ast}\theta + g\cdot(\pr_2^{\ast}\theta)
\end{equation}
for all $(g,h) \in \G_2$, where $\pr_1,\pr_2: \G_2 \to \G$ denote the canonical projections. 
\end{definition}
 
 See subsection \ref{Lie groupoids} for the definition of representation.

\begin{example}[Linear forms; the classical linear Cartan form]\label{ex: linear forms} \rm A vector bundle $\pi:F \to M$ can be seen as a Lie groupoid with multiplication given by fiberwise addition (a bundle of abelian groups). In this case, any vector bundle $E$ over $M$ can be seen as a trivial representation of $F$ ($e\cdot f= f$). In this case, a multiplicative form $\theta\in \Omega^k(F, \pi^{\ast}E)$ is called a linear form (see definition \ref{lin-form}). Thus, if $a: F\times_{\pi}F \to F$ denotes the addition of $F$, then $\theta$ is linear if and only if
\[a^{\ast}\theta = \pr_1^{\ast}\theta + \pr_2^{\ast} \theta.\]

An example is the linear Cartan 1-form associated to a vector bundle $E$, 
\[ \theta\in \Omega^1(\Jet^1E, E).\] 
Here $F= \Jet^1E\to M$ is the first jet bundle of $E$. The 1-form $\theta$ is the Cartan form given in \ref{PDEs; the Cartan forms:} taking $R=E$ a vector bundle, and using the canonical identification of $T^\pi E$ with $E$.
\end{example}
 
An important examples are classical Cartan forms on jet groupoids. However, they will be treated separately in chapter \ref{Pfaffian groupoids}.

\begin{example}[Cohomologically trivial forms] \rm \label{ex: form on base}
Any form $\omega \in \Omega^k(M,E)$ induces a multiplicative form $\delta(\omega)\in \Omega^k(\G, t^{\ast}E)$:
\[\delta(\omega)_g = g\cdot s^{\ast}\omega - t^{\ast}\omega .\]
Forms of this type will be called cohomologically trivial (for indications of the terminology, see also subsection \ref{1-cocyles and relations to the van Est map}). 


Note that the classical Cartan form $\theta\in \Omega^1(\Pi^1(M), TM)$ is of this type: it is $\delta(\iota)$, where $\iota\in \Omega^1(M, TM)$ is the identity of $TM$. For higher jets however, $\theta\in \Omega^k(\Pi^k(M), J^{k-1}TM)$ is not cohomologically trivial. 
\end{example}

\begin{remark}[Multiplicativity and bisections] \label{bisections}\rm \ Here is a point which, at least implicitly, is at the heart of our approach to multiplicative forms: recall form definition \ref{definition-bisections} that the set of bisections $\Bis(\G)$ forms a group. 

Given a multiplicative form $\theta \in \Omega^k(\G, t^{\ast}E)$, one can talk about $\theta$-holonomic bisections of $\G$, i.e. those with the property that
$b^{\ast}\theta= 0$. The multiplicativity of $\theta$ ensures that the set $\textrm{Bis}(\G,\theta)$ of $\theta$-holonomic bisections is a subgroup of $\textrm{Bis}(\G)$.
\end{remark}

\section{Spencer operators of degree $k$}
\label{Spencer operators of degree $k$}

Passing to the infinitesimal side, let
 $E$ be a representation of a Lie algebroid $A$ with $A$-connection 
 \begin{eqnarray*}
 \nabla:\Gamma(A)\times \Gamma(E)\To \Gamma(E).
 \end{eqnarray*}
 See subsection \ref{Lie algebroids}.
 Each $\alpha\in \Gamma(A)$ induces a Lie derivative operator $\Lie_{\alpha}$ acting on $\Omega^k(M, E)$, which acts as $\Lie_{\rho(\alpha)}$ on forms and as $\nabla_{\alpha}$ on $\Gamma(E)$:
\[(\Lie_{\al}\omega)(X_1, \ldots, X_k) = \nabla_{\al}(\omega(X_1, \ldots, X_k)) - \sum_i\omega(X_1, \ldots, [\rho(\al),X_i], \ldots , X_k).\]

\begin{definition}\label{def-Spencer-oprt}\rm
Given a Lie algebroid $A$ over $M$ and a representation $E$ of $A$, an \textbf{$E$-valued $k$-Spencer operator on $A$} is a linear operator
\[ D: \Gamma(A) \To \Omega^k(M, E)\]
together with a vector bundle map 
\[ l:A \To \wedge^{k-1}T^{\ast}M\otimes E,\]
called the \textbf{symbol} of the Spencer operator, satisfying the Leibniz identity 
\begin{equation}\label{eq: Leibniz identity}
D(f\al) = fD(\al) + \d f\wedge l(\al),
\end{equation}
and the compatibility conditions:
\begin{eqnarray} 
 & D([\al,\be] ) = \Lie_{\al}D(\be) - \Lie_{\be}D(\al) \label{eq: compatibility-1}\\
 & l([\al,\be]) = \Lie_{\al}l(\be) - i_{\rho(\be)}D(\al) \label{eq: compatibility-2}\\
 & i_{\rho(\al)}l(\be) = -i_{\rho(\be)}l(\al), \label{eq: compatibility-3} 
\end{eqnarray}
for all $\al,\be \in \Gamma(A)$, and $f \in \C(M)$.
\end{definition}

\begin{remark} \rm  When we are not in the ``special case'' $k= \textrm{dim}(M)+ 1$, the entire information is contained in $D$ and we only have to require 
(\ref{eq: compatibility-1}) and (\ref{eq: compatibility-3}). Indeed, in this case $l$ will be unique and (\ref{eq: compatibility-2}) follows from (\ref{eq: compatibility-1}) and Leibniz identities (plug in the first equation $f\beta$ instead of $\beta$ and expand using Leibniz). The fact that one has to adopt the previous definition so that it includes the ``special case'' $k= \textrm{dim}(M)+ 1$ is unfortunate because this case is not important for our
main motivating purpose (when $k= 1$). Keeping all these in mind, we will simply say that $D$ is a Spencer operator and that $l$ is the symbol of $D$.
\end{remark}

Again, a large class of examples is the classical Spencer operator on the $k$-jet bundles of a vector bundle. This was treated in example \ref{Spencer tower}.

\begin{remark}\label{rk-convenient}  \rm 
Note that a general $E$-valued Spencer operator of degree $1$ as in the previous definition can be encoded/interpreted in a vector bundle map
\[ j_{D}: A\To \Jet^1E .\]
in the same way as in remark \ref{rk-convenient'}.

Again, $D$ itself can be recovered as the composition $D^{\textrm{clas}}\circ j_{D}$. For the classical Spencer operator, it corresponds to $j_{D^{\textrm{clas}}}= \textrm{Id}$. 
\end{remark}

\begin{example}\label{ex: form on base-2}\rm Here is the infinitesimal analogue of the cohomologically trivial forms of example \ref{ex: form on base}: for any algebroid $A$ and any representation $E$ of $A$, any form $\omega\in \Omega^k(M, E)$ induces an $E$-valued Spencer operator by 
 \[D(\al) = \Lie_{\al}\omega, \quad l(\al) = i_{\rho(\al)}\omega .\]
 \end{example}

\section{The integrability theorem for multiplicative forms; Theorem 1}
\label{The Lie functor}

\begin{theorem}\label{t1}
Let $E$ be a representation of a Lie groupoid $\G$ and let $A$ be the Lie algebroid of $\G$. Then any multiplicative form $\theta \in \Omega^k(\G, t^{\ast}E)$ induces an
$E$-valued Spencer operator $D_{\theta}$ of order $k$ on $A$, given by 
\begin{equation}\label{eq: explicit formula}\small{
\left\{\begin{aligned}
D_{\theta}(\al)_x(X_1, \ldots, X_k) &= \frac{\d}{\d \eps}\Big|_{\eps = 0} \phi^{\eps}_{\al}(x)^{-1}\cdot\theta((\d \phi^{\eps}_{\al})_x(X_1), \ldots, (\d \phi^{\eps}_{\al})_x(X_k)), \\ \\l_{\theta}(\al) &= u^{\ast}(i_{\al}\theta) .\end{aligned}\right.}
\end{equation} 
where $\phi_\al^\epsilon:M\to R$ is the flow of $\alpha$ (see remark \ref{flows of sections}).

If $\G$ is $s$-simply connected, then this construction defines a 1-1 correspondence between  $E$-valued $k$-forms on $\G$ and 
$E$-valued $k$-Spencer operators on $A$.  
\end{theorem}

 \begin{remark}\label{linear Spencer}\rm 
Let us look again at the case when $A= F$ is a Lie algebroid with trivial bracket and anchor; then theorem \ref{t1} gives a one-to-one correspondence between linear forms $\theta \in \Omega^k(F,\pi^{\ast}E)$ and operators $D: \Gamma(F) \to \Omega^k(M, E)$, with a symbol map $l:F \to \wedge^{k-1}T^{\ast}M\otimes E$, satisfying Leibniz (all compatibility conditions are automatically satisfied). 

This actually indicates a possible strategy, in the spirit of \cite{BursztynCabrera}, but which we will not follow here, to prove theorem \ref{t1}: given $\theta\in\Omega^k(\G, t^{\ast}E)$, first ``linearize'' $\theta$ to a linear form $\theta_{0}\in \Omega^k(A, t^{\ast}E)$ then consider the associated Spencer operator $D$, carefully book-keeping all the equations involved. 
\end{remark}

\subsection{The case $k=0$: 1-cocycles and the van Est map}\label{1-cocyles and relations to the van Est map}

Another interesting case of the main theorem is $k= 0$, when we recover the van Est map relating differentiable and algebroid cohomology \cite{WeinsteinXu, Crainic:Van}, in degree $1$. Since this case will be used later on and also in order to fix the terminology, we discuss it separately here. In particular, we will provide a simple direct argument, based on  Lie's II theorem \cite{MoerdijkMrcun, Mackenzie:General} which says that, if $\G$ is $s$-simply connected and $\H$ is any Lie groupoid, then for any Lie algebroid morphism $\varphi: A(\G) \to A(\H)$, there exists a unique  Lie groupoid morphisms $\Phi: \G \to \H$ such that $\varphi = \d \Phi |_{A(\G)}$.

Let $\G$ be a Lie groupoid over $M$, and $E$ representation of $\G$. We will denote by $\G^{(p)}$ the space of strings of $p$ composable arrows on $\G$, and by $t: \G^{(p)} \to M$ the map which associates to $(g_1, \ldots, g_p)$ the point $t(g_1)$. 
A differentiable $p$-cochain on $\G$ with values in $E$ is a smooth section $c: \G^{(p)} \to t^{\ast}E$. We denote the space of all such cochains by  $C^p(\G,E)$. The differential $\delta: C^p(\G,E) \to C^{p+1}(\G,E)$ is
\[\begin{aligned}
\delta c (g_1, \ldots, g_{p+1}&) = g_1c(g_2, \ldots, g_{p+1}) +\\
&+\sum_{i=1}^p(-1)^i c(g_1, \ldots , g_ig_{i+1}, \ldots ,g_{p+1}) +(-1)^{p+1}c(g_1, \ldots, c_p).
\end{aligned}\]
For each $p$ we consider the space of $p$-cocycles
\[Z^p(\G,E) = \ker (\delta: C^p(\G,E) \to C^{p+1}(\G,E)).\]
We recognize multiplicative forms $c\in \Omega^0(\G, E)$ as elements of $Z^1(G,E)$. 

At the infinitesimal level, given a representation $\nabla$ of $A$ on $E$, one defines the de Rham cohomology of $A$ with coefficients in $E$ as the cohomology of the complex $(C^{\ast}(A,E), d)$, where $C^{p}(A,E) = \Gamma(\wedge^pA^{\ast}\otimes E)$ and 
\[\begin{aligned}
d\omega(\al_0, \ldots, \al_p) =  \sum_{i}&(-1)^i\nabla_{\al_i}\omega(\al_0, \ldots, \widehat{\al_i},\ldots, \al_p) +\\
&+\sum_{i<j}(-1)^{i+j}\omega([\al_i, \al_j], \al_0, \ldots, \widehat{\al_i}, \ldots, \widehat{\al_j}, \ldots, \al_p).
\end{aligned}\]
As above, the space of $p$-cocycles is denoted by $Z^p(A ,E)$ and 
we recognize the $E$-valued $0$-Spencer operators as $1$-cocycles on $\A$ with values in $E$.

The van Est map is a map of cochain complexes
\[\ve: C^{\ast}(\G, E) \To C^{\ast}(A,E),\]
which induces an isomorphism in cohomology under certain connectedness conditions on the $s$-fibers \cite{WeinsteinXu, Crainic:Van}. We will concentrate on degree $1$. In this case, for $c\in C^{1}(\G, A)$, $\ve(c)$ is given by: 
\[\vartheta_x(c)(\al) = \frac{\d}{\d \eps}\big{|}_{\eps = 0}g_{\eps}^{-1}c(g_{\eps}),\]
where $\al \in A_x$, and $g_{\eps}$ is any curve in $s^{-1}(x)$ such that $g_0 = 1_x$, $\frac{\d}{\d \eps}\big{|}_{\eps = 0}g_{\eps} = \al$. 
Using $g_{\epsilon}= \phi^{\epsilon}_{\alpha}(x)$, we recognize our formula for the Spencer operator from theorem \ref{t1}. Hence our main theorem gives:


\begin{proposition}\label{prop: Van Est}
If $\G$ is $s$-simply connected, then the Van Est map induces an isomorphism
\[\ve: Z^1(\G, E) \To \Omega^1_{\mathrm{cl}}(A,E),\]
where $\Omega_{\mathrm{cl}}^1(A,E)$ denotes the closed elements of $\Omega^1(A,E)$.
\end{proposition}

\begin{proof}
Consider the semi-direct product $\G\ltimes E \tto M$, whose space of arrows consists of pairs 
$(g,v) \in \G \times E$ with $t(g) = \pi(v)$ and 
\[ s(g,v) = s(g), \quad t(g,v) = t(g) = \pi(v), \quad (g,v)(h,w) = (gh, v+gw) .\] 
The fact that $c \in Z^1(\G,E)$ is a cocycle condition is equivalent to the fact that 
\[ \tilde{c}:= (\textrm{Id}, c): \G\To \G\ltimes E\]
is a morphism of groupoids. The Lie algebroid of $\G\ltimes E$ is $A\ltimes E = A\oplus E$ with 
\[\rho(\al,s) = \rho(\al),\quad\text{and}\quad [(\al,s),(\tilde \al,s')] = ([\al,\tilde \al], \nabla_{\al}s' - \nabla_{\tilde \al}s).\]
As before, $\omega \in C^1(A,E)$ is a cocycle if and only if  $\tilde{\omega}:= (\textrm{Id}, \omega) : A\To A \ltimes E$
is a Lie algebroid morphism. It is easy to see that, for $c \in Z^1(\G,E)$, the Lie functor applied to $\tilde{c}$ is $\widetilde{\ve(c)}$; hence the result follows from Lie II.
\end{proof}


\begin{remark}\label{along-s}\rm \ Sometimes it is more natural to consider cochains along $s$, i.e. sections of the pull-back of $E$ via 
$s: \G^{(p)} \to M$, $(g_1, \ldots, g_p)\mapsto s(g_p)$. In degree $1$ one deals with $c\in \Gamma(\G, s^{\ast}E)$ and the 
cocycle condition becomes
\[ c(gh)= h^{-1}\cdot c(g)+ c(h) .\]
Of course, one can just pass to cocycles in the previous sense by considering 
\[ \overline{c}(g)= g^{-1}\cdot c(g).\]
Note that the associated algebroid cocycle is simply $\ve(\overline{c})(\alpha_x)= (dc)_x(\alpha_x)$.
\end{remark}

\begin{remark}[cocycles as representations]\label{remark-cocycles}\rm 
Yet another interpretation of $1$-cocycles is obtained when $E= \mathbb{R}$ is the trivial representation of $\G$ (and of $A$).
On the groupoid side, any 1-cocycle $c\in C^{1}(\G)$ induces a representation, denoted $\mathbb{R}_c$, of $\G$ on the trivial line bundle:
\[ g\cdot t:= e^{c(g)}t .\]
When the $s$-fibers of $\G$ are connected, it is not difficult to see that this gives a 1-1 correspondence. Similarly, one has a 
1-1 correspondence between $1$-cocycles on $A$ and structures of representations of $A$ on the trivial line bundle; given $a\in Z^1(A)$,
the corresponding representation, denoted $\mathbb{R}_{a}$,  is determined by 
\[ \nabla_{\al}(1)= a(\al) .\]

It is clear that, in this case, the van Est map (and proposition \ref{prop: Van Est}) becomes the Lie functor between representations of $\G$ and those of $A$.
\end{remark}

For later use, we give the following: 

\begin{proposition}\label{prop: appendix}
Let $v \in T_g\G$ be a vector tangent to the $t$-fiber of $\G$, and let $c \in Z^1(\G,E)$ be a cocycle. Then
\[\d_gc(v) = g\cdot \d_{s(g)} c(\d_gL_{g^{-1}}v).\]
\end{proposition}

\begin{proof}
Let $\gamma: I \to \G$ be a curve in the $t$-fiber through $g$ whose velocity at $\eps = 0$ is $v$. Since $c$ is a cocycle, it follows that
\[c(g^{-1}\cdot\gamma(\eps)) = g^{-1}\cdot c(\gamma(\eps)) + c(g^{-1}),\]
for all $\eps \in I$. Finally, one differentiates w.r.t. $\eps$.
\end{proof}

\subsection{Example: trivial coefficients}
\label{sec: trivial coefficients}

Another interesting case of the main theorem is when $E$ is the trivial representation. This case was well studied
because of its relevance to Poisson geometry. Our Theorem \ref{t1} recovers the most general results in this context. 
The key remark is that, when $E= \mathbb{R}$ with the trivial action, any bundle map $l:A \to \wedge^{k-1}T^*M$ is canonically the symbol of a $E$-Spencer operator, namely $\al\mapsto d(l(\al))$. We see that any $E$-Spencer operator can be decomposed as
\[ D(\al)= \nu(\al)+ d(l(\al)),\]
where, this time, $\nu$ is tensorial. Rewriting everything in terms of $\nu$ and $\al$ we obtain the result of \cite{CrainicArias, BursztynCabrera}:

\begin{proposition}\label{prop: trivial coefficients}
Let $\G$ be an $s$-simply connected Lie groupoid with Lie algebroid $A$. Then there is a one-to-one correspondence between multiplicative forms $\theta \in \Omega^k(\G)$ and pairs $(\nu, l)$ consisting of vector bundle maps 
\begin{equation}
\left\{\begin{aligned}
\nu:A&\To \wedge^kT^*M\\
l:A&\To \wedge^{k-1}T^*M
\end{aligned}\right.
\end{equation}
satisfying 
\begin{equation}\label{eq: compatibility-tc}\left\{
\begin{aligned}
\nu([\al,\be] ) &= \Lie_{\rho(\al)}\nu(\be) - i_{\rho(\be)}d\nu(\al)\\
i_{\rho(\be)}\nu(\al) &= \Lie_{\rho(\al)}l(\be)  - i_{\rho(\be)}dl(\al)- l([\al,\be])\\
i_{\rho(\al)}l(\be) &= -i_{\rho(\be)}l(\al),
\end{aligned}\right.
\end{equation}
for all $\al,\be \in \Gamma(A)$. The correspondence $\theta\mapsto (\nu_{\theta}, l_{\theta})$ is given explicitly by
\[  \nu_{\theta}(\al)=u^{\ast}(i_{\al}\d\theta), \ \  l_{\theta}(\al)=u^{\ast}(i_\al\theta).\]
\end{proposition}


For $\phi\in \Omega^{k+1}(M)$,
the cohomologically trivial form $\delta(\phi)= s^{\ast}\phi- t^{\ast}\phi$ gives (see examples \ref{ex: form on base} and \ref{ex: form on base-2})
\[ \nu_{\delta(\phi)}(\al)=  -i_{\rho(\al)}(d\phi) ,\qquad   l_{\delta(B)}(\phi)= -i_{\rho(\al)}(\phi),\]
hence one obtains an infinitesimal characterization for cohomological triviality.

Note that in the case of trivial coefficients it makes sense to talk about the DeRham differential of a multiplicative form (itself multiplicative).
From the last formulas in the proposition, we have:
\[ \nu_{d\theta}= 0,\qquad l_{d\theta}= \nu_{\theta},\]
hence one immediately obtains infinitesimal characterizations of multiplicative forms which are closed. More generally, given $\phi\in \Omega^{k+1}(M)$
closed, one says that $\theta\in \Omega^k(\G)$ is $\phi$-closed if $d\theta= s^{\ast}\phi- t^{\ast}\phi$. One obtains for instance the following, which when $k= 2$
gives the main result of \cite{BCWZ}.

\begin{corollary}\label{cor-B-field}
Assume that $\G$ is $s$-simply connected, $\phi\in \Omega^{k+1}(M)$ closed. Then there is a bijection between $\phi$-closed, multiplicative $\theta \in \Omega^k(\G)$ and
\[l: A \To \Lambda^{k-1}T^{\ast}M\]
which are vector bundle maps satisfying
\[\left\{\begin{aligned}
l([\al,\be]) &= \Lie_{\rho(\al)}l(\be) - i_{\rho(\be)}\d l (\al)+ i_{\rho(\al)\wedge \rho(\be)}(B) \\
i_{\rho(\al)} l(\be) &= -i_{\rho(\be)}l(\al).
\end{aligned}\right.\]
\end{corollary}

\begin{example}\label{sympl-gpds-from-t1} \rm We briefly recall here the Poisson case. 
Let $(M, \pi)$ be a Poisson manifold and let $A= T^{\ast}M$ be endowed with the induced algebroid structure (see e.g. \cite{Zung}),
which is assumed to come from an $s$-simply connected Lie groupoid $\Sigma$. 
Then proposition \ref{prop: trivial coefficients} can be applied to $k= 2$, $\phi= 0$ and $l= \textrm{Id}_{T^{\ast}M}$; this gives rise to $\omega\in \Omega^2(\Sigma)$ 
closed and multiplicative, and also non-degenerate (because $l$ is an isomorphism), making $\Sigma$ into a symplectic groupoid. 
\end{example}

\subsection{Proof of Theorem \ref{t1}}
\label{sec: linearization}

\subsubsection{Rough idea and some heuristics behind the proof}
\label{Rough idea and some heuristics behind the proof}

Let us briefly indicate the intuition behind our approach (the pseudogroup point of view). The main idea is to reinterpret $\theta$ in terms of bisections: it gives rise to (and is determined by) a family of $k$-forms $\{ \theta_b: b\in \textrm{Bis}(\G)\}$; here $\theta_b\in \Omega^k(M, E)$ is obtained by pull-backing $\theta$ to $M$ via $b$ and using the action of $\G$ on $E$. In other words, $\theta$ is encoded in the map
\[ \Theta: \textrm{Bis}(\G)\rmap \Omega^k(M, E),\ \ b\mapsto \theta_b ;\]
the multiplicativity of $\theta$ translates into a cocycle condition for $\Theta$ on the group $\textrm{Bis}(\G)$. Hence, morally (because we are in infinite dimensions), the infinitesimal counterpart of $\theta$ is encoded in the linearization $\ve(\Theta)$ of $\Theta$ (as in subsection \ref{1-cocyles and relations to the van Est map}). While $\Gamma(A)$ plays the role of the Lie algebra of $\textrm{Bis}(\G)$ (cf. Remark \ref{flows of sections}), one arrives at $\ve(\Theta)= D$ given in the theorem. However, to prove the theorem, we have to avoid the infinite dimensional problem and work (still in the spirit of Lie pseudogroups) with jet spaces: since $\Theta$ depends only on first order jets of bisections, it can be reinterpreted as a finite dimensional object- a map
\[ c: \Jet^1\G  \longrightarrow \hom(\wedge^kTM, E),\]
which is a $1$-cocycle for the groupoid $\Jet^1\G$. Hence, instead of applying the integration of cocycles (as in Proposition \ref{prop: Van Est}) to the infinite dimensional $\textrm{Bis}(\G)$, we will apply it to the groupoid $\Jet^1\G$.
Here one encounters a small technical problem: $\Jet^1\G$ may have $s$-fibers which are not simply connected (not even connected), so we will have to pass to the closely related groupoid $\widetilde{\Jet^1\G}$, which has the same Lie algebroid $\Jet^1A$, but which is $s$-simply connected:
\[\tilde{c}: \widetilde{\Jet^1\G} \longrightarrow \hom(\wedge^kTM, E).\]
Then one has to concentrate on the linearization cocycle 
\[ \eta: \Jet^1A \longrightarrow \hom(\wedge^kTM, E),\]
which, together with the decomposition (\ref{J-decomposition}) in mind, is precisely the pair $(D, l)$ consisting of the Spencer operator and its symbol. 
The rest is about working out the details and finding out the precise equations that $c$ and $\eta$ have to satisfy. 

Throughout this section, $\G$ denotes a Lie groupoid over $M$, $E$ is a representation of $\G$, $\Jet^1\G$ denotes the Lie groupoid of $1$-jets of bisections of $\G$. Each one of the next subsections is devoted
to one of the 1-1 correspondences
\[ \theta \longleftrightarrow c \longleftrightarrow \tilde{c} \longleftrightarrow \eta .\]

\subsubsection{From multiplicative forms to differentiable cocycles}\label{sec: step 1}

Recall that $\lambda: \Jet^1\G \to \GL(TM)$ denotes the canonical representation of $\Jet^1\G$ on $TM$ and $\Ad: \Jet^1\G \to \GL(A)$ denotes the adjoint representation of $\Jet^1\G$ on the Lie algebroid $A$ of $\G$ (see subsection \ref{Jet groupoids and algebroids}). Combining these two actions, 
\[ \hom(\wedge^kTM, E)\]
becomes a representation of $\Jet^1\G$. Using notation from  subsection \ref{1-cocyles and relations to the van Est map}, we will denoted the associated space of $1$-cocycles by
\[ Z^{1}(\Jet^1\G, \hom(\wedge^kTM, E)).\]
Moreover, as in remark \ref{when working with jets}, we will view the elements of $\Jet^1\G$ as splittings $\sigma_g: T_{s(g)}\to T_g\G$ of $ds$. In particular, for $\sigma_g, \sigma_g'\in \Jet^1\G$ sitting above the same $g\in \G$, $\sigma_g- \sigma_g'$ takes values in $Ker(ds)_g$; identifying the last space with $A_{t(g)}$, we consider the resulting map
\[ \sigma_g \ominus \sigma_g':= R_{g^{-1}}\circ (\sigma_g - \sigma'_g): T_{s(g)}M\To A_{t(g)}.\]

\begin{proposition}\label{prop: multiplicative forms as cocycles}
There is a one-to-one correspondence between multiplicative $k$-forms $\theta \in \Omega^k(\G, t^{\ast}E)$, and pairs $(c, l)$ with 
\[\left\{\begin{aligned}
c\in Z^{1}(\Jet^1\G, \hom(\wedge^kTM, E))\\
l: A \longrightarrow \hom(\wedge^{k-1}TM, E) 
\end{aligned}\right.\]
satisfying the following equations:
\begin{equation}\label{eq: condition 1}\begin{aligned}
c&(\sigma_g)(\lambda_{\sigma_g}v_1, \ldots , \lambda_{\sigma_g}v_k) - c(\sigma_g')(\lambda_{\sigma_g'}v_1, \ldots , \lambda_{\sigma_g'}v_k) = \\
&=\sum_{i = 1}^k(-1)^{i+1}l((\sigma_g \ominus \sigma_g')(v_i))(\lambda_{\sigma_g}v_1,\ldots,\lambda_{\sigma_g}v_{i-1}, \lambda_{\sigma_g'}v_{i+1},\ldots, \lambda_{\sigma_g'}v_k), 
\end{aligned}\end{equation}
\begin{equation}\label{eq: condition 2}
i_{\rho(\al)}l(\be) = -i_{\rho(\be)}l(\al),
\end{equation}
\begin{equation}\label{eq: condition 3}\begin{aligned}
l(\Ad_{\sigma_g}\al)&(\lambda_{\sigma_g}v_1, \ldots , \lambda_{\sigma_g}v_{k-1}) - g\cdot l(\al)(v_1,\ldots,v_{k-1}) = \\
&=c(\sigma_g)(\lambda_{\sigma_g}\rho(\al),\lambda_{\sigma_g}v_1, \ldots , \lambda_{\sigma_g}v_{k-1}), 
\end{aligned}\end{equation}
for all splittings $\sigma_g, \sigma_g'\in \Jet^1\G$, $v_1, \ldots, v_k \in T_{x}M$, and $\al, \be \in A_x$, where $x= s(g)$. 
 \end{proposition}

 For the proof of the proposition we will need the following lemma:
 \begin{lemma}\label{lemma: theta vertical}\rm
For any multiplicative form $\theta \in \Omega^k(\G,t^{\ast}E)$, and any $\al_g \in \ker \d_gs$, we have that
\begin{equation}\label{eq: theta vertical 1}
\theta_g(\al_g, X_1, \ldots, X_{k-1}) = \theta_{t(g)}(\al_{t(g)},\d_g t(X_1), \ldots, \d_g t(X_{k-1})),
\end{equation}
for all $X_1, \ldots, X_{k-1} \in T_g\G$, where $\al_{t(g)} = R_{g^{-1}}(\al_g)$.
 \end{lemma}
 
 \begin{proof}
Notice that we can express $\al_g$ and $X_i$ as 
\[\al_g = \d_{t(g)}R_g(\al_{t(g)}) = \d_{(t(g),g)}m(\al_{t(g)}, 0_g), \  X_i = \d_{(t(g),g)}m(\d_gt(X_i), X_i).\]
Equation \eqref{eq: theta vertical 1} then follows from the multiplicativity equation (\ref{eq: multiplicative}) on the vectors (tangent to $\G_2$) $(\al_{t(g)}, 0_g)$, $(\d_gt(X_1), X_1), \ldots$. 
 \end{proof}

\begin{proof}[Proof of proposition \ref{prop: multiplicative forms as cocycles}]
In one direction, given $\theta \in \Omega^k(\G, t^{\ast}E)$, 
\[c(\sigma_g)(w_1, \ldots, w_k) = \theta_g(\sigma_g(\lambda_{\sigma_g}^{-1}(w_1)), \ldots, \sigma_g(\lambda_{\sigma_g}^{-1}(w_k))), \]
 for all $\sigma_g\in \Jet^1\G$ and $w_1, \ldots, w_k \in T_{t(g)}M$, and 
\[l(\al)(v_1, \ldots, v_{k-1}) = \theta_x(\al, v_1, \ldots, v_{k-1}),\]
for all $\al \in A_x$, and $v_1, \ldots, v_{k-1}\in T_xM$. The desired equations for $(c,l)$ will be proven for $k= 2$, which reveals all the necessary
arguments but keeps the notation simpler. Note that in this case lemma \ref{lemma: theta vertical} translates into
\begin{equation}\label{eq: theta vertical}
\theta_g(\al_g, X) = l(\al_{t(g)})(\d_gt(X)),
\end{equation}
for all $X \in T_g\G$, $\al_g \in \ker \d_gs$, where $\al_{t(g)} = R_{g^{-1}}(\al_g)$. To prove \eqref{eq: condition 1} (for $k= 2$), using the definition of $c$, we find that the hand side of the
equation is 
\[\begin{aligned}
\theta_g(\sigma_g(v_1), & \sigma_g(v_2)) - \theta_g(\sigma_g'(v_1),\sigma_g'(v_2)) = \\
 & =\theta_g(\sigma_g(v_1) - \sigma'_g(v_1),  \sigma_g(v_2)) + \theta_g(\sigma_g'(v_1), \sigma_g(v_2) - \sigma_g'(v_2)).
\end{aligned}\]
Applying \eqref{eq: theta vertical} with $\al_{g}= \sigma_g(v_i) - \sigma_g'(v_i)$, $i\in \{1, 2\}$, we obtain \eqref{eq: condition 1}. The same \eqref{eq: theta vertical}, combined
with skew symmetry of $\theta$, gives \eqref{eq: condition 2}:
\[l(\al)(\rho(\be)) = \theta(\al,\be) = -\theta(\be, \al) = -l(\be)(\rho(\al)).\]

Next we prove \eqref{eq: condition 3}. Using  the formula \eqref{eq: Ad} for the adjoint representation,
\[
l(\Ad_{\sigma_g}\al)(\lambda_{\sigma_g}v) = 
\theta_{t(g)}(R_{g^-1} \circ \d_{(g, s(g))}m(\sigma_g(\rho(\al)), \al),\lambda_{\sigma_g}v).\]
Using again the equation \eqref{eq: theta vertical}, the last expression is 
\[\theta_g(\d_{(g, s(g))}m(\sigma_g(\rho(\al)), \al),\sigma_g(v)).\] 
Combining with $\sigma_g(v) = \d_{(g,s(g))}m(\sigma_g(v), v)$ and then applying the multiplicativity equation for $\theta$, we arrive at the
right hand side of \eqref{eq: condition 3}.

We are left with proving the cocycle equation $\delta c = 0$. Let $(\sigma_g, \sigma_h) \in \Jet^1\G^{(2)}$ be a pair of composable arrows. Then $\delta c(\sigma_g, \sigma_h)$ is the map $\wedge^2T_{t(g)}M \to E_{t(g)}$,
\[ \delta c(\sigma_g, \sigma_h)(w_1, w_2) = c(\sigma_g)(w_1, w_2)  +  g\cdot(c(\sigma_h)(\lambda_{\sigma_g}^{-1}w_1,  \lambda_{\sigma_g}^{-1}w_2)) -  c(\sigma_g\cdot\sigma_h)(w_1, w_2).\]
Let $v_i = \lambda_{\sigma_g}^{-1}w_i \in T_{t(h)}M$. For the sum of the first two terms in the right hand side, after applying the definition of $c$, we find
\[\theta_g(\sigma_g(v_1), \sigma_g(v_2))+  g\cdot\theta_h(\sigma_h(\lambda_{\sigma_h}^{-1}v_1),  \sigma_h(\lambda_{\sigma_h}^{-1}v_2)).\]
For the last term, using the description (\ref{mult-i-J1}) for $\sigma_g\cdot\sigma_h$, we find 
\[ \theta_{gh}(\d_{(g,h)}m(\sigma_g(v_1),\sigma_h(\lambda_{\sigma_h}^{-1}v_1)),  \d_{(g,h)}m(\sigma_g(v_2),\sigma_h(\lambda_{\sigma_h}^{-1}v_2)).\]
Finally, the multiplicativity of $\theta$ implies that the last two expressions coincide. 

For the reverse direction, let $c$ and $l$ be given, and we construct $\theta$. In order to avoid clumsier notation, we extend $l$ to the entire
$\ker ds$: for $\alpha_g\in \ker(ds)_g$:
\[ l(\alpha_g):= l(R_{g^{-1}}\alpha_g).\]
Given $g$, choose $\sigma_g \in \Jet^1\G$ and use it to split a vector $X \in T_g\G$ into
\[X = \sigma_g(v) + \al_X,\]
where $v = \d_gs(X) \in T_{s(g)}M$, and $\al_X = X - \sigma_g(v) \in \ker\d_gs$. 
We define
\begin{eqnarray*} \begin{split}
&\theta_g(X_1, \ldots, X_k) = c(\sigma_g)(\lambda_{\sigma_g}v_1, \ldots \lambda_{\sigma_g}v_k) + \\
&\sum_{\substack{p+q=k \\ \tau \in S(p,q)}}(-1)^{|\tau|}l(\al_{X_{\tau(1)}})(\rho(\al_{X_{\tau(2)}}),\ldots, \rho(\al_{X_{\tau(p)}}), \lambda_{\sigma_g}v_{\tau(p+1)}, \ldots, \lambda_{\sigma_g}v_{\tau(k)}),
\end{split}\end{eqnarray*}
where $p \geq 1$, the second summation is taken over all $(p,q)$-shuffles of $\set{1,\ldots,k}$. The rest of the proof will be given again in the case $k= 2$, when the previous formula becomes
\[\begin{aligned}
\theta_g(X_1, X_2) = c(\sigma_g)&(\lambda_{\sigma_g}v_1,\lambda_{\sigma_g}v_2)  - l(\al_{X_2})(\lambda_{\sigma_g}v_1) +\\
&+ l(\al_{X_1})(\lambda_{\sigma_g}v_2) + l(\al_{X_1})(\rho(\al_{X_2})).
\end{aligned}\]
Note that \eqref{eq: condition 2} implies that $\theta_g$ is skew-symmetric. Let us show that it does not depend on the choice of $\sigma_g$.  
Choose another splitting $\sigma_g'$ and write $\tilde \al_X = X-\sigma_g'(v)$. Let $\theta'_g$ be the form obtained by using the splitting $\sigma_g'$. Then
\[\begin{aligned}
(\theta_g - \theta_g')(X_1, X_2) =  c(&\sigma_g)(\lambda_{\sigma_g}v_1, \lambda_{\sigma_g}v_2) - c(\sigma_g')(\lambda_{\sigma_g'}v_1, \lambda_{\sigma_g'}v_2) +\\
& - l(\al_{X_2})(\lambda_{\sigma_g}v_1) + l(\tilde \al_{X_2})(\lambda_{\sigma_g'}v_1)\\
&+ l(\al_{X_1})(\lambda_{\sigma_g}v_2) - l(\tilde \al_{X_1})(\lambda_{\sigma'_g}v_2) + \\
& + l(\al_{X_1})(\rho(\al_{X_2})) - l(\tilde \al_{X_1})(\rho(\tilde \al_{X_2})).
\end{aligned}\]

Let us denote by $\al_{v_i} = \sigma_g(v_i) - \sigma'_g(v_i)$ and notice that 
\begin{equation}\label{eq: alpha - alpha'}
\al_{X_i} - \tilde \al_{X_i} = - \al_{v_i}, \text{ and } \lambda_{\sigma_g}v_i - \lambda_{\sigma'_g}v_i = \rho(\al_v).
\end{equation} 
By using \eqref{eq: condition 1} and the polarization formula for $\theta$, it follows that
\[\begin{aligned}
(\theta_g - \theta_g')(X_1, X_2) &=  l(\al_{v_1})(\lambda_{\sigma_g}v_2) - l(\al_{v_2})(\lambda_{\sigma_g'}v_1) \\
&\quad + l(\al_{v_2})(\lambda_{\sigma_g}v_1) - l(\tilde \al_{v_2})(\rho(\al_{v_1}))\\
&\quad- l(\al_{v_1})(\lambda_{\sigma_g}v_2) + l(\tilde \al_{X_1})(\rho(\al_{v_2})) \\
&\quad - l(\al_{v_1})(\rho(\al_{X_2})) - l(\tilde \al_{X_1})(\rho(\al_{v_2})).
\end{aligned}\]

Thus, if we substitute into this expression the consequence
\[l(\al_{v_2})(\lambda_{\sigma_g'}v_1) = l(\al_{v_2})(\lambda_{\sigma_g}v_1) - l(\al_{v_2})(\rho(\al_{v_1}))\]
of \eqref{eq: alpha - alpha'}, almost all of the terms cancel out two-by-two and we are left with
\[\begin{aligned}
(\theta_g - \theta_g')(X_1, X_2) &= l(\al_{v_2})(\rho(\al_{v_1})) - l(\tilde \al_{X_2})(\rho(\al_{v_1})) - l(\al_{v_1})\rho(\al_{X_2})\\
&= l(\al_{v_2} - \tilde \al_{X_2})(\rho(\al_{v_1})) - l(\al_{v_1})(\rho(\al_{X_2}))\\
&= -l(\al_{X_2})(\rho(\al_{v_1})) - l(\al_{v_1})(\rho(\al_{X_2})).
\end{aligned}\]
Because of \eqref{eq: condition 2}, this expression vanishes and $\theta$ is well defined. 

We are left with verifying that $\theta$ is multiplicative. Let $g,h \in \G$ be composable arrows, and let $(X_i,Y_i) \in T_{(g,h)}\G^{(2)}$ so that $\d_ht(Y_i) = \d_gs(X_i)$. We fix splittings $\sigma_g,\sigma_h \in \Jet^1\G$ and use them to write
\[X_i = \al_i + \sigma_g(v_i),\quad \text{and}\quad Y_i = \be_i +\sigma_h(w_i).\]
From $(X_i,Y_i) \in T_{(g,h)}\G^{(2)}$ it follows that $v_i = \rho(\be_i) + \lambda_{\sigma_h}(w_i)$, hence
\[X_i = \al_i +\sigma_g(\rho(\beta_i))+ \sigma_g(\lambda_{\sigma_h}w_i).\]
Decomposing
\[\d m(X_i,Y_i) = \d m(\al_i,0) + \d m(\sigma_g(\rho(\beta_i)), \be_i)+ \d m(\sigma_g(\lambda_{\sigma_h}w_i) ,\sigma_h(w_i)).\]
$m^{\ast}\theta_{(g,h)}((X_1,Y_1),  (X_2, Y_2))$ gives six types of terms (where $1\leq i \neq j \leq2$):
\begin{description}
\item[Type 1:]  $\theta_{gh}(\d m(\al_1,0), \d m(\al_2, 0)),$
\item[Type 2:]  $\theta_{gh}(\d m(\al_i,0),\d m(\sigma_g(\rho(\beta_j)), \be_j)),$
\item[Type 3:]   $\theta_{gh}(\d m(\al_i,0),  \d m(\sigma_g(\lambda_{\sigma_h}w_j) ,\sigma_h(w_j))),$
\item[Type 4:]  $\theta_{gh}(\d m(\sigma_g(\rho(\beta_1)), \be_1),\d m(\sigma_g(\rho(\beta_2)), \be_2)),$
\item[Type 5:]    $\theta_{gh}(\d m(\sigma_g(\rho(\beta_i)), \be_i),  \d m(\sigma_g(\lambda_{\sigma_h}w_j) ,\sigma_h(w_j)))$
\item[Type 6:]    $\theta_{gh}(\d m(\sigma_g(\lambda_{\sigma_h}w_1) ,\sigma_h(w_1)),  \d m(\sigma_g(\lambda_{\sigma_h}w_2) ,\sigma_h(w_2)))$
\end{description}

In order to simplify the terms of type 1,2, and 3, we note that \begin{eqnarray*}\d_{(g,h)}m(\al_g,0_h) = R_h(\al_g)\end{eqnarray*} for all $\alpha_g\in \ker(ds)_g$ and thus, by the definition of $\theta$
\[\theta_{gh}(\d m(\al_1,0), \d m(\al_2, 0)) = l(\al_1)(\rho(\al_2)),\]\
\[\theta_{gh}(\d m(\al_i,0),\d m(\sigma_g(\rho(\beta_j)), \be_j)) = l(\al_i)(\lambda_{\sigma_g}\rho(\be_j)), \]\
\[\theta_{gh}(\d m(\al_i,0),  \d m(\sigma_g(\lambda_{\sigma_h}w_j) ,\sigma_h(w_j))) = l(\al_i)(\lambda_{\sigma_g\sigma_h}w_j).\]\

On the other hand, using again the formula \eqref{eq: Ad} for the adjoint action, as well as condition \eqref{eq: condition 3}, we simplify the terms of type 4 and 5 into
\[\begin{aligned}
\theta_{gh}(\d m(\sigma_g(\rho(\beta_1)), \be_1),&\d m(\sigma_g(\rho(\beta_2)), \be_2)) = l(\Ad_{\sigma_g}\be_1)(\lambda_{\sigma_g}(\rho(\be_2))), \\
&=c(\sigma_g)(\lambda_{\sigma_g}\rho(\be_1), \lambda_{\sigma_g}\rho(\be_2)) +g\cdot l(\be_1)(\rho(\be_2))
\end{aligned}\]\
\[\begin{aligned}
\theta_{gh}(\d m(\sigma_g(\rho(\beta_i)), \be_i), & \d m(\sigma_g(\lambda_{\sigma_h}w_j) ,\sigma_h(w_j))) = l(\Ad_{\sigma_g}\be_i)(\lambda_{\sigma_g\sigma_h}w_j) \\
&= c(\sigma_g)(\lambda_{\sigma_g}\rho(\be_i),\lambda_{\sigma_g\sigma_h}w_j) + g\cdot l(\be_i)(\lambda_{\sigma_h}w_j).
\end{aligned}\]

Finally, we use condition \eqref{eq: condition 2} to express
\[\begin{aligned}
\theta_{gh}(\d m(\sigma_g(\lambda_{\sigma_h}w_1) &,\sigma_h(w_1)),  \d m(\sigma_g(\lambda_{\sigma_h}w_2) ,\sigma_h(w_2))) = \\&=\theta_g(\sigma_g(\lambda_{\sigma_h}w_1),\sigma_g(\lambda_{\sigma_h}w_2)) + g\cdot \theta_h(\sigma_h(w_1),\sigma_h(w_2)).
\end{aligned}\]
Adding everything up, we recognize $\theta_g(X_1,X_2) +g\cdot \theta_h(Y_1,Y_2)$, 
thus concluding the proof of the proposition.
\end{proof}

\subsubsection{Realizing source-simply connectedness}\label{sec: step 2}

As mentioned at the start of the section, we need to pass from $\Jet^1\G$ to $\widetilde{\Jet^1\G}$ -- the source simply connected Lie groupoid with the same Lie algebroid $\Jet^1A$ as $\Jet^1\G$. See subsection \ref{Lie functor} for the construction. Recall that it comes equipped with a groupoid map
\[p: \widetilde{\Jet^1\G} \To \Jet^1\G,\]
whose image is the subgroupoid $(\Jet^1\G)^{0}$ made of the connected component of the identities in $\Jet^1\G$. For elements in $\widetilde{\Jet^1\G}$ we will use the notation 
$\sigma_g$ whenever we want to indicate the point $g\in \G$ onto which $\sigma_g$ projects. For $X\in T_{s(g)}\G$ we will use the notation
\[ \sigma_g(X):= p(\sigma_g)(X)\]
and, for $\sigma_g, \sigma_g'\in \widetilde{\Jet^1\G}$, consider
\[ \sigma_{g}\ominus \sigma_g':= p(\sigma_g)\ominus p(\sigma_g'): T_{s(g)}M\To A_{t(g)}.\]
We will use the map $p$ to pull-back structures from $\Jet^1\G$ to $\widetilde{\Jet^1\G}$. 
For instance, any representation of $\Jet^1\G$ can also be seen as a representation of $\widetilde{\Jet^1\G}$
and there is an induced pull-back map at the level of the resulting cocycles. In particular, we will consider 
\begin{equation}
\label{p-ast}
p^{\ast}: Z^1(\Jet^1\G,\hom(\wedge^kTM,E)) \To Z^{1}(\widetilde{\Jet^1\G}, \hom(\wedge^kTM,E)).
\end{equation}
It is clear that the pairs $(c, l)$ of Proposition \ref{prop: multiplicative forms as cocycles} and the equations that they satisfy have an analogue with $\Jet^1\G$ replaced by
$\widetilde{\Jet^1\G}$, giving rise to pairs $(\tilde{c}, l)$ satisfying identical equations.

\begin{proposition}\label{corol: passage to cover}
Let $\G$ be an $s$-simply connected Lie groupoid. Then $(c, l) \mapsto (p^{\ast}(c), l)$ defines a 
1-1 correspondence between pairs $(c, l)$ satisfying the conditions from proposition \ref{prop: multiplicative forms as cocycles}
and pairs $(\tilde{c}, l)$ consisting of 
\[\left\{\begin{aligned}
\tilde{c}\in Z^{1}(\widetilde{\Jet^1\G}, \hom(\wedge^kTM, E))\\
l: A \longrightarrow \hom(\wedge^{k-1}TM, E) 
\end{aligned}\right.\]
satisfying the conditions from proposition \ref{prop: multiplicative forms as cocycles} but with $\Jet^1\G$ replaced by $\widetilde{\Jet^1\G}$. 
\end{proposition}

\begin{proof}
Of course, the statement is about $c\mapsto p^{\ast}c$, i.e. we can fix $l$. We begin by showing that $p^{\ast}$ is injective when restricted to the set of $c$ for which \eqref{eq: condition 1} holds. To do so we first show that $c$  is determined by its value on the the Lie groupoid $(\Jet^1\G)^0$ whose $s$-fibers are the connected component of the identity in the $s$-fibers of $\Jet^1\G$. Observe that for any $g \in \G$, there exists $\sigma_g \in (\Jet^1\G)^0$. In fact, since $(\Jet^1\G)^0 \to \G$ is a submersion, and $\G$ is $s$-connected, we can lift any path in $s^{-1}(s(g))$, starting at the identity and ending at $g$, to a path in $(\Jet^1\G)^0$ starting at the identity and ending over $g$. The corresponding end point is an element $\sigma_g \in (\Jet^1\G)^0$. It follows from \eqref{eq: condition 1} that for any other $\sigma'_g \in \Jet^1\G$, $c(\sigma'_g)$ is determined by $c(\sigma_g)$, and $l$.  

Next, we note that if $p^{\ast}c = p^{\ast}c'$, then $c$ and $c'$ coincide on $(\Jet^1\G)^0$. In fact, for any $\sigma_g \in (\Jet^1\G)^0$ we can find a path inside the $s$-fiber of $\sigma_g$, joining the identity $\d_{s(g)}u$ of $(\Jet^1\G)^0$, and $\sigma_g$. Taking its homotopy class, this path gives rise to an element $\xi_g$ of $\widetilde{\Jet^1\G}$ which projects to $\sigma_g$. But then,
\[c(\sigma_g) = c(p(\xi_g)) = c'(p(\xi_g)) = c'(\sigma_g).\]

Finally, we prove that if $(\tilde{c},l)$ satisfies \eqref{eq: condition 1}, then $\tilde{c}$ lies in the image of $p^{\ast}$. For this, note that if $p(\xi_g) = p(\xi'_g)$, then the actions of $\xi_g$ and $\xi'_g$ on $TM$ coincide. Moreover, they induce the same splittings of $\d_gs$. Thus, the right hand side of \eqref{eq: condition 1} vanishes, which implies that $\tilde{c}(\xi_g) = \tilde{c}(\xi'_g)$. It follows that $\tilde{c}$ induces a map $c: \Jet^1\G \to t^{\ast}\hom(\wedge^kTM,E)$ such that $p^{\ast}c = \tilde{c}$.

Thus, we have just proven that $p^{\ast}$ determines a one-to-one correspondence 
\begin{equation*}
  \left\{
    \begin{split}
      c \in Z^1&(\Jet^1\G,\hom(\wedge^kTM,E)) \\
      &(c,l) \text{ satisfies \eqref{eq: condition 1}}
    \end{split}
  \right\}
  \longleftrightarrow
  \left\{
    \begin{split}
      \tilde{c} \in Z^1&(\widetilde{\Jet^1\G},\hom(\wedge^kTM,E)) \\
      &(\tilde{c},l) \text{ satisfies \eqref{eq: condition 1}}
    \end{split}
  \right\}.
\end{equation*}
A simple verification shows that $(c,l)$ satisfies \eqref{eq: condition 2} and \eqref{eq: condition 3} if and only if $(\tilde{c},l)$ satisfies \eqref{eq: condition 2} and \eqref{eq: condition 3}. This concludes the proof.
\end{proof}

\subsubsection{Passing to algebroid cocycles}\label{sec: algebroid cocycles}

\begin{proposition}\label{prop: algebroid cocycles}
Let $\G$ be $s$-simply connected. Then there is a one-to-one correspondence between pairs $(\tilde{c}, l)$ as in Proposition \ref{corol: passage to cover} 
and pairs $(\eta,l)$ with
\[\left\{\begin{aligned}
\eta \in Z^1(\Jet^1A, \hom(\wedge^kTM, E))\\
l: A \longrightarrow \hom(\wedge^{k-1}TM, E) 
\end{aligned}\right.\]
 satisfying the equations:
\begin{equation}\label{eq: condition 1''}
\eta(\d f \otimes \al) = \d f \wedge l(\al),
\end{equation}
\begin{equation}\label{eq: condition 2''}
i_{\rho(\al)}l(\be) = - i_{\rho(\be)}l(\al),
\end{equation}
\begin{equation}\label{eq: condition 3''}
l([\al,\be]) - \Lie_{\al}l(\be) = i_{\rho(\al)}\eta(\jet^1\be),
\end{equation}
for all $\al, \be \in \Gamma(A)$, and all $f \in \mathrm{C}^{\infty}(M)$.
\end{proposition}

For \eqref{eq: condition 1''} we are using the inclusion $i$ from the exact sequence
\[ 0\To \textrm{Hom}(TM, A)\stackrel{i}{\To} \Jet^1A \stackrel{\pr}{\To} A\To 0\]

\begin{proof}
We use the isomorphism 
\[\ve: Z^1(\widetilde{\Jet^1\G},\hom(\wedge^kTM, E)) \To Z^1(\Jet^1A, \hom(\wedge^kTM, E))\]
induced by the Van Est map (Proposition \ref{prop: Van Est}); of course, $\eta= \ve(\tilde{c})$. 
Fix $(\tilde{c},l)$ and $x\in M$. We prove that \eqref{eq: condition 1} and \eqref{eq: condition 3} for $(\tilde{c},l)$ are equivalent to \eqref{eq: condition 1''} and \eqref{eq: condition 3''} for $(\ve(\tilde{c}),l)$.

We start with the equivalence of \eqref{eq: condition 3} with \eqref{eq: condition 3''}. We interpret $l$ as 
\[ l\in \Gamma(A^{\ast}\otimes\wedge^{k-1}T^{\ast}M\otimes E)= C^0(\widetilde{\Jet^1\G}, A^{\ast}\otimes\wedge^{k-1}T^{\ast}M\otimes E),\]
a zero-cochain on $\widetilde{\Jet^1\G}$. Of course, $\ve(l)= l$, with $l$ interpreted as a $0$-cochain on $\Jet^1A$. 
On the other hand, the anchor $\rho$ induces a morphism of representations
\[\rho^{\ast}: \wedge^kT^{\ast}M\otimes E \To A^{\ast}\otimes\wedge^{k-1}T^{\ast}M\otimes E\] 
hence also a map of complexes
\[\rho^{\ast}: C(\widetilde{\Jet^1\G}, \wedge^kT^{\ast}M\otimes E) \To C(\Jet^1A, A^{\ast}\otimes\wedge^{k-1}T^{\ast}M\otimes E),\]
and similarly on algebroid cohomology, compatible with $\ve$. In particular,
\[ \ve(d(l)- \rho^{\ast}(\tilde{c}))= d(\ve(l))- \rho^{\ast}(\ve(\tilde{c}))= d(l)- \rho^{\ast}(\eta).\]
Finally, note that \eqref{eq: condition 3} is just the explicit form of the equation $d(l)= \rho^{\ast}(\tilde{c})$, while \eqref{eq: condition 3''} is just $d(l)= \rho^{\ast}(\eta)$; 
hence one just uses the injectivity of $\ve$.

We are left with proving the equivalence of \eqref{eq: condition 1} with \eqref{eq: condition 1''}. We fix $x\in M$ and we show that \eqref{eq: condition 1''} is satisfied at $x$ if and only
if \eqref{eq: condition 1} is satisfied for all $g$ that start at $x$. In the sequence
\[ \widetilde{\Jet^1(\G)}\stackrel{p}{\To} \Jet^1(\G) \stackrel{\pr}{\To} \G,\]
we consider the $s$-fibers of the three groupoids above $x$, denoted by
\[ \tilde{P}\To P\To B .\]
Both $P$ and $\tilde{P}$ become principal bundles over $B$, with structure groups 
\[ K= \pr^{-1}(1_x),\ \hat{K}= (\pr\circ p)^{-1}(1_x),\]
respectively (the action is the one induced by the groupoid multiplication). Note also that
the map $p: \tilde{P}\to P$ has as image the connected component $P^0$ of $P$ containing the unit at $x$
and $p: \tilde{P}\to P^0$ is a covering projection. 
The following is immediate. 

\begin{lemma}\label{lemma: K}
$\hat{K}$ is connected. 
\end{lemma}

Assume now that $(\ve(\tilde{c}),l)$ satisfies \eqref{eq: condition 1} for all $g\in B$. Let $\eta = \ve(\tilde{c})$ and $\d f \otimes \al \in T^{\ast}M\otimes A$. Using $p : \tilde{P} \to P$ to identify a neighborhood of the identity in $\tilde{P}$, with a neighborhood of the identity in $P$, we can view, for $\eps$ small enough,
\[\gamma_x(\eps) = \d_xu + \eps(\d_x f \otimes \al_x)\]
as a path in $\tilde{P}$ such that 
\[\gamma(0) = \d_x u, \quad \frac{\d}{\d\eps}\big{|}_{\eps = 0} \gamma_x = \d_x f \otimes \al_x.\] 

Since for each $\eps$, $\gamma_x(\eps)$ acts trivially on $E$, it follows that 
\begin{equation}\label{eq: around the identity}
\eta(\d_x f \otimes \al_x) = \frac{\d}{\d\eps}\big{|}_{\eps = 0} \tilde{c}(\gamma_x(\eps)).
\end{equation}
However, since $\tilde{c}$ is a cocycle, it follows that $\tilde{c}(\d_xu) = 0$, thus \eqref{eq: condition 1} implies that
\[\tilde{c}(\gamma_x(\eps))(v_1, \ldots, v_k) = \eps(\d_xf\wedge l(\al_x))(v_1,\ldots, v_k), \]
for all $\eps$ and all $v_i \in T_xM$. By differentiating at $\eps = 0$ one obtains \eqref{eq: condition 1''} at $x$.

We now prove the converse, namely that \eqref{eq: condition 1''} at $x$ implies \eqref{eq: condition 1} at all $g\in B$. Fix $g$ and let $\xi_g$ and $\xi_g'$ be two elements of $\widetilde{\Jet^1\G}$ which lie over $g$. Let $\gamma_1$ be any path in $\tilde{P}$ joining $\d_xu$ to $\xi_g$, and let $\gamma_2$ be any path fiber of $\tilde{P} \to B$ joining $\xi_g$ to $\xi'_g$, which exists because of Lemma \ref{lemma: K}. Furthermore, we may assume that
\[\gamma_1(\eps) = \xi_g \text{ for all } \frac{1}{2}\leq\eps\leq 1,\ \ \gamma_2(\eps) = \xi_g \text{ for all } 0\leq\eps\leq\frac{1}{2}.\]\\
Thus, we obtain two smooth paths
\[\gamma_{\xi_g}(\eps) = \left\{\begin{aligned}&\gamma_1(2\eps) &\text{ if } 0 \leq \eps \leq\frac{1}{2}\\ 
&\xi_g &\text{ if } \frac{1}{2}\leq \eps\leq 1,
\end{aligned}\right. ,\ \ \gamma_{\xi_g'}(\eps) = \left\{\begin{aligned}&\gamma_1(2\eps) &\text{ if } 0 \leq \eps \leq\frac{1}{2}\\ 
&\gamma_2(2\eps - 1) &\text{ if } \frac{1}{2}\leq \eps\leq 1
\end{aligned}\right.\] 

\vskip0.5ex
Finally, we consider the path $f: [0,1] \to E$ given by
\[\begin{aligned}
&f(\eps) = \tilde{c}(\gamma_{\xi_g}(\eps))(\lambda_{\gamma_{\xi_g}(\eps)}v_1, \ldots , \lambda_{\gamma_{\xi_g}(\eps)}v_k) - \tilde{c}(\gamma_{\xi_g'}(\eps))(\lambda_{\gamma_{\xi_g'}(\eps)}v_1, \ldots , \lambda_{\gamma_{\xi_g'}(\eps)}v_k) + \\
&-\sum_{i = 1}^k(-1)^{i+1}l((\gamma_{\xi_g} \ominus \gamma_{\xi_g'})(\eps)(v_i))(\lambda_{\gamma_{\xi_g}(\eps)}v_1,\ldots,\lambda_{\gamma_{\xi_g}(\eps)}v_{i-1}, \lambda_{\gamma_{\xi_g'}(\eps)}v_{i+1},\ldots, ) 
,\end{aligned} \]
where $v_i \in T_{s(g)}M$. We must show that $f(\eps)$ is contained in the zero section $E$ for all $\eps \in [0,1]$. It is clear that this is true for $\eps \in [0, 1/2]$. On the other hand, for $\eps \in [1/2, 1]$, the path $f(\eps)$ lies inside the fiber $E_{t(g)}$. Thus, it suffices to show that the derivative of $f$ at $\eps$ vanishes for all $\eps \in [1/2,1]$. However, by proposition \ref{prop: appendix} this is reduced to the computation \eqref{eq: around the identity} performed at the identity $\d_{s(g)}u$, which vanishes by virtue of \eqref{eq: condition 1''}.
\end{proof}

\begin{proof}[End of proof of Theorem \ref{t1}] 
We put propositions \ref{prop: multiplicative forms as cocycles}, \ref{corol: passage to cover} and \ref{prop: algebroid cocycles} together. Of course, we recognize the $\eta$'s from the last proposition as the bundle maps $j_{D}$ associated to Spencer operators (cf. remark \ref{rk-convenient}); hence the relation between $\eta$ and $D$ is:
\[ \eta(\jet^1\al)=: D(\alpha) .\]
Using that 
\[\Lie_{\al}(D(\be)) = \nabla_{\jet^1\al}D(\be),\]
it is immediate that the equations on $(\eta, l)$ from the previous proposition are equivalent to the equations that ensure that $(D,l)$ is a Spencer operator.
\end{proof}

\chapter{Spencer Operators}\label{Relative connections on Lie algebroids}
\pagestyle{fancy}
\fancyhead[CE]{Chapter \ref{Relative connections on Lie algebroids}} 
\fancyhead[CO]{Spencer Operators} 

Spencer operators are the infinitesimal counterpart of point-wise surjective multiplicative 1-forms (theorem \ref{t1} in the case $k=1$). They also arise as the linearization of Pfaffian groupoids along the unit map and therefore they are the natural notion of relative connection in the setting of Lie algebroids (see conditions \eqref{horizontal} and \eqref{horizontal2}). It is remarkable that the Pfaffian groupoid itself can be recovered from its Spencer operator (theorem \ref{t2}). Of course, the main example is the classical Spencer operator associated to a Lie algebroid (see \ref{higher jets on algebroids}). We will show that in this setting all the notions developed in chapter \ref{Relative connections} become Lie theoretic. \\

Throughout this chapter $A$ is a Lie algebroid with Lie bracket denoted by $[\cdot,\cdot]$ and anchor map $\rho:A\to TM$, $(D,l):A\to E$ is a relative connection with $\g\subset A$ the symbol space of $D$. For ease of notation, $T$ and $T^*$ will denote the tangent and cotangent bundle of $M$, respectively.

We will use of the notions given on subsection \ref{Lie algebroids} and the Lie algebroid structure on $J^1A$ described in subsection \ref{Jet groupoids and algebroids}. 

\section{Spencer operators}

In this section we study relative connection compatible with the Lie algebroid structures. 

\subsection{Definitions and basic properties}

Let $A$ be a Lie algebroid over $M$, $E$ a vector bundle over $M$, and $l:A\to E$ a surjective vector bundle map.

\begin{definition}\label{def1} An {Spencer operator} relative to $l$ is an $l$-connection 
\begin{eqnarray*}
D:\Gamma(A)\To \Omega^1(M,E)
\end{eqnarray*} 
satisfying the following two compatibility conditions 
\begin{eqnarray}\label{horizontal}
D_{\rho(\beta)}\alpha=-l[\alpha, \beta],
\end{eqnarray} 
\begin{eqnarray}
\label{horizontal2}
D_X[\alpha,\alpha']=\nabla_\alpha(D_X\alpha')-D_{[\rho(\alpha),X]}\alpha'-\nabla_{\alpha'}(D_X\alpha)+D_{[\rho(\alpha'),X]}\alpha\end{eqnarray}
for all $\alpha, \tilde \al\in \Gamma(A)$, $\be\in \Gamma(\B)$, $X\in\mathfrak{X}(M)$, $f\in C^{\infty}(M)$ and where
\begin{equation}\label{coneccion}
\nabla:\Gamma(A)\times \Gamma(E)\to \Gamma(E), \quad \nabla_\alpha(l(\alpha'))=D_{\rho(\alpha')}\alpha+l[\alpha,\alpha'].
\end{equation}
\end{definition}


\begin{remark}\rm
Condition \eqref{horizontal} implies that $\nabla$ given by \eqref{coneccion} is indeed well-defined. 
\end{remark}

\begin{lemma}\label{action-lema}
$\nabla$ makes $E$ into a representation of $A$. 
\end{lemma}

\begin{proof}
We need to show that 
\begin{eqnarray}\label{flat}
\nabla_\al\nabla_\be(l(\tilde \al))-\nabla_\be\nabla_\al(l(\tilde \al))=\nabla_{[\al,\be]}(l(\tilde \al))
\end{eqnarray}
for any $\al,\be,\tilde \al\in\Gamma(A)$. This will follow from the compatibility condition \eqref{horizontal2} applied to $X=\rho(\tilde \al)$. On the one hand, one has that 
\begin{eqnarray*}
\nabla_{[\al,\beta]}(l(\tilde \al))=D_{\rho(\tilde \al)}[\al,\beta]+l[[\al,\be],\tilde \al].
\end{eqnarray*}
On the other,
\begin{eqnarray*}
\nabla_\al\nabla_\be(l(\tilde \al))=\nabla_\al(D_{\rho(\tilde \al)}\be)+D_{\rho[\be,\tilde \al]}\al+l[\al,[\be,\tilde \al]]
\end{eqnarray*}
and
\begin{eqnarray*}
\nabla_\be\nabla_\al(l(\tilde \al))=\nabla_\be(D_{\rho(\tilde \al)}\al)+D_{\rho[\al,\tilde \al]}\be+l[\be,[\al,\tilde \al]].
\end{eqnarray*}
Using the Jacobi identity for $[[\al,\be],\tilde \al]$ and the fact that $\rho$ commutes with the Lie bracket, replacing the above equations into \eqref{flat} one has that \eqref{flat} becomes the compatibility condition \eqref{horizontal2} for $X=\rho(\tilde \al).$ 
\end{proof}

\begin{remark} \label{remark-dual}\rm We see that definition \ref{def1} is just a reformulation of the notion of $1$-Spencer operator of section \ref{Spencer operators of degree $k$} in the case when the symbol map is surjective.
Indeed \eqref{eq: compatibility-2} becomes \eqref{coneccion} and also implies \eqref{horizontal} from the previous definition;
also, \eqref{horizontal2} is just \eqref{eq: compatibility-1} for $k= 1$. 
\end{remark}

\begin{remark}\rm In analogy with theorems \ref{t:linear distributions} and \ref{cor: linear one forms}, there is a 1-1 correspondence between Spencer operators $D$ on $A$ and 
\begin{itemize}
\item distributions $H\subset TA$ with the property that $H\to TM$ is a Lie subalgebroid of the tangent algebroid $TA\to TM$, and
\item point-wise surjective linear $1$-forms $\theta\in\Omega^1(A,E)$, with the property that the map
\begin{eqnarray*}
\theta:TA\To A\ltimes E, \ \ \ T_aA\ni v\mapsto(a,\theta(v))
\end{eqnarray*} is a Lie algebroid morphism.  
\end{itemize}
For an idea of the proof we refer to \cite{BursztynCabrera} where they treat the case in which $E=\Rr_M$ is the trivial representation... The general case can be treated in a similar way.
\end{remark}

\begin{lemma} For any Spencer operator $D$,
$\Sol(A,D)\subset\Gamma(A)$ is a Lie sub-algebra.
\end{lemma}

\begin{proof}
This is immediate from the compatibility condition \eqref{horizontal2}
\end{proof}

Another interpretation of a Spencer operator is in terms of Lie algebroid cocycles on $J^1A$ with values in $T^*\otimes E\in\Rep(J^1A)$. Here we use that $TM\in\Rep(J^1A)$ by the adjoint action (see \eqref{actions}), $E\in\Rep(J^1A)$ by pulling back the action of lemma \ref{action-lema} via $pr:J^1A\to A$, and we consider $T^*\otimes E\in\Rep(J^1A)$ the tensor product representation. Hence the action is
\begin{equation}\label{equation:nabla'}
  \begin{split}
    &\tilde \nabla:\Gamma(J^1A)\times \Gamma(T^*\otimes E)\To \Gamma(T^*\otimes E)  \\
    &(\tilde \nabla_\xi\omega)(X)=\nabla_{pr(\xi)}(\omega(X))-\omega(\ad_\xi(X))
  \end{split}
\end{equation}
for $X\in \X(M)$ and $\omega\in\Omega^1(M,E)$.

\begin{lemma}\label{J^1_DA}
An $l$-connection $D$ is a Spencer operator relative to $l$ if and only if the vector bundle map
\begin{eqnarray}\label{a-map}
a:J^1A\to T^*M\otimes E,\quad a(j^1_xb)=D(b)(x) 
\end{eqnarray}
is an algebroid cocycle.
\end{lemma}

\begin{proof}
Using the Spencer decomposition \eqref{Spencer decomposition}, recall from remark \ref{cocyclea} that $a$ is given at the level of sections by 
\begin{eqnarray*}
a(\al,\omega)=D(\al)-l\omega,
\end{eqnarray*}
where $(\al,\omega)\in\Gamma(A)\oplus\Omega^1(M,A)$. What we must show is that equations \eqref{horizontal} and \eqref{horizontal2} are fulfilled if and only if 
\begin{eqnarray*}
\tilde \nabla_\xi a(\eta)-\tilde \nabla_\eta a(\xi)-a([\xi,\eta])=0
\end{eqnarray*}
for any $\xi,\eta\in\Gamma(J^1A)$. Let $\al,\beta\in\Gamma(A)$ and $\omega,\theta\in\Omega^1(M,A)$ be such that 
\begin{eqnarray*}
\xi=(\al,\omega),\qquad\eta=(\beta,\theta).
\end{eqnarray*}
By formula \eqref{bracket}, we have that the Lie bracket of $\xi$ and $\eta$, using again the Spencer decomposition \eqref{Spencer decomposition}, is given by
\begin{eqnarray}\label{1}
[\xi,\eta]=([\al,\beta],L_\xi(\theta)-L_\eta(\omega)).
\end{eqnarray}
Therefore, using formula \eqref{adjunta}
\begin{eqnarray*}
\begin{aligned}
a([\xi,\eta])(X)=&D_X[\al,\beta]-l[\al,\theta(X)]-l\omega(\rho(\theta(X)))+l\theta([\rho(\al),X])\\&+l[\be,\omega(X)]+l\theta(\rho(\omega(X)))-\omega([\rho(\be),X])
\end{aligned}
\end{eqnarray*}
for any $X\in\X(M).$ On the other hand, using formulas \eqref{actions} and \eqref{equation:nabla'}
\begin{eqnarray*}
\begin{aligned}
\tilde \nabla_\xi a(\eta)(X)=&\nabla_{\al}(D_X\beta-l\theta(X))-(D\beta-l\theta)(\ad_\xi(X))\\
=&\nabla_{\al}(D_X\beta-l\theta(X))-D_{[\rho(\al),X]}-D_{\rho(\omega(X))}\beta+l\theta([\rho(\al),X])\\
&+l\theta(\rho(\omega(X))),
\end{aligned}
\end{eqnarray*}
and similarly 
\begin{eqnarray*}
\begin{aligned}
\tilde \nabla_\eta a(\xi)(X)=&\nabla_{\be}(D_X\al-l\omega(X))-(D\al-l\omega)(\ad_\eta(X))\\
=&\nabla_{\be}(D_X\al-l\omega(X))-D_{[\rho(\be),X]}-D_{\rho(\theta(X))}\al+l\omega([\rho(\be),X])\\
&+l\omega(\rho(\theta(X))).
\end{aligned}
\end{eqnarray*}
Replacing the above equations, we get that the right hand side of \eqref{1} is equal to
\begin{eqnarray*}
\begin{aligned}
&\nabla_\al (D_X\be)-\nabla_\al(l\theta(X))-D_{[\rho(\al),X]}\be-D_{\rho(\omega(X))}\be-\nabla_\be(D_X\al)\\&+\nabla_\be(l\omega(X))+D_{[\rho(\be),X]}\al+D_{\rho(\theta(X))}\al-D_X[\al,\be]+l[\al,\theta(X)]-l[\be,\omega(X)],
\end{aligned}
\end{eqnarray*}
but this equation is identically zero for any $\al,\beta\in\Gamma(A)$, $\omega,\theta\in\Omega^1(M,A)$ and $X\in\X(M)$ if and only if 
\begin{eqnarray*}
\nabla_\alpha(D_X\be)-D_{[\rho(\alpha),X]}\be-\nabla_{\be}(D_X\alpha)+D_{[\rho(\be),X]}\alpha-D_X[\alpha,\be]=0,
\end{eqnarray*}
and for any section $\ga\in\Gamma(A)$
\begin{eqnarray*}
\nabla_\alpha(l(\ga))=D_{\rho(\ga)}\alpha+l[\alpha,\ga]
\end{eqnarray*}
which completes our proof :-)
\end{proof}

\subsection{Lie-Spencer operators}\label{Lie-Spencer operators}

The notion of Spencer operators from the previous subsection has more remarkable properties when the surjective vector bundle map is a Lie algebroid morphism
\begin{eqnarray*}
l:A\To E.
\end{eqnarray*}
To emphasize that we are in such a setting (i.e. $E$ is a Lie algebroid and $l$ is a Lie algebroid morphism) we will say that $D$ is a {\bf Lie-Spencer operator relative to $l$} (although the definition remains unchanged). Also, in this case we will use the letter $L$ instead of $E$ to denote the target of $l$. By abuse of notation we will denote the anchor of $L$ by $\rho$ and the Lie bracket by $[\cdot,\cdot]$, as we do with $A$. Remark that, in this case
\begin{eqnarray*}
\g\subset \ker(\rho);
\end{eqnarray*}
hence the condition \eqref{horizontal} is superfluous, the definition \eqref{coneccion} of $\nabla$ becomes more direct, and therefore equation \eqref{horizontal2} can be expressed without reference to $\nabla.$ Hence definition \ref{def1} becomes:

\begin{definition}\label{Lie-Spencer}
Let $l:A\to L$ be a surjective Lie algebroid morphism. A {\bf Lie-Spencer operator relative to $l$} is an $l$-connection 
\begin{eqnarray*}
D:\Gamma(A)\To \Omega^1(M,L)
\end{eqnarray*}
with the property that 
\begin{eqnarray}\label{L-S}
\begin{split}
D_X[\al,\be]-[D_X\al,l(\be&)]-[l(\al),D_X\be]\\&=D_{\rho D_X(\be)+[\rho(\be),X]}\al-D_{\rho D_X(\al)+[\rho(\al),X]}\be
\end{split}
\end{eqnarray}
for $\al,\beta\in\Gamma(A)$ and $X\in\X(M)$.
\end{definition}

\begin{remark}\rm Intuitively, \eqref{L-S} measures the compatibility of $D$ with the Lie bracket and $l$. This is clear from the left hand side. But the left hand side of \eqref{L-S} is not $C^\infty(M)$-linear, and the right hand side can be seen as making the expression of the left hand side $C^\infty(M)$-linear. The rather complicated right hand side has some more conceptual meaning, as we shall explain later on.
\end{remark}

\subsubsection{Symbols}

From the general theory of $l$-connections (see definitions \ref{definitions}), one associates to $D$ the symbol bundle 
\begin{eqnarray*}
\g=\ker(l)\subset A
\end{eqnarray*}
and the symbol map 
\begin{eqnarray*}
\partial_D:\g\To \hom(TM,L),\qquad \partial_D(v)(X)=D_X(v).
\end{eqnarray*}
In the Lie-Spencer case, all these objects become Lie theoretic. First of all:

\begin{lemma}For any surjective Lie algebroid map $l:A\to L,$ $\g:=\ker(l)$ has a natural structure of bundle of Lie algebras, uniquely determined by the condition that $\Gamma(\g)$ becomes a Lie sub-algebra of $\Gamma(A)$.
\end{lemma}

\begin{proof}For $a,b\in \g_x$ its Lie bracket is defined by 
\begin{eqnarray}\label{lie}
[a,b]_x:=[\al,\be](x)
\end{eqnarray}
where $\al,\be\in\Gamma(A)$ are any sections such that $\al(x)=a$ and $\be(x)=b$. That the previous formula does define a Lie bracket on $\g_x$ follows from the fact that $\g_x \subset \ker\rho$, where $\rho$ is the anchor of $A$. Indeed, 
\begin{eqnarray*}
\rho(a)=\rho(l(a))=0,
\end{eqnarray*}
where in the left hand side $\rho$ refers to the anchor of $L$. Hence, that \eqref{lie} is well-defined and it is bilinear is a consequence of the formula
\begin{eqnarray*}
[\al,f\be](x)=f(x)[\al,\be](x)+L_{\rho(\al(x))}(f)\be(x)=f(x)[\al,\be](x)
\end{eqnarray*}
for any $f\in C^{\infty}(M).$ 
\end{proof}

\begin{lemma}\label{[]} For any vector bundle map $\rho:L\to T$, the bracket $[\cdot,\cdot]_\rho$ on $\hom(T,L)$ given by
\begin{eqnarray*}
[\xi,\eta]_\rho=\xi\circ\rho\circ \eta-\eta\circ\rho\circ\xi
\end{eqnarray*}
makes $\hom(T,L)$ into a bundle of Lie algebras.
\end{lemma}

\begin{proof}
The Jacobi identity for $[\cdot,\cdot]_\rho$ is a straightforward computation.
\end{proof}

Finally,

\begin{lemma}\label{int}
For any Lie-Spencer operator $D$,
\begin{eqnarray*}
\partial_D:\g\To T^*\otimes L
\end{eqnarray*}
is a morphism of bundles of Lie algebras.
\end{lemma}

\begin{proof} Let $\al,\be\in\Gamma(\g)$. By the condition \eqref{L-S} and the facts that $\al,\be\in\ker (l)$, and thus $\al,\be\in\ker(\rho)$,
\begin{eqnarray*}
\begin{split}
\partial_D([\al,\be])(X)&=D_X[\al,\be]=D_{\rho D_X\be}\al-D_{\rho D_X\al}\be\\&=\partial_D(\al)\circ\rho\circ \partial_D(\be)(X)-\partial_D(\be)\circ\rho\circ \partial_D(\al)(X)\\&=[\partial_D(\al),\partial_D(\be)]_\rho(X)
\end{split}
\end{eqnarray*}
for any $X\in \X(M).$
\end{proof}

The previous lemma is a direct consequence of the fact that $\partial_D$ is part of a Lie algebroid map. Keeping in mind that 
\begin{itemize}
\item $\g=\ker(l:A\to L)$,
\item $\hom(TM,L)=\ker(pr:J^1L\to L)$,
\item The relative connection $D$ can be interpreted as a vector bundle map
\begin{eqnarray*}
j_D:A\To J^1L,
\end{eqnarray*}
\end{itemize}
lemma \ref{int} can be ``improved'' as $j_D$ is a Lie algebroid morphism as follows.

\begin{lemma}\label{J^1E}Let $l:A\to L$ be a surjective Lie algebroid morphism. An $l$-connection $D$ is a Lie-Spencer operator if and only if \begin{eqnarray*}
j_D:A\To J^1L
\end{eqnarray*}
is a Lie algebroid map.
\end{lemma}

\begin{proof}
Using the Spencer decomposition \eqref{Spencer decomposition}, recall that $j_D(\al)=(l(\al),D(\al))$ for $\al\in\Gamma(A)$. So, if $\be\in\Gamma(A)$ then, as in equation \eqref{1}, we have that 
\begin{eqnarray*}
\begin{split}
 [j_D(\al),j_D(\beta)]_{J^1L}&=[(l(\al),D(\al)),(l(\be),D(\be))]_{J^1L}\\&=([l(\al),l(\be)],L_{j_D(\al)}(D(\be))-L_{j_D(\be)}(D(\al))).
 \end{split}
\end{eqnarray*}
On the other hand, 
\begin{eqnarray*}
j_D[\al,\be]=(l[\al,\be],D[\al,\be]).
\end{eqnarray*}
Hence, $j_D$ is a Lie algebroid morphism if and only if 
\begin{eqnarray}\label{eqn:A}
[l(\al),l(\be)]=l[\al,\be]
\end{eqnarray}
and 
\begin{eqnarray}\label{eqn:B}
L_{j_D(\al)}(D(\be))-L_{j_D(\be)}(D(\al))-D[\al,\be]=0.
\end{eqnarray}
As $l:A\to L$ is surjective, equation \eqref{eqn:A} is equivalent to the fact that $l$ is a Lie algebroid morphism. Indeed, the only thing to check is that $\rho\circ l=\rho$ whenever \eqref{eqn:A} is satisfied. From Leibniz we have that, for any $f\in C^{\infty}(M)$,
\begin{eqnarray*}
\begin{aligned}
fl[\al,\beta]+L_{\rho(\al)}(f)l(\be)=l[\al,f\be]=[l(\al),fl(\be)]=f[l(\al),l(\be)]+L_{\rho(l(\al))}l(\be).
\end{aligned}
\end{eqnarray*} 
Since this equation is valid for any $\al,\be\in\Gamma(A)$ and $l$ is surjective, then $\rho(\al)=\rho(l(\al))$.
Now, if we assume that $l$ is a Lie algebroid morphism, then for $X\in\X(M)$ 
\begin{eqnarray*}
\begin{split}
L_{j_D(\al)}(D(\be&))(X)-L_{j_D(\be)}(D(\al))(X)-D_X[\al,\be]\\
=&[l(\al),D_X\be]+D_{\rho D_X\be}\al-D_{[\rho(l(\al)),X]}\be-[l(\be),D_X\al]\\&-D_{\rho D_X\al}\be+D_{[\rho(l(\be)),X]}\al-D_X[\al,\be]\\
=&[l(\al),D_X\be]+D_{\rho D_X\be}\al-D_{[\rho(\al),X]}\be-[l(\be),D_X\al]\\&-D_{\rho D_X\al}\be+D_{[\rho(\be),X]}\al-D_X[\al,\be]\\=&
\nabla_\alpha^{L}(D_X\beta)-D_{[\rho(\alpha),X]}\beta-\nabla_{\beta}^{L}(D_X\alpha)+D_{[\rho(\beta),X]}\alpha-D_X[\alpha,\beta]\\=
&R_D(\al,\be)(X)
\end{split}
\end{eqnarray*}
Therefore, equation \eqref{eqn:B} holds if and only if $D$ is an Lie-Spencer operator relative to $l:A\to L$.
\end{proof}

\subsubsection{The associated representations}

Some of the expressions appearing in \eqref{L-S} and \eqref{horizontal2} can be put in the form of two operators ($A$-connections):
\begin{eqnarray}\label{equation:A}
\nabla^{T}:\Gamma(A)\times \Gamma(TM)\To \Gamma(TM), \quad\nabla^T_\al(X)=\rho D_X(\al)+[\rho(\al),X]
\end{eqnarray}
and
\begin{eqnarray}\label{equation:B}
\nabla^{L}:\Gamma(A)\times \Gamma(L)\To \Gamma(L), \quad\nabla^{L}_\al(\lambda)= D_{\rho(\lambda)}(\al)+[l(\al),\lambda],
\end{eqnarray}
where $\al\in\Gamma(A)$, $X\in\Gamma(TM)$ and $\lambda\in\Gamma(L)$. We call the previous $A$-connections {\bf the basic connections on $TM$ and $L$} respectively. 

\begin{lemma}\label{rep} For any Lie-Spencer operator $D$ relative to $l:A\to L$, $\nabla^T$ and $\nabla^L$ make $TM$ and $L$ into representations of $A$.
\end{lemma}

\begin{proof}
Let $\al,\beta\in\Gamma(A)$ and $X\in\X(M)$. Using that $\rho$ commutes with the Lie bracket and that $\rho\circ l=\rho$, where in the right hand side we refer to the anchor of $A$, we have that
\begin{eqnarray*}
\begin{aligned}
\nabla&^{T}_\al\nabla^T_\be(X)-\nabla_{\be}^{T}\nabla^T_\al(X)-\nabla_{[\al,\be]}^{T}(X)\\=&
\rho D_{\nabla^T_\be(X)}\al+[\rho(\al),\nabla^T_\be(X)]-\rho D_{\nabla^T_\al(X)}\be-[\rho(\be),\nabla^T_\al(X)]\\&-\rho D_X[\al,\be]-[\rho[\al,\be],X]\\=&
\rho D_{\nabla^T_\be(X)}\al+[\rho(\al),\rho D_X\be+[\rho(\be),X]]-\rho D_{\nabla^T_\al(X)}\be\\&-[\rho(\be),\rho D_X\al+[\rho(\al),X]]-\rho D_X[\al,\be]-[\rho[\al,\be],X]\\=&
\rho D_{\nabla^T_\be(X)}\al-\rho D_{\nabla_\al^{T}(X)}\be+[\rho D_X\al,\rho(\be)]+[\rho(\al),\rho D_X\be]-\rho D_X[\al,\be]\\=&
\rho(D_{\nabla^T_\be(X)}\al-D_{\nabla_\al^{T}(X)}\be+[D_X\al,l(\be)]+[l(\al),D_X\be]-D_X[\al,\be])\\=&
\rho(0)=0,
\end{aligned}
\end{eqnarray*}
where in the passage from the third to the four line we used the Jacobi identity for $[[\rho(\al),\rho(\be)],X]$, and from the fifth to the sixth line we used the condition \eqref{L-S}.

On the other hand, for $\lambda\in\Gamma(L)$,
\begin{eqnarray*}
\begin{aligned}
\nabla&^{L}_\al\nabla^L_\be(\lambda)-\nabla_{\be}^{L}\nabla^L_\al(\lambda)-\nabla_{[\al,\be]}^{L}(\lambda)\\=&
\nabla^L_\al(D_{\rho(\lambda)}\be+[l(\be),\lambda])-\nabla^L_\be(D_{\rho(\lambda)}\al+[l(\al),\lambda])-D_{\rho(\lambda)}[\al,\be]-[l[\al,\be],\lambda]\\=&
\nabla^L_\al(D_{\rho(\lambda)}\be)-\nabla^L_\be(D_{\rho(\lambda)}\al)+D_{\rho[l(\be),\lambda]}\al+[l(\al),[l(\be),\lambda]]-D_{\rho[l(\al),\lambda]}\be\\&-[l(\be),[l(\al),\lambda]]-D_{\rho(\lambda)}[\al,\be]-[l[\al,\be],\lambda]\\=&
\nabla^L_\al(D_{\rho(\lambda)}\be)-D_{[\rho(\al),\rho(\lambda)]}\be-\nabla^L_\be(D_{\rho(\lambda)}\al)+D_{[\rho(\be),\rho(\lambda)]}\al-D_{\rho(\lambda)}[\al,\be]\\=&0,
\end{aligned}
\end{eqnarray*}
where we used that $l$ commutes with the Lie bracket and the Jacobi identity for $[[l(\al),l(\be)],\lambda]$, and condition \eqref{horizontal2}. \end{proof}

\begin{remark}\rm Let $l:A\to L$ be a Lie algebroid morphism. For any $l$-connection
\begin{eqnarray*}
D:\Gamma(A)\To \Omega^1(M,L)
\end{eqnarray*}
one defines $\nabla^T$ and $\nabla^L$ given by \eqref{equation:A} and \eqref{equation:B} respectively, and also {\bf the basic curvature of $D$}
\begin{eqnarray*}R_D\in\hom(\wedge^2A,\hom(TM,L))
\end{eqnarray*} 
given by 
\begin{eqnarray}\label{intermsofE}
\begin{aligned}
R_D(\al,\beta)(X)=\nabla_\alpha^{L}(D_X\beta)-D_{[\rho(\alpha),X]}&\beta-\nabla_{\beta}^{L}(D_X\alpha)\\&+D_{[\rho(\beta),X]}\alpha-D_X[\alpha,\beta],
\end{aligned}
\end{eqnarray}
where $\al,\beta\in\Gamma(A)$ and $X\in\X(M)$. The proof of lemma \ref{rep} shows that 
\begin{eqnarray*}
R_{\nabla^T}=\rho\circ R_D,\qquad R_{\nabla^L}=\rho^*R_D,
\end{eqnarray*}
where 
\begin{eqnarray*}
R_{\nabla^T}\in\hom(\wedge^2A,\hom(TM,TM)),\qquad R_{\nabla^L}\in\hom(\wedge^2A,\hom(L,L))
\end{eqnarray*}
are the basic curvatures of $\nabla^T$ and $\nabla^L$ respectively.

We find ourselves in the setting of representations up to homotopy of \cite{homotopy}:
\begin{eqnarray*}
\rho:L\To TM
\end{eqnarray*}
becomes a representation up to homotopy of $A$.
\end{remark}
 
 Let us state one of the conclusions of this discussion:
 
 \begin{corollary} A relative connection $D$ is a Lie-Spencer operator if and only if its basic curvature $R_D$ vanishes.
 \end{corollary}

\subsection{Examples}

\begin{example}[Jet algebroids]\label{higher jets on algebroids}\rm Of course in the setting of Lie algebroids one has the classical Spencer operator of $A$, denoted by
\begin{eqnarray*}
D^{\clas}:\Gamma(J^1A)\To \Omega^1(M,A).
\end{eqnarray*} 
The construction and properties of $D^\clas$ in the setting of vector bundles were explained in example \ref{Spencer tower}. As one may expect, $D^\clas$ is compatible with the Lie algebroid structure of $J^1A$. More precisely, recall that $J^1A$ acts on $A$ and $TM$ with the adjoint actions (see subsection \ref{Jet groupoids and algebroids} and definition \eqref{actions}).

\begin{proposition} For a Lie algebroid $A$, the classical Spencer operator
\begin{eqnarray*}
D^{\clas}:\Gamma(J^1A)\To \Omega^1(M,A)
\end{eqnarray*}  
is a Lie-Spencer operator relative to the projection $pr:J^1A\to A$, and the induced action on $A$ and $M$ are the adjoint actions.
\end{proposition}

\begin{proof}
We should check that the curvature map $$R_{D^\clas}\in \hom(\wedge^2J^1A,\hom(TM,A))$$ vanishes. Note that it is enough to check that it does on sections of $\Gamma(J^1A)$ of the form $j^1\al$, $\al\in\Gamma(A)$, as $\Gamma(J^1A)$ is (locally) generated by sections of this form as a $C^\infty(M)$-module (see subsection \ref{Jet groupoids and algebroids}), and $R_{D^\clas}$ is $C^{\infty}(M)$-bilinear. Let $\al,\be\in \Gamma(A)$. Then, for any $X\in \X(M)$,
\begin{eqnarray*}
\begin{split}
R_{D^\clas}(j^1\al,j^1\be)(X)&=D^\clas_X[j^1\al,j^1\be]-[D^\clas_Xj^1\al,pr(\be)]-[pr(\al),D^\clas_Xj^1\be]\\&\quad-D^\clas_{\rho D^\clas_X(j^1\be)+[\rho(j^1\be),X]}j^1\al+D^\clas_{\rho D^\clas_X(j^1\al)+[\rho(j^1\al),X]}j^1\be\\
&=D_X^\clas j^1[\al,\beta]=0,
\end{split}
\end{eqnarray*}
where we used that $D^\clas$ vanishes on holonomic sections and that $[j^1\al,j^1\be]=j^1[\al,\be]$ is again holonomic. It is straightforward to check, using the formulas \eqref{actions}, that the induced representations on $A$ and $TM$ are indeed the adjoint actions. 
\end{proof}

The above proposition has, of course, a version for higher jets (see subsection \ref{Jet groupoids and algebroids}). More precisely, following the construction of example \ref{Spencer tower}, the classical Spencer operator on $J^kA$
\begin{eqnarray*}
D^{\clas}:=D^{k\text{-\clas}}:\Gamma(J^kA)\To\Omega^1(M,J^{k-1}A)
\end{eqnarray*}
is a Lie-Spencer operator relative to the projection $pr:J^kA\to J^{k-1}A$.
There are a few things to remark here:
\mbox{}
\begin{itemize}
\item The structure of bundle of Lie algebras $\hom(TM,A)\subset J^1A$ as a Lie subalgebroid of $J^1A$, is given at the point $x\in A$ by 
\begin{eqnarray*}
[\phi,\psi]_x(X)=\phi(\rho(\psi(X)))-\psi(\rho(\phi(X)))
\end{eqnarray*}
for $\phi,\psi\in\hom(T_xM,A_x)$
and $X\in T_xM.$ Note that this structure coincides with that of $\hom(TM,A)$ as the symbol space of the Lie-Spencer operator $D^\clas.$ 

In the same way, the structure of bundle of Lie algebras of $S^kT^*\otimes A\subset J^kA$ (see example \ref{Spencer tower} ) coincides with that of the symbol space of $D^{k\text{-\clas}}$, which, for $k\geq2$, is trivial.
\item From example \ref{classicalSpencertower}, it follows that $(J^kA,D^{\clas})$ is an example of a Lie prolongation of $(J^{k-1}A,D^{\clas})$ (to be defined in definition \ref{algb-prol}) for $k\geq 2$. Moreover, the tower
\[\begin{aligned}
\cdots \To J^{k+1}A\xrightarrow{D^{\clas}}{}J^kA\xrightarrow{D^{\clas}}{} \cdots\To J^1A\xrightarrow{D^\clas}{}A
\end{aligned}\]
is the classical example of a standard Spencer tower to be defined in subsection \ref{algebroid prolongation and spencer towers}. 
\end{itemize}
\end{example}

\begin{example}\rm For a smooth generalized Lie pseudogroup $\Gamma\subset \Bis(\G)$, one gets a sequence of Lie algebroids together with Lie-Spencer operators
\begin{eqnarray*}
\cdots\To A^{(k)}(\Gamma)\xrightarrow{D^{(k)}}{}\cdots\To A^{(1)}(\Gamma)\xrightarrow{D^{(1)}}{} A^{(0)}(\Gamma)
\end{eqnarray*}
where $D^{(k)}$ is the restriction of the classical Lie-Spencer operator $D^\clas:J^kA\to J^{k-1}A$, for $A=Lie(\G).$
See subsection \ref{Lie pseudogroups}
\end{example}

\begin{example}[Composition]\rm

Let 
\begin{eqnarray*}
\tilde A\overset{\tilde D,\tilde l}{\To}A\overset{D,l}{\To}E
\end{eqnarray*}
be two Spencer operators where $\tilde l:\tilde A\to A$ is a Lie algebroid map. We saw in subsection \ref{first examples and operations} that we can compose  the above relative connections in two different ways to get connections 
\begin{eqnarray*}
D^{i}:\Gamma(\tilde A)\To\Omega^1(M,E),\quad i:1,2
\end{eqnarray*}
relative to the vector bundle map $l\circ \tilde l:\tilde A\to E$, and defined by
\begin{eqnarray*}
D^1(\al)=l\circ \tilde D(\al)\quad\text{and}\quad D^2(\al)=D(\tilde l(\al)),
\end{eqnarray*}
where $\al\in\Gamma(\tilde A)$. In the setting of Lie algebroids we have:
\mbox{}
\begin{itemize}
\item $D^2$ is a Spencer operator,
\item $D^1$ is a Spencer operator if 
\begin{eqnarray}\label{comp}
l\circ \tilde D_{\rho_A(\gamma)}(\al)=D_{\rho_A(\gamma)}(\tilde l(\al)),
\end{eqnarray}
or equivalently
\begin{eqnarray}\label{compcon}
l\circ\nabla^{A}_\al(\gamma)=\nabla^E_{\tilde l(\al)}l(\gamma)
\end{eqnarray}
for $\al\in \Gamma(\tilde A)$ and $\gamma\in\Gamma(A)$.
\end{itemize}
Let's check that $D^2$ is a Spencer operator relative to $l\circ \tilde l$, i.e. it satisfies equations \eqref{horizontal} and \eqref{horizontal2}. Let $\beta\in\Gamma(\tilde A)$ be such that $l(\tilde l(\be))=0$, or equivalently $\tilde l(\be)\in\ker l$. Then, as $\rho_{\tilde A}(\be)=\rho_A(\tilde l(\be))$,
\begin{eqnarray*}
D^2_{\rho_{\tilde A}}(\al)=D_{\rho_{\tilde A}(\be)}(\tilde l(\al))=D_{\rho_A(\tilde l(\be))}(\tilde l(\al))=-l[\tilde l(\al),\tilde l(\be)]_A=-l\circ \tilde l[\al,\be]_{\tilde A},
\end{eqnarray*}
which shows that equation \eqref{horizontal} holds. To check that equation \eqref{horizontal2} holds, we take the well-defined $\tilde A$-connection $\nabla:\Gamma(\tilde A)\times \Gamma(E)\to \Gamma(E)$ given by
\begin{eqnarray}\label{eq:3}\begin{split}
\nabla_\al(l\circ \tilde l(\tilde \al))&=D^2_{\rho_{\tilde A}(\tilde \al)}(\al)+l\circ \tilde l[\al,\tilde \al]_{\tilde A}\\&=D_{\rho_{\tilde A}(\tilde \al)}(\tilde l(\al))+l\circ \tilde l[\al,\tilde \al]_{\tilde A}\\
&=D_{\rho_A(\tilde l(\tilde \al))}(\tilde l(\al))+l[\tilde l(\al),\tilde l(\tilde \al)]_{A}=\nabla^E_{\tilde l(\al)}(l\circ \tilde l(\tilde \al))),
\end{split}\end{eqnarray}
where $\tilde \al\in \Gamma(\tilde A)$, and $\nabla^E:\Gamma(A)\times \Gamma(E)\to \Gamma(E)$ is the associated flat $A$-connections associated to $D$. As $D$ satisfies \eqref{horizontal2}, we have that 
\begin{eqnarray}\label{conexion}\begin{aligned}
\nabla^E_{\tilde l(\alpha)}(D_X\tilde l(\alpha'))-D&_{[\rho_{A}(\tilde l(\alpha)),X]}\tilde l(\alpha')-\nabla^E_{\tilde l(\alpha')}(D_X\tilde l(\alpha))\\&+D_{[\rho_{A}(\tilde l(\alpha')),X]}\tilde l(\alpha)-D_X[\tilde l(\alpha),\tilde l(\alpha')]_A=0.
\end{aligned}
\end{eqnarray}
Using a splitting $\Psi:E\to \tilde A$ of $l\circ \tilde l$ and equation \eqref{eq:3} we can rewrite 
\begin{eqnarray}\label{prim}
\begin{aligned}
\nabla^E_{\tilde l(\alpha)}(D_X\tilde l(\alpha'))&=\nabla^E_{\tilde l(\alpha)}(l\circ \tilde l(\Psi D_X\tilde l(\alpha')))\\&=\nabla_\al(l\circ \tilde l(\Psi D_X\tilde l(\alpha')))=\nabla_{\alpha}(D_X\tilde l(\alpha'));
\end{aligned}
\end{eqnarray} therefore, since $\tilde l$ is a Lie algebroid map, equation \eqref{conexion} becomes
\begin{eqnarray*}\begin{aligned}
\nabla_{\alpha}(D^2_X(\alpha'))-D^2_{[\rho_{\tilde A}(\alpha),X]}&(\alpha')-\nabla_{\alpha'}(D^2_X(\alpha))\\&+D_{[\rho_{\tilde A}(\alpha'),X]}\alpha-D^2_X[\alpha,\alpha']_{\tilde A}=0.
\end{aligned}
\end{eqnarray*}

Let's check that whenever equation \eqref{comp} holds, $D^1$ is a Spencer operator. Let $\be\in\Gamma(\tilde A)$ be such that $l(\tilde l(\be))=0$, or equivalently $\tilde l(\be)\in\ker l$. Then,
\begin{eqnarray*}\begin{split}
D^1_{\rho_{\tilde A}(\be)}(\al)=l\circ \tilde D_{\rho_{\tilde A}(\be)}(\al)&=l\circ \tilde D_{\rho_{A}(\tilde l(\be))}(\al)=D_{\rho_A(\tilde l(\be))}(\tilde l(\al))\\&=\nabla^E_{\tilde l(\al)}(l(\tilde l\be))-l[\tilde l(\al),\tilde l(\be)]_A=-l\circ \tilde l[\al,\be]_{\tilde A}.\end{split}
\end{eqnarray*}
In order to check that $D^1$ satisfies the compatibility condition \eqref{horizontal2}, we consider the well-defined $\tilde A$-connection $\nabla:\Gamma(\tilde A)\times \Gamma(E)\to \Gamma(E)$ 
\begin{eqnarray*}
\begin{split}
\nabla_\al(l\circ \tilde l(\tilde \al))&=D^1_{\rho_{\tilde A}(\tilde \al)}(\al)+l\circ \tilde l[\al,\tilde \al]_{\tilde A}\\&=l(\tilde D_{\rho_{\tilde A}(\tilde \al)}(\tilde l(\al))+ \tilde l[\al,\tilde \al]_{\tilde A})\\
&=l(\tilde D_{\rho_A(\tilde l(\tilde \al))}(\tilde l(\al))+[\tilde l(\al),\tilde l(\tilde \al)]_{A})=l\circ\nabla^{A}_{\al}(\tilde l(\tilde \al)).
\end{split}
\end{eqnarray*}
Moreover, as $\tilde D$ satisfies the compatibility condition \eqref{horizontal2}, then
\begin{eqnarray}\label{conexion2}\begin{aligned}
l(\nabla^{A}_{\alpha}(\tilde D_X(\alpha'))-\tilde D&_{[\rho_{\tilde A}(\alpha),X]}(\alpha')-\nabla^{A}_{\alpha'}(\tilde D_X(\alpha))\\&+\tilde D_{[\rho_{\tilde A}(\alpha'),X]}\tilde l(\alpha)-D_X[\alpha,\alpha']_{\tilde A})=0.
\end{aligned}
\end{eqnarray}
Using again a splitting of vector bundle maps $\Phi:A\to \tilde A$ of $\tilde l:\tilde A\to A$, we rewrite
\begin{eqnarray}\label{seg}\begin{aligned}
l\circ\nabla^{A}_{\alpha}(\tilde D_X(\alpha'))&=l\circ\nabla^{A}_{\alpha}(\tilde l(\Psi \tilde D_X(\alpha')))\\&=\nabla_{\al}(l\circ \tilde l(\Psi \tilde D_X(\alpha')))=\nabla_{\al}(l\circ \tilde D_X(\alpha'))).
\end{aligned}\end{eqnarray}
Plugging in equation \eqref{seg} into equation \eqref{conexion2}, we get the compatibility condition \eqref{horizontal2} for $D^1$.

Note that by \eqref{compcon}, $\nabla$ is the same $\tilde A$-connection \eqref{prim} associated to $D^2$.
\end{example}

\section{Prolongations of Spencer operators}

As a continuation of section \ref{section:prolongation}, we will discuss prolongations of Spencer operators. We will see how in this setting the objects are ``Lie theoretic'', for example the classical prolongation becomes now a Lie subalgebroid, and the curvature map becomes a cocycle. 

Throughout this section $(D,l):A\to E$ is a Spencer operator relative to the vector bundle map $l:A\to E$, with associated flat $A$-connection $\nabla:\Gamma(A)\times\Gamma(E)\to \Gamma(E)$.

\subsection{The classical Lie prolongation}

 From the general theory of relative connections, recall from subsection \ref{basic notions} that one has the inclusion :
 \begin{eqnarray*}
 P_D(A)\subset J^1_DA\subset J^1A.
 \end{eqnarray*}
 In the setting of Spencer operators these spaces are all Lie theoretical,
 
 \begin{proposition}\label{prop:jet_subalgebroid} For any Spencer operator $D$ on $A$,
 $J^1_DA$ is a Lie subalgebroid of $J^1A.$
 \end{proposition}
 
 \begin{proof} From lemma \ref{a-map} one has that $J^1_DA$ is the kernel of a cocycle on $J^1A$, and therefore it is a subalgebroid. Indeed, for any 1-cocycle $c:B\to E$ of constant rank on a Lie algebroid $B$ with coefficients in some representation $E\in\Rep(A)$, its kernel is a Lie subalgebroid: for $\al,\be\in\ker c$, the cocycle condition for $c$ implies that 
 \begin{eqnarray*}
 0=\nabla_\al c(\be)-\nabla_\be c(\al)-c[\al,\be]=-c[\al,\be],
 \end{eqnarray*}
 where $\nabla$ is the $A$-connection associated to $E$. Hence $[\al,\be]\in\ker c$.
 \end{proof}
 
 \begin{proposition}\label{522}
 $P_D(A)$, whenever smooth, is a Lie subalgebroid of $J^1A$.
 \end{proposition}
 
 The proof uses the same idea as the proof of proposition \ref{prop:jet_subalgebroid}. In particular, we use the description of $P_D(A)$ as the kernel of the curvature map
   \begin{eqnarray}\label{varkappa-map}
\varkappa_D:J^1_DA\To \hom(\wedge^2TM,E)
\end{eqnarray}
introduced in definition \ref{def:curvature}. In this setting the curvature map becomes a cocycle with values in $\hom(\wedge^2TM,E)\in\Rep(J^1_DA)$, where here we used the representation $TM\in\Rep(J^1_DA)$ given by the restriction of the adjoint action of $J^1A$ on $TM$ (see \eqref{actions}), $E\in \Rep(J^1_DA)$ by the pullback of the representation $E\in\Rep(A)$ via $pr:J^1_DA\to A$, and $\hom(\wedge^2TM,E)$ is endowed with the tensor product representation. Hence the action is given by the flat $J^1_DA$-connection
\begin{eqnarray*}
\tilde \nabla:\Gamma(J^1_DA)\times \Omega^2(M,E)\To \Omega^2(M,E)
\end{eqnarray*} 
defined by 
\begin{eqnarray}\label{nabla'}
\tilde \nabla_\xi\omega(X,Y)=\nabla_{pr(\xi)}(\omega(X,Y))-\omega(\ad_\xi(X),Y)-\omega(Y,\ad_\xi(Y)),
\end{eqnarray}
where $\omega\in\Omega^2(M,E)$, $\xi\in\Gamma(J^1_DA)$ and $X,Y\in\X(M)$. 

\begin{lemma}\label{aja}
The curvature of the Spencer operator $D$
\begin{eqnarray*}
\varkappa_D:J^1_DA\To \hom(\wedge^2TM,E)
\end{eqnarray*}
defined by equation \eqref{varkappa-map} is an algebroid cocyle. 
\end{lemma}

\begin{proof}
Let $\xi,\eta\in\Gamma(J^1A)$ be two sections of $J^1_DA$, i.e. for $\al,\beta\in\Gamma(A)$, $\omega,\theta\in\Omega^1(M,A)$  be such that $\xi=(\al,\omega)$ and $\eta=(\be,\theta)$, then 
\begin{eqnarray*}
D(\al)=l\circ\omega\quad\text{and}\quad D(\be)=l\circ\theta.
\end{eqnarray*}
Here we are using the Spencer decomposition \eqref{Spencer decomposition}. We must show that 
\begin{eqnarray}\label{A}
\tilde \nabla_\xi\varkappa_D(\eta)-\tilde \nabla_\eta\varkappa_D(\xi)-\varkappa_D[\xi,\eta]=0.
\end{eqnarray}
Expanding formula \eqref{1} given in the proof of lemma \ref{J^1_DA}, we have that 
\begin{eqnarray}\label{B}\begin{split}
\varkappa_D[\xi,\eta](X,Y)=&T(X,Y,\al,\theta,\omega)-T(X,Y,\be,\omega,\theta)-T(Y,X,\al,\theta,\omega)\\&+T(Y,X,\be,\omega,\theta)-L(X,Y,\al,\theta,\omega)+L(X,Y,\be,\omega,\theta),
\end{split}
\end{eqnarray} 
where for $X,Y\in\X(M)$  
\begin{eqnarray*}
T(X,Y,\al,\theta,\omega)=D_X([\al,\theta(Y)]+\omega(\rho(\theta(Y))))-\theta[\rho(\al),Y]),
\end{eqnarray*}
and 
\begin{eqnarray*}
L(X,Y,\al,\theta,\omega)=l([\al,\theta[X,Y]]+\omega(\rho(\theta[X,Y]))-\theta[\rho(\al),[X,Y]]).
\end{eqnarray*}
On the other hand, using the definition of the flat $J^1_DA$-connection $\tilde \nabla$ (see equation \eqref{nabla'}), we have that 
\begin{eqnarray}\label{C}\begin{split}
(\tilde \nabla_\xi\varkappa_D(\eta))(X,Y)=&\nabla_\al(D_X\theta(Y)-D_Y\theta(X)-l\theta[X,Y])-N(X,Y,\al,\theta,\omega)\\&+N(Y,X,\al,\theta,\omega),
\end{split}
\end{eqnarray}
where
\begin{eqnarray*}\begin{split}
N(X,Y,\al,\theta,\omega)=&D_X(\theta([\rho(\al),Y]+\rho(\omega(Y))))-D_{[\rho(\al),Y]}\theta(X)\\&-D_{\rho(\omega(Y))}\theta(X) -l\theta([X,[\rho(\al),Y]+\rho(\omega(Y))]).\end{split}
\end{eqnarray*}
Similarly one finds that 
\begin{eqnarray}\label{D}\begin{split}
(\tilde \nabla_\eta\varkappa_D(\xi))(X,Y)=&\nabla_\be(D_X\omega(Y)-D_Y\omega(X)-l\omega[X,Y])-N(X,Y,\be,\omega,\theta)\\&+N(Y,X,\be,\omega,\theta).
\end{split}
\end{eqnarray}
Using formulas \eqref{B}, \eqref{C}, \eqref{D}, and the Jacobi identity for $[[\rho(\al),X],Y]$ and for $[[\rho(\be),X],Y]$, the left hand side of equation \eqref{A} becomes
\begin{eqnarray}\label{equation:final}
\begin{split}
V(X,Y,\al,\theta&,\omega)-V(Y,X,\al,\theta,\omega)-V(X,Y,\beta,\omega,\theta)+V(Y,X,\beta,\omega,\theta)\\
-\nabla_\al (l\theta&[X,Y])+l[\al,\theta[X,Y]])+l\omega(\rho(\theta[X,Y]))\\&+\nabla_\beta(l\omega[X,Y]-l[\be,\omega[X,Y]])-l\theta(\rho(\omega[X,Y])),
\end{split}
\end{eqnarray}
where
\begin{eqnarray}\label{into}\begin{split}
V(X,Y,\al,\theta,\omega)=&\nabla_\al(D_X\theta(Y))-D_{[\rho(\al),X]}\theta(Y)-l\omega[X,\rho(\theta(Y))]\\&-D_X[\al,\omega(Y)]-D_{\rho(\omega(X))}\theta(Y).
\end{split}
\end{eqnarray}
Since $D\al=l\omega$, by the compatibility condition \eqref{coneccion} we have that 
\begin{eqnarray*}
D_{\rho(\omega(X))}\theta(Y)=\nabla_{\theta(Y)}-l[\theta(Y),\omega(X)];
\end{eqnarray*}
replacing this equation into \eqref{into} we get by the compatibility condition \eqref{horizontal2}, that 
\begin{eqnarray*}
V(X,Y,\al,\theta,\omega)=l[\theta(Y),\omega(X)].
\end{eqnarray*}
Moreover, by the compatibility condition \eqref{coneccion} 
we have that 
\begin{eqnarray*}
-\nabla_\al (l\theta[X,Y])+l[\al,\theta[X,Y]])+l\omega(\rho(\theta[X,Y]))=0
\end{eqnarray*}
and 
\begin{eqnarray*}
-\nabla_\be (l\omega[X,Y])+l[\be,\omega[X,Y]])+l\theta(\rho(\omega[X,Y]))=0
\end{eqnarray*}
since $D\al=l\omega$ and $D\be=l\theta$; therefore \eqref{equation:final} becomes
\begin{eqnarray*}
l[\theta(Y),\omega(X)]-l[\theta(X),\omega(Y)]-l[\omega(Y),\theta(X)]+l[\omega(X),\theta(Y)]=0.
\end{eqnarray*}
This finally shows equation \eqref{A}.
\end{proof}

\begin{remark}\rm 
If $P_D(A)\subset J^1A$ is smoothly defined in the sense of \ref{smoothly-defined}, then the restriction of the classical Lie-Spencer operator \eqref{Spencer operator} to $P_D(A)$
\begin{eqnarray*}
D^{(1)}:\Gamma(P_D(A))\To\Omega^1(M,A)
\end{eqnarray*}
 is a Lie-Spencer operator relative to $pr:P_D(A)\to A$. The fact that $D^{(1)}$ satisfies the compatibility condition of a Lie-Spencer operator follows from example \ref{higher jets on algebroids}.
\end{remark}

In analogy with the discussion of relative connections (see section \ref{compatible connections}), we can talk about compatible Spencer operators and the universal nature of $P_D(A).$ 

\begin{definition}\label{algb-prol} Let 
\begin{eqnarray*}
 \tilde A\overset{(\tilde D,\tilde l)}{\To}A\overset{(D,l)}{\To} E
\end{eqnarray*}
be Spencer operators. We say that $(D,\tilde D)$ are {\bf compatible Spencer operators} if 
\begin{enumerate}
\item $D$ and $\tilde D$ are compatible connections, and
\item $\tilde D$ is a Lie-Spencer operator.
\end{enumerate}
W say that $\tilde D$ is a {\bf Lie prolongation of} $(A,D)$ if, moreover, $\tilde D$ is standard. 
\end{definition}

\begin{remark}\rm
The condition \eqref{compatibility1} for compatible Spencer operators implies that the representations $\tilde \nabla$ and $\nabla$ are {\bf compatible}, in the sense that for any $\tilde\al\in\Gamma(\tilde A)$ and any $\al\in\Gamma(A)$,
\begin{eqnarray}\label{compatible-rep}
l(\tilde \nabla_{\tilde\al}\al)=\nabla_{\tilde l(\tilde\al)}l(\al).
\end{eqnarray}
This is true as equation \eqref{compatible-rep} is equivalent to
\begin{eqnarray*}
l\circ \tilde D_{\rho_A(\gamma)}(\al)=D_{\rho_A(\gamma)}(\tilde l(\al)).
\end{eqnarray*}
\end{remark}

The classical Lie prolongation $P_D(A)$ of a Spencer operator $(D,l):A\to E$ is also characterized by the universal property, stated in this setting as follows.

\begin{proposition}
$(D^{(1)},pr):P_D(A)\to A$ is a Lie prolongation of $(A,D)$ which is universal in the sense that for any other Lie prolongation
\begin{eqnarray*}
\tilde A\overset{(\tilde D,\tilde l)}{\To} A\overset{(D,l)}{\To} E
\end{eqnarray*}
there exists a unique Lie algebroid map $j:\tilde A\to P_D(A)$ such that 
\begin{eqnarray*}
\tilde D=D^{(1)}\circ j.
\end{eqnarray*} 
Moreover, $j$ is injective.
\end{proposition}

\begin{proof}
Note that $j=j_D$ is a Lie algebroid map by lemma \ref{J^1E}.
\end{proof}

\begin{example}\rm The classical Lie prolongation of the Lie-Spencer operator $(J^kA,D^\clas)$ is 
\begin{eqnarray*}
(J^{k+1}A,D^\clas).
\end{eqnarray*}
\end{example}

\subsection{Higher prolongations; formal integrability}

In the setting of Spencer operators all the notions and results given in chapter \ref{Relative connections} are still valid in this setting. Of course, in this context the objects are ``Lie theoretic''. For example, whenever smooth, the classical $k$-prolongation space of a Spencer operator $(A,D)$ (definition \ref{k-prolongation-space})
\begin{eqnarray*}
P_D^k(A)\subset J^kA
\end{eqnarray*}
is a Lie subalgebroid thanks to proposition \ref{522}. If, moreover, it is smoothly defined (definition \ref{smoothly-defined}) then the relative connection
\begin{eqnarray*}
D^{(k)}:\Gamma(P^k_D(A))\To \Omega^1(M,P^{k-1}_D(A))
\end{eqnarray*} 
is a Lie-Spencer operator relative to the Lie algebroid map $pr:P_D^k(A)\to P_D^{k-1}(A)$. This is a consequence of example \ref{higher jets on algebroids}. In this case we call 
\begin{eqnarray*}
(P_D^k(A),D^{(k)}): P_{D}^k(A)\xrightarrow{D^\clas,pr}{}P^{k-1}_D(A)
\end{eqnarray*}
the {\bf classical $k$-Lie prolongation of $(A,D)$}. If $(A,D)$ is formally integrable in the sense of definition \ref{fi}, then we obtain the tower of Spencer operators
\[\begin{aligned}\
(P^\infty_D(A),D^{(\infty)}): \cdots\To P_D^{k}(A)\overset{D^{(k)}}{\To}\cdots\To P_D(A)\overset{D^{(1)}}{\To}A\overset{D}{\To}E
\end{aligned}\]
called the {\bf classical Spencer resolution}. 

Here are some results of relative connections that we want to remind in this setting:

\begin{corollary}\label{ultimo-porfavor} 
If the classical $k$-prolongation space $P^{k}_D(A)$ is smoothly defined then 
the symbol space of the classical $k$-prolongation of $(A,D)$ is equal to the $k$-prolongation of the symbol map $\partial_D$
\begin{eqnarray*}
\g(P_D^{k}(A),D^{(k)})=\g^{(k)}(A,D).
\end{eqnarray*}
\end{corollary}

\begin{remark}\rm
As explained in subsection \ref{Lie-Spencer operators}, the symbol space of a Lie-Spencer operator is a bundle of Lie algebras. Hence, if $P^k_D(A)$ is smoothly defined then for $k>1$, \begin{eqnarray*}\g(P_D^{k}(A),D^{(k)})\subset P_D^k(A)\end{eqnarray*} is a bundle of trivial Lie algebras. In the case $k=1$, \begin{eqnarray*}\g(P_D^k(A),D^{(1)})=\g^{(1)}(A,D)\subset P_D(A)\end{eqnarray*}
is the bundle of Lie algebras with bracket given by $[\cdot,\cdot]_\rho$ where $\rho:A\to TM$ is the anchor map (see lemma \ref{[]}). See example \ref{higher jets on algebroids}.
\end{remark}

\begin{corollary}
If $P_D^{k_0}(F)$ is smoothly defined, then for any integer $0\leq k\leq k_0$, there is a one to one correspondence between the Lie algebras $\text{\rm Sol}(P^k_D(A),D^{(k)})$ and $\text{\rm Sol}(A,D)$ given by the surjective Lie algebroid map
\begin{eqnarray*}
pr^k_0:P^k_D(A)\To A,
\end{eqnarray*} 
where $pr^k_0$ is the composition $pr\circ pr^2\circ\cdots\circ pr^k$.
\end{corollary}

Recall that the main importance of formally integrable Spencer operators is the following existence result in the analytic case:

\begin{theorem}
Let $(A,D)$ be an analytic Spencer operator which is formally integrable. Then, given $p\in P^k_D(A)$ with $\pi(p)=x\in M$, there exists an analytic solution $s\in\text{\rm Sol}(A,D)$ over a neighborhood of $s$ such that $j^k_xs=p.$
\end{theorem}
 
 And of course, we again have the same workable criteria for formal integrability:

\begin{theorem}
If
\begin{enumerate}
\item $pr:P_D(A)\to A$ is surjective,
\item $\g^{(1)}(A,D)$ is a smooth vector bundle over $M$, and 
\item $H^{2,k}(\g)=0$ for $k\geq0$.
\end{enumerate}
Then, $(A,D)$ is formally integrable.
\end{theorem}

\subsection{Abstract prolongations and Spencer towers}\label{algebroid prolongation and spencer towers}

Not to make this subsection redundant, we want to remark that all the definitions and results of subsection \ref{subsection:prolongation and Spencer towers} make sense in the setting of Lie algebroids and Spencer operators. For instance:

\begin{itemize}
\item A {\bf Spencer tower} $(A^\infty,D^\infty,l^\infty)$ is a tower of relative connections
\begin{eqnarray*}
\cdots\To A^k \overset{(D^k,l)}{\To}\cdots \To A^2\overset{(D^2,l)}{\To} A^1\overset{(D^1l)}{\To} A
\end{eqnarray*} 
in which all $A^k$ are lie algebroids, and all $D^k$ are Lie-Spencer operators. We say that $(A^\infty,D^\infty,l^\infty)$ is a {\bf standard Spencer tower} when all the connections are standard.
\item A {\bf standard Spencer resolution} of a Spencer operator $D$ is a standard Spencer tower $(A^{\infty},D^\infty)$ with the property that $(D,D^1)$ are compatible. 
\item A {\bf morphism $\Psi$ of Spencer towers} is a morphism of towers 
\begin{eqnarray*}
\Psi: (A^\infty,D^\infty)\To (\tilde A^\infty,\tilde D^\infty)
\end{eqnarray*} such that each $\Psi_k:A_k\to \tilde A_k$ is a Lie algebroid morphism
\end{itemize}
and so on... 

\begin{example}[The classical standard Spencer tower]\rm The sequence 
\[\begin{aligned}
(J^\infty A,D^{\infty\text{-\clas}}):\cdots \To J^{k+1}A\overset{D^\clas}{\To}J^kA\overset{D^\clas}{\To} \cdots\To J^1A\overset{D^\clas}{\To}A
\end{aligned}\]
is an example of a standard Spencer  tower. 
\end{example}

\begin{example}[The classical Spencer resolution]\rm 
 If $(D,l):A\to E$ is formally integrable, we obtain the classical Lie-Spencer resolution
\[\begin{aligned}\label{pro2}
(P^\infty_D(A),D^{(\infty)}): \cdots\To P_D^{k}(A)\overset{D^{(k)}}{\To}\cdots\To P_D(A)\overset{D^{(1)}}{\To}A\overset{D}{\To}E.
\end{aligned}\]
Note that \eqref{pro2} is a subtower of $(J^\infty A,D^{\infty\text{-\clas}})$. Using corollary \ref{k-prolongations} and remark \ref{extender}, this sequence still makes sense even if $(A,D)$ is not formally integrable, but this involves non-smooth subbundles.
\end{example}

\begin{example}\rm Going back to the notion of a Lie pseudogroup $\Gamma\subset \Bis(\G)$ (see subsection \ref{Lie pseudogroups}), one has an induced standard Spencer tower
\[\begin{aligned}
(A^{\infty}(\Gamma),D^{\infty}): \cdots\To A^{(k)}(\Gamma)\overset{D^\clas}{\To}\cdots\To A^{(1)}(\Gamma)\overset{D^{\clas}}{\To}A\overset{D^{\clas}}{\To}A^{(0)}(\Gamma)
\end{aligned}\]
\end{example}
 
 In the setting of Spencer operators the classical Spencer resolution is also universal in the following sense.

\begin{theorem}
Let $D$ be not necessarily formally integrable Spencer operator relative to the vector bundle map $l:A\to E$. The possibly non-smooth subtower $(P_D^\infty(A),D^{(\infty)})$ of $(J^\infty A,D^\infty)$ is universal among the Spencer resolution of $(A,D)$ in the sense that for any other Spencer resolution 
\begin{eqnarray*}
A^\infty\xrightarrow{(D^\infty,l^\infty)}{}A\xrightarrow{(D,l)}E
\end{eqnarray*}
there exists a unique morphism $\Psi:(A^\infty,D^{\infty})\to (J^\infty A,D^{\infty\text{-\clas}})$ of Spencer towers such that $\Psi^0:A\to A$ is the identity map, and for $k\geq 1$
\begin{eqnarray*}
\Psi^k(A_k)\subset P_D^k(A).
\end{eqnarray*}
Moreover, if $(A^\infty,D^\infty)$ is a standard Spencer resolution of $(A,D)$ then $\Psi$ is injective.
\end{theorem}

\begin{proof}
It remains to check that the maps $\Psi_k:A_k\to J^kA$ constructed in the proof of theorem \ref{J^1E} are Lie algebroid morphisms. We will give an inductive argument. By construction $\Psi_1$ is equal to 
\begin{eqnarray*}
j_{D^1}.:A_1\to J^1A
\end{eqnarray*}
which is a Lie algebroid morphism by lemma \ref{J^1E}. Let's assume now that for a fixed $k$, $\Psi_k:A_k\to J^kA$ is a Lie algebroid morphism. By definition $\Psi_{k+1}:A_{k+1}\to J^{k+1}A$ is given by  the composition 
\begin{eqnarray*}
A_{k+1}\xrightarrow{j_{D^{k+1}}} J^1(A_k)\xrightarrow{prol(\Psi_k)}J^1(J^kA).
\end{eqnarray*}
As $\Psi_k:A_k\to J^kA$ is a Lie algebroid morphism then it follows that $prol(\Psi_k):J^1(A_k)\to J^1(J^kA)$ is a Lie algebroid morphism, and therefore $\Psi_{k+1}$ is a Lie algebroid morphism since it is the composition of two of them. Indeed, as $\rho_{J^kA}\circ\Psi_k=\rho_{A_k}$, where $\rho_{A_k}$ and $\rho_{J^kA}$ are the anchor maps of $A_k$ and $J^kA$ respectively, and the anchor of the first jet of a Lie algebroid $A$ is the composition of the projection $J^1A\to A$ with the anchor of $A$, then 
\begin{eqnarray*}
\begin{aligned}
\rho_{J^1(J^kA)}\circ prol(\Psi_k)(\al,\omega)=&\rho_{J^1(J^kA)}(\Psi_k\circ\al,\Psi_k\circ\omega)\\&=\rho_{J^kA}(\Psi_k\circ\al)=\rho_{A_k}(\al)=\rho_{J^1(A_k)}(\al,\omega)
\end{aligned}
\end{eqnarray*}
for $(\al,\omega)\in\Gamma(A_k)\oplus \Omega^1(M,A_k)$, where here we are using the Spencer decomposition \eqref{Spencer decomposition}.
On the other hand, as $[j^1\al,j^1\be]_{J^1A_k}=j^1([\al,\be]_{A_k})$ for $\al,\be\in\Gamma(A_k)$, then
\begin{eqnarray*}\begin{split}
[prol(\Psi_k(j^1\al)),&prol(\Psi_k(j^1\al))]_{J^1(J^kA)}=[j^1(\Psi_k\circ\al),j^1(\Psi_k\circ\be)]_{J^1(J^kA)}\\&=j^1[\Psi_k\circ\al,\Psi_k\circ\be]_{J^kA}=j^1[\al,\be]_{A_k}=[j^1\al,j^1\be]_{J^1(A_k)}.
\end{split}
\end{eqnarray*}
As the space of sections of $J^1(A_k)$ is generated by the holonomic sections as a $C^{\infty}(M)$-module, then by the above equation and the Leibniz identity it follows that 
\begin{eqnarray*}
[\cdot,\cdot]_{J^1(J^kA)}\circ\Psi_k=[\cdot,\cdot]_{J^1(A_k)}.
\end{eqnarray*}
Hence, $\Psi_{k+1}$ is a Lie algebroid morphism. 
\end{proof}

\section{Particular cases}\label{Cartan algebroids}

\subsection{Cartan algebroids}

Recalling example \ref{inf-act}, let  
\begin{eqnarray*}
\rho:\mathfrak{h}\To \X(M)
\end{eqnarray*}
be an infinitesimal action of a Lie algebra $\mathfrak{h}$ on $M$. The associated action algebroid $\mathfrak{h}\ltimes M$ comes equipped with the canonical flat linear connection 
\begin{eqnarray*}
\nabla^{\text{flat}}:\X(M)\times C^\infty(M,\mathfrak{h})\To C^\infty(M,\mathfrak{h}).
\end{eqnarray*}
This is the main example of a Cartan algebroid defined in \cite{Blaom}, and in fact it is the only flat Cartan algebroid, up to isomorphism, when $M$ is simply connected as we will see in proposition \ref{only}. 

Cartan algebroids were defined as a Lie algebroid $A$ together with a linear connection 
\begin{eqnarray*}
\nabla: \X(M)\times \Gamma(A)\To \Gamma(A)
\end{eqnarray*}  
such that its basic curvature $R_\nabla$ vanishes. By definition \ref{Lie-Spencer} and lemma \ref{J^1E}, two equivalent interpretations of Cartan algebroids are:
\begin{itemize}
\item as a Lie-Spencer operator $D$ relative to the identity map $id:A\to A$,
\item as a Lie algebroid splitting
\begin{eqnarray*}
j_D:A\To J^1A
\end{eqnarray*}
of the projection $pr:J^1A\to A.$
\end{itemize}

\begin{proposition}\label{only}
Let $(A,\nabla)$ be Cartan algebroid over a simply connected manifold $M$. If $\nabla$ is flat then there exists an infinitesimal action $\rho:\mathfrak{h}\to \X(M)$ and a Lie algebroid isomorphism
\begin{eqnarray*}
\phi:\mathfrak{h}\ltimes M\To A
\end{eqnarray*}
such that under this identification, $\nabla$ becomes the canonical flat linear connection $\nabla^{\text{\rm flat}}.$
\end{proposition}

\begin{proof}
It is a well-known fact that a vector bundle $A$ over a simply connected manifold $M$, with a flat linear connection $\nabla$, is isomorphic to the trivial bundle $V_M$ with fiber $V\subset\Gamma(A)$, the finite dimensional vector space of parallel sections. The isomorphism is given by 
\begin{eqnarray}\label{lam}\begin{aligned}
V_M&\overset{\phi}{\To} A\\
(s,x)&\mapsto s(x).
\end{aligned}\end{eqnarray}
Moreover, under this isomorphism, the canonical flat connection $\nabla^\text{flat}$ of $V_M$ is equal to $\nabla$.
Now, from the compatibility condition \eqref{horizontal} we have that the Lie bracket of $\Gamma(A)$ restricts to a Lie bracket $[\cdot,\cdot]$ on $V$. Set $\mathfrak{h}:=V$ be the Lie algebra with Lie bracket $[\cdot,\cdot]_{\mathfrak{h}}:=[\cdot,\cdot]$. We have an obvious action
\begin{eqnarray*}
\rho':\mathfrak{h}\To\X(M)
\end{eqnarray*}
given by $\rho'(\al)(x)=\rho(\al(x))$, where $\rho$ is the anchor of $A$. It is clear from the definition of $\rho$ and the Lie algebroid structure of $\mathfrak{h}\ltimes M$ that \eqref{lam} is a Lie algebroid isomorphism.
\end{proof}

\subsection{Spencer operators of finite type}

Let $D$ be a Spencer operator of finite type (see section \ref{finite type}) relative to the vector bundle map $l:A\to E$. We will prove the following stronger version of theorem \ref{finitecase}.

\begin{theorem}\label{F-T}
Suppose that $D$ is a Spencer operator over a simply connected manifold $M$, and let $k\geq 1$ be the order of $D$. If 
\begin{enumerate}
\item $P^k_D(A)$ is smoothly defined, and 
\item $pr:P^{k+1}_D(A)\to P^{k}_D(A)$ is surjective,
\end{enumerate}
then there exists an infinitesimal action 
\begin{eqnarray*}
\psi:\mathfrak{h}\To\X(M)
\end{eqnarray*}
of a Lie algebra $\mathfrak{h}$ on $M$ and a surjective morphism $p:(\mathfrak{h}\ltimes M,\nabla^{\text{\rm flat}})\to (A,D)$ of Spencer operators, inducing  a bijection in the space of solutions. Moreover, $\Sol(A,D)\subset\Gamma(A) $ is a finite dimensional Lie algebra of dimension 
\begin{eqnarray*}
r=\rank A+\rank\g^{(1)}+\rank\g^{(2)}+\cdots+\rank\g^{(k-1)}.
\end{eqnarray*}
\end{theorem}

\begin{remark}\rm In theorem \ref{F-T}, a morphism of Spencer operators is a morphism of relative connections  $p:(\mathfrak{h}\ltimes M,\nabla^{\text{flat}})\to (A,D)$ with $p:\mathfrak{h}\ltimes M\to A$ a Lie algebroid map. 

The message is clear here: a Spencer operator satisfying the hypotheses of theorem \ref{F-T} is the quotient of an action algebroid with the canonical flat connection $\nabla^{\text{\rm flat}}$. Intuitively the situation is as follows 
\begin{eqnarray*}
\xymatrix{
\mathfrak{h}\ltimes M \ar[r]^{\nabla^{\text{flat}}} \ar[d]_p & \mathfrak{h}\ltimes M \ar[d]^{l\circ p}\\
A \ar[r]^{D} & E.
}
\end{eqnarray*}
Although the above diagram is not precise, it helps to illustrate our situation.\end{remark}

\begin{lemma}\label{lemma:auxiliary}
With the hypotheses of the previous theorem, one has that
\begin{enumerate}
\item the Lie algebroid $P^{k}_D(A)$ is isomorphic to the action algebroid $\mathfrak{h}\ltimes M$,
\item $pr:P^{k}_D(A)\to P^{k-1}_D(A)$ is a Lie algebroid isomorphism.
\end{enumerate} 
Moreover, under the identification $pr$, $D^{(k)}$ becomes the trivial connection $\nabla^{\text{\rm flat}}.$
\end{lemma}

\begin{proof}[Proof of lemma \ref{lemma:auxiliary} and theorem \ref{F-T}]
From lemma \ref{lema} one has that under the Lie algebroid isomorphism $pr:P^{k}_D(A)\to P^{k-1}_D(A)$, $(P^{k}_D(A),D^{(k)})$ is a flat Cartan algebroid and therefore, by lemma \ref{only}, it is isomorphic to the action groupoid $\mathfrak{h}\ltimes M$ where $\mathfrak{h}$ is the Lie algebra of parallel sections of $D^{(k)}.$

That $p:\mathfrak{h}\ltimes M\to A$ is a Lie algebroid morphism follows from the fact that it is the composition of lie algebroid morphisms as one can see in the proof of theorem \ref{finitecase}. From this, theorem \ref{finitecase}, and lemma \ref{lema} the results follow.
\end{proof}

\begin{corollary}
Suppose that $D$ is a Spencer operator over a simply connected manifold $M$, and let $k\geq 1$ be the order of $D$. If 
\begin{enumerate}
\item $pr:P_D(A)\to A$ is surjective, 
\item $\g^{(1)}$ is a smooth vector bundle over $M$, and 
\item $H^{2,l}(\g)=0$ for $0\leq l\leq k-1$,
\end{enumerate}
then there exists an infinitesimal action 
\begin{eqnarray*}
\psi:\mathfrak{h}\To\X(M)
\end{eqnarray*}
of a Lie algebra $\mathfrak{h}$ over $M$ and a morphism $p:(\mathfrak{h}\ltimes M,\nabla^{\text{\rm flat}})\to (A,D)$ of Spencer operators, inducing  a bijection in the space of solutions. Moreover, $\Sol(A,D)\subset\Gamma(A) $ is a finite dimensional Lie algebra of dimension 
\begin{eqnarray*}
r=\rank A+\rank\g^{(1)}+\rank\g^{(2)}+\cdots+\rank\g^{(k-1)}.
\end{eqnarray*}
\end{corollary}

\begin{proof}
  Our hypotheses imply that $(A,D)$ is formally integrable and therefore we are left in the situation of theorem \ref{F-T}.
\end{proof}

\chapter{Pfaffian groupoids}\label{Pfaffian groupoids}
\pagestyle{fancy}
\fancyhead[CE]{Chapter \ref{Pfaffian groupoids}} 
\fancyhead[CO]{Pfaffian groupoids}

Pfaffian groupoids are Lie groupoids $\G\tto M$ endowed with a multiplicative distribution $\H$ of $\G$ (see definition \ref{def-pf-syst}) compatible with the groupoid structure and satisfying some extra conditions (see definition \ref{defn:pfaffian_grpds}). For such a groupoid, the source map $s:(\G,\H)\to M$ is a Pfaffian bundle, to which we can apply the theory of chapter \ref{Pfaffian bundles}. In the dual picture we deal with a Pfaffian form $\theta\in\Omega^1(\G,t^*E)$, where in this case $E$ is a representation of $\G$ and $\theta$ is multiplicative, i.e. it is compatible with the multiplication of $\G$. In this setting all objects become Lie theoretic and therefore their linearizations are of Spencer type on the Lie algebroid $A$ of $\G$ (see chapter \ref{Relative connections on Lie algebroids}). 
Again, one advantage of working with forms is that it is slightly more general (it allows $\theta$ to have non-constant rank). An advantage of the point of view of distributions is that some integrability conditions are more natural and easier to handle globally.\\

Here are some connections with the existing literature. Multiplicative distributions in a sense more general than ours, but which are required to be involutive, were studied in \cite{Hawkins} in the context of geometric quantization and, more recently, in \cite{JotzOrtiz}. 
Moving towards Cartan's ideas, our Cartan connections from subsection \ref{Cartan connections} are the global counterpart of Blaom's Cartan algebroids \cite{Blaom}. 
The flat Cartan connections are the same ``flat connections on groupoids'' used by Behrend in the context of equivariant cohomology \cite{behrend}.
On the other hand, due to our approach, there is long list of literature on Lie pseudogroups and the geometry of PDE that serves as inspiration for this thesis \cite{Cartan1904, Cartan1905, Cartan1937, Spencer, KumperaSpencer, GuilleminSternberg:deformation, Gold1, Gold2, Gold3, Olver:MC, Kamran, BC, Seiler}. Of course, the appearance of the classical Cartan form and Spencer operator is an indication of this relationship with the theory of Lie pseudogroups. 

\section{Pfaffian groupoids}

\subsection{Multiplicative distributions}

To discuss multiplicativity of distributions, recall that one has a Lie groupoid $T\G\tto TM$ associated to any Lie groupoid $\G\tto M$; its structure maps are just the differentials of the structure
maps of $\G$. 

\begin{definition}\label{def-pf-syst} \mbox{}
 A {\bf multiplicative distribution on $\G$} is any distribution $\H\subset T\G$ which is also a Lie subgroupoid of $T\G\tto TM$ (with the same base $TM$).
 \end{definition}

Here we have some remarks about the previous definitions

\begin{remark}\label{mult-equiv}\rm Given a distribution $\H\subset T\G$, the multiplicativity of $\H$ is equivalent to:
\begin{enumerate}
\item $\H$ is closed under $\d m$, i.e. for any $X_g\in \H_g$, $Y_h\in \H_h$ for which $\d s(X_g)=\d t(Y_h)$, $\d_{(g,h)}m(X_g,Y_h)\in \H_{gh}.$
\item $\H$ is closed under $\d i$, i.e. $\d i(\H_g)=\H_{g^{-1}}.$
\item At units $x=1_x$, $\H_x$ contains $T_xM$. 
\item $\H$ is $s$-transversal, i.e. $T\G=\H+T^s\G$.
\end{enumerate}
The proof of the last condition is analogous to the part of proof of lemma \ref{linear-transversal} where one shows that a linear distribution is $\pi$-transversal. In this case the roles of the zero section $z:M\to F$ and the fiber-wise addition $a:F\times_MF\to F$ are played by the unit map $u:M\to\G$ and the multiplication map $m:\G_2\to \G$, respectively.

Note that the last condition is equivalent to the surjectivity of $ds: \H\to TM$; it actually implies that the last map is not only point-wise surjective, but also a submersion 
(which is necessary for $\H$ to be a Lie groupoid over $TM$) as it is a vector bundle map over $s:\G\to M$
\end{remark}

\begin{remark}\rm
The dual objects of multiplicative distributions are the point-wise surjective multiplicative one forms. This will be treated in subsection \ref{the dual point of view}. 
\end{remark}

\begin{definition}\label{defn:pfaffian_grpds}\mbox{}\begin{itemize}
\item A {\bf Pfaffian groupoid} is a Lie groupoid $\G\tto M$ endowed with a multiplicative distribution $\H\subset T\G$ which is also a Pfaffian distribution with respect to the source map. 
\item A Pfaffian groupoid $(\G,\H)$ is said to be of \textbf{Lie type} (or a \textbf{Lie-Pfaffian} groupoid) if 
\begin{eqnarray*}
\H\cap\ker ds=\H\cap\ker dt.
\end{eqnarray*}
 \end{itemize}
\end{definition}

\begin{remark}\rm In other words, a Pfaffian groupoid is Lie groupoid endowed with an $s$-involutive multiplicative distribution. See remark \ref{mult-equiv}.
\end{remark}

\begin{remark}\label{transversalidad}\rm
Equivalently, one could define a Pfaffian groupoid as a Lie groupoid $\G$ endowed with a multiplicative distribution $\H$ with the property that $\H$ is a Pfaffian distribution with respect to the target map (see chapter \ref{Pfaffian bundles}). That both notions are equivalent comes from the fact that the differential of the inverse map $i:\G\to \G$ maps isomorphically $H$ to $H$, and $T^s\G$ to $T^t\G$. From this one has that the multiplicative distribution $\H$ is Pfaffian with respect to the source map $s:\G\to M$ if and only if it is Pfaffian with respect to the target map $t:\G\to M$.  
\end{remark}

\begin{remark}\rm
As we shall see, via the linearization along the unit map, Pfaffian groupoids correspond to Spencer operators in the sense of chapter \ref{Relative connections on Lie algebroids}. In this correspondence, Lie-Pfaffian groupoids correspond to Lie-Spencer operators (see definition \ref{Lie-Spencer}). This also motivates the terminology of ``Lie-Pfaffian''. See theorem \ref{t2} and proposition \ref{in-dis}.
\end{remark}

\begin{proposition}
Let $(\G,\H)$ be a Lie-Pfaffian groupoid. There exists a Lie algebroid structure on $$E:= A/(\H^s|_M)$$ such that the projection $A\to E$ is a Lie algebroid map.
\end{proposition}

\begin{proof}
Notice that the condition $\H^s=\H^t$ implies that the anchor map $\rho:A\to TM$ descends to the quotient $E\to M$ as $\H^s|_M\subset \ker\rho$. On the other hand, we define the bracket on $E$ on the equivalence class of $\al,\be\in\Gamma(A)$ by
\begin{eqnarray}\label{bra-quo}
[\bar{\al},\bar{\be}]=[\al,\be]\mod \H^s|_M.
\end{eqnarray}
To see that the bracket is well-defined, take $\be\in \Gamma(\H^s|_M)$ and $\al\in\Gamma(A)$. Then for the associated right invariant vector fields $\be^r\in\Gamma(\H^s)(=\Gamma(\H^t))$ and $\al^r\in\X^{\text{inv}}(\G)$,
\begin{eqnarray*}
\begin{split}
[\al^r,\be^r](g)&=\frac{d}{d\eps}|_{\eps=0}d_{\varphi_{\al^r}^\eps(g)}\varphi^{-\eps}_{\al^r}(\be^r_{\varphi_{\al^r}^\eps(g)})\\&=\frac{d}{d\eps}|_{\eps=0}dm(\phi^{-\eps}_{\al}(dt(\be^r_{\varphi_{\al^r}^\eps(g)})),\be^r_{\varphi_{\al^r}^\eps(g)})\\&=\frac{d}{d\eps}|_{\eps=0}dm(0_{t(\varphi_{\al^r}^{-\eps}(g))},\be^r_{\varphi_{\al^r}^\eps(g)})=\frac{d}{d\eps}|_{\eps=0}L_{t(\varphi_{\al^r}^{-\eps}(g))}(\be^r_{\varphi_{\al^r}^\eps(g)}),
\end{split}
\end{eqnarray*}
where $\varphi_{\al^r}^\eps(g)=\phi_\al^{\eps}(t(g))g$ is the flow of $\al^r$ (see remark \ref{flows of sections}), and $L:T^t\G\to T^t\G$, left multiplication, is defined by  $L(v_g)= dm(0_{g^{-1}},v_g)$. As $\H^t$ is closed under left multiplication then $[\al^r,\be^r]\in\Gamma(\H^t)(=\Gamma(\H^s))$. This implies that $[\al,\be]\in\Gamma(\H^s|_M)$ and therefore formula \eqref{bra-quo} is well-defined. 
\end{proof}

\begin{remark}\rm The previous proof shows that any multiplicative distribution $\H$ with the property that $\H^s=\H^t$ is of Pfaffian-type, i.e. $\H$ is $s$-involutive.  
\end{remark}

\begin{remark}\label{618}\rm There is a natural action of a Lie-Pfaffian groupoid $(\G,\H)$ on $TM$, the tangent space of the base manifold: for $g\in\G$,
\begin{eqnarray*}
g:T_{s(g)}M\To T_{t(g)}M,\quad g\cdot X=dt(\tilde X_g),
\end{eqnarray*}
where $\tilde X_g\in \H_g$ is any vector that $s$-projects to $X$, i.e. $ds(\tilde X_g)=X$. Note that this does not depend on the choice of $\tilde X_g$. Indeed, if $X'_g\in \H_g$ is any other such vector then 
\begin{eqnarray*}
\tilde X_g-X'_g\in\H_g^s=\H^t_g\quad\Longrightarrow\quad dt(\tilde X_g)=dt(X'_g).
\end{eqnarray*}  
To check the axioms of the action notice that $TM\hookrightarrow \H$ by remark \ref{mult-equiv}, and therefore $1_x:T_xM\to T_xM$ is the identity. Now, if $(h,g)$ are composable and $X\in T_{s(g)}M$, let $Y_h\in \H_h$ be such that 
\begin{eqnarray*}
ds(Y_h)=dt(\tilde X_g)=g\cdot X.
\end{eqnarray*} 
Take 
$
Z_{hg}=dm(Y_h,\tilde X_g)\in \H_{hg}.
$
Then $ds(Z_{hg})=ds(\tilde X_g)=X$ and therefore
\begin{eqnarray*}
(hg)\cdot X=dt(Z_{hg})=dt(Y_h)=h\cdot(g\cdot X). 
\end{eqnarray*}
\end{remark}

Since a Pfaffian groupoid is a Pfaffian bundle $s:(\G,\H)\to M$, we can apply the notions of Pfaffian bundles as:
\begin{itemize}
\item the symbol space $\g(\H)$ of $\H$ given by the vector bundle over $\G$
\begin{eqnarray*}
\g(\H):=\H^s,
\end{eqnarray*} 
\item the symbol map given by the vector bundle map over $\G$
\begin{eqnarray*}
\partial_\H:\g(\H)\To \hom(s^*TM,T\G/\H),
\end{eqnarray*}
\item the prolongations of $\g(\H)$
\begin{eqnarray*}
\g^{(k)}(\G,\H)\subset s^*(S^kT^*)\otimes (T\G/\H).
\end{eqnarray*}
\end{itemize}
In the case of groupoids, we will see that these objects over $\G$ satisfy some invariance conditions and that they actually come from $M$. For instance, for the symbol space, consider the induced vector bundle over $M$ given by 
\begin{eqnarray*}
\g_M(\H):=\g(\H)|_M\subset T^s\G|_M=A.
\end{eqnarray*}

\begin{lemma}\label{isom}Let $\H\subset T\G$ be a multiplicative distributions. Then 
\begin{eqnarray*}
\g(\H)\simeq t^*\g_M(\H).
\end{eqnarray*}
\end{lemma}

\begin{proof}
While right translations induce an isomorphism of vector bundles over $M$, $R: T^{s}\G \stackrel{\sim}{\To} t^{\ast}A$, $R(X_g)= R_{g^{-1}}(X_g)$,
they restrict to an isomorphism $\H^s\cong t^{\ast} \mathfrak{g}$:
this is a consequence of $\H$ being closed under the differential of the multiplication, and that $R(X_g)=dm(X_g,0_{g^{-1}})$ for any $X_g\in T^s_g\G.$
\end{proof}

In this setting the natural objects to consider as solutions and partial integral elements of the pair $(\G,\H)$ are the following:

\begin{definition} Let $\H\subset T\G$ be a multiplicative distribution.
\begin{itemize}
\item A {\bf solution of $(\G,\H)$} is a bisection $b\in\Bis(\G)$ which is also a solution of $(\G,\H)$ in the sense of definition \ref{def: pfaffian distributions}. the set of solutions of $(\G,\H)$ is denoted by
\begin{eqnarray*}
\Bis(\G,\H).
\end{eqnarray*}
\item A {\bf partial integral element of $\H$} is a linear subspace $V\subset T_g\G$ with the property that $V\subset \H_g$ and 
\begin{eqnarray*}
T_g\G=V\oplus T^s_g\G\qquad{\text{and}}\qquad T_g\G=V\oplus T_g^t\G.
\end{eqnarray*}
\end{itemize}
\end{definition}

\subsection{The dual point of view}\label{the dual point of view}

In general, multiplicative distributions arise as kernels of point-wise surjective multiplicative one forms and this provides a dual point of view on Pfaffian groupoids, which generalizes the following 1-1 correspondence for bundles $R\to M$
\begin{enumerate}
\item point-wise surjective $1$-forms $\theta\in \Omega^1(R, E')$, where $E'$ is some vector bundle over $R$. 
\item distributions $\H$ on $R$, i.e. vector sub-bundles $\H\subset TR$.
\end{enumerate}
In one direction,  $\H= \textrm{Ker}(\theta)$; conversely, $E= TR/\H$ and $\theta$ is the canonical projection. 

It is clear that the kernel $\H$ of any point-wise surjective multiplicative one form $\theta\in\Omega^1(\G,E)$ (with coefficients in some representation $E$) is multiplicative. We explain here the converse.

\begin{proposition}\label{from-theta-H} Let $\G$ be a Lie groupoid. Then for any representation $E$ of $\G$ and any point-wise surjective $E$-valued multiplicative form $\theta\in \Omega^1(\G, t^{\ast}E)$,
\[ \H_{\theta}:= \textrm{Ker}(\theta)\subset T\G\]
is a multiplicative distribution on $\G$. Moreover, any multiplicative distribution arises in this way.
\end{proposition}

We use
\[  \mathfrak{g}:= \H^s|_{M}\subset A \]
from the previous subsection, where $A$ is the Lie algebroid of $\G$. As coefficients we take
\[ \ E:= A/\mathfrak{g} ,\]
As we saw in lemma \ref{isom}, while right translations induce an isomorphism of vector bundles over $M$, $R: T^{s}\G \stackrel{\sim}{\To} t^{\ast}A$, $r(X_g)= R_{g^{-1}}(X_g)$,
 restricts to an isomorphism $\H^s\cong t^{\ast} \mathfrak{g}$ of 
vector bundles over $\G$ 
\[ T\G/\H\simeq T^s\G/\H^s \overset{R}{\To} t^{\ast}(E) .\]
Hence the canonical projection $T\G\to T\G/\H$ can be interpreted as a form
\[ \theta_{\H}\in \Omega^1(\G, t^*E).\]
Finally, there is an induced ``adjoint action'' of $\G$ on $E$: for $g\in \G$, 
\[\Ad^\H_g: E_{s(g)}\To E_{t(g)} ,\ \ \Ad^\H_g(\al \mod \B) = (\Ad_{\sigma_g}\al)\mod \B,\]
where $\sigma_g: T_{s(g)}M \to \H_g\subset T_g\G$ is any splitting of $d_gs$ 
and where $\Ad$ is the adjoint representation 
of $\Jet^1\G$ on $A$ (see section \ref{Jet groupoids and algebroids}). With this,

\begin{proposition}\label{lemma-from-H-to-theta} $E$ is a representation of $\G$ and $\theta_{\H}\in \Omega^1(\G, E)$ is 
multiplicative.
\end{proposition}

\begin{proof} Through this proof we will use that the canonical Cartan form \begin{eqnarray*}\theta_{\mathrm{can}}\in \Omega^1(J^1\G,t^*A)\end{eqnarray*} is multiplicative with respect to the adjoint action. This will be proved in subsection \ref{example: jet groupoids}. 

To see that $\Ad^\H$ is well defined we note that, if $\be \in \mathfrak{g}$, then
\[\Ad_{\sigma_g}\be = R_{g^{-1}}\d m (\sigma_g(\rho(\be)), \be),\]
which belongs to $\mathfrak{g}$ due to $\H$ being multiplicative. Moreover, if $\sigma'_g$ is another splitting of $\d s$ whose image lies in $\H$, then
\[\begin{aligned}
\Ad_{\sigma_g}\al - \Ad_{\sigma'g}\al& =R_{g^{-1}}(\d m (\sigma_g(\rho(\al)), \al) - \d m (\sigma'_g(\rho(\al)), \al))\\
&= R_{g^{-1}}\d m (\sigma_g(\rho(\al)) - \sigma'_g(\rho(\al)), 0_{s(g)}),
\end{aligned}\]
which also belongs to $\mathfrak{g}$, for all $\al \in \Gamma(A)$. It follows that $\Ad^\H_g$ is independent of the choice of splitting $\sigma_g$.

We now show that $\theta_{\H}$ is multiplicative for this representation. Observe that for $\xi \in T_g\G$, if $\tilde{\xi}$ is any lift of $\xi$ to $T_{\sigma_g}\Jet^1\G$, then
\[\theta_g(\xi) = \theta_{\mathrm{can},\sigma_g}(\tilde{\xi})\mod \g,\]
where, again $\sigma_g$ is any splitting of $\d s$ whose image lies in $\H$. Also, since $\H$ is multiplicative, if $\sigma_g$ and $\sigma_h$ are splittings of $\d s$ whose image lie in $\H$, then also the image of $\sigma_g\cdot \sigma_h$ lies in $\H$ (whenever the product is defined). It follows that
\[\begin{aligned}
\theta_{gh}(\d m(\xi_1, \xi_2)) & = (\theta_{\mathrm{can}, \sigma_g\sigma_h}(\d m(\tilde{\xi_1},\tilde{\xi_2})))\mod \g \\
&=(\theta_{\mathrm{can},\sigma_g}(\tilde{\xi_1}) + \Ad_{\sigma_g}\theta_{\mathrm{can},\sigma_h}(\tilde{\xi_2}))\mod \g\\
&=\theta_g(\xi_1) + \Ad^\H_{\sigma_h}(\xi_2).
\end{aligned}\]
\end{proof}

Recall from remark \ref{618} that for a Lie-Pfaffian groupoid there is a canonical action of $\G$ on $TM$.

\begin{proposition}\label{6113}
With the above notation, let $(\G,\H)$ be a Lie-Pfaffian groupoid.
There exists an action of $\G$ on $TM$ which makes $\rho:E\to TM$ $\G$-equivariant with respect to the adjoint action $\Ad^\H$ of $\G$ on $E$.
\end{proposition}

\begin{proof}
Consider the action of $\G$ on $TM$ described in remark \ref{618}. Let $g\in\G$ and $[v]\in E_{s(g)}$, and take $\sigma\in J^1\G$ any element such that $\sigma(T_{s(g)}M)\subset \H_g$. Then
\begin{eqnarray*}
\rho(\Ad^\H_g[v])=\rho(R_{g^{-1}}dm(\sigma(\rho(v)),v))=dt(\sigma_g(\rho(v)))=g\cdot(\rho([v])).
\end{eqnarray*}
\end{proof}


\subsection{Examples}

\begin{example}[Jet groupoids]\label{example: jet groupoids}\rm

Our motivating class of examples of Lie-Pfaffian groupoids comes from the Cartan forms on jet groupoids (and their subgroupoids). They are, of course, the multiplicative version of the Cartan forms for jet bundles (see example \ref{cartandist}). In the case $k=1$, we talk about the multiplicative Cartan form on $J^1\G$
\begin{eqnarray*} \theta_{\mathrm{can}}\in \Omega^1(\Jet^1\G, t^{\ast}A).\end{eqnarray*} 
The Cartan form is the multiplicative form with values on the adjoint representation (see subsection \ref{Jet groupoids and algebroids}). 
Recall that $\theta_{\mathrm{can}}$ is described as follows. Let $\pr:\Jet^1\G\to \G$ be the canonical projection and let $\xi$ be a vector tangent to $\Jet^1\G$ at some point
$\jet^1_xb\in \Jet^1\G$. Then the difference
\[\d_{\jet^1_xb}pr(\xi)-\d_x b\circ\d_{\jet^1_xb}s(\xi) \in T_g\G\] 
lies in $ \ker \d s$, hence it comes from an element in $A_{t(g)}$:
\[\theta_{\mathrm{can}}(\xi)=R_{b(x)^{-1}}(\d_{\jet^1_xb}pr(\xi)-\d_x b\circ\d_{\jet^1_xb}s(\xi))\in A_{t(g)}.\]

\begin{proposition}\label{cartan form} Let $\G$ be a Lie groupoid and $A$ a Lie algebroid over $M$. Then:
\begin{enumerate}
\item The Cartan form $\theta_{\mathrm{can}}\in \Omega^1(\Jet^1\G, t^{\ast}A)$ is a multiplicative form with values in the adjoint representation.
\item If $A= Lie(\G)$, the Lie-Spencer operator of $\theta_{\mathrm{can}}$ (cf. theorem \ref{t1} and example \ref{higher jets on algebroids}) is $D^\clas$.
\end{enumerate}
\end{proposition}

\begin{proof} We first show that $\theta_{\mathrm{can}}$ is multiplicative, i.e. that:
\[(m^{\ast}\theta_{\mathrm{can}})|_{(\sigma_g,\sigma_h)} = \pr_1^{\ast}\theta_{\mathrm{can}} + \Ad_{\sigma_g}\pr_2^{\ast}\theta_{\mathrm{can}}.\]

We use the description from remark \ref{when working with jets}. Let $\xi_1\in T_{\sigma_g}\Jet^1\G$ and $\xi_2\in T_{\sigma_h}\Jet^1\G$ be such that $\d s(\xi_1)=\d t(\xi_2)$. Denote by $X_1 = \d \pr(\xi_1) \in T_g\G$ and $v_1 = \d s(X_1) = \d s(\xi_1) \in T_{s(g)}\G$. Similarly, let $X_2 = \d \pr(\xi_2) \in T_h\G$ and $v_2 = \d s(X_2) = \d s(\xi_2)\in T_{s(h)}M$. Computing $\theta_{\mathrm{can}}(\d m (\xi_1, \xi_2))$ we find
\[\begin{aligned}
&\  R_{(gh)^{-1}}(\d \pr(\d m(\xi_1,\xi_2)) - (\sigma_g\cdot\sigma_h)(\d s (\d m(\xi_1,\xi_2))))= \\
&=R_{(gh)^{-1}}(\d m(X_1,X_2) - (\sigma_g\cdot\sigma_h)(v_2))\\
&=R_{(gh)^{-1}}(\d m(X_1,X_2) - \d m(\sigma_g(\lambda_{\sigma_2}(v_2)), \sigma_h(v_2)))\\
&=R_{(gh)^{-1}}(\d m(X_1-\sigma_g(\lambda_{\sigma_2}(v_2)), X_2-\sigma_h(v_2)))\\
&=R_{g^{-1}}(\d m(X_1-\sigma_g(\lambda_{\sigma_2}(v_2)),R_{h^{-1}}( X_2-\sigma_h(v_2))))\\
&=R_{g^{-1}}(\d m(X_1- \sigma_g(v_1), 0_{s(g)}) + \d m(\sigma_g(v_1)-\sigma_g(\lambda_{\sigma_2}(v_2)),R_{h^{-1}}( X_2-\sigma_h(v_2))))\\
&=R_{g^{-1}}(X_1- \sigma_g(v_1) + \d m(\sigma_g(v_1)-\sigma_g(\lambda_{\sigma_2}(v_2)),R_{h^{-1}}( X_2-\sigma_h(v_2))))\\
&=R_{g^{-1}}(X_1- \sigma_g(v_1)) + \Ad_{\sigma_g}(R_{h^{-1}}( X_2-\sigma_h(v_2)))\\
&=\theta_{\mathrm{can}}(\xi_1) + \Ad_{\sigma_g}\theta_{\mathrm{can}}(\xi_2),
\end{aligned}\] 
where we have used the fact that $\pr:\Jet^1\G\to \G$ is a Lie groupoid morphism.
Let $(D,l)$ denote the Spencer operator of $\theta_{\mathrm{can}}$. It is clear from the definition of $l$ that $l= \pr$ and it suffices to prove that $D$ satisfies the holonomicity condition 
$D(\jet^1\al)=0$, for all $\al \in \Gamma(A)$. Let $\zeta=\jet^1\al$. In the explicit formula \eqref{eq: explicit formula} for $D$, we remark that  $\phi_\zeta^\epsilon(x)=\d_x\phi^\epsilon_\al$, hence
\[\begin{aligned}
\theta_{\mathrm{can}}(\d_x\phi^\epsilon_\zeta(X_x))&=R_{(\phi^\eps_\al(x))^{-1}}(\d \pr(\d_x\phi_{\zeta}^{\eps}(X_x)) - \d_x \phi^{\eps}_\al(\d s(\d_x\phi^\epsilon_\zeta(X_x))))\\
&=R_{(\phi^\eps_\al(x))^{-1}}(\d_x \phi^{\eps}_\al(X_x) - \d_x \phi^{\eps}_\al(X_x)) = 0,
\end{aligned}
\]
Hence $D(\jet^1\al)(X)= 0$. 
\end{proof}

Of course proposition \ref{cartan form} holds for all the Cartan forms
\[ \theta^k \in \Omega^1(J^k\G, t^*\Jet^{k-1}A),\]
on higher jet groupoids. This is easily seen as
 $\theta^k$ is given by the restriction to $J^k\G\subset J^1(J^{k-1}\G)$ of the Cartan form associated to the $(k-1)$-jet groupoid $J^{k-1}\G$. 

On the dual side one recovers the so called Cartan multiplicative distributions on $\mathcal{C}_k\subset TJ^k\G$ defined to be the kernel of the Cartan form $\theta^k$. Alternatively, the Cartan multiplicative distributions on $J^k\G$ can by described as the intersection 
\begin{eqnarray}\label{alter}
C_k\cap TJ^k\G,
\end{eqnarray}
where $C_k$ is the Cartan distribution of the $k$-jet bundle associated to the source map $s:\G\to M$. 

\begin{proposition}
The Lie groupoid $J^k\G$ endowed with the Cartan multiplicative distribution $\mathcal{C}_k\subset TJ^k\G$ is Lie-Pfaffian.
\end{proposition}

\begin{proof}
That $\mathcal{C}_k$ is multiplicative is a consequence of lemma \ref{from-theta-H}  and the fact that it is defined by the kernel the multiplicative form $\theta^k$. That it is $s$-involutive follows from description \eqref{alter}. 
To see that $\mathcal{C}_k^s=\mathcal{C}_k^t$, let $X\in\mathcal{C}_k^s$. Then $dpr(X)=0$ follows by an easy computation of $\theta^k(X)=0$, and therefore \begin{eqnarray*}dt(X)=dt(dpr(X))=0.\end{eqnarray*} This means that $X\in\mathcal{C}_k^t$ which implies that $\mathcal{C}_k^s\subset\mathcal{C}_k^t$. By dimension counting we conclude that $\mathcal{C}_k^s=\mathcal{C}_k^t$.
\end{proof}

\begin{remark}\rm
As a curiosity: the kernel of Lie groupoid map $pr:J^1\G\to\G$ is equal to the bundle of Lie groups $\mathcal{K}$ over $M$, with fiber at $x$ given by
$$\mathcal{K}_x:=\{\phi:T_{x}M\to A_{x}\mid \rho\circ\phi+ id\text{ is an isomorphism}\}$$
with multiplication of two linear maps $\phi,\psi\in \mathcal{K}_x$ given by
\begin{eqnarray*}
\phi\cdot\psi:=\phi\circ\rho\circ\psi+\psi+\phi.
\end{eqnarray*}
Hence, we get an exact sequence of Lie groupoids
\begin{eqnarray*}
\mathcal{K}\hookrightarrow J^1\G\overset{pr}{\To}\G.
\end{eqnarray*}
For $k\geq2$, the situation is even simpler: the kernel of $pr:J^k\G\to J^{k-1}\G$ is isomorphic to the bundle of abelian Lie groups $S^kT^*\otimes A$. With this we have an exact sequence of Lie groupoids
\begin{eqnarray*}
S^kT^*\otimes A\hookrightarrow J^k\G\overset{pr}{\To}J^{k-1}\G.
\end{eqnarray*}
\end{remark}

For later use we state the following result: consider the exact sequence of vector bundles over $J^k\G$
\begin{eqnarray*}
0\To S^kT^*\otimes T^s \G\overset{i}{\To} T^sJ^k\G\overset{dpr}{\To} T^sJ^{k-1}\G\To 0.
\end{eqnarray*}

\begin{lemma}\label{multi}
Let $\gamma,b\in J^k\G$ be a pair of composable arrows.
 Then right translating by $b$ a vector $\Psi\in S^kT_{s(\gamma)}^*\otimes T^s_\gamma \G$ one obtains 
 \begin{eqnarray*}
 R_b(\Psi)\in S^kT_{s(b)}^*\otimes T^s_{\gamma b} \G
 \end{eqnarray*}
 defined at $X_1,\ldots X_k\in T_s(g)$ by
 \begin{eqnarray*}
R_b(\Psi)(X_1,\ldots,X_{k})&=R_{pr(b)}(\Psi(\lambda_b^{-1}(X_{1}),\ldots,\lambda_b^{-1}(X_{k-1}))),
\end{eqnarray*}
where $R_{pr(b)}:T^s_{pr(\gamma)}\G\to T^s_{pr(\gamma b)}\G$ is right translation by $pr(b)$. 
\end{lemma}

\end{example}

\begin{example}[$G$-structures] \label{ex: G-structures}\rm
Recall from example \ref{gauge-groupoids} that any principal $G$-bundle $\pi: P \to M$ has an associated Lie groupoid, called the gauge groupoid, and denoted by $\G\text{auge}(P)\tto M$.
In particular, if $G$ is a subgroup of $\GL_n$, and $$P = \mathrm{F}_G(M)$$ is a $G$-structure on $M$ -- a principal $G$ sub-bundle of the frame bundle of $M$ (the bundle whose fiber over $x$ consists of all linear isomorphisms $p: \Rr^n \to T_xM$), one obtains $\G(\mathrm{F}_G(M)):=\G\text{\rm auge}(\mathrm{F}_G(M))$.
Note that in this case, $\G(\mathrm{F}_G(M))$ is a Lie subgroupoid of $\GL(TM)$ (see example \ref{difeo}). 
Thus, we obtain a point-wise surjective multiplicative form $\tau$ on $\G(\mathrm{F}_G(M)) \subset \GL(TM)$ with values in $TM$, called the \textbf{tautological form}, defined by
\begin{eqnarray*}
\tau=\theta^1|_{\G(\mathrm{F}_G(M))},
\end{eqnarray*} 
where $\theta^1\in\Omega^1(\GL(TM),TM)$ is the multiplicative Cartan form associated to the pair groupoid $\Pi(M).$ Of course the symbol space $\g(\tau)|_M\subset T^*M\otimes TM$ of $\tau$ restricted to $M$ is given by the kernel of 
\begin{eqnarray*}
dpr:T\G(\mathrm{F}_G(M))\To TM
\end{eqnarray*}
at the units. By the construction of $\G(\mathrm{F}_G(M))$, we see that this is just
\begin{eqnarray*}
\g(\tau)|_M=T^*\otimes \g, \quad \g:=Lie(G)\subset \gl_n.
\end{eqnarray*}

The main importance of the tautological form is that it detects solutions of the $G$-structure in the following sense: recall that bisections of $\GL(TM)$ are simply automorphisms of $TM$ (covering a diffeomorphism of $M$). The tautological form satisfies the following important property:
\begin{proposition}\label{prop: tautological form}
Let $\tau \in \Omega^1(\G(\mathrm{F}_G(M)), t^{\ast}TM)$ denote the tautological form. A bisection $\Phi \in \Bis(\G(\mathrm{F}_G(M)))$ is of the form $\Phi = \d \phi$, where $\phi = t \circ \Phi : M \to M$ if and only if $\Phi^{\ast}\tau = 0$. 
\end{proposition}

This of course follows from proposition \ref{cartan form} applied to $J^1\Pi(M) \cong \GL(TM)$.\\

On the infinitesimal side, the Lie algebroid of $\GL(TM)$ is $\ggl(TM)$ (for the definition of this Lie algebroid see remark \ref{rk-algbds}), and thus the Lie algebroid of $\G(\mathrm{F}_G(M))$ is a transitive subalgebroid $A(\mathrm{F}_G(M)) \subset \ggl(TM)$. A direct application of proposition \ref{cartan form} shows that the Spencer operator associated to the tautological form is
\[D_X(\partial) = [\sigma_{\partial},X] - \partial(X), \quad l(\partial) = -\sigma_{\partial},\]
for all $\partial \in \Gamma(A(\mathrm{F}_G(M))$, and $X\in \X(M)$.

\end{example}


\subsection{The integrability theorem for Pfaffian groupoids; Theorem 2}

Let $\H\subset T\G$ be a multiplicative distribution.
The multiplicativity of $\H\subset T\G$ implies that the unit map $u:M\to \G$ is a solution of $(\G,\H)$. The linearization of $(\G,\H)$ along $u$, in the sense of subsection \ref{linearization of distributions}, consists of the Lie algebroid $A$ of $\G$ 
\begin{eqnarray*}
A=L_u(\G):=u^*T^s\G
\end{eqnarray*}
endowed with the connection 
\begin{eqnarray*}
D:\X(M)\times\Gamma(A)\To (E)
\end{eqnarray*}
relative to the projection $A\to E$, where $E$ is the vector bundle over $M$ defined to be
\begin{eqnarray*}
E:=A/\g
\end{eqnarray*}
for $\g=u^*(\H^s)$. Actually, in the Lie groupoid setting, $E$ becomes a representation of $\G$ and $D$ is now a Spencer operator in the sense of chapter \ref{Relative connections on Lie algebroids}. Moreover, if $\G$ is a source simply connected Lie groupoid, $\H$ is determined by $D$ and $\g$.\\

In the rest of this subsection we explain and prove the following integrability result (which is actually a consequence of theorem \ref{t1}). This can also be found in \cite{Maria}.

\begin{theorem}\label{t2} Let $\G\rightrightarrows M$ be a $s$-simply connected Lie groupoid with Lie algebroid $A\to M$. There is a one to one correspondence between 
\begin{enumerate}
\item multiplicative distributions $\H\subset T\G$, 
\item sub-bundles $\g\subset A$ together with a Spencer operator $D$ on $A$ relative to the projection $A\to A/\g$.
\end{enumerate}

In this correspondence, $\g$ is the symbol space of $\H$ and 
\begin{eqnarray*}
D_X\alpha(x)=[\tilde X,\alpha^r]_x\mod \H^s_{1_x} ,
\end{eqnarray*}
where $\tilde X\in \Gamma(\H)\subset\mathfrak{X}(\G)$ is any vector field which is $s$-projectable to $X$ and extends $u_*(X)$ (for $\alpha^r$, see remark \ref{flows of sections}). 
\end{theorem}

\begin{remark}\rm Let $\G$ be a Lie groupoid whose source fibers are not necessarily simply connected.
As noticed before, out of a multiplicative distribution $\H\subset T\G$ one has a Spencer operator as in theorem \ref{t2}. Of course in this case the correspondence may fail to hold. 
\end{remark}

Before going to the proof of theorem \ref{t2}, we extract some properties of $\H$ out of its Spencer operator.

\begin{proposition}\label{in-dis} Let $\G \tto M$ be a (not necessarily s-simply connected) Lie groupoid. Let $\H\subset T\G$ be a multiplicative distribution with $D:\Gamma(A)\to \Omega^1(M,A/\g)$ the associated Spencer operator as in theorem \ref{t2}. Then,
\begin{enumerate}
\item $\H$ is of Pfaffian type (i.e. $\H$ is $s$-involutive) if and only if $\g\subset A$ is a Lie subalgebroid of $A$.
\item $\H$ is of Lie-type (i.e. $\H^s=\H^t$) if and only if $D$ is a Lie Spencer operator.
\end{enumerate} 
\end{proposition}

\begin{proof} 
For item 1 recall that right translation 
\begin{eqnarray}\label{r}
R:(\Gamma(A),[,]_A)\To (\Gamma(T^s\G),[,]_\G),\quad\al\mapsto R(\al)_g=R_g(\al_{t(g)})
\end{eqnarray} 
is an isomorphism of Lie algebras. As $\H$ is multiplicative, $R$ restricts to an isomorphism $R:\Gamma(\g)\to \Gamma(\H^s)$, and therefore  $\Gamma(\g)\subset\Gamma(A)$ is Lie subalgebra (i.e. $\g\subset A$ is a Lie subalgebroid) if and only if its image (i.e. $\Gamma(\H^s)$) is a Lie subalgebra of $\Gamma(T^s\G)$. 

For item 2 suppose first that $\H^s=\H^t$. Then by corollary \ref{6113} the quotient map $pr:A\to A/\g$ is a Lie algebroid morphism and therefore the Spencer operator $D$ becomes a Lie-Spencer operator. For the converse notice that $\g\subset \ker\rho$ as $\rho\circ pr=\rho$. This means 
that $\g=\H^s|_M\subset \H^t|_{M}$ and by dimension counting $\H^s|_M=\H^t|_M$. Now, fixing $g\in\G$ with $x=s(g)$, $R_{g^{-1}}:\H^s_{1_x}\to \H^s_{g^{-1}}$ is an isomorphism. As $\H^t_{1_x}=\H^s_{1_x}$ and $dt\circ R_{g^{-1}}=dt$ one has that $\textrm{im } R_{g^{-1}} \subset \H^t_g\cap \H^s_g$. Again by dimension counting we conclude that $\H^t_g=\H^s_g.$ 
\end{proof}

The next proposition states that we obatin (local) solutions of a non-linear system of PDE's out of a solution of its linearization. Of course, the immediate advantage is to solve the easier infinitesimal (linear) problem to obtain a solution to the more complicated and interesting global problem.

\begin{proposition}Let $\H\subset T\G$ be a multiplicative distribution and let $D$ be the associated Spencer operator. If $\al\in\Gamma(A)$ is a solution of $(A,D)$, then for every $x\in M$,  
\begin{eqnarray*}
\phi^\eps_\al\in\Bis_{\text{\rm loc}}(\G,\H)
\end{eqnarray*}
is a solution around $x$, for $\eps$ such that $\phi^\eps_\al$ is defined (around $x$). See definition \ref{flows of sections}. 
\end{proposition}

\begin{proof}
As we will see in the proof of theorem \ref{t2}, the Spencer operator associated to $\H$ is the one given in theorem \ref{t1} for $\theta:=\theta_\H$. Let $x\in M$ and let $U\subset M$ be a neighborhood of $x$. Take $\eps_0$ to be such that $\phi^{\eps_0}_\al$ is defined on $U$. Consider the curve 
\begin{eqnarray*}
\eps\mapsto c(\eps)\in\Omega^1(U,E|_U),\quad c(\eps)(X_y)=\phi_\al^{\eps}(y)^{-1}\cdot\theta(d_y\phi_\al^\eps(X_y)).
\end{eqnarray*}
By the definition of $D$ (see theorem \ref{t2}), 
\begin{eqnarray*}
\frac{d}{d\eps}c(\eps)|_{\eps=0}=D(\al). 
\end{eqnarray*} 
We will prove that $\frac{d}{d\eps}c(\eps)|_{\eps=r}=0$ for any $0<r<\eps_0$. This, together with the fact that $c(0)=0$, implies, of course, that $c(\eps_0)=0$. Indeed, 
\begin{eqnarray*}
\phi^{\eps+r}_\al(y)=\phi^\eps_\al(t(\phi^r_\al(y)))\cdot \phi^\eps_\al(y).
\end{eqnarray*}
Hence, the differential of $\phi^{\eps+r}_\al$ at $y$ satisfies a similar equation involving the differential of the functions.
Therefore, by multiplicativity of $\theta$,
\begin{eqnarray*}
\theta(d_y\phi^{\eps+r}_\al(X))=\theta(d\phi^\eps_\al(dt(d\phi_\al^r(X))))+ \phi^\eps_\al(t(\phi^r_\al(y)))\cdot \theta(d\phi^r_\al(X)),
\end{eqnarray*}
and from this,
\begin{eqnarray*}\begin{split}
\frac{d}{d\eps}c(\eps+r)|_{\eps=0}&=\frac{d}{d\eps}\phi^r_\al(y)^{-1}\cdot(\phi_\al^{\eps}(t(\phi^r_\al(y)))^{-1}\cdot\theta(d\phi^\eps_\al(dt(d\phi_\al^r(X))))+\theta(d\phi^r_\al(X)))\\&=\phi^r_\al(y)^{-1}\cdot D_{dt(d\phi_\al^r(X)))}\al=0.
\end{split}
\end{eqnarray*}
\end{proof}

Of course, theorem \ref{t2} follows from theorem \ref{t1} applied to $\theta_{\H}$, combined with the 
reformulation of Spencer operators (remark \ref{remark-dual}). What we still have to prove is that the explicit formula (\ref{eq: explicit formula}) for $D$ (from 
theorem \ref{t1}) gives the explicit formula (\ref{D}) (from theorem \ref{t2}). With the right hand side of (\ref{eq: explicit formula}) in mind, we consider more general expressions
of type:
\[(\Lie_{\al}\omega)_g:=\frac{d}{d\eps}\big{|}_{\eps=0}(\varphi^{\eps}_{\al^r}(g))^{-1}\cdot(\varphi^\epsilon_{\al^r})^*\omega|_{\varphi^\epsilon_{\al^r}(g)}\]
for $\alpha\in \Gamma(A)$ and $\omega\in \Omega^k(\G,t^*E)$ (see also remark \ref{flows of sections}). This defines
\[\Lie_{\al}:\Omega^k(\G,t^*E)\To \Omega^k(\G,s^*E).\]

\begin{lemma}\label{commutator} For any vector field $\xi\in \X(\G)$, $\omega \in \Omega^k(\G, t^{\ast}E)$, $g\in \G$:
\[[i_\xi,\Lie_{\al}](\omega)_g=g^{-1}\cdot\omega_g([\xi,\al^r]).\]
\end{lemma}

Note that this implies theorem \ref{t2}. Indeed, if $\tilde{X}$ is as in the statement of theorem \ref{t2}, using the above lemma for $g= 1_x= x$, $\omega= \theta$, since $D(\alpha)(x)= \Lie_{\al}(\theta)(x)$ and $\theta(\tilde{X}_x)= 0$, 
\[ D_{X}(\al)(x)=[i_{\tilde{X}}, \Lie_{\alpha}](\theta)_x= \theta_{x}([\tilde X,\al^r])= \text{[}\tilde{X},\al^r]_{x} \mod \mathfrak{g}.\]

\begin{proof}[Proof of lemma \ref{commutator}]
We apply the chain rule to the composition 
\[\frac{d}{d\eps}\big{|}_{\eps=0}(\varphi^{-\eps}_{\al^r}(g))^{-1}\cdot\omega_{\varphi^{-\eps}_{\al^r}(g)}(\d_g\varphi^{-\eps}_{\al^r}(\xi_g))= f_1\circ f_2,\]
where
\[f_1: I\times s^{-1}(s(g)) \To E_{s(g)}, \qquad f_1(\eps,h)= h^{-1}\cdot\omega_h(\d_{\varphi^{\eps}_{\al^r}(h)}\varphi^{-\eps}_{\al^r}(\xi_{\varphi^{\eps}_{\al^r}(h)})), \]
and
\[f_2: I \To s^{-1}(s(g)), \qquad f_2(\eps) = \varphi^{-\eps}_{\al^r}(g).\]
We obtain,
\begin{equation*}
\begin{aligned}
\frac{d}{d\eps}&\big{|}_{\eps=0}(\varphi^{-\eps}_{\al^r}(g))^{-1}\cdot\omega_{\varphi^{-\eps}_{\al^r}(g)}(\d_g\varphi^{-\eps}_{\al^r}(\xi_g))=\\
=&\frac{d}{d\epsilon}\big{|}_{\eps=0}g^{-1}\cdot\omega_{g}(\d_{\varphi^{\eps}_{\al^r}(g)}\varphi^{-\eps}_{\al^r}(\xi_{\varphi^{\eps}_{\al^r}(g)}))+
\frac{d}{d\eps}\big{|}_{\eps=0}(\varphi^{-\eps}_{\al^r}(g))^{-1}\cdot\omega_{\varphi^{-\eps}_{\al^r}(g)}(\xi_{\varphi^{-\eps}_{\al^r}(g)}),
\end{aligned}
\end{equation*}
or in other words,
\begin{equation*}
-(\Lie_{\al}\omega)(\xi)(g)=g^{-1}\cdot\omega([\al^r,\xi])-\Lie_{\al}(\omega(\xi))(g).
\end{equation*}
i.e. the equation in the statement.
\end{proof}

\section{Lie-Prolongations of Pfaffian groupoids}

In this chapter we will show that the notions introduced in section \ref{prolongations of Pfaffian bundles} become Lie theoretic. We will also introduce the notions of morphism and abstract Lie prolongations of Pfaffian groupoids, notions closely related to the Maurer-Cartan equation! (see subsection \ref{Lie-prolongations and the Maurer-Cartan equation}). We will see the advantages of passing from the more general but ``wild'' picture of Pfaffian bundles to the easier to handle picture of Pfaffian groupoids, whose rich structures will allow us to have a better understanding of the process of prolonging. Throughout the whole exposition we will point out what the analogous notions are in the setting of Spencer operators (see chapter \ref{Relative connections on Lie algebroids}) when taking the Lie functor. 

This can be slightly generalized to the case where $\G$ is endowed with a multiplicative form $\theta$. We will point out 
which results are still valid in this setting.

\subsection{Lie prolongations (abstract prolongations); Theorem 3}

Throughout this section let $\tG$ and $\G$ be two Lie groupoids over $M$;  $$\tt\in\Omega^1(\tG,t^*\tilde E) \quad \text{ and } \quad \theta\in\Omega^1(\G,t^*E),$$ are multiplicative point-wise surjective $1$-forms
with $$\tH:=\ker\tt\subset T\tG\quad\text{ and }\quad\H:=\ker\theta\subset T\G$$ being the multiplicative distributions given by the kernel of $\tt$ and $\theta$ respectively. Assume that $$(\tilde D,\tilde l):\tilde A\to \tilde E \quad \text{ and } \quad (D,l):A\to E$$ are the associated Spencer operators of $\tt$ and $\theta$ respectively.\\ 

Suppose that we have a Lie groupoid map $p:\tG\to \G$ with the property that 
\begin{eqnarray}\label{con1}
dp(\tH)\subset \H.
\end{eqnarray}
Then there is an induced vector bundle map $p_0:\tilde E\to E$ making the diagram
\begin{eqnarray}\label{la}
\xymatrix{
\tilde A \ar[r]^{\tilde l} \ar[d]_{Lie(p)} & \tilde E \ar@{.>}[d]^{p_0}\\
A \ar[r]^l  & E
}
\end{eqnarray}
commutes. Actually, $p_0$ is determined by the formula $$p_0(\tilde l (v))=l(Lie(p)(v)),\quad v\in \tilde A$$ which is well-defined thanks to the condition \eqref{con1} and the definitions of $\tilde l,l$ ($\tilde l=\tt|_{\tilde A}, l=\theta|_{A}$).

\begin{definition} Let $(\tG,\tt)$ and $(\G,\theta)$ be Pfaffian groupoids, with $\tt\in\Omega^1(\tG,t^*\tilde E)$ and $\theta\in\Omega^1(\G,t^*E)$. A {\bf morphism of Pfaffian groupoids}, denoted by
\begin{eqnarray*}
p:(\tG,\tt)\To(\G,\theta),
\end{eqnarray*}
is a Lie groupoid map $p:\tG\to\G$ with the following two properties:
\begin{enumerate}
\item $dp(\tilde \H)\subset \H$, and 
\item for any $X,Y\in \tH$, $p_0(\delta\tt(X,Y))=\delta\theta(dp(X),dp(Y)).$
\end{enumerate}
\end{definition}

The following corollary is immediate from the definition.

\begin{corollary} Let $p:(\tG,\tt)\to(\G,\theta)$ be a morphism of Pfaffian groupoids. Then the morphism of groups $p:\Bis(\tG)\to\Bis(\G)$ restricts to 
\begin{eqnarray*}
p:\Bis(\tG,\tt)\To\Bis(\G,\theta).
\end{eqnarray*}
\end{corollary}

\begin{example}\label{canonical-with-dist}\rm For a multiplicative point-wise surjective form $\theta\in\Omega^1(\G,t^*E)$, with $\H=\ker\theta\subset T\G$, the identity map 
\begin{eqnarray*}
id:(\G,\theta_\H)\To (\G,\theta)
\end{eqnarray*}
is an isomorphism of Pfaffian groupoids, $\theta_\H\in \Omega^1(\G,t^*A/\g)$ being the multiplicative form associated to $\H$ (see subsection \ref{the dual point of view}). \end{example}

Assume now that our Lie groupoid map $p:\tG\to \G$ satisfying the condition \eqref{con1}, has the extra property that 
\begin{eqnarray*}
dp(\tH^s)=0.
\end{eqnarray*}
Hence $\tilde\g:=\tH^s|_M\subset \ker Lie(p)$ and we have an induced vector bundle map $L:\tilde E\to A$, making the diagram 
\begin{eqnarray*}
\xymatrix{
\tilde A \ar[r]^{\tilde l} \ar[d]_{Lie(p)} & \tilde E \ar@{.>}[ld]_{L} \ar[d]^{p_0}\\
A \ar[r]^l  & E
}
\end{eqnarray*}
commute. Actually, $L$  is determined by the well-defined formula
\begin{eqnarray*}
L(\tilde l(v))=Lie(p)(v),\quad v\in \tilde A.
\end{eqnarray*}

The following definition is motivated by the fact that the infinitesimal counterparts of the objects defined below are compatible Spencer operators. See definition \ref{algb-prol} and theorem \ref{prol-comp}.

\begin{definition} Let $(\G,\theta)$ be a Pfaffian groupoid. A {\bf Lie-prolongation} of $(\G,\theta)$ is a Pfaffian groupoid $(\tG,\tt)$ together with a morphism of Pfaffian groupoids  
\begin{eqnarray*}
p:(\tG,\tt)\To(\G,\theta),
\end{eqnarray*}
which is a surjective submersion satisfying:
\begin{enumerate}
\item $dp( \tH^s)=0$,
\item  for any $X,Y\in \tH$, $\delta\theta(dp(X),dp(Y))=0$, and
\item\label{cond4} $L:\tilde E\to A$ is an isomorphism.
\end{enumerate}
\end{definition}

\begin{remark}\rm Under the identification $L:\tilde E\simeq A$, $Lie(p):\tilde A\to A$ and $p_0:\tilde E\to E$ become $\tilde l:\tilde A\to A$ and  $l:A\to E$ respectively. Therefore we are left in the situation where $$p:(\tG,\tt)\To (\G,\theta)$$ is a surjective submersion, with the property that $\tt$ is of Lie-type taking values in the Lie algebroid $A$ of $\G$, and satisfying:
\begin{enumerate}
\item $Lie(p)=\tilde l$,
\item $dp(\tilde \H)\subset \H$, and 
\item for any $X,Y\in \tH$, $\delta\theta(dp(X),dp(Y))=0.$
\end{enumerate}
We will make this more precise in the results below.
\end{remark}

\begin{corollary}
Let $p:(\tG,\tt)\to(\G,\theta)$ be a Lie-prolongation of $(\G,\theta)$. Then $(\tG,\tt)$ is a Lie-Pfaffian groupoid and 
\begin{eqnarray*}
L:A\simeq \tilde E
\end{eqnarray*}
is an isomorphism of Lie algebroids.
\end{corollary}

The previous corollary follows from the next lemma and proposition \ref{in-dis}.

\begin{lemma}
Let $p:(\tG,\tt)\to(\G,\theta)$ be an abstract Lie prolongation of $(\G,\theta)$. Then there exists a unique Lie algebroid structure on $\tilde E$ with the property that $$\tilde l:\tilde A\To \tilde E$$ is a Lie algebroid map. Moreover, the vector bundle map $L:\tilde E\to A$ becomes an isomorphism of Lie algebroids.
\end{lemma}

\begin{proof}
We use the identification $L:\tilde E\to A$ to put an algebroid structure on $\tilde E$. Now, as $Lie(p):\tilde A\to A$ is a Lie algebroid map and as $L\circ \tilde l=Lie(p)$, it is clear that with this structure $\tilde l:\tilde A\to \tilde E$ is a Lie algebroid map. 
\end{proof}

\begin{definition}
We say that 
$$p:(\tG,\tH)\To (\G,\H)$$
is a {\bf Lie prolongation} of $(\G,\H)$ if $p:(\tG,\theta_{\tH})\to (\G,\theta_{\H})$ is a Lie prolongation of $(\G,\theta_{\H})$. See subsection \ref{the dual point of view}.
\end{definition}

\begin{remark}\rm Of course in the previous definitions and results we don't need that our multiplicative distribution $\H\subset T\G$ is $s$-involutive. Actually all the following results are still valid without this condition.
\end{remark}

One result that shows that the notion of Lie-prolongations is the global counterpart of compatible Spencer operators is the following integrability theorem.

\begin{theorem}\label{prol-comp}  If $p:(\tG,\tt)\to (\G,\theta)$ is a Lie-prolongation of $(\G,\theta)$ with $\tt$ taking values in the Lie algebroid $A$ of $\G$, then the associated Spencer operators
\begin{eqnarray*}
\tilde A\xrightarrow{\tilde D,\tilde l}{}A\xrightarrow{D,l} E
\end{eqnarray*} 
are compatible. 

If $\tG$ is source-connected and $p:(\tG,\tt)\to (\G,\theta)$ is a surjective submersion with the property that $Lie(p)=\tt|_A$, then the converse also holds.
\end{theorem}

See subsection \ref{Proofs} for the proof of theorem \ref{prol-comp}.\\

Actually, as a consequence of theorems \ref{t1} and \ref{prol-comp}, theorem 3 says that under some topological conditions, there is a correspondence between Lie-prolongations and compatible Spencer operators. More precisely,

\begin{theorem}\label{siguiente1} Let $\tG$ and $\G$ two Lie groupoid over $M$ with $\tG$ $s$-simply connected and $\G$ $s$-connected, with Lie algebroids $\tilde A$ and $A$ respectively. Let  $\theta\in\Omega^1(\G,t^*E)$ be a point-wise surjective multiplicative form and denote by $(D,l):A\to E$ the Spencer operator associated to $\theta$. There is a 1-1 correspondence between:
\begin{enumerate}
\item\label{a1} Lie-prolongations $p:(\tG,\tt)\to (\G,\theta)$ of $(\G,\theta)$, and
\item\label{a2} Lie-Spencer operators $(\tilde D,\tilde l):\tilde A\to A$ compatible with $(D,l):A\to E$.
\end{enumerate} 
In this correspondence, $\tilde D$ is the associated Lie-Spencer operator of $\tt$. 
\end{theorem}

\begin{proof}
Theorem \ref{prol-comp} establishes the correspondence going from \eqref{a1} to \eqref{a2}. Conversely, applying theorem \ref{t1} we get a a point-wise surjective multiplicative form $\tt\in\Omega^1(\tG,t^*A)$ with the property that $\tt|_{\tilde A}=\tilde l$. Moreover, integrating the Lie algebroid map $\tilde l:\tilde A\to A$ we get a Lie groupoid map $p:\tG\to \G$. Since $\tilde l$ is surjective then $p$ is a surjective submersion onto the $s$-connected component of $\G$ (which is actually $\G$). As $Lie(p)=\tt|_{\tilde A}$ we can apply \ref{prol-comp} to conclude that the compatibility of $(\tilde D,D)$ implies that $p:(\tG,\tt)\to (\G,\theta)$ is a Lie-prolongation.
\end{proof}

\begin{remark}\rm For $p:(\tG,\tt)\to (\G,\theta)$ a Lie-prolongation of $(\G,\theta)$, we have a canonical Lie groupoid morphism
\begin{eqnarray*}
j_{\tt}:\tG\To J^1\G,
\end{eqnarray*}
where $j_{\tt}(g):T_{s(g)}M\to T_{p(g)}\G$ is defined by the equation
\begin{eqnarray*}
j_{\tt}(g)(X)=dp(\tilde X_g),
\end{eqnarray*}
for any vector $\tilde X_g\in \tH_g$ with the property that $ds(\tilde X_g)=X$.\end{remark}

\begin{corollary}\label{628}Let $p:(\tG,\tt)\to (\G,\theta)$ be a Lie-prolongation of $(\G,\theta)$. Then there exists a unique Lie groupoid map
\begin{eqnarray*}
j:\tG\To J^1\G
\end{eqnarray*}
with the properties that $pr\circ j=p$ and
\begin{eqnarray*}
  j^*\theta^1=\tt,
\end{eqnarray*}
where $\theta^1\in \Omega^1(J^1\G,t^*A)$ is the Cartan form.
\end{corollary}

\begin{proof}
Take $j=j_{\tt}$. Let's first show that $j$ is well-defined, i.e.
 $j(g)(X)$ does not depend on the choice of $\tilde X_g$. Indeed, if $X'\in\tH_g$ is any other such vector, then $\tilde X-X'\in \tH^s$ and therefore 
\begin{eqnarray*}
dp(\tilde X-X')=\tt(\tilde X-X')=0.
\end{eqnarray*}
On the other hand, 
\begin{eqnarray*}
dt(j(g)(X))=dt(dp(\tilde X_g))=dt(\tilde X_g)=g\cdot X,
\end{eqnarray*}
where $g:T_{s(g)}M\to T_{t(g)}$ is the action of $\tG$ on $TM$ (see remark \ref{618}), which shows that $dt\circ j(g)$ is an isomorphism and therefore $j(g)$ is indeed an element of $J^1\G.$ It's left to the reader to show that $j$ is Lie groupoid morphism. Let's see now that $j^*\theta^1=\theta$. For $X\in T_g\tG$, we have 
\begin{eqnarray}\label{eq11}\begin{split}
j^*\theta^1(X)&=r_{p(g)^{-1}}(dpr(dj(X))-j(g)(ds(dj(X))))\\
&=r_{p(g)^{-1}}(dpr(X)-j(g)(ds(X))),
\end{split}
\end{eqnarray}
where we used that $pr\circ j=p$. Now, for $\tilde X\in\tH_g=\ker\tt_g$ $s$-projectable to $ds(X)$,
\begin{eqnarray}\label{eq2}
j(g)(ds(X)):=dp(\tilde X)=dp(X)-r_{p(g)}\tt(X)
\end{eqnarray}
as $dp(\tilde X-X)=r_{p(g)}\tt(\tilde X-X)=-r_{p(g)}\tt(X)$. Plugging in equation \eqref{eq2} into equation \eqref{eq11} we get that $j^*\theta^1(X)=\tt_g(X)$. 

For the uniqueness notice that for any other such Lie groupoid map $j':\tG\to J^1\G$, then for any $X\in T_gT\G$,
\begin{eqnarray*}
0=j'^*\theta^1(X_g)-j^*\theta^1(X_g)=j'(g)(ds(X))-j(g)(ds(X)).
\end{eqnarray*}
\end{proof}

\begin{remark}\label{generali}\rm The previous proof shows $j=j_{\tt}$ and that for any $X\in T_{s(g)}M$
\begin{eqnarray*}
j_{\tt}(g)(X)=dp(\tilde X_g)-r_{p(g)^{-1}}\tt(\tilde X_g),
\end{eqnarray*}
where $\tilde X_g\in T_g\tG$ is any vector (not necessarily in $\tH_g$) $s$-projectable to $X$. 
\end{remark}

\subsection{Lie prolongations and the Maurer-Cartan equation; Theorem 4}\label{Lie-prolongations and the Maurer-Cartan equation}

Throughout this section let $\G$ be a Lie groupoid over $M$ and let $\theta\in\Omega^1(\G,t^*E)$ be a multiplicative form. Denote by
\begin{eqnarray*}
D:\Gamma(A)\To\Omega^1(M,t^*E),
\end{eqnarray*} 
the Spencer operator associated to $\theta$, relative to the vector bundle map $l:A\to E$.

Out of the $l$-Spencer operator $D$, one can construct an antisymmetric, $C^{\infty}(M)$-bilinear map
\begin{eqnarray*}
\{\cdot,\cdot\}_D:A\times A\To E
\end{eqnarray*}
defined at the level of section by 
\begin{eqnarray*}
\frac{1}{2}\{\al,\be\}_D:=D_{\rho(\al)}(\be)-D_{\rho(\be)}(\al)-l[\al,\be]
\end{eqnarray*}
where $\al,\be\in\Gamma(A)$. On the other hand, one has a differential operator relative to $D$
\begin{eqnarray*}
d_D:\Omega^*(\tG,t^*A)\To\Omega^{*+1}(\tG,t^*E)
\end{eqnarray*}
explicitely defined by
\begin{equation}\begin{aligned}
d_D\omega(X_0,\ldots,&X_k)=\sum_{i}(-1)^{i}D^t_{X_i}(\omega(X_0,\ldots,\hat{X_i},\ldots,X_k))\\&+\sum_{i<j}(-1)^{i+j}l(\omega([X_i,X_j],X_0,\ldots,\hat{X_i},\ldots,\hat{X_j},\ldots,X_k)),
\end{aligned}\end{equation}
where $D^t:\Gamma(t^*A)\to \Omega^1(\tG,t^*E)$ is the pullback of $D$ via the target map $t:\tG\to M$ (see subsection \ref{first examples and operations} for the precise definition of $D^t$).

\begin{definition}
Let $\tG$ be a Lie groupoid over $M$ and $\tt\in\Omega^1(\tG,t^*A)$ a multiplicative form. We denote by $MC(\tt,\theta)\in\Omega^2(\tG,t^*E)$ the two form given by 
\begin{eqnarray}\label{MC}
MC(\tt,\theta):=d_{D}\tt-\frac{1}{2}\{\tt,\tt\}_{D}
\end{eqnarray}
We say that the pair $(\tt,\theta)$ satisfies the {\bf Maurer-Cartan equation} if 
\begin{eqnarray*}
MC(\tt,\theta)=0.
\end{eqnarray*}
\end{definition}

\begin{remark}\rm
In general, in order to obtain a two form as in \eqref{MC}, one only needs a one form $\tt\in\Omega^1(R,t^*A)$, with $t:R\to M$ a surjective submersion, $A\to M$ a Lie algebroid, and $D:\Gamma(A)\to \Omega^1(M,E)$ a relative connection.\end{remark}

Theorem 4 shows the relation between Lie-prolongations and the Maurer-Cartan equation. Roughly speaking they are they same thing:

\begin{theorem}\label{theorem-mc}
Let $p:\tG\to\G$ be a Lie groupoid map which is also a surjective submersion and let $\tt\in\Omega^1(\tG,t^*A)$ be a multiplicative form with $A=Lie(\G)$. If  
 $$p:(\tG,\tt)\To(\G,\theta)$$
 is a Lie prolongation of the Pfaffian groupoid $(\G,\theta)$ then the pair $(\theta,\tt)$ satisfies the Maurer-Cartan equation.
 
 If $\tG$ is source connected and $Lie(p)=\tt|_{\tilde A}$, then the converse also holds. 
\end{theorem}

See subsection \ref{Proofs} for the proof.


 
\subsection{The partial Lie prolongation}
 
 Let $\H\subset T\G$ be a multiplicative distribution.
The partial Lie prolongation of $(\G,\H)$, denoted by $J^1_\H\G$, is our first attempt to construct a Lie-prolongation of the Pfaffian groupoid $(\G,\H)$. However, the compatibility conditions for a Lie-prolongation will hold for a `smaller' subgroupoid of $J^1_\H\G$ as we will see in the next subsection.
Explicitly, $J^1_\H\G$ is the subspace of the jet groupoid $J^1\G$ given by
\begin{eqnarray}\label{J^1_HG}\Jet^1_\H\G=\{\sigma_g\in \Jet^1\G\mid \sigma_g:T_{s(g)}M\to \H_g\subset T_g\G\}.\end{eqnarray}

\begin{proposition}\label{prop421}
Let $\H\subset T\G$ be a multiplicative distribution. Then,
\begin{enumerate} 
\item $J^1_\H\G\subset J^1\G$ is a Lie subgroupoid and $pr:J^1_\H\G\to\G$ is a surjective submersion,
\item the Lie algebroid of $J^1_\H\G$ is the partial prolongation $J^1_DA\subset J^1A$, where $D:\Gamma(A)\to \Omega^1(M,E)$ is the Spencer operator associated to $\theta$ (see definition \ref{partial prolongation} and theorem \ref{t2}).
\end{enumerate}
\end{proposition}

The proof is a consequence of the following lemma. Let $E$ be representation of $\G$ and consider  
\begin{eqnarray*}
\hom(TM,E)\in\Rep(J^1\G)
\end{eqnarray*}
with the tensor product action of $J^1\G$ on $TM$ by the canonical adjoint action, and on $E$ by the pullback via $pr:J^1\G\to \G$ of the action of $\G$ on $E$.

\begin{lemma}\label{s}
Let $\theta\in\Omega^1(\G,t^*E)$ be a multiplicative form. Then the map
\begin{eqnarray}\label{ev}
ev:J^1\G\To s^*\hom(TM,E),\quad j^1_xb\mapsto b(x)^{-1}\cdot (b^*\theta)_x
\end{eqnarray}
is a cocycle and its linearization is equal to vector bundle map $a:J^1A\to \hom(TM,E)$ defined in lemma \ref{J^1_DA}. More precisely, 
\begin{eqnarray*}
ev(j^1_yb_2\cdot j^1_xb_1)=(j^1_xb_1)^{-1}\cdot ev(j^1_yb_2)+ev(j^1_xb_1)
\end{eqnarray*}
for any pair of composable arrows $j^1_xb_1,j^1_yb_2\in J^1\G$, and 
\begin{eqnarray}\label{Lie(ev)}
Lie(ev)(j^1_x\al)(X)=D_X(\al)
\end{eqnarray}
for $j^1_x\al \in J^1A$ and $X\in T_xM.$
\end{lemma}

\begin{proof}
We use the description of $J^1\G$ given in remark \ref{when working with jets}. Let $\sigma_h,\sigma_g\in J^1\G$ be a pair of composable arrows. Then,
\begin{eqnarray*}
\begin{split}
ev(\sigma_h\cdot\sigma_g)(X)&=(hg)^{-1}\cdot \theta(dm(\sigma_h(\lambda_{\sigma_g}(X)),\sigma_g(X)))\\
&=(hg)^{-1}\cdot(\theta(\sigma_h(\lambda_{\sigma_g}(X)))+h\cdot\theta(\sigma_g(X)))\\
&=g^{-1}\cdot h^{-1}\theta(\sigma_h(\lambda_{\sigma_g}(X)))+g^{-1}\cdot\theta(\sigma_g(X))\\
&=\sigma_g\cdot ev(\sigma_h)(X)+ev(\sigma_g)(X)
\end{split}
\end{eqnarray*} 
for any $X\in T_{s(g)}M$. 

Let's compute $Lie (ev)(j^1_x\al)$. Consider the flow $\phi^\epsilon_{\alpha^r}:\G\to\G$ of the right invariant vector field $\alpha^r$ induced by $\alpha$. Denote by $\phi^\epsilon_\alpha:M\to \G$ the bisection defined by $\phi^\epsilon_\alpha:=\phi^\epsilon_{\alpha^r}\circ\u$. Then, $\epsilon\mapsto \d_x\phi_\alpha^\epsilon :T_xM\to T_{\phi_\alpha^\epsilon(x)}\G$ is a curve on $\Jet^1\G$ lying inside the source fiber at $x$ such that   
\begin{eqnarray*}
\jet^1_x\alpha=\frac{d}{d\epsilon}\d_x\phi_\alpha^\epsilon|_{\epsilon=0}
\end{eqnarray*}
By definition of $Lie(ev)$ we get that $Lie(ev)(\jet_x^1\alpha)(X)$ is given by
\begin{eqnarray*}
\frac{d}{d\epsilon}ev(\d_x\phi_\alpha^\epsilon)(X)|_{\epsilon=0}=\frac{d}{d\epsilon}(\phi_\alpha^\epsilon(x))^{-1}\cdot\theta_{\phi_\alpha^\epsilon(x)}(\d_x\phi_\alpha^\epsilon(X))|_{\epsilon=0}=D_X(\alpha).
\end{eqnarray*}
\end{proof}

\begin{proof}[Proof of Proposition \ref{prop421}] Since $J^1_\H\G\subset J^1\G$ is the kernel of $ev$, a cocycle on $J^1\G$, then it is a subgroupoid. The smoothness of $J^1_\H\G$ follows from the fact that it is the intersection a smooth submanifold of the first jet bundle of the source map of $s:\G\to M$, namely the partial prolongations of $\H$ with respect to the source map (see definition \ref{part-prol} and lemma \ref{smothness}), with $J^1\G$ which is an open subset of the first jet bundle of the source map. Now, that $pr:J^1_\H\G\to \G$ is surjective is clear. That it is a submersion follows from the fact that $J^1_\H\G$ is an open set of the partial prolongations of $\H$ with respect to the source map and lemma \ref{smothness}. 

That $J^1_DA\subset J^1A$ is the algebroid of $J^1_\H\G$ follows from the fact that $J^1_DA$ is the kernel of $Lie(ev)$ (see remark \ref{cocyclea}).
\end{proof}

\begin{remark}\rm For a general multiplicative one form $\theta\in\Omega^1(\G,t^*E)$, the subspace $J^1_\theta\G\subset J^1\G$ defined by 
\begin{eqnarray*}
J^1_\theta\G=\{j^1_xb\in J^1\G\mid (b^*\theta)_x=0\}
\end{eqnarray*}
still makes sense and it is again the kernel of the cocycle $ev$ defined exactly in the same way as in \eqref{ev}. Therefore, $J^1_\theta\G\subset J^1\G$ is again a subgroupoid (but it may fail to be smooth), and the linearization $Lie(ev)$ of $ev$ is again given by formula \eqref{Lie(ev)}, where in this case, $D:\Gamma(A)\to\Omega^1(M,E)$ is defined to be the Spencer operator associated to $\theta$ given in theorem \ref{t1}.
\end{remark}

\begin{definition}Let $\H\subset T\G$ be a multiplicative distribution.
The {\bf partial Lie prolongation of $(\G,\H)$} is the Lie groupoid $J^1_\H\G$ endowed with the Lie-Pfaffian multiplicative distribution $\H^{(1)}\subset TJ^1_\H\G$ 
\begin{eqnarray*}
\H^{(1)}:=\mathcal{C}_1\cap TJ^1_\H\G,
\end{eqnarray*}
where $\mathcal{C}_1\subset TJ^1\G$ is the multiplicative Cartan distribution (see example \ref{example: jet groupoids}).
\end{definition}

That $\H^{(1)}$ is indeed a Pfaffian multiplicative form is a consequence of proposition \ref{proposition1} and subsection \ref{Jet groupoids and algebroids}, as it follows that $J^1_\H\G$ is an open subset of the partial prolongation of $s:(\G,H)\to M$ where $H:=\H$ as a Pfaffian distribution. As $\H^{(1)}=H^{(1)}|_{J^1_\H\G}$ by equation \eqref{alter}, and $H^{(1)}$ is a Pfaffian distribution, it follows that $\H$ is itself a Pfaffian distribution. That $\H^{(1)}$ is multiplicative of Lie-type follows from the definition.

One obtains one of the main properties of $J^1_\H\G$ (see proposition \ref{proposition1}):

\begin{proposition} For a multiplicative distribution $\H\subset T\G$ the morphism of groups
\begin{eqnarray*}
pr:\Bis(J^1_\H\G,\H^{(1)})\To \Bis(\G, \H)
\end{eqnarray*}
is a bijection with inverse $j^1:\Bis(\G, \H)\to\Bis(J^1_\H\G,\H^{(1)})$.
\end{proposition}

\begin{remark}\rm Related to the abstract notion of Lie-prolongation, one has that if $p:(\tG,\tH)\to (\G,\H)$ is a Lie-prolongation of $(\G,\H)$, then there is an induced Lie groupoid map
\begin{eqnarray*}
j_{\tH}:\tG\To J^1_{\H}\G
\end{eqnarray*}
 such that $j_{\tH}^*\theta^{(1)}=\theta_{\tH}$. See corollary \ref{628} to conclude that $\Im j_{\tH}\subset J^1_{\H}\G$.
\end{remark}

\subsection{The classical Lie prolongation: its groupoid structure}

The classical Lie prolongation of a Pfaffian groupoid $(\G,\H)$ is, under some smoothness assumptions, a Lie-prolongation of $(\G,\H)$. Again, it may be thought as the complete infinitesimal data of solutions $\H$, and it also satisfies the universal property stated in proposition \ref{prop-un}. Also we can apply all the theory and notions developed in subsection \ref{the-classical-prolongation} for Pfaffian bundles, but once again, the objects become Lie theoretic and simpler thanks to the multiplicative structure. See for example \cite{Gold3} for the concept of prolongation in the setting of partial differential equations.\\

Closely related to the notion of Lie-prolongation and the involutivity of $\H$, we recall that, using $\theta_\H$ to identify $T\G/\H$ with $t^*E$, we get the bracket modulo $\H$:

\begin{definition}Let $\H\subset T\G$ be a multiplicative distribution. The {\bf curvature map} is defined by
\[\label{bra} c_{\H}: \H\times \H \To t^*E, \ c_{\H}(X, Y)= \theta_\H([X, Y]).\]
\end{definition}

See subsection \ref{the dual point of view} and definition \ref{curvature-map} in the case of Pfaffian bundles.

 \begin{definition}The {\bf classical Lie-prolongation space of $\H$}, denoted by
 \begin{eqnarray*}
 P_{\H}(\G)\subset J^1_\H\G\subset J^1\G
 \end{eqnarray*}
 is defined by 
 \begin{eqnarray*}
 P_\H(\G):=\{j^1_xb\in J^1\G\mid d_xb(T_xM)\subset \H_{b(x)}\text{ and }c_\H(d_xb(\cdot),d_xb(\cdot))=0\}.
 \end{eqnarray*}
 We say that $P_\H(\G)$ is {\bf smooth} if it is a smooth submanifold of $J^1\G$, and that it is {\bf smoothly defined} if, moreover, $pr:P_\H(\G)\to \G$ is a surjective submersion.
\end{definition} 

\begin{remark}\label{open2}\rm By subsection \ref{Jet groupoids and algebroids} and the previous definition one has that $P_\H(\G)$ is an open set of $P_H(\G),$ the partial prolongation of $H:=\H$ with respect to the source map $s:\G\to M$, given by the intersection of 
$P_H(\G)$ with $J^1\G.$
\end{remark}

\begin{proposition}\label{yase} Let $\H\subset T\G$ be a multiplicative distribution. Then 
$P_\H(\G)$ is a subgroupoid of $J^1\G$. Moreover, if $P_\H(\G)$ is smooth then 
\begin{eqnarray*}
Lie(P_\H(\G))=P_D(A),
\end{eqnarray*}  
where $P_D(A)$ is the classical prolongation space of the Spencer operator $D$ associated to $\H$ (see theorem \ref{t2} and subsection \ref{basic notions}).
\end{proposition}

To prove proposition \ref{yase} we will need to introduce the {\bf $1$-curvature map}, which in this setting becomes a $1$-cocycle. This, of course, comes from the $1$-curvature map in the setting of Pfaffian bundles (see definition \ref{primer}), but in this setting all the spaces involved are simpler. Of course there is a price to pay: some formulas become more complicated due to the identifications of the spaces.

\begin{definition} Let $\H\subset T\G$ be a multiplicative distribution. The {\bf $1$-curvature map} 
\[ c_1:  \Jet^{1}_{\H}\G \To s^{\ast} \hom(\Lambda^2TM, E)\]
is given by
\[ c_1(\sigma_g)
(X_x, Y_x)= \Ad^{\H}_{g^{-1}} c^{\H}_{g} (\sigma_g(X_x), \sigma_g(Y_{x}))\]
 for $\sigma_g\in \Jet^{1}_{\H}\G$ with $s(g)= x$, $X_x, Y_x\in T_xM$.
 \end{definition} 

The following lemma follows from the definitions.

\begin{lemma}\label{lemma:102984}Let $\H\subset T\G$ be a multiplicative distribution. Then
\begin{eqnarray*}
P_\H(\G)=\ker c_1.
\end{eqnarray*}
\end{lemma}

In the statement of proposition \ref{yase}, the fact that $P_{\H}(\G)$ is a subgroupoid of $\Jet^1_{\H}\G$  follows from $P_\H\G$ being the kernel of a cocycle (see lemmas \ref{lemma:102984} and \ref{cocycle}). Hence it remains to prove that $c_1$ is indeed a $1$-cocycle on $\Jet^{1}_{\H}\G$.  
The action of $\Jet^1_\H\G$ on  $\hom(\wedge^2TM, E)$ is the one induced by the representations $\lambda$, and $\Ad^{\H} \circ \pr$ on $TM$ and on $E$ respectively:
for $\sigma_{g}\in \Jet^1_\H\G$, with $s(g)= x$, $t(g)= y$ and for $T_x\in \hom(\Lambda^{2}T_{x}M, E_x)$, 
\[ g(T_x)\in \hom(\Lambda^{2}T_{y}M, E_y),\ \ g(T_x)(X_y, Y_y)= \Ad^{\H}_g T_x(\lambda_{g}^{-1}X_y, \lambda_{g}^{-1}Y_y).\]

\begin{lemma}\label{cocycle}
The map $c_1: \Jet^1_\H\G \to s^{\ast}\hom(\wedge^2TM, E)$ is a cocycle.
\end{lemma}

For the proof we need the following lemma:

\begin{lemma}\label{lemma: delta theta is multiplicative}
Let $\H\subset T\G$ by a multiplicative distribution. The curvature map $c_{\H}: \H \times\H \to t^{\ast}E$ satisfies
\[ c_{\H}(\d m(\xi_1, \xi_2), \d m(\xi_1',\xi_2')) =  c_{\H}(\xi_1,\xi_1')  + \Ad^\H_g c_{\H}(\xi_2,\xi_2'),\]
where $\xi_1,\xi_1'\in\H_g$, and $\xi_2,\xi_2'\in \H_h$ are such that $\d t(\xi_2) = \d s(\xi_1)$, and similarly $\d t(\xi_2') = \d s(\xi_1')$ 
\end{lemma}

\begin{proof} In general, for a fiber-wise surjective $1$-form $u\in \Omega^1(P, F)$, denote by $I_{u}$ the resulting bilinear form
$c_u$ on $K_u:= \textrm{Ker}(u)$. If $f: Q\to P$ is a submersion, then $c_{f^{\ast}u}= f^{\ast}(c_u)$, i.e. 
\[ c_{f^*u}(X, Y)= c_u (df(X), df(Y))\ \ \ \ (X, Y\in K_{f^*u}= (df)^{-1}(K_u)).\]
This is the content of lemma \ref{delta-commutes}.
In particular, $c_{m^{\ast}\theta}= m^{\ast}c_{\theta}$, $c_{\pr_{1}^{\ast}\theta}= \pr_{1}^{\ast}c_{\theta}$. A variation of this argument also gives
$c_{g^{-1}\pr_{2}^{\ast}\theta}= g^{-1}\pr_{2}^{\ast}c_{\theta}$, where $g^{-1}$ refers to the multiplication by the inverse of the first component on $E$.
Another general remark is that, for $u, v\in \Omega^1(P, E)$, $c_{u+v}= c_u+ c_v$ on $K_u\cap K_v$. Putting everything together we find that 
\[ m^{\ast}(c_{\theta}) = \pr_{1}^{\ast}(c_{\theta})+ g^{-1} \pr_{2}^{\ast}(c_{\theta})\]
on all pairs $(U, V)$ of vectors tangent to $\G_{2}$ with 
\[ U, V\in (d\pr_1)^{-1}(\H)\cap (d\pr_2)^{-1}(\H)\]
from which the lemma follows. 
\end{proof}

\begin{proof}[Proof of lemma \ref{cocycle}]
Let $\sigma_g, \sigma_h\in \Jet^{1}_{\H}\G$ composable, $X, Y\in T_{s(h)}M$. Using the formula describing the composition $\sigma_g\cdot\sigma_h$, i.e. applying (\ref{mult-i-J1}) for $u= X$ and $u= Y$ and then applying the multiplicativity of $c_{\H}$ from Lemma \ref{lemma: delta theta is multiplicative}, we find
\[\begin{aligned}
c_{gh}(\sigma_g\cdot\sigma_h(X), \sigma_g\cdot\sigma_h(Y)) & = c_g(\sigma_g(\lambda_h(X)), \sigma_g(\lambda_h(Y))) + \\ 
                                                                                                  & + \Ad^\H_g c_h(\sigma_h(X), \sigma_h(Y)).
\end{aligned}\]
Rewriting this in terms of $c_1$, we obtain, after also applying $\Ad^H_{h^{-1}g^{-1}}$,
\[ c_1(\sigma_g\cdot\sigma_h)(X, Y)= \Ad^\H_{h^{-1}} c_1(\sigma_g)(\lambda_h(X), \lambda_h(Y))+ c_1(\sigma_h)(X, Y),\]
i.e. the cocycle condition $c_1(\sigma_g\cdot\sigma_h)= \Ad^\H_{h^{-1}}(c_1(\sigma_g))+ c_1(\sigma_h)$.
\end{proof}

Of course, the next step is to linearize $c_1$. Hence we pass to the Lie algebroid $\Jet^{1}_{D}A$ of $\Jet^{1}_{\H}\G$ where $D:\Gamma(A)\to\Omega^1(M,E)$ is the Spencer operator associated to $\H$. See theorem \ref{t2}. 

\begin{lemma}\label{curvatura} The linearization of the cocycle $c_1$ is equal to the curvature map
\[ \varkappa_D: \Jet^{1}_{D}A \To \hom(\Lambda^2TM, E).\]
See definition \ref{varkappa-map}.
\end{lemma}

To conclude the proof of proposition \ref{yase}, it suffices to check that $Lie(c_1)=\varkappa_D$ as
$$ Lie(P_{\H} \G) = Lie(\ker (Lie(c_1))) = Lie(\varkappa_D) = P_{D}(A).$$
This is precisely the content of lemma \ref{curvatura}. 

\begin{proof} 
For the proof, it is better to consider
\[\tilde\theta_g = \Ad^\H_{g^{-1}}(\theta_\H)_g \in \Omega^1(\G,s^{\ast}E). \]
We claim that, for any connection $\nabla$ on $E$, using the induced derivative operator $d_{\nabla}$, one has:
\begin{equation}\label{help-I} 
c_1(\sigma_g)(X, Y)= d_{\nabla}\tilde\theta (\sigma_g(X), \sigma_g(Y)).
\end{equation}
Indeed, using the definition of $c_1$ and $\tilde\theta$, this reduces to $d_{\nabla}\tilde\theta (X, Y)= \tilde\theta ([X, Y])$ for all $X, Y\in \textrm{Ker}(\theta)$, which is clear.

We now compute the linearization $\varkappa_D$. Let $(\alpha, \omega)$ representing a section $\zeta$ of $\Jet^{1}_{D}A$.
From the definition of $\varkappa_D$, 
\[ \varkappa_D(\alpha, \omega)(x) = (dc_1)_{1_x}(\zeta_x).\]
Note that, thinking of elements of $\Jet^{1}_{D}A$ in terms of splittings, 
\[ \zeta_x= \d_x\al -\omega_x: T_xM\To T_{\alpha(x)}A, \ X\mapsto (\d \alpha)_{x}(X)- \omega_x(X),\]
where $\omega_x(X)\in A_x$ is viewed inside $T_{\alpha(x)}A$ by the natural inclusion
\[ A_x\hookrightarrow T_{\alpha(x)}A, \ v\mapsto \frac{d}{d\eps}\big{|}_{\eps=0}(\alpha(x)+ \epsilon v).\]
To compute $(dc)(\d_x\al +x\omega_x)$, we will consider the curve $\sigma^\eps_g: I \to s^{-1}(x) \subset \Jet^1\G$ given by
\[\sigma^\eps(X_x) = \d_x\phi^\eps_{\al}(X_x - \eps\omega(X_x)),\]
for all $X_x \in T_xM$ and $\eps$ small enough (for $\phi_{\alpha}^{\epsilon}$ see remark \ref{flows of sections}). Note that $\sigma^\eps$ is a curve in the $s$-fiber of $\Jet^1\G$ (not necessarily in $\Jet^1_\H\G$), whose derivative at $\eps = 0$ is $\d_x\al - \omega_x$. In fact, one has that
\[\begin{aligned}
\frac{d}{d\eps}\big{|}_{\eps=0}\sigma_g^\eps(X_x) &= \frac{d}{d\eps}\big{|}_{\eps=0}( \d_x\phi^\eps_{\al}(X_x) - \eps\cdot \d_x \phi^\eps_{\al}(\omega(X_x))) \\
&= \d_x\al(X_x) -\d_x\phi^0_{\al}\omega(X_x)\\
& = (\d_x\al - \omega_x)(X_x).
\end{aligned}\]

Next, we fix a splitting of $\d s: \H \to s^{\ast}TM$, and for each $X\in\X(M)$ we denote by $\tilde{X}\in \Gamma(\H)$ the corresponding horizontal lift. Then
\[\tilde\sigma^\eps(X)_g = \d_{\varphi_{\al^r}^{-\eps}(g)}\varphi^\eps_{\al^r}(\tilde X - \eps \omega(X)^r)\]
defines an extension of $\sigma^\eps(X_x)$ to $\X(\G)$.

From the equation (\ref{help-I}) we deduce that 
\begin{eqnarray*}\label{derivada}
\varkappa_D(\alpha, \omega)(X_x,Y_x)= \frac{d}{d\eps}\big{|}_{\eps=0}\d_\nabla\tilde\theta_{g_\epsilon}(\tilde\sigma^\epsilon(X),\tilde\sigma^\epsilon(Y))(x),
\end{eqnarray*}
for all $X,Y\in\X(M)$.

Finally, to perform the computation, we let $\nabla$ be the pull-back via $s$ of a connection on $E$ (which we also denote by $\nabla$). We have:
\[\begin{aligned}
\d_{\nabla^s}\tilde\theta_{g_\eps}(\tilde\sigma^\eps(X),\tilde\sigma^\eps(Y))&=\nabla_{\tilde\sigma^\eps(X)}\tilde\theta(\tilde\sigma^\eps(Y))-\nabla_{\tilde\sigma^\eps(Y)}\tilde\theta(\tilde\sigma^\eps(X))-\tilde\theta([\tilde\sigma^\eps(X),\tilde\sigma^\eps(Y)])\\
&=-\eps\nabla_{\tilde\sigma^\eps(X)}\tilde\theta(\d\varphi^\eps_{\alpha^r}(\omega(Y)^r))+\nabla_{\tilde\sigma^\eps(X)}\tilde\theta(\d\varphi^\eps_{\alpha^r}(\tilde Y))\\
&\quad+\eps\nabla_{\tilde\sigma^\eps(Y)}\tilde\theta(\d\varphi^\eps_{\alpha^r}(\omega(X)^r))-\nabla_{\tilde\sigma^\eps(Y)}\tilde\theta(\d\varphi^\eps_{\alpha^r}(\tilde X))\\
&\quad-\eps^2\tilde\theta([\d\varphi^\eps_{\alpha^r}(\omega(X)^r),\d\varphi^\eps_{\alpha^r}(\omega(Y)^r)])+\eps\tilde\theta(\d\varphi^\eps_{\alpha^r}[\omega(X)^r,\tilde Y])\\
&\quad-\eps\tilde\theta(\d\varphi^\eps_{\alpha^r}[\omega(Y)^r,\tilde X])-\tilde\theta(\d\varphi^\eps_{\alpha^r}[\tilde X,\tilde Y]).
\end{aligned}\]
We now take the derivative when $\eps=0$ and evaluate the expression at $x$. Using the fact that $\nabla$ is the pull-back of a connection on $E$, the first term of the right hand side of the second equality gives us
\[\begin{aligned}
-\nabla_{\tilde\sigma^0(X)}\tilde\theta(\omega(Y)^r)(x) = -\nabla_X\tilde\theta(\omega(Y))(x)=-\nabla_Xl(\omega(Y))(x),
\end{aligned}\]  
while the second term gives
\[\begin{aligned}\frac{d}{d\eps}\big{|}_{\eps=0}\nabla_{\tilde\sigma^\eps(X)}\tilde\theta(\d\varphi^\eps_{\alpha^r}(\tilde Y))(x) &=\nabla_{\frac{d}{d\eps}\tilde\sigma^\eps(X)}\tilde\theta(\tilde Y)(x)+\nabla_{\tilde X}\frac{d}{d\eps}\big{|}_{\eps=0}(\varphi^\eps_{\alpha^r})^*\tilde\theta(\tilde Y)(x)\\
&=\nabla_X\frac{d}{d\eps}\big{|}_{\eps=0}(\phi^\eps_{\alpha})^*\tilde\theta(Y)(x)\\
&=\nabla_XD_Y\alpha,
\end{aligned}\]
where in the passage from the first to the second line we have used that $\tilde\theta(\tilde Y) = 0$ because $\tilde Y\in \Gamma(\H)$. It then follows from $D_Y(\al) = l(\omega(Y))$ that the the first line of the expression vanishes. The same argument shows also that the second line of the expression is equal to zero. So we are left with calculating the last  three terms of the expression. We obtain:
\[\begin{aligned}
\tilde\theta([\omega(X)^r,\tilde Y])(x) - \tilde\theta([\omega(Y)^r,\tilde X])(x) - \frac{\d}{\d \eps}\big{|}_{\eps = 0}\tilde\theta(\d\varphi^\eps_{\alpha^r}[\tilde X,\tilde Y])(x).
\end{aligned}\] 
From the first two terms we obtain
\[D_Y(\omega(X))-D_X(\omega(Y)).\]
Finally, for the last term we use the fact that $\tilde X$ and $\tilde Y$ are projectable extensions of $X$ and $Y$ to obtain
\begin{eqnarray*}
\frac{d}{d\eps}\big{|}_{\eps=0}\tilde\theta(\d\varphi^\eps_{\alpha^r}[\tilde X,\tilde Y])(x)=\frac{d}{d\eps}\big{|}_{\eps=0}\tilde\theta(\d\varphi^\eps_{\alpha^r}\circ\d u([X,Y]_x))=\frac{d}{d\eps}\big{|}_{\eps=0}\tilde\theta(\d\phi^\eps_{\alpha}[X,Y]_x)\\=\frac{d}{d\eps}\big{|}_{\eps=0}(\Ad^\H_{\phi^\eps_{\alpha}})^{-1}\theta(\d\phi^\eps_{\alpha}[X,Y]_x)=D_{[X,Y]}\alpha(x)=l(\omega([X,Y]_x)).
\end{eqnarray*}
Putting these pieces together concludes the proof.
\end{proof}

Going back to jet groupoids (see subsections \ref{example: jet groupoids} and \ref{Jet groupoids and algebroids}) one finds that the multiplicative Cartan distributions $\mathcal{C}_k\subset J^k\G$ are ``compatible under prolongations''. That is,

\begin{proposition}\label{claslie}Let $\G\tto M$ be a Lie groupoid and $k>0$ an integer. Then,
\begin{eqnarray*}
P_{\mathcal{C}_k}(J^k\G)=J^{k+1}\G\qquad\text{and}\qquad\mathcal{C}_k^{(1)}=\mathcal{C}_{k+1}.
\end{eqnarray*}
\end{proposition}

\begin{proof} This follows form proposition \ref{prol-jet}, and the definition of $\mathcal{C}_k$ as in \eqref{alter}, and the fact that $J^k\G$ and are open submanifolds of the $k$-jet bundle of $s:\G\to M$.
\end{proof}

\begin{remark}\label{el-caso-forma}\rm
For a multiplicative one form $\theta\in\Omega^1(\G,t^*E)$ of constant rank, one can describe the $1$-curvature map directly in terms of $\theta$: 
\[ c_1:  \Jet^{1}_{\theta}\G \To s^{\ast} \hom(\Lambda^2TM, E)\]
 given by
\[ c_1(\sigma_g)
(X_x, Y_x)= g^{-1}\cdot \delta\theta_{g} (\sigma_g(X_x), \sigma_g(Y_{x}))\]
 for $\sigma_g\in \Jet^{1}_{\theta}\G$ with $s(g)= x$, $X_x, Y_x\in T_xM$, where $\delta\theta$ is defined by equation \eqref{deltatheta}. This follows from the fact that in this case, $\delta\theta$ satisfies a similar multiplicativity condition as in lemma \ref{lemma: delta theta is multiplicative}, where in this case we replace the action $\Ad^\H$ by the action with respect to which $\theta$ is multiplicative. The classical Lie prolongation, denoted by 
 \begin{eqnarray*}
 P_\theta(\G)\subset J^1_\theta\G\subset J^1\G
 \end{eqnarray*}
 is set to be the kernel of $c_1$.
 
 The linearization of $c_1$ is again given by \eqref{varkappa-map}, where in this case  $D:\Gamma(A)\to \Omega^1(M,E)$ is the the Spencer operator associated to $\theta$ (see theorem \ref{t1}). 
\end{remark}

\subsection{The classical Lie-prolongation: smoothness}

In order to define a canonical Lie-prolongation of a multiplicative distribution $\H\subset T\G$ on the space $P_\H(\G)$ we need to require the smoothness of $P_\H(\G)$. Thus, we now concentrate on the smoothness of $P_\H(\G)$.\\

 For the following results let $D:\Gamma(A)\to \Omega^1(M,E)$ be the Spencer operator associated to $\H$ (see theorem \ref{t2}) and let $\g^{(1)}(A,D)\subset T^*\otimes A$ be the prolongation of the symbol map $\partial_D:\g\to \hom(TM, E)$ (see definition \ref{definitions}).

\begin{proposition}\label{prop4213} Assume that $\G$ has connected $s$-fibers and that $pr:P_\H(\G)\to \G$ is surjective. Then the following are equivalent:
\begin{enumerate}
\item $P_\H(\G)\subset J^1\G$ is smooth and $pr:P_\H(\G)\to \G$ is a submersion.
\item $\g^{(1)}(A,D)$ has constant rank and $pr:P_D(A)\to A$ is surjective.
\end{enumerate}
\end{proposition}

The following two lemmas, which are interesting in their own right, are going to be used together with results on relative connections in the proof of the proposition \ref{prop4213}.

\begin{lemma}\label{4217}
$P_\H(\G)\subset J^1\G$ is a Lie subgroupoid if and only if $P_D(A)\subset J^1A$ is a Lie subalgebroid.
\end{lemma}

\begin{proof}
By proposition \ref{yase} we know that if $P_\H(\G)$ is smooth then $P_D(A)$ is its Lie algebroid. Conversely, as $P_D(A)$ is the kernel of $\varkappa_D$, then $P_D(A)$ is smooth if and only if $\varkappa_D$ has constant rank. By definition of the linearization of $c_1$ and lemma \ref{curvatura}, one has that $\varkappa_D=dc_1|_A$ and this implies that 
\begin{eqnarray}\label{c=l+m}
\rank d_{\sigma_g}c_1=\rank(\varkappa_D|_{t(g}))+\dim M
\end{eqnarray} 
for any $\sigma_g\in J^1_\H\G$. From this one has that $c_1:J^1_\H\G\to \hom(\wedge^2TM,E)$ has constant rank and therefore $P_\H(\G)=\ker c_1$ is smooth (see proposition 2.1 in \cite{Gold2}). Indeed, to show \eqref{c=l+m} consider a $1$-cocycle $c$ on a Lie groupoid $\Sigma\tto M$ with values in a representation $F\in \Rep(\Sigma)$. Choose any $g\in\Sigma$, and let $b\in\Bis(\Sigma)$ be such that $b(x)=g$ for $x=s(g)$. Using $b$ we split $T_g\Sigma$ as the direct sum
\begin{eqnarray}\label{dec}
T_g\Sigma\simeq T_xM\oplus T^s_g\Sigma,
\end{eqnarray}
where a vector $X\in T_xM$ corresponds to $d_xb(X)$. As the composition $c\circ b:M\to F$ is a section of $F$, then $d_gc$ is injective in the $T_xM$ component of the decomposition \eqref{dec}. On the other hand, the cocycle condition for $c\in Z^1(\Sigma,s^*F)$ implies that 
\begin{eqnarray}\label{c-l}
d_gc(v)=g^{-1}\cdot d_{t(g)}c(r_{g^{-1}}v)=g^{-1}\cdot Lin(c)(r_{g^{-1}}v)
\end{eqnarray}
and therefore $\rank (d_gc|_{T_g^s\Sigma})=\rank Lin(c)_{t(g)}$ since $r_{g^{-1}}:T^s_g\Sigma\to A_{t(g)}$ is an isomorphism. The argument to prove equation \eqref{c-l} is completely analogous to the one given on proposition \ref{prop: appendix}. From this one has that 
\begin{eqnarray*}
\rank d_{g}c=\rank (Lin (c))|_{t(g)})+\dim M.
\end{eqnarray*} 
\end{proof}

\begin{lemma}\label{con}
Assume that $\G$ has connected $s$-fibers. If $\g^{(1)}(A,D)$ has constant rank and $pr:P_D(A)\to A$ is surjective then $P_\H(\G)$ is smooth and $pr:P_\H(\G)\to \G$ is a submersion.
\end{lemma}

\begin{proof}
From lemma \ref{workable} one obtains that $P_D(A)$ is smooth, and therefore by the previous lemma, $P_\H(\G)\subset J^1\G$ is smooth. 

That $pr:P_\H(\G)\to\G$ is a submersion is an instance of the following general case: let $pr:\Sigma\to \G$ be a Lie groupoid map over the identity map of $M$ with the property that $\G$ has connected $s$-fibers. If $Lie(pr):A(\Sigma)\to A(\G)$ is surjective then $pr:\Sigma\to \G$ is a surjective submersion. To see this we will show that $pr$ is a submersion onto its image and therefore $\Im(pr)\subset \G$ is an open set. As $\G$ has connected $s$-fibers and $M\subset \Im(pr)$ then $\Im(pr)=\G$. Let's prove then that $pr$ is a submersion. First of all, notice that $pr$ is a submersion at the units of $\Sigma$. Indeed, for any $x\in M$,
\begin{eqnarray*}
T_{1_x}\Sigma\simeq T_xM\oplus A(\Sigma)_x,
\end{eqnarray*}
where a vector of $X\in T_xM$ corresponds to $d_xu(X)\in T_{1_x}\Sigma$. Similarly one can decompose $T_{1_x}\G$ as $T_xM\oplus A(\G)_x$. As $pr$ is the identity on the units then $dpr$ is the on $T_xM$. Now, as 
\begin{eqnarray*}
dpr|_{A(\Sigma)_x}=Lie(pr)
\end{eqnarray*} 
and $Lie(pr)$ is surjective by assumption, then $d_{1_x}pr$ is surjective. Secondly, for any $g\in \Sigma$, let $b\in\Bis(\Sigma)$ be such that $b(s(g))=g$. If $L_b:\Sigma\to \Sigma, h\mapsto b(t(h))\cdot h$ denotes the diffeomorphism given by left multiplication by $b$, then the fact that $pr$ respects multiplication implies that we have the commutative diagram where the columns are diffeomorphisms
\begin{eqnarray*}
\xymatrix{
\Sigma \ar[r]^{pr} \ar[d]_{L_b} & \G \ar[d]^{L_{pr\circ b}} \\
\Sigma \ar[r]^{pr} & \G.
}
\end{eqnarray*}
Taking differentials we get that 
\begin{eqnarray*}
\xymatrix{
T_{s(g)}\Sigma \ar[r]^{dpr} \ar[d]_{dL_b} &T_{s(g)} \G \ar[d]^{dL_{pr\circ b}} \\
T_g\Sigma \ar[r]^{dpr} & T_{pr(g)}\G.
}
\end{eqnarray*}
As the columns are isomorphisms and the top arrow is surjective, by the commutativity we have that the bottom arrow is also surjective.
\end{proof}

\begin{proof}[Proof of proposition \ref{prop4213}] That statement 1 implies 2 is clear from remark \ref{smooth-prol}, and the fact that the Lie functor of the Lie groupoid map $pr:P_\H(\G)\to \G$ is $pr:P_D(A)\to A$. Lemma \ref{con} gives the converse.
\end{proof}

\begin{definition} Let $\H\subset T\G$ be a multiplicative distribution. If $P_\H(\G)$ is smoothly defined we call 
\begin{eqnarray*}
pr:(P_\H(\G),\H^{(1)})\To (\G,\H)
\end{eqnarray*}
{\bf the classical Lie prolongation of $(\G,\H)$}, where
$$\H^{(1)}:=\mathcal{C}_1\cap TP_\H(\G)$$
with $\mathcal{C}_1\subset TJ^1\G$ being the multiplicative Cartan distribution.
\end{definition}

\begin{remark}\rm For $\H\subset T\G$ as in the previous definition. That $\H^{(1)}\subset T P_{\H}(\G)$ is indeed a Pfaffian distribution with respect to the source map can be proven using the ideas of lemma \ref{lemaese}. The Lie-multiplicativity condition follows by definition.  \end{remark}

\begin{remark}\label{yase2}\rm If $D$ is the associated Spencer operator of $\H$, the induced Spencer operator associated to $\H^{(1)}\subset TP_\H(\G)$ is 
\begin{eqnarray*}
D^{(1)}:\Gamma(P_D(A))\To\Omega^1(M,A).
\end{eqnarray*}
This follows from propositions \ref{yase} and \ref{cartan form}.
\end{remark}

\begin{proposition}\label{si-es} If $P_\H(\G)$ is smoothly defined, then
\begin{eqnarray*}
pr:(P_\H(\G),\H^{(1)})\overset{}{\To}(\G,\H)
\end{eqnarray*}
is a Lie-prolongation of $(\G,\H).$
\end{proposition}

\begin{proof} This is a consequence of the key properties \ref{key-properties} of $pr:P_\H(\G)\to \G$.
\end{proof}

\subsection{The classical Lie-prolongation: the universal property}

The classical Lie-prolongation of $(\G,\H)$ is actually universal among Lie-prolongations of $(\G,\H)$ in the following sense:


\begin{proposition}\label{prop-un}The Lie prolongation space $P_\H(\G)$ of $(\G,\H)$ is universal among Lie-prolongations of $(\G,\H)$. More precisely, if
 \begin{eqnarray*}p:(\tG,\tH)\To (\G,\H)\end{eqnarray*} is a Lie prolongation of $(\G,\H)$, then there exists a unique Lie groupoid map
$j:\tG\to P_\H(\G)$
with the property that $p\circ j=pr$ and 
\begin{eqnarray}\label{cond3}
j^*\theta^{(1)}=\theta_{\tH},\qquad j^*\delta\theta^{(1)}=\delta\theta_{\tH},
\end{eqnarray}
where $\theta_{\tH}$ is the multiplicative form associated to $\tH$.
\end{proposition}

\begin{remark}\rm One can write the condition \eqref{cond3} directly in terms of the multiplicative distributions as follows.  By assumption $\tilde A/\tilde\g\simeq A$. On the other hand, recall that right translation induces an isomorphism 
\begin{eqnarray*}
T\tG/\tH\simeq T^s\tG/\tH^s\overset{R}{\simeq} t^*(\tilde A/\tilde\g)\simeq t^*A.
\end{eqnarray*} 
Similarly,
\begin{eqnarray*}
TP_\H/\H^{(1)}\simeq t^*A.
\end{eqnarray*}
Condition \eqref{cond3} means that, under this identification, 
\begin{eqnarray*}
X\mod \tH=dj_{\tt}(X)\mod \H^{(1)},\qquad c_{\tH}(Y,Z)=c_{\H^{(1)}}(dj_{\tt}(Y),dj_{\tt}(Y))
 \end{eqnarray*}
 for any vectors $X,Y,Z$ of $\tG$ with the property that $Y,Z$ belong to $\tH$. 
\end{remark}

\begin{proof}
Take $j=j_{\theta_{\tH}}$ of corollary \ref{628}. By definition of $P_\H(\G)$ and proposition \ref{410} it follows that $\Im j\subset P_{\H}(\G)$. That $j^*\theta^{(1)}=\theta_{\tH}$ follows from corollary \ref{628}. To prove that $j^*\delta\theta^{(1)}=\delta\theta_{\tH}$
notice that $j^*\theta^{(1)}=\theta_{\tH}$ implies that $dj:T\tG\to J^1\G$ maps $\tH$ to $\H^{(1)}$. Take now $X,Y\in \Gamma(\tH)$ $s$-projectable and let $g\in\tG$. Consider $s$-projectable extensions $\widetilde{dj(X)},\widetilde{dj(Y)}\in\Gamma(\mathcal{H}^{(1)})$ of $ds(X)$ and $ds(Y)$, respectively with the property that 
\begin{eqnarray*}
\widetilde{dj(X)}_g=d_gj(X)\quad\text{and}\quad \widetilde{dj(Y)}_g=d_gj(Y).
\end{eqnarray*}
 Then
\begin{eqnarray*}
\begin{split}
j^*\delta\theta^{(1)}(X,Y)_g&=\theta^1_{j_{\tt}(g)}[\widetilde{dj(X)},\widetilde{dj(Y)}]\\
&=r_{p(g)^{-1}}(dpr[\widetilde{dj(X)},\widetilde{dj(Y)}]-j(g)(ds[\widetilde{dj(X)},\widetilde{dj(Y)}])).
\end{split}
\end{eqnarray*}
But applying lemma \ref{delta-commutes} (or the remark below) to the submersion $pr:P^1_{\H}\G\to \G$, and recalling that $dpr(\H^{(1)})\subset \tH$ by remark \ref{key-properties}, one has that 
\begin{eqnarray*}
dpr[\widetilde{dj(X)},\widetilde{dj(Y)}]=c_\H(dpr(dj(X)),dpr(dj(Y)))=0.
\end{eqnarray*}
On the other hand, as $\widetilde{dj(X)}$ and $\widetilde{dj(Y)}$ are $s$-projectable to $ds(X)$ and $ds(Y)$ respectively, using remark \ref{generali} 
\begin{eqnarray*}
\begin{split}
j(g)(ds[\widetilde{dj(X)},\widetilde{dj(Y)}])&=j(g)[ds(X),ds(Y)]=j(g)(ds[X,Y])\\
&=d_gp([X,Y])-r_{p(g)}\theta_{\tH}([X,Y]_g).
\end{split}
\end{eqnarray*} 
Using again lemma \ref{delta-commutes}, one gets that 
\begin{eqnarray*}
d_gp([X,Y])=c_{\H}(dp(X),dp(Y))=0,
\end{eqnarray*}
and therefore
\begin{eqnarray*}
j^*\delta\theta^{(1)}(X,Y)_g=\theta_{\tH}([X,Y]_g)=\delta\theta_g(X,Y).
\end{eqnarray*}
\end{proof}

\begin{corollary}\label{compatibility-prolongations}
If $pr:P_\H(\G)\to\G$ is smoothly defined, then  
\begin{eqnarray}\label{mc}
MC(\theta^{(1)},\theta_{\H})=0.
\end{eqnarray} 
\end{corollary}

\begin{proof}
This is a consequence of theorem \ref{theorem-mc} and proposition \ref{si-es}.
\end{proof}

Going back to jet groupoids (see subsections \ref{example: jet groupoids}), we have the main example of Lie-prolongations.

\begin{corollary}\label{cor:1349} For $k\geq 1$ a natural number, each jet groupoid endowed with the Cartan form
\begin{eqnarray*}
pr:(J^{k+1}\G,\theta^{k+1})\overset{}{\To}(J^k\G,\theta^{k})
\end{eqnarray*}
is the classical Lie-prolongation of $(J^k\G,\theta^k)$. Hence, 
\begin{eqnarray*}
MC(\theta^{k+1},\theta^k)=0.
\end{eqnarray*} 
\end{corollary}

\begin{proof}
This is a consequence of proposition \ref{claslie} and theorem \ref{theorem-mc}.
\end{proof}

\subsection{Proofs of the compatibility theorems \ref{prol-comp} and \ref{theorem-mc}}\label{Proofs}

Through this subsection we assume that:
\begin{itemize}
\item $\tG$ and $\G$ are Lie groupoids over $M$ with Lie algebroids $\tilde A$ and $A$ respectively,
\item the Lie groupoid map $p:\tG\to \G$ is a surjective submersion,
\item $\tt\in\Omega^1\in\Omega^1(\tG,t^*A)$ and $\theta\in\Omega^1(\G,t^*E)$
are point-wise surjective multiplicative forms,
\item $(\tilde D,\tilde l):\tilde A\to A$ and $(D,l):A\to E$ are the associated Spencer operators of $\tt$ and $\theta$,
\item $Lie(p)=\tilde l:\tilde A\to A$, and   
\item we denote by
$\tH:=\ker\tt\subset T\tG$ and by $\H:=\ker\theta\subset T\G$ the associated multiplicative distributions.
\end{itemize}

We will prove the following two propositions:

\begin{proposition}\label{410}
If the differential $dp:T\tG\to T\G$ is such that
\begin{eqnarray}\label{cont}
dp(\tH)\subset \H,
\end{eqnarray} then 
\begin{eqnarray}\label{co}
D\circ\tilde l-\tilde D\circ l=0.
\end{eqnarray}
Moreover, if $dp(\tH)\subset \H$ holds, then
\begin{eqnarray}\label{b1}
\delta\theta(dp(\tilde X),dp(\tilde Y))=0,\quad\text{for any }\tilde X,\tilde Y\in\tH
\end{eqnarray}
implies that 
\begin{eqnarray}\label{b2}
D_X\tilde D_Y-D_Y\tilde D_X-l\tilde D_{[X,Y]}=0,\quad\text{for any }X,Y\in\X(M).
\end{eqnarray}

If $\tG$ has connected $s$-fibers then conditions \eqref{cont} and \eqref{b1} are equivalent to \eqref{co} and \eqref{b2} respectively.
\end{proposition}

\begin{remark}\label{523}\rm Note that as $T\tG=\tH+T^s\tG$ and $dp|_{T^s(\tG)}=\tt|_{T^s(\tG)}$, then condition \eqref{cont} is equivalent to the equation
\begin{eqnarray*}
\theta\circ dp=\theta\circ \tt=l\circ\tt.
\end{eqnarray*}
\end{remark}

It is clear that proposition \ref{410} implies theorem \ref{prol-comp}; in order to achieve a proof of the former, we will first prove two key lemmas (see lemmas \ref{lemma:3459724} and \ref{4418}). However, before proceeding to a proof, we introduce two cocycles. The first one takes care of condition \eqref{cont}. Explicitely, let 
\begin{eqnarray*}
ev_{\tH}:J^1_{\tH} \tG\To s^*\hom(TM,E)
\end{eqnarray*}
be defined at $\sigma_g\in J^1_{\tH}\tG$ and $X\in T_{s(g)}M$ by 
\begin{eqnarray*}
ev_{\tH}(\sigma_g)(X)=p(g)^{-1}\cdot\theta_{p(g)}(dp(\sigma_g(X))).
\end{eqnarray*}
As $dp(\tH^s)=\tt(\tH^s)=0$ since $Lie(p)=\tt_{\tilde A}$, then it is clear that \eqref{cont} holds if and only if $ev_{\tH}$ is identically zero.

Under the assumption that $dp(\tH)\subset \H$ (see also remark \ref{523}), we introduce the second $1$-cocycle which takes cares of the condition \eqref{b1}. 
More precisely, consider the well-defined map
\begin{eqnarray*}
\tilde c_1:J^1_{\tH}\tG\To s^*\hom(\wedge^2TM,E)
\end{eqnarray*}
defined at $\sigma_g\in J^1_{\tH}\tG$ and $X,Y\in T_{s(g)}M$ by
\begin{eqnarray*}
\tilde c_1(\sigma_g)(X,Y)=p(g)^{-1}\delta\theta_g(dp(X),dp(Y)).
\end{eqnarray*}

\begin{remark}\label{ayuda}\rm Of course if equation \eqref{b1} holds then $\tilde c_1$ vanishes. For the converse, let $g\in\tG$ and let $\sigma_g:T_{s(g)}M\to T_g\tG$ be any element of $J^1_{\tH}\tG$. As
\begin{eqnarray}\label{eseno}
\tH_g=\sigma_g(T_{s(g)}M)\oplus \tH^s_g,
\end{eqnarray}
we can write 
\begin{eqnarray*}
X=\sigma_g(U)+u,\qquad Y=\sigma_g(V)+v,
\end{eqnarray*}
where $u,v\in \tH^s_g$. Then
\begin{eqnarray*}
\begin{split}
\delta\theta(dp(X),dp(Y))&=\delta\theta(dp(\sigma_g(U)),dp(\sigma_g(V)))+\delta\theta(dp(u),dp(\sigma_g(V)))\\&\quad+\delta\theta(dp(\sigma_g(U)),dp(v))+\delta\theta(dp(u),dp(v))\\
&=\delta\theta(dp(\sigma_g(U)),dp(\sigma_g(V)))=p(g)\tilde c_1(\sigma_g)(U,V),
\end{split}
\end{eqnarray*}
where we use that $dp(\tH^s)=0$. This shows the equivalence.

The previous computation also shows that if $\tG$ has connected $s$-fibers, then condition \eqref{b1} holds if and only if 
\begin{eqnarray*}
\tilde c_1:(J^1_{\tH}\tG)^0\To s^*\hom(\wedge^2TM,E)
\end{eqnarray*}
vanishes, where $(J^1_{\tH}\tG)^0$ is the connected component of the $s$-fibers. This is true since $pr:J^1_{\tH}\tG\to\tG$ is a surjective a submersion and therefore $pr:(J^1_{\tH}\tG)^0\to\tG$ is still surjective. Hence, for any $g\in \tG$, we can choose our $\sigma_g$ as in \eqref{eseno} belonging to $(J^1_{\tH}\tG)^0$.
\end{remark}

The action of $J^1_{\tH} \tG$ on $\hom(TM,E)$ (and on $\hom(\wedge^2TM,E)$) is the one induced by the representation $\lambda$ on $TM$ and the pullback of the action of $\G$ on $E$ via $p\circ pr:J^1_{\tH} \tG\to \G$: for $\sigma_g\in J^1_{\tH} \tG$ with $s(g)=x$,$t(g)=y$ and $T\in \hom(T_xM,E_x)$,
\begin{eqnarray*}
\sigma_g\cdot T\in \hom (T_yM,E_y),\qquad \sigma_g\cdot T(X_y)=p(g)\cdot T(\lambda_{\sigma_g}^{-1}(X_y)).
\end{eqnarray*}

For the following lemma
recall that $J^1_{\tilde D}\tilde A\subset J^1\tilde A$ is the Lie algebroid of $J^1_{\tH}\tG$, whose sections are given by the Spencer decomposition \eqref{Spencer decomposition} by pairs $(\al,\omega)\in\Gamma(\tilde A)\oplus \Omega^1(M,\tilde A)$ with the property that
\begin{eqnarray}\label{eq}
\tilde D(\al)=\tilde l\circ\omega.
\end{eqnarray}
See remark \ref{cocyclea}.

\begin{lemma}\label{lemma:3459724}
The map $ev_{\tH}:J^1_{\tH}\tG\to s^*\hom(TM,E)$ is a cocycle with linearization 
\begin{eqnarray*}
Lie(ev_{\tH}):J^1_{\tilde D}\tilde A\To \hom (TM,E)
\end{eqnarray*}
given on sections $(\al,\omega)\in \Gamma(J^1_{\tilde D}A)$ by
\begin{eqnarray*}
D(\tilde l(\al))-l(\tilde D(\al)).
\end{eqnarray*}
\end{lemma}

\begin{proof}
Note that $ev_{\tH}$ is the pullback of the cocycle $ev:J^1\G\to \hom s^*(TM,E)$ in lemma \ref{s}, along the Lie groupoid map \begin{eqnarray*}j^1p:J^1_{\tH}\tG\to J^1\G,\quad\sigma_g\mapsto dp\circ \sigma.\end{eqnarray*}
Therefore $ev_{\tH}$ is a cocycle and its linearization is given by the composition
\begin{eqnarray*}
J^1_{\tilde D}\tilde A\overset{Lie(j^1p)}{\To}J^1A\overset{a}{\To}\hom(TM,E),
\end{eqnarray*}
where $a=Lie(ev)$. Explicitely, for $(\al,\omega)\in\Gamma(J^1_{\tilde D}\tilde A)$ as in \eqref{eq}
\begin{eqnarray*}
\begin{split}
Lie(ev_{\tH})(\al,\omega)&=a(Lie(j^1p)(\al,\omega))=a(Lie(p)\circ\al,Lie(p)\circ\omega)\\&=a(\tilde l(\al),\tilde l(\omega))=a(\tilde l(\al), \tilde D(\al))=D(\tilde l(\al))-l(\tilde D(\al)).
\end{split}
\end{eqnarray*}
\end{proof}

\begin{lemma}\label{4418}Assume that $dp(\tH)\subset \H$. The map $\tilde c_1:J^1_{\tH}\tG\to\hom(\wedge^2TM,E)$ is a cocycle with linearization
\begin{eqnarray*}
Lie(\tilde c_1):J^1_{\tilde D}\tilde A\To\hom(\wedge^2TM,E)
\end{eqnarray*}  given on section $(\al,\omega)\in J^1_{\tilde D}\tilde A$ as in \eqref{eq} by
\begin{eqnarray*}
D_X\tilde D_Y(\al)-D_Y\tilde D_X(\al)-l\tilde D_{[X,Y]}(\al)
\end{eqnarray*}
for $X,Y\in\X(M)$.
\end{lemma}

\begin{proof}
Notice that $\tilde c_1$ is the pullback of the $1$-cocycle \begin{eqnarray*}c_1:J^1_\theta\G\To s^*\hom(\wedge^2TM,E)\end{eqnarray*} defined in remark \ref{el-caso-forma}, along the Lie groupoid map 
\begin{eqnarray*}
j^1p:J^1_{\tH}\tG\To J^1_\theta\G,\quad\sigma_g\mapsto dp\circ\sigma_g.
\end{eqnarray*}
Hence $\tilde c_1$ is a cocycle with linearization given by the composition
\begin{eqnarray*}
J^1_{\tilde D}\tilde A\overset{Lie(j^1p)}{\To}J^1_DA\overset{Lie(c_1)}{\To}\hom(\wedge^2TM,E).
\end{eqnarray*}
Explicitly, for $(\al,\omega)\in\Gamma(J^1_{\tilde D}\tilde A)$ as in \eqref{eq} and $X,Y\in\X(M)$,
\begin{eqnarray*}
\begin{split}
Lie(\tilde c_1)(\al,\omega)(X,Y)&=Lie(c_1)(\tilde l(\al),\tilde l(\omega))(X,Y)=Lie(c_1)(\tilde l(\al),\tilde D(\al))(X,Y)\\
&=D_X\tilde D_Y(\al)-D_Y\tilde D_X(\al)-l\tilde D_{[X,Y]}(\al).
\end{split}
\end{eqnarray*}
\end{proof}

Lemmas \ref{lemma:3459724} and \ref{4418} yield proposition \ref{410} as shown below.

\begin{proof}[Proof of proposition \ref{410}]
As mentioned before, conditions \eqref{cont} and \eqref{b1} are equivalent to the fact that $ev$ and $\tilde c_1$ vanish respectively. From this it is clear that if conditions \eqref{cont} and \eqref{b1} hold then so do \eqref{co} and \eqref{b2} respectively. For the converse, assume that $\tG$ has connected $s$-fibers. As $ev$ and $\tilde c_1$ are 1-cocycles with $Lie(ev)=Lie(\tilde c_1)=0$, then $dev$ and $d\tilde c_1$ vanish on vectors tangent to the $s$-fibers (see equation \eqref{c-l} in the proof of lemma \ref{4217}). This implies that $ev$ vanishes on $(J^1_{\tH}\tG)^0$--the connected component of the $s$-fibers. On the other hand, as the map $pr:J^1_{\tH}\tG\to \tG$ is a surjective submersion (see proposition \ref{prop421}), then $(J^1_{\tH}\tG)^0$ is mapped onto $\tG^0$, i.e. the image equals $\tG$ since $\tG$ is $s$-connected. This means that for any $g\in\tG$, one can choose a splitting $\sigma_g:T_{s(g)}M\to T_g\tG$ of the source map, whose image lies in $\tH$, with the property that 
\begin{eqnarray*}
\theta(dp\circ\sigma_g(T_{s(g)}M))=0.
\end{eqnarray*}
As the vector space $\tH_g$ can be decomposed as the direct sum
\begin{eqnarray*}
\tH_g=\sigma_g(T_{s(g)}M)\oplus \tH^s_g
\end{eqnarray*}
and $dpr(\tH^s_g)=R_{pr(g)}dpr(\tH^s_{t(g)})=R_{pr(g)}\tt(\tH^s_{t(g)})=0$, then we can conclude that $dpr(\tH_g)=0$ for any $g\in\tG$. A similar computation shows that if $Lie(\tilde c_1)$ vanishes on $(J^1_{\tH}\tG)^0$, then it vanishes everywhere (see remark \ref{ayuda}).
\end{proof}

Now we proceed with the proof of theorem \ref{theorem-mc}.

\begin{proof}[Proof of theorem \ref{theorem-mc}]
As $\tH$ is $s$-transversal, any vector on $\tG$ can be written as a linear combination (with coefficients in $C^{\infty}(\tG)$) of vector fields tangent to $\tH$ together with right invariant vector fields on $\tG.$ For this reason, to compute $MC(\tt,\theta)$, it suffices to calculate it on three types of pairs of vectors $(X,Y)$, namely, 
\begin{enumerate}
\item $X=\al^r\in\X^r(\tG)$ is a right invariant vector field and $Y$ is tangent to $\tH$,
\item $X=\al^r$ and $Y=\be^r$ are both right invariant vector fields, and
\item $X,Y$ are both tangent to $\tH$.
\end{enumerate}

In the first case, we have that
for any $g\in\tG$
\begin{eqnarray}\label{Mc}
\begin{split}
MC(\tt,\theta)(\al^r_{g},Y_{g})&=-D^t_Y(\tt(\al^r))(g)-l\circ\tt[\al^r,Y](g)\\
&=-D^t_Y(\tt(\al)^r)(g)+l\circ\tt[\al^r,Y](g)\\
&=-D^t_Y(\tilde l(\al)^r)(g)+l\circ\tt[\al^r,Y](g),
\end{split}
\end{eqnarray}
where in the second equality we use that multiplicativity of $\tt$ implies that  
\begin{eqnarray*}
\tt(\al^r)=\tt(\al)^r=\tilde l(\al)^r.
\end{eqnarray*}
On the one hand, one can easily show that $D^t_Y(\tilde l(\al)^r)=D_{dt(Y)}(\tilde l(\al))\circ t$. On the other, lemma \ref{commutator} implies that 
\begin{eqnarray*}\begin{split}
\tt([\al^r,Y]_{g})&=g\cdot (L_\al\tt)(Y)\\
&=g\cdot\frac{d}{d\eps}\big{|}_{\eps=0}(\varphi^{\eps}_{\al^r}(g))^{-1}\cdot(\varphi^\epsilon_{\al^r})^*\tt|_{\varphi^\epsilon_{\al^r}(g)}\\
&=g\cdot\frac{d}{d\eps}\big{|}_{\eps=0}(\varphi_{\al}^\eps(t(g))\cdot g)^{-1}\cdot (\tt(dm(d\varphi^\epsilon_\al(dt(Y)),Y))\\
&=g\cdot\frac{d}{d\eps}\big{|}_{\eps=0}g^{-1}\cdot(\varphi^{\eps}_{\al}(t(g)))^{-1}\cdot (\tt (dm(d\varphi^\epsilon_\al(dt(Y)),Y))\\
&=\frac{d}{d\eps}\big{|}_{\eps=0}(\varphi^{\eps}_{\al}(t(g)))^{-1}\cdot (\tt(d\varphi^\epsilon_\al(dt(Y)))-\varphi_{\al}^\eps(t(g))\cdot\tt(Y))\\
&=\frac{d}{d\eps}\big{|}_{\eps=0}(\varphi^{\eps}_{\al}(t(g)))^{-1}\cdot \tt(d\varphi^\epsilon_\al(dt(Y)))=\tilde D_{dt(Y)}(\al)(t(g)),
\end{split}
\end{eqnarray*}
where we used that the flow of a right invariant vector field $\al^r$ is given by $\varphi_{\al^r}^\eps(g)=\varphi_{\al}^\eps(t(g))\cdot g$ and therefore for a fixed $\eps$, $d\varphi_{\al^r}^\eps=dm(d\varphi_{\al}^\eps\circ dt,id)$. Expression \eqref{Mc} becomes then
\begin{eqnarray*}
MC(\tt,\theta)(\al^r,Y)=D_{dt(Y)}(\tilde l(\al))(t(g))-l(\tilde D_{dt(Y)}(\al)(t(g))).
\end{eqnarray*}
As $\tH$ is $t$-transversal (see remarks \ref{mult-equiv} and \ref{transversalidad}), then the above equation implies that $MC(\tt,\theta)=0$ if and only if $D\circ \tilde l-l\circ \tilde D=0.$

For the second case,
\begin{eqnarray}\label{mC}
\begin{split}
MC(\tt,\theta)(\al^r,\be^r)&=D^t_{\al^r}(\tt(\be^r))-D^t_{\be^r}(\tt(\al^r))\\&\quad-l\circ\tt[\al^r,\be^r]-\{\tt(\al^r),\tt(\be^r)\}_D\\&
=D^t_{\al^r}(\tt(\be)^r)-D^t_{\be^r}(\tt(\al)^r)\\&\quad-l\circ dp([\al^r,\be^r])-\{\tt(\al)^r,\tt(\be)^r\}_D,
\end{split}
\end{eqnarray}
where we use that $dp|_{T^s\tG}=\tt|_{T^s\tG}$.
As $d_{g}t(\al^r)=\rho(\al_{t(g)})$ for any $g\in \tG$, then it is easy to see that 
\begin{eqnarray*}D^t_{\al^r}(\tt(\be)^r)=D_{\rho(\al)}(\tt(\be))\circ t\quad\text{and}\quad D^t_{\be^r}(\tt(\al)^r)=D_{\rho(\be)}(\tt(\al))\circ t.\end{eqnarray*} Also,
\begin{eqnarray*}
\begin{split}
T(\tt(\al)^r,\tt(\be)^r)&=T(\tt(\al),\tt(\be))\circ t\\
&=(D_{\rho(\al)}(\tt(\be))-D_{\rho(\be)}(\tt(\al))-l[\tt(\al),\tt(\be)])\circ t;
\end{split}
\end{eqnarray*}
hence, \eqref{mC} becomes 
\begin{eqnarray*}
\begin{split}
-l&\circ dp[\al^r,\be^r]+l[\tt(\al),\tt(\be)]\circ t=-l[dp(\al^r),dp(\be^r)]\\&+l[dp(\al),dp(\be)]\circ t
=-l[dp(\al),dp(\be)]^r+l[dp(\al),dp(\be)]\circ t=0.
\end{split}
\end{eqnarray*}
where we used that $dp|_{\tilde A}=\tt.$

For the third case,
\begin{eqnarray*}
MC(\tt,\theta)(X,Y)=-l(\tt[X,Y])=-\theta(dp[X,Y])=-\delta\theta(dp(X),dp(Y))
\end{eqnarray*}
where in the the right equality we use that $p$ is a submersion and therefore $p^*\delta\theta=\delta p^*\theta$ (see lemma \ref{delta-commutes}).

From the above 3 cases it is clear that if $p:(\tG,\tt)\to(\G,\theta)$ is a Lie-prolongation, then $MC(\tt,\theta)=0$. Conversely, if $\tG$ is source-connected then we saw from the first and third case that $(\tt,\theta)$ satisfies the Maurer-Cartan equation if and only if conditions \eqref{co} and \eqref{b1} hold, which in this case, by proposition \ref{410}, it is equivalent to the fact that $p:(\tG,\tt)\to(\G,\theta)$ is a Lie-prolongation.
\end{proof}

\section{Higher Lie-prolongations}

\subsection{Cartan towers; Corollary 1}

Let 
\begin{eqnarray}\label{eq: Cartan tower}
\cdots \To(\G^{k+1},\H^{k+1})\overset{p^{k+1}}{\To}(\G^{k},\H^{k})\overset{p^{k}}{\To}\cdots \To(\G^{2},\H^{2})\overset{p^{2}}{\To}(\G^1,\H^1)
 \end{eqnarray}
be an infinite sequence of Pfaffian groupoids. For ease of notation, when there is no risk of confusion, we omit the upper index of the map $p^{k}:\G^k\to \G^{k-1}$. 

\begin{definition}\mbox{}
\begin{itemize}
\item A {\bf Cartan tower} $(\G^{\infty}, \H^{\infty},p^{\infty})$ is a sequence as in \eqref{eq: Cartan tower} where any Pfaffian groupoid is a Lie-prolongation of the previous one.
\item A {\bf Cartan resolution} of a Pfaffian groupoid $(\G,\H)$ is a Cartan tower with the property that there exists a Lie groupoid map $p:\G^{1}\to \G$ such that
\begin{eqnarray*}
p:(\G^{1},\H^1)\To (\G,\H)
\end{eqnarray*}
is a Lie-prolongation of $(\G,\H)$.
\end{itemize}
\end{definition}

Corollary 1 says that (under the usual conditions) there is a one to one correspondence between Cartan towers and Spencer resolutions. More precisely,

\begin{corollary} Given a tower of $s$-simply connected Lie groupoids
\begin{eqnarray}\label{to1}
\cdots \To\G^{3}\overset{p^3}{\To}\G^{2}\overset{p^{2}}{\To}\G^1
 \end{eqnarray}
(i.e. all the maps are Lie groupoid morphisms and surjective submersions), and the induced tower of Lie algebroids
\begin{eqnarray}\label{to2}
\cdots \To A_{3}\overset{l_3}{\To}A_{2}\overset{l_{2}}{\To}A_1,
 \end{eqnarray}
there is a 1-1 correspondence between:
\begin{enumerate}
\item multiplicative forms $\theta^k\in\Omega^1(\G^k,t^*A_{k-1})$ making \eqref{to1} into a Cartan tower,
\item Spencer operators $D^k$ relative to $l_k$ making \eqref{to2} into a Spencer tower.
\end{enumerate}
\end{corollary}

This, of course, is a consequence of \ref{prol-comp}.

\begin{example}[The classical Cartan tower]\rm
The sequence
\begin{eqnarray*}
(J^{\infty}\G, \mathcal{C}^{\infty}):\cdots \To(J^{k+1}\G,\mathcal{C}^{k+1})\overset{pr}{\To}(J^{k}\G,\mathcal{C}^{k})\overset{pr}{\To}\cdots\To(J^1\G,\mathcal{C}^1)
 \end{eqnarray*}
 is an example of a Cartan tower. Taking the Lie functor we get the classical Spencer tower $(J^{\infty}A, D^{\infty\text{-}\clas},pr^{\infty})$ on the Lie algebroid $A$ of $\G$. See corollary \ref{cor:1349}. 
\end{example}

\begin{example}\rm Going back to the theory of Lie pseudogroups (see subsection \ref{Lie pseudogroups}), for any smooth Lie pseudogroup $\Gamma$, one has a Cartan tower 
\[\begin{aligned}
(\Gamma^{\infty}(\Gamma),\mathcal{C}^{\infty}): \cdots\To (\Gamma^{(k)},\mathcal{C}_k)\To\cdots\To (\Gamma^{(2)},\mathcal{C}_2)\To(\Gamma^{(1)},\mathcal{C}_1).
\end{aligned}\]
See \cite{Ngo} where the author discusses the universal property (in the sense of proposition \ref{prop-un}) of the tower above for $\Gamma\subset\Diff(M)$. 
 \end{example}

\subsection{The classical $k$-Lie prolongation}

Now we proceed to define the classical $k$-Lie prolongation space inductively, as in the case of Pfaffian bundles (see definition \ref{higher-prol}). We remark that the smoothness results of section \ref{sec:class-prol-pfaff} are still valid in this setting, for example propositions \ref{integrable} and \ref{affine}, and corollary \ref{337}.

\begin{definition} Let $(\G,\H)$ be a Pfaffian groupoid. We say that the {\bf classical $k$-Lie prolongation space $P^k_\H(\G)$} is {\bf smooth} if 
\begin{enumerate}
\item $(P_\H(\G), \H^{(1)}),\dots,(P_\H^{k-1}(\G),\H^{(k-1)})$ are smoothly defined, and
\item the classical prolongation space of $(P_\H^{k-1}(\G),\H^{(k-1)})$
\begin{eqnarray*}
P^k_\H(\G):=P_{\H^{(k-1)}}(P_\H^{k-1})
\end{eqnarray*}
is smooth.
\end{enumerate}
In this case, we define the {\bf $k$-Lie prolongation of $\H$}:
\begin{eqnarray*}
\H^{(k)}:=(\H^{(k-1)})^{(1)}\subset TP^k_\H(\G).
\end{eqnarray*}
We say that the classical Lie prolongation space $P_\H^k(\G)$ is {\bf smoothly defined} if, moreover,
\begin{eqnarray*}
pr: P^k_\H(\G)\To P_\H^{k-1}(\G)
\end{eqnarray*}
is a surjective submersion. In this case,
\begin{eqnarray*}
\pi:(P_\H^k(\G),\H^{(k)})\To M
\end{eqnarray*}
(a Lie-Pfaffian groupoid: see proposition \ref{si-es}) is called the {\bf classical $k$-Lie prolongation of $(\G,\H)$.}
\end{definition}

 \begin{proposition}\label{groupoid-prop2} Let $(\G,\H)$ be a Pfaffian bundle and suppose that $P^k_\H(\G)$ is smoothly  defined. Then, 
 \begin{eqnarray*}
 pr^k_0:\Bis(P^k_\H(\G),\H^{(k)})\To \Bis(\G,\H)
 \end{eqnarray*}
 is a bijection of groups with inverse $j^k:\Bis(\G,\H)\to \Bis(P^k_\H(\G),\H^{(k)})$.
 \end{proposition}
 
 \begin{proof} The proof is completely analogous to that of proposition \ref{bundle-prop2}.
 \end{proof}

 \begin{remark}\label{non-smooth-prolongations}\rm Of course as in the case of Pfaffian bundles one can define the $k$-prolongation of $(\G,\H)$ without smoothness assumptions as 
\begin{eqnarray*}
P^k_\H(\G)=J^{k-1}(P_\H(\G))\cap J^k\G,
\end{eqnarray*}
where one takes jets of bisections. Note that proposition \ref{jet-prol} also holds in this setting.
\end{remark}

Let $D:\Gamma(A)\to\Omega^1(M,E)$ be the Spencer operator associated to $\H$.

\begin{corollary}\label{corollary:auxiliar}Let $k>0$ be an integer. If $P^k_\H(\G)\subset J^k\G$ is smoothly defined then it is a Lie subgroupoid and  
\begin{eqnarray*}
Lie(P^k_\H(\G))=P^k_D(A).
\end{eqnarray*}
Moreover, $D^{(k)}:\Gamma(P^k_D(A))\to \Omega^1(M,A)$ is the associated Lie-Spencer operator $\H^{(k)}.$
\end{corollary}

\begin{proof}
Apply proposition \ref{yase} and remark \ref{yase2} inductively.
\end{proof}

\begin{corollary}\label{corollary:auxiliar2} Let $k>0$ be an integer. If $P^k_\H(\G)\subset J^k\G$ is smoothly defined then 
\begin{eqnarray*}
\g(P^k_\H(\G),\H^{(k)})|_M=\g^k(A,D).
\end{eqnarray*}
\end{corollary}

\begin{proof}
From corollary \ref{ultimo-porfavor}, we have that 
\begin{eqnarray*}
\g(P^k_\H(\G),\H^{(k)})|_M=\g(P^k_D(A),D^{(k)})=\g^{(k)}(A,D)
\end{eqnarray*}
by corollary \ref{corollary:auxiliar} and proposition \ref{inv-prol}.
\end{proof}

\subsection{Formally integrable Pfaffian groupoids}

\begin{definition}
A Pfaffian groupoid $(\G,\H)$ is called {\bf formally integrable} if all the classical $k$-Lie prolongations 
\begin{eqnarray*}
P_\H(\G),P^1_\H(\G),\ldots, P^k_\H(\G),\ldots
\end{eqnarray*} 
are smoothly defined.
\end{definition}

If $(\G,\H)$ is formally integrable, we obtain the {\bf classical Cartan resolution}, 
\begin{eqnarray*}
(P^\infty_\H(\G),\H^{(\infty)}):\cdots \To (P_\H^{k}(\G),\H^{(k)})\overset{pr}{\To}\cdots\To (P_\H(\G),\H^{(1)})\overset{pr}{\To}(\G,\H).
\end{eqnarray*} 

\begin{remark}\rm Taking the Lie functor of $(P^\infty_\H(\G),\H^{(\infty)})$ we obtain the classical Spencer resolution \begin{eqnarray*}(P_D^{\infty}(A),D^{(\infty)}).\end{eqnarray*}
\end{remark}

The main importance of formally integrable Pfaffian groupoids is the following existence result for analytic Pfaffian groupoids, and is a consequence of theorem \ref{resolution}.

\begin{theorem}
Let $(\G,\H)$ be an analytic Pfaffian groupoid. If $(\G,\H)$ is formally integrable, then for every point $g\in\G$ there exists real analytic solutions of $(\G,\H)$ passing through $g$.
\end{theorem}

In the spirit of theorem \ref{workable-pfaffian-bundles}, we will prove the following workable criteria for formally integrable Pfaffian groupoids.

\begin{theorem}\label{workable-pfaffian-groupoids}
Let $(\G,\H)$ be a Pfaffian groupoid and let $D:\Gamma(A)\to \Omega^1(M,E)$ be the Spencer operator associated to $\H$. If
\begin{enumerate}
\item $pr:P_\H(\G)\to \G$ is surjective.
\item $\mathfrak{g}^{(1)}(A,D)$ is a vector bundle over $M$, 
\item $H^{2,l}(\g(A,D))=0$ for $l\geq0.$
\end{enumerate}
Then, $(\G,\H)$ is formally integrable.
\end{theorem}

In the previous theorem we are considering the $\partial_D$-Spencer cohomology of $\g^{(1)}(A,D)$. See definition \ref{exten}.

The main difference between the previous result and theorem \ref{workable-pfaffian-groupoids} is that of replacing the global symbol space $\g({\H})$ and its cohomology (which are bundles over $\G$), by the easier-to-handle cohomology of the symbol space $\g(A,D)$ (both the cohomology and $\g(A,D)$ are bundles over $M$!). This situation reflects again the phenomenon occurring in Lie groupoids: their multiplicative structure allows us to obtain global information out of its infinitesimal data. In this case the result will follow from the invariance of some objects in consideration, such as the symbol map
\begin{eqnarray*}
\partial_\H:\g(\H)\To\hom(s^*TM,T\G/\H)
\end{eqnarray*} 
and the $\partial_\H$-Spencer cohomology.

\subsubsection{Invariance of the $\partial_\H$-Spencer cohomology}

In this section fix $\H\subset T\G$ a multiplicative Pfaffian distribution with $D:\Gamma(A)\to \Omega^1(M,E)$ the associated Spencer operator given in theorem \ref{t2}.\\

Recall that the symbol map of $\H\subset T\G$ is given by
\begin{eqnarray*}
\partial_\H:\g(\H)\To\hom(s^*TM,T\G/\H), \quad v_g\mapsto\partial_\H(X_{s(g)})=r_gc_\H(v_g,\tilde X_g),
\end{eqnarray*} 
where $\tilde X_g\in \H_g$ is any vector $s$-projectable to $X\in T_{s(g)}M$. Notice that by definition of $\partial_D$, the symbol map of $D$, 
\begin{eqnarray}\label{pb}
\partial_D(r_{g^{-1}}(v))(\lambda_{\sigma_g}(X))=c_\H(r_{g^{-1}}(v),\lambda_{\sigma_g}(X))
\end{eqnarray}
for any $v\in \g(\H)_g$, $X\in T_{s(g)}M$ and $\sigma_g\in J^1_\H(\G).$

\begin{lemma}\label{label}
For a fixed $g\in\G$, $v\in \g(\H)_g$, and $X\in T_{s(g)}M,$ the expression \eqref{pb} does not depend on the choice of $\sigma_g\in J^1_\H\G$. Moreover, 
\begin{eqnarray*}
r_{g^{-1}}\partial_\H(v)(X)=\partial_D(r_{g^{-1}}(v))(\lambda_{\sigma_g}(X)).
\end{eqnarray*}
\end{lemma}

\begin{proof}
This is a consequence of lemma \ref{lemma: delta theta is multiplicative} as follows. Writing $r_{g^{-1}}(v)$ and $\lambda_{\sigma_g}(X)$ as 
\begin{eqnarray*}
r_{g^{-1}}(v)=dm(v,0_{g^{-1}}),\qquad \lambda_{\sigma_g}(X)=dm(\sigma_g(X),di\circ\sigma_g(X)),
\end{eqnarray*}
one has that 
\begin{eqnarray*}\begin{split}
\partial_D(r_{g^{-1}}(v))(\lambda_{\sigma_g}(X))&=c_\H(r_{g^{-1}}(v),\lambda_{\sigma_g}(X))\\&=c_\H(v,\sigma_g(X))+\Ad^\H_{g^{-1}}c_\H(0,di\circ\sigma_g(X))\\&=c_\H(v,\sigma_g(X))=r_{g^{-1}}\partial_\H(v)(X).
\end{split}\end{eqnarray*}
\end{proof}

\begin{proposition}\label{inv-prol}Let $k\geq0$ be a natural number. Then the following statements are equivalent:
\begin{enumerate}
\item $\g^{(k)}(A,D)\subset S^kT^*\otimes (\g)$, the $k$-prolongation of $\partial_D$, is a vector bundle over $M$.
\item $\g^{(k)}(\H)\subset s^*S^kT^*\otimes (\H^s)$, the $k$-prolongation of $\partial_\H$, is a vector bundle over $\G$.
\end{enumerate}
Moreover, for any $g\in \G$ there is an isomorphism of vector spaces
\begin{eqnarray*}
\g^{(k)}(\H)_g\simeq\g^{(k)}(A,D)_{t(g)}.
\end{eqnarray*}
\end{proposition}

\begin{remark}\label{expli}\rm
The isomorphism of \begin{eqnarray*}\g^{(k)}(\H)_g\subset S^kT^*_{s(g)}\otimes \H^s_g\end{eqnarray*} with \begin{eqnarray*}\g^{(k)}(D)_{t(g)}\subset S^kT^*_{t(g)}\otimes \g(D)_{t(g)}\end{eqnarray*} is given on $S^kT^*M$ by the action 
$\lambda_{\sigma_g}:T_{s(g)}M\to T_{t(g)}M$
where $\sigma_g\in J^1_\H\G$ is any element, and on $\H^s$ by right translation
$R_{g^{-1}}:\H^s_h\to \g(D)_{t(g)}$.
\end{remark}

\begin{proof}[Proof of proposition \ref{inv-prol}]
For $k=0$, the statement is true by lemma \ref{s} as $\g_M(\H)=\g(A,D)$, where the isomorphism $t^*\g(A,D)\simeq \g(\H)$ is given by right translation. For $k=1$, let $g\in\G$ and choose any $\sigma_g\in J^1_\H\G$. For a linear map $\varphi:T_{s(g)}M\to \H^s_g$, let $\bar\varphi:T_{t(g)}M\to \g_{t(g)}$ be the unique linear map such that the diagram
\begin{eqnarray*}
\xymatrix{
T_{t(g)}M \ar[r]^{\bar\varphi} & \g_{t(g)} \\
T_{s(g)}M \ar[u]^{\lambda_{\sigma_g}}_{\simeq} \ar[r]_{\varphi} &\H^s_g \ar[u]^{\simeq}_{r_{g^{-1}}}
}
\end{eqnarray*}
commutes. It is clear that this correspondence gives an isomorphism
\begin{eqnarray*}
\hom(T_{s(g)}M,\H^s_g)\simeq\hom(T_{t(g)}M,\g_{t(g)}).
\end{eqnarray*}
Now, $\varphi\in \g^{(1)}(\H)_g$ if and only if $\partial_\H(\varphi(X))(Y)-\partial_\H(\varphi(Y))(X)=0$ for any $X,Y\in T_{s(g)}$M. By lemma \ref{label},
\begin{eqnarray*}\begin{split}
\partial_\H(\varphi(X))(Y)-&\partial_\H(\varphi(Y))(X)=\\&=r_g\partial_D(\bar\varphi(\lambda_{\sigma_g}(X)))(\lambda_{\sigma_g}(Y))-r_g\partial_D(\bar\varphi(\lambda_{\sigma_g}(Y)))(\lambda_{\sigma_g}(X)).
\end{split}
\end{eqnarray*}
Hence, $\g^{(1)}(A,D)_{t(g)}=\g^{(1)}(\H)_g$. Using the isomorphisms $\lambda_{\sigma_g}:T_{s(g)}M\to T_{t(g)}M$ and $R_{g^{-1}}:\H^s_g\to \g_{t(g)}$ and the case $k=1$, the general case follows.
\end{proof}

\begin{corollary}\label{final}
For any $g\in\G$, and any $p,k\geq 0$ integers, there is an isomorphism of vector spaces
\begin{eqnarray*}
H^{(p,k)}(\g(A,D))_{t(g)}\simeq H^{(p,k)}(\g(\H))_g,
\end{eqnarray*}
where $H^{(p,k)}(\g(A,D))$ and $H^{(p,k)}(\g(\H))$ are the cohomology groups of the $\partial_D$ and $\partial_\H$-Spencer cohomology respectively (see definition \ref{exten}).
\end{corollary}

\begin{proof}Let $g\in \G$, and fix $k\geq 1$ an integer. We want to show that there is an isomorphism
\begin{eqnarray}
l_k:\wedge^pT_{t(g)}^*\otimes\g^{(k)}(A,D)_{t(g)}\simeq \wedge^pT_{s(g)}^*\otimes\g^{(k)}(\H)_{g}
\end{eqnarray}
which makes the diagram
\begin{eqnarray}\label{partial-commutes}
\xymatrix{
\wedge^pT_{t(g)}^*\otimes\g^{(k)}(A,D)_{t(g)} \ar[r]^{\partial} \ar[d]_{l_k} & \wedge^{p+1}T_{t(g)}^*\otimes\g^{(k-1)}(A,D)_{t(g)} \ar[d]^{l_{k-1}} \\
\wedge^pT_{s(g)}^*\otimes\g^{(k)}(\H)_{g} \ar[r]^{\partial}  & \wedge^{p+1}T_{s(g)}^*\otimes\g^{(k-1)}(\H)_{g} 
}
\end{eqnarray}
commutative, where $\partial$ is the formal differentiation (see definition \ref{fdo}). We will carry out the proof for $p=0$, the general case can be proven analogously.

Throughout we will be using the exact sequence of vector bundles over $J^k\G$
\begin{eqnarray*}
0\To S^kT^*\otimes T^s \G\overset{i}{\To} TJ^k\G\overset{dpr}{\To} TJ^{k-1}\G\To 0.
\end{eqnarray*}
Choose an element $\sigma_g\in J^1_\H\G$
and let $(b,\xi)\in J^{k+1}\G$ be any element with the property that it projects to $\sigma_g$. Regarding $(b,\xi)$ as a splitting $\xi:T_{s(g)}M\to T_bJ^{k}\G$, then it has the property that  $\xi\subset C_k$ (see proposition \ref{prol-jet}). By example \ref{cartandist}, we have that for an element $v\in S^kT^*_{s(g)}\otimes T^s_g\G=\mathcal{C}_k^s|_g$
\begin{eqnarray*}
\partial(v)(X)=\partial_k(v)(X),
\end{eqnarray*}
where $\partial_k$ is the symbol map of $\mathcal{C}_k$. By lemma \ref{label}, and taking into account that the associated Spencer operator of $\mathcal{C}_k$ is $D^{k\text{-}\clas}$, we have that 
\begin{eqnarray*}
\partial_k(v)(X)=R_{b}\partial_{D^{k\text{-}\clas}}(R_{b^{-1}}v)(\lambda_\xi(X))=R_b\partial(R_{b^{-1}}v)(\lambda_{\sigma_g}(X)),
\end{eqnarray*}
where in the last equality we used that the actions $\lambda_\xi,\lambda_{\sigma_g}:T_{s(g)}M\to T_{t(g)}M$ are equal ($\xi$ projects to $\sigma_g$), and $\partial_{D^{k\text{-}\clas}}=\partial$ (see example \ref{classicalSpencertower}). Now, by lemma \ref{multi}, right translation by $b\in J^{k}\G$ on an element
\begin{eqnarray*}
\Psi\in S^{k-1}T_{t(g)}^*\otimes T^s_{t(g)}\G\subset T_{t(g)}J^{k-1}\G\simeq TJ^k\G/\mathcal{C}_k|_{t(g)}
\end{eqnarray*}
evaluated at $X_1,\ldots,X_{k-1}\in T_{s(g)}M$, is given by
\begin{eqnarray*}
\begin{split}
R_b(\Psi)(X_1,\ldots,X_{k-1})&=R_{pr(b)}(\Psi(\lambda_b^{-1}(X_{1}),\ldots,\lambda_b^{-1}(X_{k-1})))\\
&=R_{g}(\Psi(\lambda_{\sigma_g}^{-1}(X_{1}),\ldots,\lambda_{\sigma_g}^{-1}(X_{k-1}))),
\end{split}
\end{eqnarray*}
where we used again that $b$ projects to $\sigma_g$ and therefore $\lambda_b=\lambda_{\sigma_g}$. Now, for the isomorphism 
\begin{eqnarray}
l_k:\g^{(k)}(A,D)_{t(g)}\simeq \g^{(k)}(\H)_{g}
\end{eqnarray}
we use the one given in remark \ref{expli}, as for 
\begin{eqnarray}
l_{k-1}:T_{t(g)}^*\otimes\g^{(k-1)}(A,D)_{t(g)}\simeq T_{s(g)}^*\otimes\g^{(k-1)}(\H)_{g}
\end{eqnarray} 
we use the action $\lambda_{\sigma_g}^{-1}$ on the $TM$ component and right translation by $b$ on $\g^{(k-1)}(A,D)_{t(g)}$. It is left to the reader to check that the diagram \eqref{partial-commutes} commutes.
\end{proof}

\begin{remark}\rm For a multiplicative one form $\theta\in\Omega^1(\G,t^*E)$, one has analogous versions of the results \ref{label}, \ref{inv-prol} and \ref{final}, where in this case $D:\Gamma(A)\to \Omega^1(M,E)$ is the Spencer operator given in theorem \ref{t1}.
\end{remark}

\begin{proof}[Proof of theorem \ref{workable-pfaffian-groupoids}] A completely analogous proof to that of theorem \ref{workable-pfaffian-bundles} adapts to check that if $P_\H(\G)\to \G$ is surjective, $\g^{(1)}(\H)$ is a vector bundle over $\G$ and $H^{(l,2)}(\g(\H))=0$ for any integer $l\geq0$, then $(\G,\H)$ is formally integrable. Now, proposition \ref{inv-prol} says that $\g^{(1)}(\H)$ is a vector bundle if and only if $\g^{(1)}(A,D)$ is a vector bundle over $M$. Finally, corollary \ref{final} implies that $H^{(l,2)}(\g(\H))_g=H^{(l,2)}(\g(A,D))_{t(g)}$. 
\end{proof}

\section{Results related to the theory developed in this thesis}

In this sections we state and prove some results very much related to the language and notions developed in this thesis. Again we come across the phenomenon occurring in Lie groupoids (under some topological conditions): global information can be recovered from its infinitesimal data.\\

Subsections \ref{involutive multiplicative distributions} and \ref{Contact groupoids} are largely based on the preprint \cite{Maria}. Parts of subsection \ref{Cartan connections} can be found in the same preprint.

\subsection{Integrability theorem for involutive multiplicative distributions; Theorem 5}\label{involutive multiplicative distributions}

Let now $\H$ be a multiplicative distribution on a Lie groupoid $\G$ which is source connected, and consider the
associated symbol space $\B= \H^{s}|_{M}$, representation $E= A/\B$, and the associated Spencer operator
\[ D: \X(M)\times \Gamma(A)\to \Gamma(E).\] 
Let
\[ \partial_{D}: \B\to \textrm{Hom}(TM, E),\]
be the symbol map of $D$. 
Remark that, if $\partial_D= 0$, then $D$ induces a connection
\[\nabla^{E}: \X(M)\times \Gamma(E)\to \Gamma(E) ,\ \nabla^{E}_{X}[\al]= D_X(\al).\]

\begin{theorem}\label{t6}
A multiplicative distribution $\H \subset T\G$ is involutive if and only if the symbol map $\partial_D$ vanishes
and the connection $\nabla^{E}$ on $E$ is flat.
\end{theorem}

\begin{example}\label{ex-flat-Cartan} \rm \ Let $\rho: \mathfrak{h}\to \X(M)$ be an infinitesimal action of a Lie algebra $\mathfrak{h}$ on $M$. Consider the associated Lie
algebroid $\mathfrak{h}\ltimes M$ (see example \ref{inf-act}). In this case
the canonical flat connection
\[ \nabla^{\textrm{flat}}: \X(M)\times C^{\infty}(M, \mathfrak{h})\To C^{\infty}(M, \mathfrak{h})\]
satisfies the conditions from the previous theorem with $E= \mathfrak{h}\ltimes M$, $l= \textrm{Id}$. Hence one obtains a flat involutive $\H$ on the integrating groupoid.
This can be best seen when the infinitesimal action comes from the action of a Lie group $H$ on $M$. Then $\mathfrak{h}\ltimes M$ is the Lie algebroid of the action groupoid 
$H\ltimes M$, which as a manifold is $H\times M$ (see example \ref{ex-action-groupoids}).
The flat involutive $\H$ on $H\times M$ is simply the foliation with the leaves $\{h\}\times M$ (for $h\in H$). See also corollary \ref{flat-Cartan}.
\end{example}

\begin{corollary} If $\G$ is $s$-simply connected then there is a 1-1 correspondence between
\begin{enumerate}
\item[1.] involutive multiplicative distributions $\H$ on $\G$.
\item[2.] a flat vector bundle $(E, \nabla^{E})$ over $M$, a $\nabla^{E}$-parallel tensor $T: \Lambda^2E\to E$ and a surjective 
vector bundle map $l: A\to E$ satisfying 
\[ l([\alpha, \beta])= \nabla^{E}_{\rho(\alpha)}(l(\beta))- \nabla^{E}_{\rho(\beta)}(l(\alpha))+ T(l(\alpha), l(\beta)), \ \ \forall\ \alpha, \beta\in \Gamma(A).\]
\end{enumerate}
\end{corollary} 

\begin{proof} The last equation defines $T$ in terms of $\nabla^{E}$ and $l$; the only problem is whether it is well-defined, but this immediately follows from (\ref{horizontal}). 
The rest follows from the fact that $D$ is determined by $\nabla^{E}$ (a condition that itself implies that $\partial_D= 0)$. Hence one just has to 
rewrite the equation (\ref{horizontal2}) in terms of $\nabla^{E}$ and $T$, and one finds the condition that $T$ is $\nabla^E$-parallel. 
\end{proof}

\subsubsection{Proof of theorem \ref{t6}}

Let $\H \subset T\G$ be a multiplicative distribution, let $(\theta, E)$ be its canonically associated Pfaffian system given in 
lemma \ref{lemma-from-H-to-theta}, and $(D,l)$ its associated Spencer operator given explicitly in Theorem \ref{t2}. Recall that 
\[ \mathfrak{g} = (\H\cap\ker \d s)|_M,\qquad E = A/\mathfrak{g}\]
and that $\partial_D$ denotes the symbol representation
\[ \partial_D: \mathfrak{g}\To \hom(TM, E), \ \ \partial_D(\beta)(X)= D_{X}(\beta).\]

As we have already pointed out, the involutivity of $\H$ is controlled by the bracket modulo $\H$; using $\theta$ to identify $T\G/\H$ with $t^*E$, this is
\[ c_{\H}: \H\times \H \To t^*E, \ c_{\H}(X, Y)= \theta([X, Y]).\]

\begin{lemma} $c_{\H}(\H, \H^s)= 0$ if and only if $\partial_D= 0$.
\end{lemma}

\begin{proof} 
For any $y\in M$, $Y_y\in T_yM$, $\beta_y\in \mathfrak{g}_y$, extending them to sections $Y \in \Gamma(\H)$ and $\beta \in \Gamma(\H^s)$
and using them in the formula for $D$ in theorem \ref{t2}, we have
\[ \partial_D(\beta_y)(X_y)= D_{Y}(\beta)(y)= \theta_{1_y}([\beta, Y])= c^{H}_{y}(\beta_y, Y_y),\]
where, as before, we identify $y$ with $1_y$. 

For arbitrary $\sigma_g\in \Jet^{1}_{\H}\G$ with $s(g)= x$, $t(g)=y$ and $X_x\in T_xM$, $\beta_y\in \mathfrak{g}_y$ we write
\[ \lambda_g(X_x)= (dt)_g(\sigma_g(X_x))= (dm)_{g, g^{-1}}(\sigma_g(X_x), (di)_g\sigma_g(X_x)),\]
\[ \beta_y= (dm)_{g, g^{-1}}(R_g(\beta_y)), 0_{g^{-1}})\]
and we use lemma \ref{lemma: delta theta is multiplicative}:
\[ c^{\H}_{g}(\sigma_g(X_x), R_{g}(\beta_y))= c^{\H}_{y}(\lambda_g(X_x), \beta_y)= \partial_D(\beta_y)(\lambda_g(X_x)).\]
Hence $\partial_D= 0$ if and only if $c^{\H}_{g}(\sigma_g(T_xM), \H^{s}_{g})= 0$ for all $\sigma_g\in \Jet^{1}_{\H}\G$.
Note that for any $\sigma_g$ and any $\xi: T_xM\to \mathfrak{g}_y$ linear,
\[ \sigma_{g}^{\epsilon}(X_x)= \sigma_{g}(X_x)+ \epsilon R_{g}(\xi(X_x))\]
belongs to $\Jet^{1}_{\H}\G$ for $\epsilon$ small  enough, so the last equation also implies that $c^{\H}_{g}(\H_g, \H^{s}_{g})= 0$,
and then the equivalence with $\partial_D= 0$ is clear. 
\end{proof}

Recall again that the cocycle 
\[ c_1:  \Jet^{1}_{\H}\G \To s^{\ast} \hom(\Lambda^2TM, E)\]
is defined by
\[ c_1(\sigma_g)
(X_x, Y_x)= \Ad^{\H}_{g^{-1}} c^{\H}_{g} (\sigma_g(X_x), \sigma_g(Y_{x}))\]
for $\sigma_g\in \Jet^{1}_{\H}\G$ with $s(g)= x$, $X_x, Y_x\in T_xM$. This cocycle, together with $\partial_D$, takes care of the involutivity of $\H$.

The following is now clear: 

\begin{lemma} 
$\H$ is involutive if and only if $\partial_D= 0$ and $c_1= 0$.
\end{lemma}

Note that, under the assumption $\partial_D \equiv 0$, $c(\sigma_g)$ only depends on $g$ and not on the entire splitting $\sigma_g$ at $g$, i.e.
$c_1= \pr^{*}\bar c_1$, the pull-back along the projection $\pr: \Jet^{1}_{\H}\G\to \G$ of the $1$-reduced curvature map
\begin{eqnarray*}\bar c_1: \G \To s^{\ast} \hom(\Lambda^2TM, E),\quad g\mapsto c_1(\sigma_g)\end{eqnarray*}
where $\sigma_g$ is any element in $J^1_\H\G$ which projects to $g$.
See definition \ref{curvature-map1}.
However, even in this case, we will continue to work with $c_1$ because $\G$ does not act canonically on $\hom(\Lambda^2TM, E)$, and solving this problem for $\bar c_1$ requires some work. 

For the following corollary recall that the linearization of $c_1$ is given by the map
\begin{eqnarray*}
\varkappa_D:J^1_DA\To \hom(\wedge^2TM,E)
\end{eqnarray*}
which, under the Spencer decomposition \eqref{Spencer decomposition}, is given at the level of sections $(\al,\omega)\in \Gamma(A)\oplus\Omega^1(M,A)$ by 
\begin{eqnarray}\label{eq:1}
D_X\omega(Y)-D_Y\omega(X)-l\omega[X,Y].
\end{eqnarray}
See proposition \ref{curvatura}.

\begin{corollary} Assume that $\partial_D= 0$ and consider the induced connection $\nabla^{E}$ on $E$ ($\nabla^{E}_{X}(l(\alpha))= D_X(\alpha)$). Then 
\[ \varkappa_D(\alpha, \omega)(X, Y)= \nabla^{E}_{X}\nabla^{E}_{Y}(\alpha)- \nabla^{E}_{Y}\nabla^{E}_{X}\alpha- \nabla^{E}_{[X, Y]}(l(\alpha)),\]
hence $\varkappa_D$ vanishes if and only if $\nabla^{E}$ is flat.
\end{corollary}

\begin{proof} From the formula for $\varkappa_D(\alpha, \omega)(X, Y)$ from equation \eqref{eq:1}, we obtain that
\[\varkappa_D(\alpha, \omega)(X, Y)=  \nabla^{E}_{X}(l\omega(Y))- \nabla^{E}_{Y}(l\omega(X))- l\omega([X, Y]).\]
Using that $l\circ \omega(Z)=  D_Z(\alpha)= \nabla^{E}_{Z}(l(\alpha))$, we obtain the formula from the statement of the corollary.
\end{proof}

The following finishes the proof of theorem \ref{t6}.

\begin{lemma}\label{lemma: passing to algebroid} 
If $\partial_D= 0$ and $\G$ has connected source fibers, then $c_1= 0$ if and only if $\varkappa_D= 0$. 
\end{lemma}

\begin{proof}As mentioned before, 
a cocycle vanishes on the connected component of the $s$-fibers if and only if its linearization vanishes. Now,
the situation is simpler here because, as
as we already remarked, $c_1$  as a section lives already on $\G$: $c_1= \pr^{\ast}(\bar c_1)$. Also, as proven before (see proof of lemma \ref{4217}), for a Lie groupoid map which is also a surjective submersion, the connected components of the $s$-fibers are mapped surjectively to the connected components of the $s$-fibers. But $pr:J^1_\H\G\to \G$ is a surjective submersion by proposition \ref{prop421}, and $\G$ is source connected itself. This implies that $\bar c_1$ vanishes and therefore $c_1$ vanishes itself (!). 
\end{proof}

\subsection{Cartan connections on groupoids}\label{Cartan connections}

The notion of Cartan connections on a Lie groupoid $\G$ arises when looking at the adjoint representation of $\G$ \cite{Abadrep}. They can be seen as the global counterpart of Blaom's cartan algebroids \cite{Blaom}. It is straightforward to see that the
definition from \cite{Abadrep} is equivalent to:

\begin{definition} A Cartan connection on a Lie groupoid $\G$ over $M$ is a multiplicative distribution $\H\subset T\G$
which is complementary to $\textrm{Ker}(ds)$.
\end{definition}

As for any Ehresmann connection, we will denote the inverse of $(ds)|_{\H}$ by 
\[ \textrm{hor}: TM\To \H\subset T\G.\]


On the infinitesimal side, we deal with the Cartan algebroids. These are classical connections
\[ \nabla: \X(M)\times \Gamma(A) \To \Gamma(A) \]
on the vector bundle underlying a Lie algebroid $A$, whose basic curvature $R_\nabla$ vanishes. See subsection \ref{Cartan algebroids}. Theorems \ref{t2} and \ref{t6} give:

\begin{theorem} For any Cartan connection $\H$ on a Lie groupoid $\G$ over $M$, 
\begin{eqnarray}\label{nabla-basic}
\nabla: \X(M)\times \Gamma(A) \To \Gamma(A), \ \nabla_X\alpha(x)= ds([\textrm{hor}(X),\alpha^r]_x) 
\end{eqnarray}
is a Cartan connection on the Lie algebroid $A$ of $\G$.

When $\G$ is $s$-simply connected, this gives a bijection between Cartan connections $\H$ on $\G$
and Cartan connections $\nabla$ on the algebroid $A$. Moreover, $\H$ is involutive if and only if $\nabla$ is flat. 
\end{theorem}

Under topological conditions, the existence of flat Cartan connections implies that the groupoid must come from the action of a Lie group. For instance:

\begin{corollary}\label{flat-Cartan}
If $\G$ is a $s$-simply connected Lie groupoid over a compact $1$-connected manifold $M$ and if $\G$ admits a 
flat Cartan connection $\H$, then $\G$ is isomorphic to an action Lie groupoid $H\ltimes M$ associated to
a Lie group $H$ acting on $M$ (as defined in example \ref{ex-flat-Cartan}).
\end{corollary}

\begin{proof} The flatness of the associated $\nabla$ and the $1$-connectedness of $M$ implies that $A$ is a trivial
bundle: $A= \mathfrak{h}\times M$ for some vector space $\mathfrak{h}$ and the constant sections correspond to flat sections. 
The vanishing of the basic curvature implies that the bracket of constant sections is again constant; hence one has an 
induced Lie algebra structure on $\mathfrak{h}$; the anchor of $A$ becomes an infinitesimal action. Due to the
compactness of $M$, one can integrate this to an action of the $1$-connected Lie group $H$ whose Lie algebra
is $\mathfrak{h}$. Then $H\ltimes M$ and $\G$ are two Lie groupoids with $1$-connected $s$-fibers and with the same Lie algebroid; hence
they are isomorphic. 
\end{proof}

In the same spirit as corollary \ref{flat-Cartan}, we have the following result for Pfaffian groupoids of finite type.

Let $(\G,\H)$ be a Pfaffian groupoid with associated Spencer operator $D:\Gamma(A)\to \Omega^1(M,A)$.

\begin{definition}
We say that the Pfaffian groupoid $(\G,\H)$ is of {\bf finite type} if there exists an integer $k$ such that $\g^{(k)}(A,D)=0$. The smallest $k$ with this property is called the {\bf order of $\H$}. 
\end{definition}

\begin{proposition} Let $(\G,\H)$ be a Pfaffian groupoid of finite type over a compact connected 1-manifold $M$. If the $s$-connected component $(P^k_\H(\G))^0$ is $1$-connected, where $k$ is the order of $\H$, then there is an inclusion of groups
\begin{eqnarray*}
H\subset \Bis(\G,\H),\quad h\mapsto \sigma_h
\end{eqnarray*}
for some Lie group $H$. Moreover, if $\G$ is source connected then for any $g\in \G$, there exists $h\in H$ such that 
\begin{eqnarray*}
\sigma_h(s(g))=g.
\end{eqnarray*}
\end{proposition}

\begin{proof}
Let $(P^k_\H(\G),\H^{k})$ be the $k$-Lie prolongation of $(\G,\H)$. From corollary \ref{corollary:auxiliar2} we have that $$\g(P^k_\H(\G),\H^{k})|_M=\g^{(k)}(A,D)=0.$$
Lemma \ref{isom} implies that $(\H^{(k)})^s=0$ which means that $\H^{(k)}$ is a Cartan connection of $P^k_\H(\G)$. Restricting $\H^{(k)}$ to the open subgroupoid $P^k_\H(\G)^0\subset P_\H^k(\G)$ we have a Cartan connection on a 1-connected groupoid. Moreover, its associated Spencer operator $D^{(k)}$ on $P_D^k(A)$ is flat (see lemma \ref{lemma:auxiliary} and corollary \ref{corollary:auxiliar}), and therefore, by theorem \ref{t6}, $\H^k$ is flat on $P^k_\H(\G)^0$. Applying corollary \ref{flat-Cartan} we have that $P^k_\H(\G)^0$ is isomorphic to an action groupoid $H\ltimes M$ (see example \ref{ex-action-groupoids}). Note that the group of solutions of $(H\ltimes M, \mathcal{F})$, where $\mathcal{F}$ is the trivial foliation on $H\ltimes M$, is canonically identified with $H$. Therefore, under the identification $(H\ltimes M, \mathcal{F})\simeq (P^k_\H(\G)^0,\H^{(k)})$, we get that $H\simeq\Bis(P^k_\H(\G)^0,\H^{(k)})$ as a group. By proposition \ref{groupoid-prop2} 
\begin{eqnarray*}
pr^k_0: H\To \Bis(\G,\H),\quad h\mapsto \sigma_h
\end{eqnarray*}  
is an injective group morphism. Moreover, if $\G$ is source connected, then the image of the submersion $pr^k_0:P^k_\H(\G)\to \G$ is $\G$. With this we get that if $g\in\G$, taking an element on the preimage $(h,(s(g)))\in H\ltimes M\simeq P^k_\H(\G)$, then $\sigma_h(s(g))=g.$
\end{proof}

\subsection{Contact groupoids; Corollary 2}
\label{Contact groupoids}

In analogy with symplectic groupoids and Poisson geometry, contact groupoids are the global counterpart of Jacobi manifolds. Although this has been known 
for a while (see e.g. \cite{Jacobi} and the references therein), the existing approaches have been
rather in-direct (by using ``Poissonization'', applying the similar results from Poisson geometry, then passing to quotients). What happens is that contact groupoids
require the use of non-trivial coefficients; therefore, our main theorem now allows for a direct approach.  Furthermore, using the slightly 
more general setting of Kirillov's local Lie algebras, the approach becomes less computational and more conceptual.

The difference between Jacobi manifolds and local Lie algebras, or the difference between their global counterparts, is completely analogous the
difference between the two related but non-equivalent notions of contact manifolds that one finds in the literature. Here we follow the terminology of \cite{contact}.

\begin{definition} Let $M$ be a manifold. 
\begin{itemize}
\item A \textbf{contact structure} on $M$ is a contact form $\theta$, i.e. a regular $1$-form $\theta\in \Omega^1(M)$ with the property that the restriction of $d\theta$ to the distribution $H_{\theta}= \textrm{Ker}(\theta)$ is
pointwise non-degenerate. 

\item A \textbf{contact structure in the wide sense} on $M$ is a contact hyperfield, i.e. a codimension one distribution $H\subset TM$ 
which is maximally non-integrable. 
\end{itemize}
\end{definition}

Here maximal non-integrability can be understood globally as follows. First, $H$ induces a line bundle
\[ L= TM/H .\]
The maximal non-integrability of $H$ means that $c$ is non-degenerate, where 
\begin{equation}\label{I-H} 
c: H\times H \To L, \ \ (X, Y)\mapsto [X, Y]\ \textrm{mod}\ H
\end{equation}
is the curvature map of $H$ (see definition \ref{curvature-map}). The contact case is obtained when $L$ is the trivial line bundle. 

Passing to groupoids:

\begin{definition} Let $\Sigma$ be a Lie groupoid over $M$. 
\begin{itemize}
\item A \textbf{contact structure} on the groupoid $\Sigma$ is a pair $(\theta, r)$ consisting of a smooth map $r: \Sigma\to \mathbb{R}$ (the Reeb cocycle) and a contact form $\theta\in \Omega^1(\Sigma)$
which is $r$-multiplicative in the sense that 
\[ m^{\ast}\theta= \pr_{2}^{\ast}(e^{-r})\pr_{1}^{\ast}\theta+ \pr_{2}^{\ast}\theta .\]

\item A \textbf{contact structure in the wide sense} on $\Sigma$ is a contact hyperfield $\H$ on $\Sigma$ which is multiplicative.
\end{itemize}
\end{definition}

Regarding the first notion, note that the equation above implies that, indeed, $r$ is a $1$-cocycle; hence it induces a representation $\mathbb{R}_r$ of $\Sigma$ (cf. remark \ref{remark-cocycles}). Also, one has the following immediate but important remark,
which will allow us to reconstruct $\theta$ from associated Spencer operators:

\begin{lemma} 
The $r$-multiplicativity of $\theta$ is equivalent to the fact that $e^{r}\theta\in \Omega^1(\Sigma)$ is multiplicative as a form with values in the representation $\mathbb{R}_r$.
\end{lemma} 

The following indicates the conceptual advantage of the ``wide'' point of view.

\begin{lemma} \label{lemma-contact-versus-wide} Assume that the $s$-fibers of $\Sigma$ are connected. Then the construction $(\Sigma, \theta, r)\mapsto (\Sigma, \textrm{Ker}(\theta))$ induces a 1-1 correspondence between contact groupoids $(\Sigma, \theta, r)$ and contact groupoids in the wide sense $(\Sigma, \H)$ with the property that the associated line bundle is trivial.
\end{lemma}

\begin{proof} It is clear that $\textrm{Ker}(\theta)$ has the desired properties. Conversely, assume that  we start with $(\Sigma, \H)$ so that $L$ is trivial. 
First of all, we know that the multiplicativity of $\H$ makes $L$ into a representation
of $\Sigma$ (cf. subsection \ref{the dual point of view}); we also know that a representation of $\Sigma$ on a trivial line bundle is uniquely determined by a $1$-cocycle (cf. e.g. remark \ref{remark-cocycles}); this
gives rise to the cocycle $r$. Then the canonical projection $T\Sigma \to L$ gives a $1$-form $\overline{\theta}\in \Omega^1(\Sigma)$ which, by proposition \ref{lemma-from-H-to-theta}  is multiplicative as a 
form with coefficients in $\mathbb{R}_r$. Hence, by the previous lemma, $\theta:= e^{-r}\overline{\theta}$ is $r$-multiplicative.
\end{proof}

We now pass to the corresponding infinitesimal structures.

\begin{definition} Let $M$ be a manifold.
\begin{itemize}
\item A \textbf{Jacobi structure} on $M$ is a pair $(\Lambda, R)$ consisting of a 
bivector $\Lambda$ and a vector field $R$ (the Reeb vector field)
satisfying 
\[ [\Lambda, \Lambda]= 2 R\wedge \Lambda,\ \ [\Lambda, R]= 0.\]

\item A \textbf{Jacobi structure in the wide sense} on $M$ is a pair $(L, \{\cdot, \cdot \})$ consisting of a line bundle $L$ over $M$ and a Lie
bracket $\{\cdot, \cdot \}$ on the space of sections $\Gamma(L)$, with the property that it is local, i.e. 
\[ \textrm{sup}(\{u, v\})\subset \textrm{sup}(u)\cap \textrm{sup}(v)\ \ \ \ \forall\ u, v\in \Gamma(L).\]
\end{itemize}
\end{definition}

The second notion appears in the literature under various names. Kirillov introduced them under the notion of local Lie algebra \cite{Kirillov}; Marle uses the term Jacobi bundle \cite{Marle}. Our term ``wide'' is ad-hoc, for compatibility with the previous definitions; however, we will also say that $L$ is a Jacobi bundle. 

For a Jacobi bundle $L$, Kirillov 
proves that $\{\cdot, \cdot\}$ must be a differential operator of order at most one in each argument. When $L= \mathbb{R}_M$ is the trivial line bundle, this implies that the
bracket must be of type 
\[  \{f, g\}_{\Lambda, R}= \Lambda(df, dg)+ \Lie_R(f)g- f\Lie_R(g)\ \ \ (f, g\in \Gamma(\mathbb{R}_M)= C^{\infty}(M))\]
for some bivector $\Lambda$ and vector field $R$. A straightforward check shows that this satisfies the Jacobi identity if and only if $(\Lambda, R)$ 
is a Jacobi structure. Hence, one obtains the following well-known:

\begin{lemma} \label{Jacobi-wide}
$(\Lambda, R)\mapsto (\mathbb{R}_{M}, \{\cdot, \cdot\}_{\Lambda, R})$ defines a bijection between Jacobi structures and local Lie algebras with trivial underlying line bundle. 
\end{lemma}

Next, we sketch the connection between contact groupoids and Jacobi structures pointing out the relevance of the Spencer operator and of theorem \ref{t2}.

\subsubsection*{The Lie functor}
In one direction (the Lie functor), starting with a contact groupoid in the wide sense $(\Sigma, \H)$,
there is an induced Jacobi bundle on $M$. The relevance of the Spencer operator $D$ associated to $\H$ is the following: since $\H$ is contact, it follows that
the vector bundle map associated to $D$ (cf. example \ref{rk-convenient'}),
\[ j_{D}: A\rmap \Jet^1L,\]
is an isomorphism, where $A$ is the Lie algebroid of $\Sigma$ and $L$ is the line bundle associated to $\H$. Identifying $A$ with $\Jet^1L$, we
obtain a Lie bracket $[\cdot, \cdot]$ on $\Jet^1L$. On $\Gamma(L)$ we define the bracket
\[ \{u, v\}:= \pr ([j^1u, j^1v]).\]

\begin{lemma} $(L, \{\cdot, \cdot\})$ is a Jacobi structure in the wide sense.
\end{lemma}

\begin{proof} The bracket is clearly local, hence we are left with proving the Jacobi identity. 
For this it suffices to show that
\[ [j^1u, j^1v]= j^1\{u, v\}\]
for all $u, v\in \Gamma(L)$. Note that, after the identification of $A$ with $\Jet^1L$, the $D$ is identified with the
classical Spencer operator (see example \ref{higher jets on algebroids}); in particular, $D(\xi)= 0$ if and only if $\xi$ is the first jet of a section. 
Fixing $u$ and $v$, the equation (\ref{eq: compatibility-1}) for the Spencer operator implies that $D$ kills $[j^1u, j^1v]$,
hence $[j^1u, j^1v]= j^1s$ for some $s$. Applying $\pr$, we find $s= \{u, v\}$.
\end{proof}

Of course, starting with a contact groupoid $(\Sigma, \theta, r)$, lemmas \ref{lemma-contact-versus-wide} and \ref{Jacobi-wide} ensure the existence of a Jacobi structure $(\Lambda, R)$ on the base.

\subsubsection*{Integrability}
Conversely, start with a Jacobi structure in the wide sense $(L, \{\cdot, \cdot\})$ on $M$. With the Lie functor in mind, the strategy is quite clear:
consider the induced Lie algebroid structure on $\Jet^1L$ with the property that
\[ [j^1u, j^1v]= j^1\{u, v\}\]
for all $u, v\in \Gamma(L)$ and show that the classical Spencer operator $D$ is a Spencer operator with respect to this Lie algebroid structure. Then,
if $\Jet^1L$ comes from a Lie groupoid $\Sigma$, assumed to be $s$-simply connected, integrating $D$ gives the multiplicative hyperfield $\H$ on $\Sigma$ and
the fact that $j_{D}$ is an isomorphism implies that $\H$ is contact. 

For instance, when $(L, \{\cdot, \cdot\})$ comes from a Jacobi structure $(\Lambda, R)$, $\Jet^1L= T^*M\oplus \mathbb{R}$ and, starting from the previous formula, one
finds the Lie algebroid
\[ A_{\Lambda, R}:= T^{\ast}M\oplus \mathbb{R},\]
with anchor $\rho_{\Lambda, R}= \rho$ given by
\[ \rho(\eta, \lambda)= \Lambda^{\sharp}(\eta)+ \lambda R\]
and bracket 
\begin{align}
& [(\eta, 0), (\xi, 0)]_{\Lambda, R}= ([\eta, \xi]_{\Lambda}- i_{R}(\eta\wedge \xi), \Lambda(\eta,\xi)),   \nonumber\\
& [(0, 1), (\xi, 0)]_{\Lambda, R}= (\Lie_{R}(\xi), 0), \nonumber  \\
& [(0, 1), (0, 1)]_{\Lambda, R}= (0, 0)\nonumber
\end{align}
(extended to general elements using bilinearity and the Leibniz identity). The classical Spencer operator becomes 
\[ D: \Gamma(A_{\Lambda, R})\To \Omega^1(M), \ \ D(\eta, f)= \eta+ \d f,\ l(\eta, f)= f  .\]
Of course, checking directly that $D$ is a Spencer operator on $A_{\Lambda, R}$ is rather tedious. The advantage of the ``wide'' point of view is that 
it provides a more compact and computationally free approach. 

So, let's return to our $(L, \{\cdot, \cdot\})$. It is rather unfortunate (and surprising) that the 
definition of the associated Lie algebroid $\Jet^1L$ is missing from the literature. The remaining part of this section is mostly devoted to this point (after that, the 
part with the Spencer operator is immediate). The starting point is the result of Kirillov mentioned above: $\{u, v\}$ must be a differential operator of order at most
one in each argument. To fix notation, recall that a differential operator 
\[ P: \Gamma(E)\To \Gamma(F)\]
of order at most one, where $E$ and $F$ are vector bundles, has a symbol
\[ \sigma_{P}\in \Gamma(TM\otimes \textrm{Hom}(E, F)).\]
Its defining property is
\[ P(fs)= fP(s) + \sigma_{P}(df)(s) \]
for all $s\in \Gamma(E)$, $f\in C^{\infty}(M)$. Of course, when $E= F= L$ is one-dimensional, we get $\sigma_{P}\in \Omega^1(M)$. 
Fixing $u\in \Gamma(L)$, applying this to the operator $\{u, \cdot\}$, we denote the associated symbol by $\rho^1(u)$. This defines a map
\[ \rho^1: \Gamma(L)\To \X(M),\]
characterised by the property that
\[ \{u, fv\}= f\{u, v\}+ \Lie_{\rho^1(u)}(f) v\]
for all $u, v\in \Gamma(L)$. A straightforward computation with the Jacobi identity for $u, fv, w$ combined with the last equations implies that $\rho^1$ is a Lie algebra map:
\[ \rho^1(\{u, v\})= [\rho^1(u), \rho^1(v)].\]
 Unlike the case of Lie algebroids, $\rho^1$ need not be $C^{\infty}(M)$-linear. However, it is a differential operator of order at most one; hence it
satisfies the equation 
\[ \rho^1(fu)= f\rho^1(u)+ \rho^2(df\otimes u) ,\]
where $\rho^2$ is the symbol of $\rho^1$, interpreted as a vector bundle map
\[ \rho^2: \textrm{Hom}(TM, L)\To TM .\]
We define the anchor of $\Jet^1L$ by putting $\rho^1$ and $\rho^2$ together (see (\ref{J-decomposition})):
\[ \rho: \Gamma(\Jet^1L)\cong \Gamma(L)\oplus \Omega^1(M, L)\stackrel{\rho^1-\rho^2}{\To} \X(M).\]
Using the classical Spencer operator, one can write more compactly:
\[ \rho(\xi)= \rho^1(\pr(\xi))- \rho^2(D^{\textrm{clas}}(\xi)).\]
Finally, the Lie bracket for $\Jet^1L$ is, as we wanted, given by
\[ [j^1u, j^1v]:= j^1(\{u, v\}),\]
extended by the Leibniz identity to arbitrary sections (see also remark \ref{curiosity}).

\begin{lemma} 
$(\Jet^1L, [\cdot, \cdot],\rho)$ is a Lie algebroid.
\end{lemma}

\begin{proof} The Leibniz identity holds by construction. The Jacobi identity $$\textrm{Jac}(\xi_1, \xi_2, \xi_3)= 0$$ 
 is clearly satisfied when the $\xi_i$'s are first jets of sections of $L$. Hence it suffices to remark that the expression 
$\textrm{Jac}$ is $C^{\infty}(M)$-linear in all arguments. Using the Leibniz identity, we see that this is equivalent to the fact that the
anchor is a Lie algebra map:
\[ \rho([\xi_1, \xi_2])= [\rho(\xi_1), \rho(\xi_2)].\]
This time, the Leibniz identity implies that the difference  between the two terms is $C^{\infty}(M)$-bilinear, hence it suffices to
check it when $\xi_1= j^1u$, $\xi_2= j^1v$. This is equivalent to $\rho^1$ being a Lie algebra map.
\end{proof}

\begin{lemma} The classical Spencer operator $D^{\textrm{\rm clas}}: \Gamma(\Jet^1L) \To \Omega^1(M, L)$ 
is a Spencer operator on the Lie algebroid $\Jet^1L$. 
\end{lemma}

Of course, the action $\nabla$ of $\Jet^1L$ on $L$ is the one induced by formula (\ref{coneccion}); hence it is characterised by
\[ \nabla_{j^1u}(v)= \{u, v\}. \]

\begin{proof} First note that equation 
(\ref {eq: compatibility-2}) is satisfied (it is $C^{\infty}(M)$-linear in the arguments and, on holonomic sections, it reduces to the previous formula
for $\nabla$). In turn, this implies that formula \eqref{eq: compatibility-1} is $C^{\infty}(M)$-linear in the arguments hence, again, it suffices
to check it on holonomic sections, when it becomes $0= 0$. 
\end{proof}

\begin{remark}\label{curiosity}\rm As a curiosity, note that $\nabla$ and $[\cdot, \cdot]$ can be written on general elements using the
Spencer operator $D= D^{\textrm{clas}}$ as:
\[ \nabla_{\xi}(v)= \{\pr(\xi), v\}+ D(\xi)(\rho^1(v)),\]
\[ [\xi, \eta]= j^1\{\pr \xi, \pr \eta\}+ \Lie_{\xi}(D\eta)- \Lie_{\eta}(D\xi).\]
See \eqref{actions}.
\end{remark}

In particular, we obtain the following integrability result (corollary 4), which should be compared with the one of \cite{Jacobi} (and please compare the proofs as well!).

\begin{corollary} Given a Jacobi structure in the wide sense $(L, \{\cdot, \cdot\})$ over $M$, if the associated Lie algebroid $\Jet^1L$ comes from an
$s$-simply connected Lie groupoid $\Sigma$, then $\Sigma$ carries a contact hyperfield $\H$ making it into a contact groupoid in the wide sense;
$\H$ is uniquely characterized by the fact that the associated Spencer operator coincides with the classical one.
If $(L, \{\cdot, \cdot\})$ comes from a Jacobi structure $(\Lambda, R)$, then we end up with a contact groupoid $(\Sigma, \theta, r)$. 
\end{corollary}

\clearpage \pagestyle{plain}


\bibliographystyle{amsplain}
\def\lllll{}

\pagestyle{fancy}
\fancyhead[CE]{Bibliography} 
\fancyhead[CO]{Bibliography}

\bibliographystyle{abbrv}
\bibliography{bibliography(3)}

\clearpage \pagestyle{plain}

\chapter*{Samenvatting}
\addcontentsline{toc}{chapter}{Samenvatting} \pagestyle{fancy} 
\fancyhead[CE]{Samenvatting} 
\fancyhead[CO]{Samenvatting}

De motivatie voor dit proefschrift komt voort uit de studie van symmetrie\"en van parti\"ele differentiaalvergelijkingen (PDE's). Parti\"ele differentiaalvergelijkingen dienen als modellen voor verscheidene fenomenen in het dagelijks leven, waaronder bijvoorbeeld de voortplanting van licht en geluid in de atmosfeer. In de wiskunde spelen ze ook een essenti\"ele rol in het beschrijven van meetkundige structuren. De Noorse wiskundige Sophus Lie \cite{Lie} begon aan het eind van de 19$^{\text{de}}$ eeuw aan een diepe studie van hun symmetrie\"en, in een poging om de meetkunde van PDE's beter te begrijpen. 

Probeert u zich, om een idee te krijgen van de symmetrie\"en van een ruimte, een fietswiel voor te stellen. Het wiel kan geroteerd worden met een hoek $x$ zonder dat, als we de spaken en andere details negeren, de vorm van het wiel veranderd. Vervolgens kan het wiel nogmaals geroteerd worden met een hoek $y$, of het wiel kon vanaf het begin direct met een hoek $x+y$ geroteerd worden. Alledrie deze rotaties zijn symmetrie\"en van het wiel en we observeren dat twee symmetrie\"en samengesteld kunnen worden om een nieuwe te vormen. In het bijzonder kunnen we roteren met een hoek van $0$ graden, dit wordt de eenheidssymmetrie genoemd. Als het wordt samengesteld met een andere symmetrie verandert dit niets, vergelijkbaar met het vermenigvuldigen met het getal \'e\'en. Tot slot, als we eerst het wiel met een hoek $x$ met de klok mee draaien en daarna met dezelfde hoek tegen de klok in (in andere woorden, een rotatie met de inverse van $x$), dan krijgen we uiteraard weer de eenheidssymmetrie. 

Het werk van Lie lag was het beginsel van de het moderne begrip van een Lie groep; welke een `glad' (ofwel differentieerbaar) concept van symmetrie verwezenlijkt. In het voorbeeld hierboven is de cirkel een Lie groep. De cirkel \emph{werkt} op het wiel: elk punt van de cirkel stelt een hoek van rotatie voor. Verder is de cirkel `glad' in de zin dat het geen hoeken heeft.

In werkelijkheid werkte Sophus Lie met een algemener concept, wat hedendaags een Lie pseudogroep wordt genoemd. Hierbij is het toegestaan dat de symmetrie\"en slechts op een deel van de ruimte werken. Een van zijn grootste ontdekkingen was dat Lie pseudogroepen beter begrepen kunnen worden door ze lineair te benaderen. Dit lineairiseringsproces stelt ons in staat om van complexe differentiaalvergelijkingen naar eenvoudigere lineaire vergelijkingen, \'en weer terug, te gaan. Hij beperkte zich in zijn beschrijving van de infinitesimale data voornamelijk tot een speciaal soort Lie pseudogroepen: die van het eindige type. Zij corresponderen met Lie groepen en hun infinitesimale data staan bekend als Lie algebra's.

Het was echter pas \'Elie Cartan \cite{Cartan1904,Cartan1905,Cartan1937} die in het begin van de 20ste eeuw vooruitgang boekte op het gebied van Lie pseudogroepen van het \emph{oneindige} type. Hij ontdekte dat de infinitesimale data voor een dergelijke Lie pseudogroepen, in analogie met Lie algebras voor Lie pseudogroepen van het eindige type, wordt gegeven door een differentiaal-1-vorm. Deze aanpak leidde tot de ontdekking van differentiaalvormen, welke hedendaags een centrale rol spelen in differentieerbare meetkunde, en tot de theorie van `Exterior Differential Systems' (EDS)\cite{BC}. In het hart van zijn interpretatie van Lie's werk ligt een \emph{jet-bundel} $X$ samen met een zogenaamde Cartan-vorm (ook wel contactvorm genoemd) $\theta$. De infinitesimale data van de originele Lie pseudogroep is vervolgens gecodeerd in een Maurer-Cartan-achtige vergelijking voor $\theta$. Dit brengt ons tot de titel van dit proefschrift: het mooie samenspel tussen $X$ en $\theta$ komt niet zozeer voort uit het feit dat $X$ een jet bundel en $\theta$ zijn Cartan-vorm is, maar uit de compatibiliteit van $\theta$ met de multiplicatieve structuur van $X$. Dit leidt tot de notie van een Pfaffiaanse groepo\"ide: een Lie groepo\"ide $X$ samen een \emph{multiplicatieve} differentiaalvorm $\theta$. Deze abstractie blijkt het meetkundige inzicht achter Cartan's elegante theorie naar boven te brengen.
\newline

In dit proefschrift bestuderen we Pfaffiaanse groepo\"iden vanuit twee equivalente uitgangspunten: de bovengenoemde definitie met een differentiaal-1-vorm en het duale beeld met een deelbundel van de raakbundel (een distributie). U zult zien dat we in alle hoofdstukken (met uitzondering van hoofdstuk 4) veel standaard theorie over de meetkunde van PDE's aandoen (zoals prolongatie en Spencer cohomologie), in een poging om een conceptueler begrip van deze theorie uit te dragen.

In hoofdstuk 1 behandelen we het benodigde inleidende werk, zoals bijvoorbeeld de theorie van jet-bundels van Ehresmann, Spencer cohomologie en Lie groepo\"ides en hun algebro\"ides.

Omdat de theorie in latere hoofdstukken gebaseerd is op de klassieke theorie (waarin wordt gewerkt met een algemene vezelbundel in plaats van een groepo\"ide) is het belangrijk om eerst een goed conceptueel beeld van deze theorie op te bouwen. Zodoende zijn hoofdstukken 2 en 3 toegewijd aan de klassieke theorie, beginnende met lineaire plaatje van relatieve connecties in hoofdstuk 2, om vervolgens naar het globale plaatje in hoofdstuk 3 te gaan. In een zekere zin is het natuurlijker om vanuit het globale af te dalen naar het lineaire. We zijn echter van mening dat de huidige volgorde tot een helderdere uiteenzetting oplevert.

In hoofdstuk 4 doet het multiplicatieve aspect zijn intrede in de vorm van Lie groepo\"ides en hun infinitesimale tegenhangers, de Lie algebro\"ides. Dit hoofdstuk is in een zekere zin onafhankelijk van de rest van het proefschrift; de resultaten in dit hoofdstuk zijn belangrijk voor de volgende hoofdstukken, maar spelen nog geen rol in het voorgaande. Het hoofdresultaat in dit hoofdstuk is de integreerbaarheidsstelling voor multiplicatieve $k$-vormen met co\"effici\"enten. Het stelt dat, onder de aannames die men zou verwachten, we de multiplicatieve $k$-vorm terug kunnen winnen uit zijn infinitesimale data (in de vorm van een $k$-Spencer operator). Hierin is vooral het geval $k=1$ relevant voor de rest van het proefschrift.

Hoofdstukken 5 en 6 vormen de kern van de theorie. De klassieke theorie wordt vereenvoudigd door de multiplicativiteitsconditie voor Pfaffiaanse groepo\"ides en het krijgt zijn ``Lie theoretische'' karakter. In tegenstelling tot hoofdstuk 3, kunnen integreerbaarheidsresultaten (stelling 2) gebruikt worden om de Pfaffiaanse groepo\"ide volledig terug te winnen uit de infinitesimale data. Hiervoor richt hoofdstuk 5 zich eerst op deze infinitesimale data van een Pfaffiaanse groepo\"ides: Spencer operatoren. Zij zijn de natuurlijke interpretaties van relatieve connecties in de wereld van Lie algebro\"iden, in de zin dat ze compatibel zijn met het anker en de Lie-haak. Tot slot behandelen we in hoofdstuk 6 enkele zelfstandige gerelateerde resultaten. Stelling 5 geeft de infinitesimale conditie die Frobenius-involutiviteit voor Pfaffiaanse distributies garandeert. Corollarium 2 beschrijft hoe Jacobi-structuren integreren to contactgroepo\"ides.

\clearpage \pagestyle{plain}

\chapter*{Acknowledgements}
\addcontentsline{toc}{chapter}{Acknowledgements} \pagestyle{fancy} 
\fancyhead[CE]{Acknowledgements} 
\fancyhead[CO]{Acknowledgements}

There are many people that I would like to thank at this point. Some of them were of fundamental help for my academic career, while others made me feel at home in Utrecht.

I would like to start with my advisor Marius Crainic. You, Marius, are truly my mathematical father. I am really grateful for guiding me so closely during my PhD. Thanks to your advice and example I am learning how to be a researcher. Your help on matters beyond my research will leave a permanent mark on my life.

Secondly, I am grateful to Prof. Ieke Moerdijk. With your seminars and walks to the forest, you formed a big nice group of mathematicians, all willing to share and discuss many different and interesting topics.   

Thanks to my dear colleagues and friends at the math department: Ionut, Pedro, Roy, David, Matias, Daniele, Ivan, Jo\~{a}o, Ori, Boris, Florian, Dana, Dave, Dmitry, Sergey, Gil... I am sure I ended up in the best mathematical group thanks to you guys. I learned a lot from you and you became my family in Utrecht. 
Special thanks go to Ivan: I enjoyed working with you a lot, and I learned a lot of interesting mathematics from you. But above all, you together with L\'ea were my adopted parents in Utrecht, and with you both I felt with my own people.
 
I would also like to thank Jean, Helga, Ria, C\'ecile, and Wilke for helping me with many practical issues at the department.

Someone else that I would like to thank is Cate. With you we have created a beautiful friendship, which will prevail regardless of the path that we will take.

Daniele, you made the period of writing pass with the scent of a sweet dream, falling in love over and over again with you ever since. I am very grateful to you because you were the one who kept me in good shape, always watching over me and helping me again and again with the corrections of my thesis.

My parents, Jorge Iv\'an and Lina deserve the greatest of thanks for being there whenever I needed it, and for their unconditional support. I have never told you this but I believe that you are wonderful parents, and that I am the luckiest daughter and sister ever.

\clearpage\pagestyle{plain}
\chapter*{Curriculum Vitae}
\addcontentsline{toc}{chapter}{Curriculum Vitae} \pagestyle{fancy} 
\fancyhead[CE]{Curriculum Vitae} 
\fancyhead[CO]{Curriculum Vitae}

Mar\'ia Amelia Salazar Pinz\'on was born in Manizales, Colombia on the 29th of September 1983. She studied at Colegio San Luis Gonzaga, where she completed her high school in 2001.

In 2002, she moved to Medell\'in, where she read mathematics at Universidad Nacional de Colombia sede Medell\'in as an undergraduate. In 2006, she obtained her bachelor's degree under the supervision of Prof. Carlos Parra.
 
After completing her bachelor's studies she moved to Bogot\'a, where she started a master in mathematics at Universidad de los Andes. She wrote her master thesis under the supervision of Prof. Bernardo Uribe and Prof. Erik Backelin, obtaining her master in 2008.

In September 2008 she moved to the Netherlands and took part in the Master Class program ``Calabi-Yau manifolds'' at Utrecht University. Her Master Class project was supervised by Prof. Marius Crainic.

Since 2009, she has been a Ph.D. student at Utrecht University under the supervision of Prof. Marius Crainic.

\end{document}